\documentclass[12pt]{article}
\usepackage{amsmath,amssymb,amsfonts,amsthm}
\newenvironment{myabstract}{\par\noindent
{\bf Abstract . } \small }
{\par\vskip8pt minus3pt\rm}
\newcounter{item}[section]
\newcounter{kirshr}
\newcounter{kirsha}
\newcounter{kirshb}
\newenvironment{enumroman}{\setcounter{kirshr}{1}
\begin{list}{(\roman{kirshr})}{\usecounter{kirshr}} }{\end{list}}
\newenvironment{enumarab}{\setcounter{kirshb}{1}
\begin{list}{(\arabic{kirshb})}{\usecounter{kirshb}} }{\end{list}}
\newenvironment{athm}[1]{\vskip3mm\par\noindent
{\bf #1 }. \slshape }
{\upshape\par\vskip10pt minus3pt}
\newtheorem{theorem}{Theorem}[section]

\newtheorem{lemma}[theorem]{Lemma}
\newtheorem{corollary}[theorem]{Corollary}
\newenvironment{demo}[1]{\noindent{\bf #1.}\upshape\mdseries}
{\nopagebreak{\hfill\rule{2mm}{2mm}\nopagebreak}\par\normalfont}
\theoremstyle{definition}
\newtheorem{remark}[theorem]{Remark}

\newtheorem{example}[theorem]{Example}
\newtheorem{definition}[theorem]{Definition}
\def\R{\mathbb{R}}

\def\C{{\mathfrak{C}}}
\def\Fm{{\mathfrak{Fm}}}

\def\At{{\bf At}}
\def\Nr{{\mathfrak{Nr}}}
\def\Fr{{\mathfrak{Fr}}}
\def\Sg{{\mathfrak{Sg}}}
\def\Fm{{\mathfrak{Fm}}}
\def\A{{\mathfrak{A}}}
\def\B{{\mathfrak{B}}}
\def\C{{\mathfrak{C}}}
\def\D{{\mathfrak{D}}}
\def\M{{\mathfrak{M}}}
\def\N{{\mathfrak{N}}}
\def\Sn{{\mathfrak{Sn}}}
\def\CA{{\bf CA}}

\def\SC{{\bf SC}}
\def\QEA{{\bf QEA}}

\def\Df{{\bf Df}}

\def\Lf{{\bf Lf}}

\def\PA{{\bf PA}}
\def\PEA{{\bf PEA}}

\def\K{{\bf K}}
\def\K{{\bf K}}

\def\RCA{{\bf RCA}}

\def\Rd{{\ Rd}}
\def\(R)RA{{\bf (R)RA}}
\def\RA{{\bf RA}}

\def\R{\mathbb{R}}

\def\Sc{{\bf Sc}}

\def\Id{{\bf Id}}
\def\c #1{{\cal #1}}
 \def\CA{{\sf CA}}
\def\B{{\sf B}}
\def\G{{\sf G}}
\def\w{{\sf w}}
\def\y{{\sf y}}
\def\g{{\sf g}}

\def\r{{\sf r}}
\def\K{{\sf K}}
\def\tp{{\sf tp}}

 \def\Cm{{\mathfrak{Cm}}}
\def\Nr{{\mathfrak{Nr}}}

\def\restr #1{{\restriction_{#1}}}
\def\cyl#1{{\sf c}_{#1}}
\def\diag#1#2{{\sf d}_{#1#2}}

\def\R{\sf R}
\def\Ra{{\mathfrak{Ra}}}
\def\Ca{{\mathfrak{Ca}}}
\def\set#1{\{#1\} }
\def\Ra{{\mathfrak{Ra}}}
\def\Nr{{\mathfrak{Nr}}}
\def\Tm{{\mathfrak{Tm}}}
\def\A{{\mathfrak{A}}}
\def\B{{\mathfrak{B}}}
\def\C{{\mathfrak{C}}}
\def\D{{\mathfrak{D}}}
\def\E{{\mathfrak{E}}}
\def\A{{\mathfrak{A}}}
\def\B{{\mathfrak{B}}}
\def\C{{\mathfrak{C}}}
\def\D{{\mathfrak{D}}}
\def\E{{\mathfrak{E}}}
\def\U{{\mathfrak{U}}}
\def\GG{{\mathfrak{GG}}}

\def\Bb{{\mathfrak{Bb}}}
\def\L{{\mathfrak{L}}}
\def\Rd{{\mathfrak{Rd}}}
\def\Bb{{\mathfrak{Bb}}}
\def\At{{\mathfrak{At}}}
\def\L{{\mathfrak{L}}}
\def\Bl{{\mathfrak{Bl}}}
\def\CA{{\bf CA}}
\def\RA{{\bf RA}}

\def\RCA{{\bf RCA}}

\def\G{{\bf G}}

\def\CRA{{\sf CRA}}

\def\F{{\mathfrak{F}}}
\def\At{{\sf{At}}}
\def\N{\mathbb{N}}
\def\R{\mathfrak{R}}

\def\CRA{{\sf CRA}}

\def\Cs{{\sf Cs}}
\def\RPEA{{\sf RPEA}}

\def\cyl#1{{\sf c}_{#1}}
\def\diag#1#2{{\sf d}_{#1#2}}

\def\c #1{{\cal #1}}

\def\pa{$\forall$}
\def\pe{$\exists$}

\def\ef{Ehren\-feucht--Fra\"\i ss\'e}

\def\nodes{{\sf nodes}}

\def\restr #1{{\restriction_{#1}}}

\def\Ra{{\mathfrak{Ra}}}
\def\Nr{{\mathfrak{Nr}}}
\def\Z{{\cal Z}}
\def\CA{{\bf CA}}
\def\RCA{{\bf RCA}}
\def\c#1{{\mathcal #1}}

\def\set#1{ \{#1\}}

\def\Ca{{\mathfrak Ca}}

\def\pe{$\exists$}
\def\pa{$\forall$}
\def\Cm{{\mathfrak Cm}}
\def\Sg{{\mathfrak Sg}}
\def\P{{\mathfrak P}}
\def\Rl{{\mathfrak Rl}}
\def\N{{\cal N}}

\def\At{{\sf At}}
\def\Uf{{\sf Uf}}

\def\rng{{\sf rng}}
\def\dom{{\sf dom}}
\def\Cm{{\sf Cm}}
\def\Mat{{\sf Mat}}
\def\w{{\sf w}}
\def\g{{\sf g}}
\def\y{{\sf y}}
\def\r{{\sf r}}
\def\tp{{\sf tp}}
\def\cyl#1{{\sf c}_{#1}}

\def\diag#1#2{{\sf d}_{#1#2}}
\def\RQEA{{\sf RQEA}}
\def\ws{winning strategy}
\def\ef{Ehren\-feucht--Fra\"\i ss\'e}

 \def\CA{{\sf CA}}
\def\Cs{{\sf Cs}}
\def\RCA{{\sf RCA}}
\def\edges{{\sf edges}}
\def\RA{{\sf RA}}
\def\PA{{\sf PA}}
\def\SC{{\sf SC}}
\def\PEA{\sf PEA}
\def\QEA{{\sf QEA}}
\def\TEA{{\sf TEA}}
\def\edges{{\sf edges}}
\def\y{{\sf y}}
\def\g{{\sf g}}
\def\r{{\sf r}}
\def\w{{\sf w}}
\def\Z{{\mathbb{Z}}}
\def\N{{\mathbb{N}}}
\def\EF{{\sf EF}}
\def\CPEA{\sf CPEA}
\def\U{{\mathfrak{U}}}

\def\TA{{\sf TA}}

\def\c{{\sf c}}
\def\s{{\sf s}}

\def\Ig{{\sf Ig}}

\def\d{ Dedekind-MacNeille}
\def\Id{{\sf Id}}
\def\Sc{{\sf Sc}}
\def\Df{{\sf Df}}
\def\Lf{{\sf Lf}}
\title{On Completions, neat reducts and omitting types, yet again}
\author{Tarek Sayed Ahmed and Mohammed Assem\\
Department of Mathematics, Faculty of Science,\\
Cairo University, Giza, Egypt.
  }
\begin{document}
\maketitle
\begin{myabstract} In this paper we investigate using the methodology of algebraic logic,  deep algebraic results
to prove three new omitting types theorems for finite variable fragments of first order logic.
As a sample, we show that it $T$ is an $L_n$ theory and $|T|=\lambda$, $\lambda$ a regular cardinal,
if $T$ admits elimination of quantifiers, then $T$ omits $< 2^{\lambda}$ many non isolated {\it maximal} types.
Modifying a known example,
we show in the countable case omitting $< {}2^{\omega}$ many non isolated types in $L_n$ theories
is independent
of $ZFC$.
Conversely, we show that the omitting types theorem, in its weakest form for countable finite first
order definable extensions of $L_n$ fails; single non isolated types may not be omitted.
This is proved algebraically by using a cylindric algebra in $\RCA_n\cap \Nr_n\CA_{n+k}$
that is not completely representable, that also serves as an example to show that condition of maximality, in the above mentioned
omitting types theorem cannot be omitted. We also show, using a rainbow construction for cylindric algebras,
that the omitting types theorem fails for $L_n$ even if we  consider clique guarded semantics. This is done by constructing a
an atomic $\A\in \PEA_n$ with countably many atoms (which are coloured graphs) who $\Sc$
(Pinter's) reduct is not in $S_c\Nr_n\Sc_{n+3}$, but $\A$ is elementary equivalent to a countable completely representable
(polyadic equality) algebra.
Various connections between the notions of strong representability and complete representability
are given in terms of neat embeddings.  Several examples, using rainbow constructions
and Monk-like algebras are also given to show that our results are best possible.
As a sample we show that, assuming the existence of certain finite relation algebras, that
for any $k\in \omega$, there exists $\A\in {\sf RPEA}_n\cap \Nr_n\PEA_{n+k}$
such that $\Rd_{\sf Sc}\Cm\At\A\notin S\Nr_n\Sc_{n+k+1}$. This implies that for any finite $n\geq 3$,
for any $k\geq 0$, there is an $L_n$ theory and a type $\Gamma$ such that
$\Gamma$ is realized in every $n+k+1$ smooth model, but cannot be isolated by a witness using
$n+k$ variables.
\end{myabstract}

\section{Introduction}

The 'yet again' in the title indicates the fact that the author has a lot of publications on the above topics, witness the references in \cite{1}.
This paper surveys very briefly some of these results, but it also contains many new results.
Among many other things, we fix loose ends in \cite{Sayed} by providing full proofs of statements concerning the omitting types
theorem for variants and modifications of first order logic, and formulating and proving
new ones in clique guarded semantics. We study the notion of neat embedding discussed extensively and thoroughly in \cite{Sayed}
in connection to the two related notions completions and atom canonicity. Rainbow constructions, as well as Monk-like  cylindric algebras
are sprinkled throughout, they will both be used a lot, to prove many of the new results. But let us start from the beginning.
We follow the notation of \cite{HMT1} and \cite{1}.

First, we admit the layout of this paper is rather unusual, in so far as we
include some technical results in the introduction. We chose to do so for (at least) three reasons.
One is that the proofs are very long but the ideas are simple,
so deferring them in the text would considerably increase the length of the paper,
which we tried to keep as short as possible, and even so, we do not think that we succeeded very much in that.
Second, is that they illustrate central concepts elaborated upon in the text, but in different
contexts. Last but not least: semi-technical proofs, usually adopted in survey papers, could be very useful.

So the introduction can be actually viewed as a rather broad survey in the literature, with a unifying
perspective, and the more technical parts are
endowed with more than a glimpse of proof.
In the text itself detailed proofs of all new results, of which there are quite a few,
are given.

The topic of this paper is algebraic logic in a broad sense. A crucial addition to the
collection of mathematical catalysts initiated at the beginning of the 20 century, is formal logic and its study
using mathematical machinery,
better known as  metamathematical investigations, or simply metamathematics.
Traced back to the works of Frege, Hilbert, Russell, Tarski, Godel and others;
one of the branches of pure mathematics that metamathematics  has precipitated to is algebraic logic.

On the other hand, algebraic logic has found new applications in Einstein's general theory of relativity, and the feedback
between the two subjects has transcended all previous anticipations.

Algebraic logic starts from certain special logical considerations, abstracts from them, places them in a general algebraic context
and via this generalization makes contact with other branches of mathematics.
It cannot be overemphasized that algebraic logic is more algebra than logic,
nor more logic than algebra, in this article we argue that algebraic logic has become sufficiently interesting and deep
to occupy a distinguished  place among other branches of pure mathematics. Not confined to the walls of pure mathematics, algebraic logic
has also found bridges to such apparently remote areas like general relativity, and hypercomputation.

Highlighting the connections with graph theory, model theory, set theory, finite combinatorics, topology,  this
article aspires to present topics of broad interest in a
way that is accessible to a large audience. The article is not at all only purely expository, for it contains, in addition, new ideas
and results and also new approaches to old ones.
We prove a diversity of new results, collect old ones that are apparently unrelated, presenting them all
in a unified framework.

\subsection{The repercussions of the seminal result of Monk back in 1969}

A typical Monk argument theme, is to construct a sequence of non-representable algebras based on graphs with finite
chromatic number converging to one that is based on a graph having infinite chromatic number,
hence, representable. (The limit on the level of algebras is the ultraproduct, and that on the graphs it can be a disjoint union, or an ultraproduct as well).
It follows immediately that the variety of representable algebras is not finitely axiomatizable.

Monk's seminal result proved in 1969 \cite{Monk}, showing that the class of representable cylindric algebras is not finitely axiomatizable had a
shattering effect on  algebraic logic, in many respects. In fact it changed the history of the development
of the subject, and inspired immensely fruitful research, satrtinf form Andr\'eka to Venema.

The conclusions drawn from this result, were that either the
extra non-Boolean basic operations of cylindrifiers and diagonal elements were not properly chosen, or that the notion of
representability was inappropriate; for sure it was concrete enough, but perhaps this is precisely the reason, it is far {\it too concrete.}

Research following both paths, either by changing the signature or/and altering the notion of concrete representability have been pursued
ever since, with amazing success.  Indeed there are  two conflicting but complementary facets
of such already extensive  research referred to in the literature, as 'attacking the representation problem'.
One is to delve deeply in investigating the complexity of potential axiomatizations for existing varieties
of representable algebras, the other is to try to sidestep
such wild unruly complex axiomatizations, often referred to as {\it taming methods}.

Those taming methods can either involve passing to (better behaved) expansions of the algebras considered,
or even completely change the signature  bearing in mind that the essential operations like cylindrifiers are
term definable or else change the very  notion of representability involved, as long as it remains concrete enough.

The borderlines are difficult to draw, we might not know what is {\it not} concrete enough, but we can
judge that a given representability notion is satisfactory, once we have one.

One can find well motivated appropriate notions of semantics by first locating them while giving up classical semantical prejudices.
It is hard to give a precise mathematical underpinning to such intuitions. What really counts at the end of the day
is a completeness theorem stating a natural fit between chosen intuitive concrete-enough, but perhaps not {\it excessively} concrete,
semantics and well behaved axiomatizations.
The move of altering semantics (the most famous is Henkin's move changing second order semantics)
has radical philosophical repercussions, taking us away from the conventional
Tarskian semantics captured by G\"odel-like axiomatization; the latter completeness proof is effective
but highly undecidable; and this property is inherited by finite variable fragments
of first order logic as long as we insist on Tarskian semantics. Relativized semantics, has led to many nice modal
fragments of first order logic, like the guarded,
loosely guarded, clique guarded and packed fragments of first order logic. Below we shall deal with another logic that can be witnessed 
as a finite varaible  guarded fragment
of first order logic, and also manifest nice modal behaviour like the finite base ptroperty , and decidability of its satisfibality problem.
In fact, we shall prove that the variety generated by complex algebras of Kripke frames of such logic has a decidable universal, hence
equational, theory

If one views relativized models as the natural semantics for predicate logic rather than some tinkering
devise which is the approach adopted in \cite{HMT1}, then
many well -established taboos of the field must be challenged.

In standard textbooks one learns that predicate logical validity is one unique notion specified once and for all by the usual
Tarskian (square) semantics and canonized by G\"odel's completeness theorem. Moreover, it is essentially complex,
being undecidable by Church's theorem.

Standard predicate logic has arisen historically by making several ad-hoc semantic decisions
that could have gone differently. Its not all about 'one completeness theorem' but rather about several completeness theorems
obtained by varying both the semantic and syntactical parameters.
This can be implemented from a classical relativized representability theory,
like that adopted in the monograph \cite{HMT1}, though such algebras were treated in op.cit off main stream,
and they were only brought back to the front of the scene by the work of Resek, Thompson,
Andr\'eka and last but not least Ferenczi,  or from a modal perspective, that has enriched the subject
considerably, and dissolved the non finite axiomatizability persistent results in algebraic logic initiated by Monk.
Taming such erratic behaviour of the class of representable algebras, has definitely led to nice mathematics,
including a very rich interplay between, combinatorics, graph theory,  modal and first order logic.

\subsection{Negative results?}

Indeed such supposedly 'negative' results themselves are proved using very nice mathematics,
starting from deep neat embedding theorems, to the use of Erdos' probabilistic
graphs, which Hirsch and Hodkinson used to show that the class of strongly
representable atom structure is not elementary.
The idea is to use an {\it anti Monk ultraproduct},
a sequence of good algebras (based on graphs with infinite chromatic number) converging to a bad one, namely, one that
is based on a graph that is
only $2$ colourable.

\subsubsection{The good and the bad}

Here we give a different approach to Hirsch Hodkinson's result; our first technical
innovation.

We use the notation and ideas in \cite{HHbook2}, cf. lemmas 3.6.4, 3.6.6. An important difference is that our cylindric algebras are binary generated,
and they their atom structures are the set of basic marices on relation algebras satisfying the same properties.
Let $\G$ be a graph. One can  define a family of first order structures in the signature $\G\times n$, denote it by $I(\G)$
as follows:
For all $a,b\in M$, there is a unique $p\in \G\times n$, such that
$(a,b)\in p$. If  $M\models (a,i)(x,y)\land (b,j)(y,z)\land (c,l)(x,z)$, then $| \{ i, j, l \}> 1 $, or
$ a, b, c \in \G$ and $\{ a, b, c\} $ has at least one edge
of $\G$.
For any graph $\Gamma$, let $\rho(\Gamma)$ be the atom structure defined from the class of models satisfying the above,
these are maps from $n\to M$, $M\in I(\G)$, endowed with an obvious equivalence relation,
with cylindrifiers and diagonal elements defined as Hirsch and Hodkinson define atom structures from classes of models,
and let $\M(\Gamma)$ be the complex algebra of this atom structure.

We define a relation algebra atom structure $\alpha(\G)$ of the form
$(\{1'\}\cup (\G\times n), R_{1'}, \breve{R}, R_;)$.
The only identity atom is $1'$. All atoms are self converse,
so $\breve{R}=\{(a, a): a \text { an atom }\}.$
The colour of an atom $(a,i)\in \G\times n$ is $i$. The identity $1'$ has no colour. A triple $(a,b,c)$
of atoms in $\alpha(\G)$ is consistent if
$R;(a,b,c)$ holds. Then the consistent triples are $(a,b,c)$ where

\begin{itemize}

\item one of $a,b,c$ is $1'$ and the other two are equal, or

\item none of $a,b,c$ is $1'$ and they do not all have the same colour, or

\item $a=(a', i), b=(b', i)$ and $c=(c', i)$ for some $i<n$ and
$a',b',c'\in \G$, and there exists at least one graph edge
of $G$ in $\{a', b', c'\}$.

\end{itemize}

$\alpha(\G)$ can be checked to be a relation atom structure. Note that some monochromatic triangles
are allowed namely the 'dependent' ones.
This allows the relation algebra to have an $n$ dimensional cylindric basis
and, in fact, the atom structure of $\M(\G)$ is isomorphic (as a cylindric algebra
atom structure) to the atom structure $\Mat_n$ of all $n$-dimensional basic
matrices over the relation algebra atom structure $\alpha(\G)$.

Now we conjecture that  $\alpha(\G)$ is strongly representable iff $\M(\G)$ is representable iff
$\G$ has infinite chromatic number. This will give the result that strongly representable atom structures of both relation algebras
and cylindric algebras of finite dimension $>2$ in one go, using Erdos' graphs of large
chromatic number and girth. The idea is that when $\G$ has an infinite chromatic number then the  shade of red
$\rho$ used to represent Monk and rainbow algebras in theorem in \cite{Hodkinson} and \cite{weak},
will appear in the ultrafilter extension of $\G$
as a reflexive node, if $\G$ has infinite chromatic number, and its $n$ copies,
can be used to completely represent
$\M(\G)^{\sigma}$.

We sketch an idea. Fix $\G$, and let  $\G^* $ be the ultrafilter extension of $\G$.
First define a strong bounded morphism $\Theta$
form $\M(\G)_+$ to $\rho(I(\G^*))$, as follows:
For any $x_0, x_1<n$ and $X\subseteq \G^*\times n$, define the following element
of $\M(\G^*)$:
$$X^{(x_0, x_1)}=\{[f]\in \rho(I(\G^*)): \exists p\in X[M_f\models p(f(x_0),f(x_1))]\}.$$
Let $\mu$ be an ultrafilter in $\M(\G)$. Define $\sim $ on $n$ by $i\sim j$ iff $\sf d_{ij}\in \mu$.
Let $g$ be the projection map from $n$ to $n/\sim$.
Define a $\G^*\times n$ coloured graph with domain $n/\sim$ as follows.
For each $v\in \Gamma^*\times n$
and $x_0, x_1<n$, we let
$$M_{\mu}\models v(g(x_0), g(x_1))\Longleftrightarrow  X^{(x_0, x_1)}\in \mu.$$
Hence, any ultrafilter $\mu\in \M(\G)$ defines $M_{\mu}$ which is  a $\G^*$ structure.
If $\Gamma$ has infinite chromatic number, then $\G^*$ has a reflexive node, and this can be used
to completely represent $\M(\G)^{\sigma}$, hence represent  $\M(\G)$ as follows:
To do this one tries to show  \pe\ has a \ws\ in the $\omega$ rounded atomic game on networks.

Her strategy can be implemented using the following argument.
Let $N$ be a given $\M(\G)^{\sigma}$ network. Let $z\notin N$ and let $y=x[i|z]\in {}^n(N\cup \{z\}={}^nM$.
Write $Y=\{y_0,\ldots y_{n-1}\}$.  We need to complete the labelling of edges of $M$.
We have a fixed $i\in n$. Defines $q_j$ $:j\in n\sim \{i\}$, the unique label of any two distinct elements in $Y\sim y_j$,
if the latter elements are pairwise distinct, and arbitrarily
otherwise.
Let $d\in \G^*$ be a reflexive node in the copy that does not contain any of the $q_j$s (there number is $n-1$),
and define $M\models d(t_0, t_1)$ if $z\in \{t_0, t_1\}\nsubseteq Y$.
Labelling the  other edges are like $N$.

\begin{definition}
\begin{enumarab}

\item A Monks algebra is good if $\chi(\Gamma)=\infty$

\item A Monk's algebra is bad if $\chi(\Gamma)<\infty$

\end{enumarab}
\end{definition}
It is easy to construct a good Monks algebra
as an ultraproduct (limit) of bad Monk algebras. Monk's original algebras can be viewed this way.
The converse is hard.
It took Erdos probabilistic graphs, to get a sequence of good graphs converging to a bad one.
If our conjecture is true, then we get in one go:

\begin{theorem}(Hirsch-Hodkinson) The class of strongly representable atom structures, for any signature between ${\sf Df}$ and ${\sf PEA}$
of finite dimensions $n\geq 3$ is not elementary
\end{theorem}
\begin{proof} This is proved by constructing a sequence of Monk algebras based on probabilistic  graphs of Erdos each having infinite
chromatic number, that converge to a good one, namely a Monk algebra based on the ultraproduct of these graphs,
which is only $2$ colourable

\end{proof}

\subsection{Neat embeddings}

The notion of neat reducts and the related one of neat embeddings is important in algebraic logic for the
very simple reason that is very much related
to the notion of representability, via the so called {\it neat embedding theorem} of Henkin.
In essence, a completeness theorem,
it says that an algebra is representable if and only if it has the neat embedding property.
In symbols, ${\sf RCA}_{\alpha}=S\Nr_{\alpha}\CA_{\alpha+\omega}$ \cite{Sayedneat}.

What Monk actually proved (using a form of Ramsey's theorem) is that all $\omega$ extra dimensions are needed for representability,
$\omega$ cannot be truncated to be finite. Indeed Monk proved for $n\geq 3$ finite that $S\Nr_n\CA_{n+k}\neq \RCA_n$.
If $\A_k\in S\Nr_n\CA_{n+k}\sim \RCA_n$ then any non trivial ultraproduct of the
$\A_k$'s (relative to a non principal ultrafilter on $\omega$)
will be representable, witnessing the non- finite axiomatizability of $\RCA_n$.

The idea of a Monk algebra is not so hard.
Such algebras are finite, hence atomic, more precisely their Boolean reducts are atomic.
The atoms are given colours, and
cylindrifications and diagonals are defined by stating that monochromatic triangles
are inconsistent. If a Monk's algebra has many more atoms than colours,
it follows from Ramsey's Theorem that any representation
of the algebra must contain a monochromatic triangle, so the algebra
is not representable.
The Monk's algebras we use are Monk algebras in
essence sometimes modified by Hirsch and Hodkinson \cite{HH}, relaxing the condition of forbidding all monochromatic triangles,
allowing independent monochromatic triangles
(In both cases all polychromatic triangles are allowed.)

But there is a very obvious problem, when we start thinking of lifting Monk's result to the transfinite, for the simple reason that in this case
we have infinitely many operations, so it down right meaningless to speak about
finite axiomatizability. However, Monk managed to sidestep such difficulty by defining
the notion of a finite schema, which is finitary in a two sorted sense. One sort for the indices, and the other
for the first order situation. Monk defined such {\it varieties definable by a schema} only for infinite
dimensions based, in a precise sense, on the first countable
ordinal $\omega$; here we generalize Monk's notion of a schema to include finite dimensions.

By doing this, Monk overcame the main hurdle, and he
proved the analogous result for infinite dimensions, using his finite dimensional algebras lifted to
the infinite dimensional case. This ingenious lifting process involved
non trivial iterated applications of expansions (to finite extra dimensions) and ultraproducts to 'infinitely' stretch the dimensions
(This will be precisely elaborated upon below).

Below we use the lifting argument due to Monk to show that for infinite ordinals $\alpha$
and $k\geq 1$, the class $S\Nr_{\alpha}\CA_{\alpha+k+1}$ is not
axiomatized over $S\Nr_{\alpha}\CA_{\alpha+k}$ by a finite schema,
our lifting arguments cover all cylindric-like algebras, using Monk Like {\it finite}
algebras constructed by Hirsch and
Hodkinson. This will be proved below, cf. theorem  \ref{2.12}.
In fact we show the following, but first fixing a notation.
$\sf Sc$ denotes Pinter's substitution algebras, where we have only cylindrifiers and substitutions corresponding to
replacements. $\Rd_{\Sc}$ denotes the $\Sc$ reduct of an algebra, where all operations except those of $\Sc$
(if present) are discarded.

\begin{theorem} For $\alpha>2$, for every $k,r\in \omega$,
$k\geq 1$, there is a $\B^r\in S\Nr_{\alpha}\sf QEA_{\alpha+k}$ such that $\Rd_{sc}\B^r\notin S\Nr_{\alpha}\Sc_{\alpha+k+1}$,
but $\prod_r \B^r/F\in \RQEA_{\alpha}$ for any cofinite
ultrafilter on $\omega$.
\end{theorem}

\begin{corollary} For any $\K$ of Pinter's substitution, quasi polyadic with and without equality, cylindric algebras,
for any $\alpha>2$, for any $k\geq 1$,
the class $S\Nr_{\alpha}K_{\alpha+k+1}$ is not axiomatizable  by a finite schema over
$SNr_{\alpha}\K_{\alpha+k}$. Furthermore,  the class $\sf RK_{\alpha}$
is not axiomatizable by any finite schema over
$S\Nr_{\alpha}\K_{\alpha+k}$ for all $k<\omega$.
\end{corollary}

So much for axiomatizations.
In another direction neat embeddings in the infinite dimensional case can be tricky and unpredictable, as far as {\it definability}
is concerned. (Indeed, the theorem to be proved in a while is very closely related to the amalgamation property
{\it not holding} for any class between the representable algebras
and the abstract ones in infinite dimension).

Given $\A\subseteq \Nr_{\alpha}\B$, we call $\B$ a dilation of $\A$.
$\B$ is a {\it minimal dilation} of $\A$, if $A$, as a set, generates $\B$ (using the all the operations
of $\B$.)

{\it It turns out that there are two non isomorphic algebras that generate the same
dilation  in $\omega$ extra dimensions,
and dually minimal dilations of an algebra can be non-isomorphic.}

This was proved in \cite{Sayedneat} theorem 5.2.4, confirming three conjectures of Tarski's, but the presented proof therein is somewhat
convoluted and  it follows a route that is extremely long and winding.
It goes by contradiction, and it does not exhibit the required algebra explicitly.
Besides it depends on a strong deep result of Pigozzi's, namely,
that the variety of representable algebras does not have the amalgamation property, even if the amalgam
is sought in the class of abstract algebras.

Now  we give a direct proof, though the amalgamation property can be discerned below the surface of the proof.
What we show next, is that the variety of representable algebras  algebras fail to have $UNEP$ (unique neat embedding property),
and $NS$ (neat reducts commuting with forming subalgebras)
as defined in \cite {Sayedneat}, definitions 5.2.1, 5.2.2; furthermore
the algebra witnessing this is explicitly described as a quotient of the free algebra with any number of generators
$\geq 4$, so in fact, we are giving infinitely many examples.

In the coming proof, we chose to write $\A^{(X)}$ for the subalgebra generated by $X$; because the operation of forming subalgebras will be iterated,
and this notation becomes easier to follow;
otherwise the notation is that adopted in \cite{Sayedneat}.

But this notation is {\it only used in the next theorem}. Later, the usual notation $\Sg^{\A}X$ (of \cite{1}),
will denote the subalgebra if
$\A$ generated by $X$.

The amalgamation and the superamalgamation
properties are defined in \cite{Sayedneat}. We present our proof for the class of
quasipolyadic equality algebras $\sf QEA$. From the chain of implications in theorem 5.2.4 (which also holds in the $\QEA$ context)
and moving backwards, we get that (ii) (ii) (i) in this
theorem for $\QEA$s are also false.

This immediately gives the required of non uniqueness of dilations, as well as the non
invertibility of the neat reduct functor,
in the following sense.  {\it If we consider the neat reduct operator from $\QEA_{\alpha+\omega}$ to $\RQEA_{\alpha}$
by truncating dimensions,  and defining images of morphisms, taken as injective homomorphisms, to be restrictions,
then the next result implies that this functor is {\it has no right adjoint}.
As proved in \cite{con} the $UNEP$ is equivalent to the existence of
a right adjoint to the neat reduct
functor defined as above. }

We notice that in sharp contrasty the class of polyadic algebras have the superamalgmation property, anf for those the 
neat reduct  functor is strongly invertible. In \cite{con} this rigorous 
mathematical distinction formulated in categorial jargon  is proposed as a measure
to distinguish between the polyadic paradigm and the cylindric one 
mathematically underpinning the metaphysical 'spirit' of each.

\begin{theorem}
For infinite $\alpha$ there exists $\A\subseteq \Nr_{\alpha}\B_i$ $\B_i\in \CA_{\alpha+\omega}$, $i\in \{0,1\}$,
$\A$ generates $\B_i$ but there is no isomorphism from $\B_0$ to $\B_1$ that fixes
$\A$ pointwise; call such an isomorphism an $\A$ isomorphism.
\end{theorem}
\begin{proof}
We use the known \cite{amal}: Let $\bold L=\{\A\in \sf QEA_{\alpha+\omega}: \A=\Sg^{\Nr_{\alpha}\A}\A\}.$
Then $\bold L$ has the superamalgamation property.
\begin{enumroman}
\item Let $\A=\Fr_4{\sf QEA_{\alpha}}$ with $\{x,y,z,w\}$ its free generators. Let $X_1=\{x,y\}$ and $X_2=\{x,z,w\}$.
Let $r, s$ and $t$ be defined as follows:
$$ r = {\sf c}_0(x\cdot {\sf c}_1y)\cdot {\sf c}_0(x\cdot -{\sf c}_1y),$$
$$ s = {\sf c}_0{\sf c}_1({\sf c}_1z\cdot {\sf s}^0_1{\sf c}_1z\cdot -{\sf d}_{01}) + {\sf c}_0(x\cdot -{\sf c}_1z),$$
$$ t = {\sf c}_0{\sf c}_1({\sf c}_1w\cdot {\sf s}^0_1{\sf c}_1w\cdot -{\sf d}_{01}) + {\sf c}_0(x\cdot -{\sf c}_1w),$$
where $ x, y, z, \text { and } w$ are the first four free generators
of $\A$.
Then $r\leq s\cdot t$.

\item Let $\D=\Fr_4\sf RQEA_{\alpha}$ with free generators $\{x', y', z', w'\}$.
Let  $\psi: \A\to \D$ be defined by the extension of the map $t\mapsto t'$, for $t\in \{x,y,x,w\}$.
For $i\in \A$, we denote $\psi(i)\in \D$ by $i'$.
Let $I={\Ig}^{\D^{(X_1)}}\{r'\}$ and $J={\Ig}^{\D^{(X_2)}}\{s'.t'\}$, and let
$$L=I\cap \D^{(X_1\cap X_2)}\text { and }K =J\cap \D^{(X_1\cap X_2)}.$$
Then $L=K$. Let $\A_0=\D^{(X_1\cap X_2)}/L$. Then there exists $\B_1$ and $\B_0$ minimal dilations of $\A_0$,
for which there is no isomorphism between $\B_0$ and $\B_1$ that fixes $\A$ pointwise.
\end{enumroman}
The first part is proved in \cite{amal}.
Let $\A_1=\D^{(X_1)}/I$ and $\A_2=\D^{(X_2)}/J$.
Let $i:\A_0\to \A_1, j: \A_0\to \A_2$ be the natural (injective) embeddings,
and assume for contradiction that any minimal dilations over $\A_0$ are isomorphic by an isomorphism that fixes $\A_0$ pointwise.
Now $\A_0$ embeds into $i(\A_0)\subseteq \A_1$.
The isomorphism can be lifted to $\bar{i}$. Similarly $j$ can be lifted to $\bar{j}.$
We find an amalgam in the big diagram (since $\bold L$ has $AP$), and we return to the original one using the
neat reduct functor. The uniqueness of dilations is used to show that the
isomorphisms $i$ and $j$ lift to $\bar{i}$ and $\bar{j}$.

In more detail, $i:\A_0\to \A_1$ is an embedding, which factors through $i:\A_0\to i(\A_0)$ and $Id:i(\A_0)\to \A_1$ and so is $j:\A_0\to \A_2$
which factors through $\A_0\to i(\A_0)$ and $Id: i(\A_0):\to \A_2$. $\A_1$ neatly embeds in $\A_1^+$ via $e_1:\A\to \Nr_{\alpha}\A^+$.
$i(A_0)$ is a generating subreduct of
$\Sg^{\A^+}(e_1(i(\A_0))).$  By {\it our assumption of uniqueness of dilations},
 $i$ lifts to $\bar{i}:\A_0^+\to \Sg^{\A_1^+}(e_1(i(\A_0))).$
Similarly, $j$ lifts to $\bar{j}:\A_0^+\to \Sg^{\A_2^+}(e_2(i(\A_0))).$
e can assume that $\A_0^+, \A_1^+$ and $\A_2^+$ are in $\bold L$;
indeed no generality is lost if we assume that $\A_0^+=\Sg^{\A_0^+}\A_0$, $\A_1^+=\Sg^{\A_1^+}\A_1$,
and $\A_2^+=\Sg^{\A_2^+}\A_2$. Now $\D^+$ is an amalgam of the outer diagram via $h,k$
and so $\Nr_{\alpha}\D^+$ is an amalgam  via $k\circ e_1$ and $h\circ e_2$.

Now let  $\A=\Fr_4\sf QEA_{\omega}$, and let $\bold R$ be the kernel of the map $\psi$ defined in the hypothesis.
\begin{equation*}\label{p19}
\begin{split}
M={\Ig}^{\A^{(X_2)}}[\{s\cdot t\}\cup (\bold R\cap A^{(X_2)})];
\end{split}
\end{equation*}
\begin{equation*}\label{p20}
\begin{split}
N={\Ig}^{\A^{(X_1)}}[(M\cap A^{(X_1\cap X_2)})\cup (\bold R\cap A^{(X_1)})].
\end{split}
\end{equation*}
Because we have proved that there is an amalgam,
then by exactly the same argument in \cite{amal},
there is an ideal $P$ of $\A$ such that
\begin{equation*}\label{p23}
\begin{split}
P\cap A^{(X_1)}=N,
\end{split}
\end{equation*}
and
\begin{equation*}\label{p24}
\begin{split}
P\cap A^{(X_2)}=M.
\end{split}
\end{equation*}
Thus we can infer that $s\cdot t\in P$ and hence  $r\in P$.
So  $r\in N$. But now
there exist elements
\begin{equation*}\label{p25}
\begin{split}
u\in M\cap A^{(X_1\cap X_2)}
\end{split}
\end{equation*}
and $b\in \bold R$ such that
\begin{equation*}\label{p26}
\begin{split}
r\leq u+b.
\end{split}
\end{equation*}

Since $u\in M$, we get by properties of ideals, that there is a $\Gamma\subseteq_{\omega} \alpha$ and
$c\in \bold R$ such that
$$u\leq {\sf c}_{(\Gamma)}(s\cdot t)+c.$$

Recall the definition of $\psi$; the homomorphism from $\A$ to $\D$ be such that
$\psi(i)=i'$ for $i\in \{x,y,w,z\}$.
We have   $ker \psi=\bold R$, we have $\psi(b)=\psi(c)=0$.
It follows that $$\psi(r)\leq \psi(u) \leq {\sf c}_{(\Gamma)}(\psi(s).\psi(t)).$$
Using the notation in the statement
of the theorem, let $r'=\psi(r)$, $u'=\psi(u)$, $s'=\psi(s)$ and $t'=\psi(t).$
Now we have succeeded to transfer this inequality to the realm of the representable algebras; we have shown
that $r'\leq s'.t'$ can be interpolated in the free representable algebra. Now we define an atomic algebra $\C$,
that we use to show that this interpolant cannot exist.

This part is identical to \cite{amal}.
Let $$\B = ( \wp (^{\alpha}{\alpha}), \cup, \cap, \sim, \emptyset,
{^{\alpha}{\alpha}}, {\sf c}_{i},  {\sf d}_{i,j})_{
i,j\in \alpha}$$
that is $\B$ is the  full quasi polyadic set algebra in the
space $ {^{\alpha}{\alpha}}$. Let $E$ be the set of all equivalence
relations on $\alpha$, and for each $ R \in E $ set
$$ X_R = \{ \varphi : \varphi  \in {^{\alpha}{\alpha}} \,\,\,\
\textrm{and for all} \,\,\,\  \xi, \eta < \alpha , \varphi_\xi =
\varphi_\eta \,\,\,\ \textrm{iff} \,\,\,\  \xi R \eta \}.$$
More succinctly
$$X_R=\{ \varphi \in {^{\alpha}{\alpha}}: ker \varphi=R\}.$$
Let
$$ C = \{ \bigcup_{R \in L} X_R : L \subseteq E \}. $$
$C$ is clearly closed under the formation of arbitrary unions, and
since
$$ \sim \bigcup_{R \in L} X_L =  \bigcup_{R \in E \sim L} X_R$$
for every $ L \subseteq E$, we see that $C$ is closed under the
formation of complements with respect to ${^{\alpha}{\alpha}}$. Thus
$ C $ is a Boolean subuniverse (indeed, a complete Boolean
subuniverse) of $\B$; moreover, it is obvious that
\begin{equation*}\label{p4}
\begin{split}
X_R \,\,\,\ \textrm{is an atom of} \,\,\,\ ( C, \cup, \cap, \sim,
0, {^{\alpha}{\alpha}}) \,\,\,\,\ \textrm{for each} \,\,\,\ R \in E.
\end{split}
\end{equation*}
For all $i,j\in \alpha$ we have ${\sf  d}_{ij} =\bigcup \{ X_R : (i,j)\in R \}$
and hence ${\sf d}_{ij} \in C$. Also for any finite $\Gamma$,
$$ {\sf c}_{(\Gamma)}X_R = \bigcup \{ X_S : S \in E, {}^{2}(\alpha \sim \Gamma) \cap S = {}^{2}(\alpha \sim \Gamma) \cap R \}$$
and $ R \in E$. Thus, because ${\sf c}_{(\Gamma)}$
is completely additive, $C$ is closed under the operation $ {\sf c}_{(\Gamma)}$
Notice that that $C$ is also closed under finite substitutions.
For any such  $\tau$, we have,
$${\sf s}_{\tau}X_R=\bigcup\{X_S: S\in E, \forall i, j<\alpha(iRj \Longleftrightarrow \tau(i)S\tau(j)\}.$$
The set on the right may of course be empty.
Since ${\sf s}_{\tau}$ is also completely additive, therefore, we have shown
that
\begin{equation*}\label{p5}
\begin{split}
C \,\,\,\ \textrm{is a subuniverse of} \,\,\,\ \B .
\end{split}
\end{equation*}
(even if we count in finite substitutions.) It can be shown exactly as in \cite{amal}
that there is a subset $Y$ of ${^{\alpha}{\alpha}}$ such that
\begin{equation*}\label{p6}
\begin{split}
&X_{Id} \cap f(r') \neq 0 \,\,\ \textrm{for
every} \,\,\ f \in Hom (\D, \B) \\
& \textrm{such that} \,\,\ f(x') = X_{Id} \,\,\
\textrm{and} \,\,\ f(y') = Y,
\end{split}
\end{equation*}
and also that for every $ \Gamma \subseteq_{\omega} \alpha$, there are subsets $Z, W$ of $ {^{\alpha}{\alpha}}$ such that
\begin{equation*}\label{p7}
\begin{split}
& X_{Id} \sim {\sf c}_{(\Gamma)}g(s'\cdot t') \neq 0 \,\,\
\textrm{for every} \,\,\ g\in Hom (\D, \B)\\
 & \textrm{such that}
\,\,\ g(x') = X_{Id}, g(z') = Z \,\,\ \textrm{and}
\,\,\ g(w') = W.
\end{split}
\end{equation*}
Here $Hom(\D, \B)$ stands for the set of all homomorphisms from $\D$ to $\B$.
We will show that this in incompatible with the existence of an interpolant.
This will allow us to show that the atom $X_{Id}$ intersects an element in the algebra and its complement non trivially,
which is a contradiction.
Indeed, assume that there exists $ \Gamma \subseteq_{\omega} \alpha$
and an interpolant $u' \in \D^{(x')},$  that is
$$r'\leq u'\leq {\sf c}_{(\Gamma)}(s'\cdot t').$$

Assume also that there also exist $ Y, Z, W \subseteq
{^{\alpha}{\alpha}}$ as indicated above.

Now take any $k \in Hom (\D, \B)$ such that $ k(x') = X_{Id}$,
$ k(y') = Y$, $k(z') = Z$, and $k(w') = W.$
This is possible by the freeness of $\D.$
Then using the fact that $X_{Id} \cap k(r')$ is non-empty,  we get
$$ X_{Id} \cap k(u') = k(x'\cdot u') \supseteq k(x'\cdot r')
\neq 0.$$
And using the fact that $X_{Id} \sim
{\sf c}_{(\Gamma)}k(s'\cdot t')$ is also non-empty,  we get
$$ X_{Id} \sim k(u') = k( x'\cdot -u') \supseteq
k(x'\cdot -{\sf c}_{(\Gamma)}(s'\cdot t')) \neq 0.$$

However,  it is impossible for $X_{Id}$
to intersect both $k(u')$ and its complement since
$ k(u') \in C$ and $X_{Id}$ is an atom;
to see that $k(u')$ is indeed contained in $C$ recall that
$ u' \in \D^{(x')}$, and then observe that because of
the fact that $X_{Id} \in C$ we
must  have $k[\D^{(x')}] \subseteq C.$
\end{proof}
Actually for every $\alpha>0$,
the amalgamation base of $\RQEA_{\alpha}$ coincides with algebras having $UNEP$,
and the superamalgamation base of $\RQEA_{\alpha}$ coincides with those algebras in the amalgamation base that have
$NS$ property. One implication is proved in \cite{Sayed} theorem 5.2.3,
the other is proved in \cite{univ}, but in a more general context, namely,
that of systems of varieties definable by
a schema. We also note that the above proof also shows that cylindric algebras with merry go round identities
does not have the amalgamation property,
answering a question posed by N\'emeti (personel communication).

\subsection{Other, though related, means}

Andr\'eka \cite{Andreka}  generalized some of  the results excluding axiomatizations using finitely many variables
a task done by Jonsson for relation algebras.
In a different direction and using more basic methods, avoiding totally the notion of schema, Andr\'eka invented the method of {\it splitting} \cite{Andreka}
to show that there is no universal axiomatization of
the class of representable cylindric algebras of any infinite dimension containing only finitely many variables.

This is stronger than Monk's result; furthermore, the method proved powerful enough to prove much more refined non-finite
axiomatizability results. Andr\'eka's result was apparently generalized to
quasipolyadic algebras with equality by Sain and Thompson in their seminal
paper \cite{SL}, but there is a
serious error in the proof, that was corrected by the present author, but only for countable dimensions.

Here we generalize this non finite axiomatizability result for any ordinal, possibly uncountable, completing the correction of the
erroneous proof in \cite{ST} to be corrected in a minute.
However, Andr\'eka's splitting does not work for diagonal free algebra; this can indeed be {\it  proved.}
The presence of diagonal elements are crucial in the
argument. Below, we shall give a rainbow construction that can lead to analogous results for diagonal free reducts
of both cylindric and polyadic algebras.

The idea, traced back to Jonsson for relation algebras, consists of constructing for every finite $k\in \omega$
a non-representable algebra, all of whose $k$ -generated subalgebras are representable.

Andr\'eka ingeniously transferred such an idea to cylindric algebras, and to fully implement it, she invented
the {\it nut cracker} method of splitting.
The subtle splitting technique invented by Andr\'eka can be summarized as follows.

Take a fairly simple representable algebra generated by an atom, and  break
up or {\it split} the atom into enough (finitely many)  $k$ atoms,
forming a larger algebra, that is in fact non-representable; in fact, its cylindric reduct will not be representable, due to the  incompatibility
between the number of atoms, and the number of elements in the domain of a representation. However, the ''small" subalgebras
namely, those generated by $k$ elements of such an algebra will be
representable.

This incompatibility is usually witnessed by an application of Ramsey's theorems, like in Monk's first Monk-like algebra,
but this is not always the case (as we will show below)
like  Andr\'eka' splitting. This is definitely an asset, because the proofs in this  case are much  more basic,
and often even proving stronger results.
While Monk and rainbow algebras can prove subtle results concerning non finite axiomatizability for finite
dimensions, splitting works best in the infinite dimensional case.

\subsection{Another technical innovation}

We give two instances of the splitting technique, but with only sketchy proofs.
The first answers a question of Andr\'eka's
formulated on p. 193 of \cite{Andreka} for any infinite ordinal
$\alpha$.

\begin{theorem}\label{complexity} Let $\alpha\geq \omega$. Then the variety $\RQEA_{\alpha}$
cannot be axiomatized with with a set $\Sigma$ of quantifier free
formulas containing finitely many variables. In fact, for any $k<\omega$, and any set of quantifier free formulas
$\Sigma$ axiomatizing $\RQEA_{\alpha}$,
$\Sigma$ contains a formula with more than $k$ variables in which some diagonal
element occurs. In particular, the variety $\RQEA_{\alpha}$ is not axiomatizable over its diagonal free reduct
with a set of universal formulas containing infinitely many variables.
\end{theorem}
\begin{proof}  The theorem is proved only for $\omega$ in \cite{t} (the proof generalizes to any countable ordinal without much ado).
The proof consisted of two parts.
In the first part algebras $\A_{k,n}$  for $k, n\in \omega\sim \{0\}$, with the following properties, where constructed:
$\A_{k,n}$ is of the form $(A_{k,n}, +, \cdot ,-, {\sf c}_i, {\sf s}_{\tau}, {\sf d}_{ij})_{i,j\in \omega, \tau\in G_n}$, where $G_n$ is the symmetric group on $n$
satisfying the following:
\begin{enumarab}
\item $\Rd_{ca}\A_{k,n}\notin \RCA_{\omega}$.
\item Every $k$-generated subalgebra of $\A_{k,n}$ is representable.
\item  There is a one to one mapping $h:\A_{k,n}\to (\B(^{\omega}W), {\sf c}_i, {\sf s}_{\tau}, {\sf d}_{ij})_{i,j<\omega, \tau\in G_n}$
such that $h$ is a homomorphism with respect to all operations of $\A_{k,n}$ except for the
diagonal elements.
\end{enumarab}
Here $k$-generated means generated by $k$ elements.
In the second part  it was shown that  for $n<m$, $\A_{k,n}$ is a subreduct (subalgebra of a reduct)
of $\A_{k,m}$.
Then the directed union $\A_k=\bigcup_{n\in \omega}\A_{k,n}$, was proved to be as required when the signature is countable.

In the above construction, infinitely many finite splittings
(not just one which is done in \cite{ST}), increasing in number but always finite, are implemented
constructing infinitely many algebras, whose similarity types contain only finitely many
substitutions. This is the main novelty occurring in \cite{t} {\it  a modification of Andr\'eka's method of splitting to adapt to
the quasipolyadic equality case.} Such constructed non-representable algebras,
form a chain, and our desired algebra is their
directed union, so that it is basically an $\omega$ step by step construction not just one step construction.

{\it The error in Sain's Thompson paper is claiming that the small subalgebras
of the non-representable algebra, obtained by performing {\it only one splitting into infinitely many atoms}, are representable; this is not necessarily
true at all.}

For finite dimensions, as it happens,  this {\it is true} because the splitting is implemented relative to a {\it finite} set of substitutions
and representability  involves a combinatorial trick depending on {\it counting the number} of substitutions.
This technique no longer holds in the presence of infinitely many substitutions because  we simply cannot count them.
So the splitting have to be done relative to reducts containing only finitely many substitutions and
their desired algebra, witnessing the complexity of
axiomatizations, is the limit
of such reducts, synchronized by the reduct operator.

The idea might be simple, but the  details of implementing the idea turn out to be somewhat intricate; one
has to make sure that certain (basically) combinatorial
properties holding for all finite reducts pass to their limit.

As it turns out, the easy thing to do is to show that ``small" subalgebras of every non-representable
algebra in the chain is representable; the hard thing to do is to show that ``small" subalgebras of the non-representable
limit remain  representable. To extend the result from the least countable ordinal to arbitrary ordinals we use, as to be expected,
transfinite induction.

Now to lift this result to possibly uncountable ordinals, we proceed inductively. Fix finite $k\geq 1$.
For any ordinal $\alpha\geq \omega$, we assume that for all $\kappa\in \alpha$, there is an algebra $\A_{k,\kappa}$
such that $\Rd_{ca}\A_{k,\kappa}\in \RCA_{\alpha}$, and it has all substitution operations corresponding to
transposition from $\kappa$,  namely,
${\sf s}_{[i,j]}$ $i,j\in \kappa$, satisfying all three conditions.

We further assume that  $\beta<\mu<\alpha$, then $\A_{k, \beta}$ is a subreduct of $\A_{k, \mu}$.
The base of the induction is valid.

Now, as in the countable case, take $\A_k=\bigcup_{\mu\in \alpha}\A_{k,\mu}$, and define the operations the natural way; this limit is of the
same similarity type as ${\sf QPEA}_{\alpha}$, and is a well defined algebra, since $\alpha$ is well ordered.
We claim that it is as desired. Clearly it is not representable (for its cylindric reduct is not representable)

Let $|G|\leq k$. Then $G\subseteq \A_{k,\mu}$ for some $\mu\in \alpha$.
We show that $\Sg^{\A_k}G$ has to be representable.
Let $\tau=\sigma$  be valid in the variety
$\RQEA_{\alpha}$. We show that it  is valid in $\Sg^{\A_k}G$.
Let $v_1,\ldots v_k$ be the variables occurring in this equation, and let $b_1,\ldots b_k$ be arbitrary elements of $\Sg^{\A_k}G$.
We show that $\tau(b_1,\ldots b_k)=\sigma(b_1\ldots b_k)$. Now there are terms
$\eta_1\ldots \eta_k$ written up from elements of $G$ such that $b_1=\eta_1\ldots b_k=\eta_k$, then we need to show that
$\tau(\eta_1,\ldots \eta_k)=\sigma(\eta_1, \ldots \eta_k).$
This is another equation written up from elements of $G$, which is also valid in $\RQEA_{\alpha}$.
Let $n$ be an upper bound for the indices occurring in this equation and let $\mu>n$ be such that $G\subseteq \A_{k,\mu}$.
Then the above equation is valid in $\Sg^{\Rd_{\mu}\A}G$ since the latter is representable.
Hence the equation $\tau=\sigma$ holds in $\Sg^{\A_k}G$ at the evaluation $b_1,\ldots b_k$ of variables.

For $\mu\in \alpha$, let  $\Sigma_{\mu}^v$ be the set of universal formulas using only $\mu$
substitutions and $k$ variables valid in $\RQEA_{\omega}$,
and let $\Sigma_{\mu}^d$ be the set of universal formulas
using only $\mu$ substitutions and no diagonal elements valid in $\RQEA_{\omega}$.
By $\mu$ substitutions we understand the set
$\{{\sf s}_{[i,j]}: i,j\in \mu\}.$
Then $\A_{k,\mu}\models \Sigma_{\mu}^v\cup \Sigma_{\mu}^d$.
$\A_{k,\mu}\models \Sigma_{\mu}^v$
because the $k$ generated subalgebras of $\A_{k,\mu}$ are representable,
while $\A_{k,\mu}\models \Sigma_{\mu}^d$ because $\A_{k,{\mu}}$
has a representation that preserves all operations except
for diagonal elements.

Indeed, let $\phi\in \Sigma_{\mu}^d$, then there is a representation of $\A_{k,{\mu}}$ in which all operations
are the natural ones except for the diagonal elements.
This means that (after discarding the diagonal elements) there is a one to one homomorphism
$h:\A^d\to \P^d$ where $\A^d=(A_{k,n}, +, \cdot , {\sf c}_k, {\sf s}_{[i,j]}, {\sf s} _i^j)_{k\in \alpha, i,j\in \mu}\text { and }
\P^d=(\B(^{\alpha}W), {\sf c}_k^W, {\sf s}_{[i,j]}^W, {\sf s}_{[i|j]}^W)_{k\in \alpha, i,j\in \mu},$
for some infinite set $W$.

Now let $\P=(\B(^{\alpha}W), {\sf c}_k^W, {\sf s}_{[i,j]}^W, {\sf s}_{[i|j]}^W, {\sf d}_{kl}^W)_{k,l\in \alpha, i,j\in \mu}.$
Then we have that $\P\models \phi$ because $\phi$ is valid
and so  $\P^d\models \phi$ due to the fact that  no diagonal elements  occur in $\phi$.
Then $\A^d\models \phi$ because $\A^d$ is isomorphic to a subalgebra of $\P^d$ and $\phi$ is quantifier free. Therefore
$\A_{k,\mu}\models \phi$.
Let $$\Sigma^v=\bigcup_{\mu\in \alpha}\Sigma_{\mu}^v
\text { and }\Sigma^d=\bigcup_{\mu\in \alpha}\Sigma_{\mu}^d$$
Hence $\A_k\models \Sigma^v\cup \Sigma^d.$

For if not then there exists a quantifier free  formula $\phi(x_1,\ldots x_m)\in \Sigma^v\cup \Sigma^d$,
and $b_1,\ldots b_m$ such that $\phi[b_1,\ldots b_n]$ does not hold in $\A_k$. We have $b_1\ldots b_m\in \A_{k,\mu}$ for some $\mu\in \alpha$.
Take $\mu$ large enough $\geq i$ so that
$\phi\in \Sigma_{\mu}^v\cup \Sigma_{\mu}^d$.   Then $\A_{k,\mu}$ does not model $\phi$, a contradiction.
Now let $\Sigma$ be  a set of quantifier free formulas axiomatizing  $\RQEA_{\alpha}$, then $\A_k$ does not model $\Sigma$ since $\A_k$ is not
representable, so there exists a formula $\phi\in \Sigma$ such that
$\phi\notin \Sigma^v\cup \Sigma^d.$
Then $\phi$ contains more than $k$ variables and a diagonal constant occurs in $\phi$.

\end{proof}

The next theorem, lifts an unpublished result of Andr\'eka's to the transfinite (this result is mentioned in \cite{Andreka}
referring to an abstract. It is attributed to  Andr\'eka and Tuza for dimension $3$
and to Andr\'eka alone for higher finite dimensions. The infinite version of this result is not addressed
at all in \cite{Andreka}.
\begin{theorem} For any $\alpha\geq \omega$, the class $\sf RQEA_{\alpha}$
is not finitely axiomatizable over its substitutions free reducts,
that is, over $\sf RQA_{\alpha}$.
\end{theorem}
\begin{demo}{Sketch of proof}
This follows by a refinement of the result in \cite{ANS}. We only give a sketch of proof.
Let $n\geq 3$.
Let $Z_0=Z_1=n=\{0,1,2, \ldots n-1\}$ and $Z_i=\{(n-1)i-1, (n-1)i\}$ for $i>1$.
Let $p:\omega\to \omega$ be defined by $p(i)=(n-1)i$. Let
$$V={}^{\omega}\N^{(p)}=\{s\in {}^{\omega}\N: |\{i\in \omega: s_i\neq (n-1)i\}|<\omega\}.$$
$V$ will be the unit of our algebra.
Let $${\sf P}Z=Z_0\times Z_1\times Z_2\ldots \cap V.$$
Let $$t=Z_2\times Z_3\times\ldots \cap V.$$
We split $t$ into two parts, measured by the deviation from $p$:
$$X=\{s\in t: |\{i\in \omega\sim 2: s(i)\neq (n-1)i\}| \text { is even }\},$$
$$Y=\{s\in t: |\{i\in \omega\sim 2: s(i)\neq (n-1)i\}| \text { is odd }\}.$$
Then we define two relations $R, B$ on $n$, such that $domR=domB=n$ and $R\cap B=\emptyset$, for example, we set:
$$R=\{(u,v): u\in n, v=u+1(mod n)\},$$
$$B=\{(u,v): u\in n, v=u+2(mod n)\}.$$

$FT_{\omega}$ denotes the set of all finite transformations on $\omega$.
Let $\pi(\omega)=\{\tau\in FT_{\omega}: \tau \text { is a bijection }\}.$
Next we define {\it the  crucial} relation on $V,$
$$a=\{s\in {\sf P}Z: (s\upharpoonright 2\in R \text { and } s\upharpoonright \omega\sim 2\in X)\text { or }
(s\upharpoonright 2\in B\text { and }s\upharpoonright \omega\sim 2\in Y\}.$$
We also set
$$d=\{s\in {\sf P}Z: s_0=s_1\},$$
and we let $P$ be the permutated versions of $a$ and $d$, that is,
$$P=\{{\sf S}_{\tau}x: \tau \in \pi(\omega),\ \  x\in \{a,d\}\}.$$
Here, and elsewhere, ${\sf S}_{\tau}$ is the usual set-theoretic substitution operation corresponding to $\tau$.

For $W\in {}^{\omega}\rng Z^{(Z)}$, let
$${\sf P}W=\{s\in V: (\forall i\in \omega) s_i\in W_i\}.$$
Let $Eq(\omega)$ be the set of all equivalence relations on $\omega$.
For $E\in Eq(\omega)$, let $e(E)=\{s\in V: ker s=E\}$.
Let
$$T=\{{\sf P}W\cdot e(E): W\in {}^{\omega}\rng Z^{(Z)},  (\forall \delta\in \pi(\omega))W\neq Z\circ \delta, E\in Eq(\omega)\}.$$
Finally, let $$\At =P\cup T,$$
and
$$A_n=\{\bigcup X: X\subseteq \At\}.$$

We have
\begin{enumarab}
\item  $A_n$ is a subuniverse of the full cylindric weak set algebra
$$\langle \wp(V), + ,\cdot,  -, {\sf c}_i, {\sf d}_{ij}\rangle_{i,j\in \omega}.$$
Furthermore $\A_n$ is atomic and $\At \A_n=\At\sim \{0\}$.
\item  $\A_n$ can be expanded to a quasi-polyadic equality algebra $\B_n$ such that $\B_n$ is not representable.
\end{enumarab}
Like  \cite{ANS} Claim 1 undergoing the obvious changes
Like \cite{ANS} Claim 2, also undergoing the obvious changes, in particular the polyadic operations
are defined by:

Let $ \tau, \delta \in FT_{\omega}$. We say that $``\tau,
\delta$ transpose" iff $( \delta0-\delta1).(\tau \delta 0 - \tau
\delta 1)$ is negative.
Let $P'\subseteq P$ be the set of permutated versions of $a$.
Now we first define $ {\sf p}_\sigma : At \rightarrow A$ for every $ \sigma
\in FT_{\omega}$.
\[ {\sf p}_\sigma({\sf S}_\delta a )  =
\begin{cases}
{\sf S}_{\sigma \circ \delta \circ [0,1] } a& \textrm{if}~~~``
\sigma, \delta
~~~\textrm{transpose}" \\
{\sf S}_{\sigma \circ \delta} a & \textrm{otherwise}
\end{cases}\]

$$ {\sf p}_\sigma (x) = {\sf S}_\sigma x \quad \textrm{if} \quad x \in At
\sim P'.$$
Then we set:
$${\sf p}_\sigma (\sum X) = \sum \{ {\sf p}_\sigma (x) : x \in X \} \quad \textrm{for}
\quad X \subseteq At.$$
The defined operations satisfy the polyadic axioms, so that the expanded algebra $\B_n$ with the polyadic operations
is a quasi-polyadic equality algebra that is not representable.

Let $F$ be the cofinite ultrafilter over $\omega,$
then $\prod \B_n/F$ is representable.
Of course the cylindric reduct of the ultraproduct is representable.
The point is to represent the substitutions; particularly those corresponding to transpositions.
The non-representability follows from an incompatibility condition between
the cardinality of $Z_0$ which is $n$ and how the abstract substitutions are defined.
This is expressed by the fact that this definition forces $|a|$ to be strictly greater than $n.$
But when one form the ultraproduct, this 'cardinality incompatibility condition' disappears, $Z_0$ is infinitely
stretched to have cardinality $\omega$. The abstract polyadic operations  coincide with usual
concrete ones. This is the main idea of the proof of this part.
We omit the highly technical details.
\end{demo}

Correcting a serious error in the seminal paper \cite{ST}, answering a question posed by Andreka, and lifting an unpublishred result of hers 
proved for finite dimension
$>2$, witness  \cite{Andreka}, we get:

\begin{corollary} For any infinite ordinal $\alpha$, the class $\sf RPEA_{\alpha}$ is not finitely axiomatizable
over its quasi-polyadic reduct,  and not finite schema axiomatizable over its cylindric reduct.
\end{corollary}

\subsection{Omitting types, algebraically}

The question as to whether the omitting types hold for finite first order
expansions first order logic holds was raised by the author.
Here using a deep construction  of Andr\'eka e all \cite{ANT},
showing that there are countable atomic $k$ neat reducts that are not completely
representable, we answer this question negatively, as well. A $k$ neat reduct $\A$
is an algebra that is a full $n$ neat reduct of an algebra $\B\in \CA_{n+k}$, in symbols,
$\A=\Nr_n\B$.
We show, taking the co atoms of this algebra,
that  even one single type in a countable $L_n$ theory $n\geq 3$ cannot be omitted.

Lately, it has become fashionable to study representations  that preserve infinitary meets and joins.
This is extensively discussed in \cite{Sayed}.
This has affinity with the algebraic notions of complete representations  for cylindric like algebras
and indeed atom-canonicity, a prominent persistence property studied in modal logic.

The typical question is: given an  algebra  and a set of meets, is there a representation  that carries
this set of meets to set theoretic intersections?
(assuming that our semantics is specified by set algebras, with the concrete Boolean operation of intersection
 among
its basic operations)
When the algebra in question is countable, and we have countably many meets;
this is an algebraic version of an  omitting types theorem.
When there it is only one meet consisting of co-atoms, this is complete representability.

In this paper, using the well developed machinery of algebraic logic, we solve many open problems in
that has to do with subneat reducts, and complete
representations. This paper is in fact  interdisciplinary between model theory, algebraic logic and (descriptive) set theory,
though the latter is not very much in the forefront, but can be discerned under the surface
of many of our results.

There two typical types of investigations in set theory, proper.
Both of those will be addressed in this paper, the former only briefly since it was dealt with in some detail in
\cite{Sayed}.
The first type consists of theorems demonstrating the independence of mathematical statements.
This type requires a thorough understanding of mathematics surrounding the
source of problem in question, reducing the ambient mathematical constructions to combinatorial statements about sets,
and finally using some method
(primarily forcing) to show  that such combinatorial statements
are independent.
A second type involves delineating the edges of the independence proofs, giving proof in $ZFC$ of statements
that on first sight would be suspected of being independent.
Proofs of this kind is often extremely subtle and surprising; very similar
statements are independent and it is hard to detect the underlying difference.

The first type was investigated in many publications of Sayed Ahmed,
culminating in independence result concerning the number of types omitted in a
countable language is formulated in \cite{Sayed}, theorem 3.2.8.

{\it But there is a missing part, that  is fully proved here.}

We show that when we consider countable algebras, omitting $< {}^{\omega}2$ types turns
out independent, whereas, if the types are maximal the statement is provable.
Such results will be presented in the context of omitting types in $\L_n$ theories.

{\it We find this discrepancy interesting from the purely set theoretic point of
view when maximality shifts us from the realm of independence to that of
provability from $ZFC$.}

Using the construction in Andr\'eka et all  in \cite{ANT}
we will show that the omitting types theorem fails for finite first
order definable extension of first order logic as defined by Tarski and Givant, and further pursued by others, like
Maddux \cite{Maddux} and Biro\cite{Biro}, a result mentioned in the above cited paper without a proof.

We give full proofs to three results mentioned in \cite{Sayed} that are either declared, or declared with a ketch of proof, referring to a pre-print,
concerning omitting types in uncountable theories using finitely
many variables, namely theorems 3.2.8, theorems 3.8.9, and the the statement mentioned in the last paragraph, all in \cite{Sayed}.
This is the pre print, expanded, modified and polished containing proofs of these results and much more.
The results concerning omitting types, for finite variable fragments, depend on a deep model-theoretic construction of Shelah's.

The ling list of non finite axiomatizability of the class of representable algebras triggered off by Monk \cite{Monk}, 
led to many others refined by other algebraic logicians like Maddux \cite{Maddux} and 
Biro \cite{Biro}. Maddux modified Monk's algebars to make them 
generated by a single two dimensional elements; this is not trivial, far from it, because it make the automophism
group of the algebra rigid.  Biro \cite{Biro} showed that non finitizability results cannot be side stepped by adding finitely many first order definable
operations, he used cylindric algebras
constructed from relation algebras possessing cylindric basis.

A simpler proof is given by Sayed Ahmed in \cite{basim}, using {\it only} cylindric algebras.
This is an asset, in so much as it simplifies Biro's proof considerably, but Biro's result is stronger,
because it says that this happens in languages with only one binary relation (since his cylindric algebras are constructed from
 $n$ dimensional basis of  atom structures structures of relation algebras,
and in this case the relation algebra embeds into the $\Ra$ reduct of the constructed $n$ dimensional algebra
via basic matrices.)

In fact, Biro's constructions depend on involved constructions
of Maddux \cite{Maddux} on relation algebras possessing $n$ dimensional basis.
Maddux 's results are important because they transferred incompleteness results for
finite variable fragments of first order logic from languages with countably many relation symbols each
of arity $n$, where $n$ is the number of variables used, to languages with
only one binary relation. Here $n\geq 3.$
(The result, without imposing any conditions on either the arity of relation symbols or their number,
easily follows from Monk's famous non-finite axiomatizability result for representable cylindric
algebras of dimension $>2$ \cite{Monk}, as indeed pointed out by Sayed Ahmed in \cite{basim})

Now we ask the same question for the omitting types theorem for the corresponding expansions of $L_n$ which is strictly stronger.
We will see that we also get a negative answer, when $n>2$, which is an utterly unsurprising result.
Our construction also uses relation algebras that have an
$n$ dimensional cylindric basis, and like Maddux's algebras, they are generated by one binary relation symbol. Such algebras are constructed
in \cite{ANT}. (In such cases, usually when an algebra satisfying the required is constructed, then this is getting over the hurdle.
Making it one generated uses
well established combinatorial ideas of Maddux, modified to the given context).

We will use a simple Monk-Maddux algebra and we use
a rainbow polyadic algebra, to show that the omitting types theorem fails even if we require that
the countable model omitting a single non principal type in an $L_n$ theory is only locally square, meaning
that witnesses to cylindrifiers can be obtained only on cliques of the model of size limited by $n$;
more over, under certain purely algebraic conditions concerning the existence of certain finite relation algebras,
there can be types (in countable $L_n$ theories) that are realized in every $n+k+1$
smooth model, but are not isolated in $n+k$ variables; this happen for every finite $k$.
In the first case we blow up and blur the Monk Maddux finite algebra.

\subsection{Splitting and blowing  up and blurring a Maddux algebra}

The splitting method due to Andr\'eka has several subtle sophisticated re-incarnations in the literature, \cite{Sayed}, \cite{ANT},
like the blow up and blur constructions in \cite{ANT}
\cite{IGPL}, \cite{MLQ}.
In the latter two references a finite atom structure is split twice to given non elementary equivalent algebras, one a neat reduct
the other is not; for the infinite dimensional case the atom structure has the same cardinality as the dimension; it is weak set algebra whose
base is the dimension.
Witness also \cite{HHbook}
(section on completions) in the  context of splitting atoms in finite rainbow algebras.
The next theorem is the main theorem in \cite{ANT}; it is also mentioned in \cite{Sayed}.

We decided to include it early on in the introduction, with a sketch, because in a large portion of the paper
we will have occasion to deal with the repercussions of
such subtle, deep construction, laying bare its potential consequences, from constructing Monk like algebras witnessing non finite axiomatizability
results for both relation and cylindric algebras,
to proving non atom canonicity of
varieties of subneat reducts that approximate the class of representable
algebras. Besides it is a very nice construction.

Follows is theorem 1.1 in op.cit; for undefined notions the reader is referred to \cite{ANT}.

\begin{theorem}\label{OTT}
Suppose that $n$ is a finite ordinal with $n>2$ and $k\geq 0$.
There is a countable
representable
relation algebra ${\R}$
such that
\begin{enumroman}
\item Its completion, i.e. the complex algebra of its atom structure is
not representable, so $\R$ is representable but not completely representable
\item $\R$ is generated by a single element.
\item The (countable) set $\B_n{\R}$ of all $n$ by $n$ basic matrices over $\R$
constitutes an $n$-dimensional cylindric basis.
Thus $\B_n{\R}$ is a cylindric atom structure
and the full complex algebra $\Cm(\B_n{\R})$
with universe the power set of $\B_n{\R}$
is an $n$-dimensional cylindric algebra
\item The {\it term algebra} over the atom structure
$\B_n{\R}$, which is the countable subalgebra of $\Cm(\B_n{\R})$
generated by the countable set of
$n$ by $n$ basic matrices, $\C=\Tm(B_n \R)$ for short,
is a countable representable $\CA_n$, but $\Cm(\B_n)$ is not representable.
\item Hence $\C$ is a simple, atomic representable but not completely representable $\CA_n$
\item $\C$ is generated by a single $2$ dimensional element $g$, the relation algebraic reduct
of $\C$ does not have a complete representation and is also generated by $g$ as a relation algebra, and
$\C$ is a neat reduct of some simple representable $\D\in \CA_{n+k}$
such that the relation algebraic reducts of $\C$ and $\D$
coincide.
\end{enumroman}
\end{theorem}
\begin{demo}{Sketch} \cite{ANT}. Here we give an idea of the proof which is a blow up and blur construction.
For the technical details one is referred to the original paper \cite{ANT}, or to the sketch in \cite{Sayed}.
Below we will return to this proof, and discuss its modifying to solve a long standing open problem in algebraic logic.
One starts with a finite Maddux relation algebra $\bold M$, that cannot be represented on finite sets.
Then this algebra is blown up and blurred. It is blown up by splitting the atoms each to infinitely many.
It is blurred by using a finite set of blurs or colours $J$. This can be expressed by the product $\At=\omega\times \At \bold M\times J$,
which will define an infinite atom structure of a new
relation algebra. One can view such a product as a ternary matrix with $\omega$ rows, and for each fixed $n\in \omega$,  we have the rectangle
$\At \bold M\times J$.
Then two partitions are defined on $\At$, call them $P$ and $E$.
Composition is defined on this new infinite atom structure; it is induced by the composition in $\bold M$, and a ternary relation $e$
on $\omega$, that synchronizes which three rectangles sitting on the $i,j,k$ $e$ related rows compose like the original algebra $\bold M$.
This relation is definable in the first order structure $(\omega, <)$.

The first partition $P$ is used to show that $\bold M$ embeds in the complex algebra of this new atom structure, so
the latter cannot be represented, because it can only be represented on infinite sets.

The second partition divides $\At$ into $\omega$ sided finitely many rectangles, each with base $W\in J$,
and the the term algebra over $\At$, are the sets that intersect co-finitely with every member of this partition.
On the level of the term algebra $\bold M$ is blurred, so that the embedding of the small algebra into
the complex algebra via taking infinite joins, do not exist in the term algebra for only finite and cofinite joins exist
in the term algebra.

The term algebra is representable using the finite number of blurs. These correspond to non-principal ultrafilters
in the Boolean reduct, which are necessary to
represent this term algebra, for the principal ultrafilter alone would give a complete representation,
hence a representation of the complex algebra and this is impossible.
Thereby an  an atom structure that is weakly representable but not strongly representable is obtained.

This atom structure has an $n$- dimensional cylindric basis, and so the $n$
basic matrices form an atom structure that is also only weakly representable.
The resulting $n$ dimensional cylindric term algebra obtained is a $k$ neat reduct that is not completely
representable. To make the algebra  one generated one uses Maddux's combinatorial techniques,
and this entails using infinitely many ternary relations.

\end{demo}

\subsection{Rainbow constructions}

One can blur up and blow a rainbow algebra or a Monk-like algebra (both finite) to get {\it representable
term algebras} that have desirable required properties.
However, as we show in the paper splitting atoms of a rainbow algebra usually gives more
delicate and refined results; because basically one can control \pa\ moves in the the \ef\ game one devises,
that is, one has control  on the number of pebbles on the board.

The rainbow construction in algebraic logic is invented by Hirsch and Hodkinson.
The rainbow construction reduces finding subtle differences between seemingly
related notions or concepts using a very simple \ef\ forth pebble game between two players \pa\ and \pe\ on two very simple structures.
From those structures a relation or cylindric algebra can be constructed and a \ws\ for either player lifts to a \ws\ on the atom structure
of the algebra, though the number of pebbles and rounds increases in the algebra.

In the case of relation algebras, the atoms are coloured, so that the games are played on colours. For cylindric algebras,
matters are a little bit more complicated
because games are played on so-called coloured graphs, which are models of the rainbow signature coded in an
$L_{\omega_1,\omega}$ theory. The atom structure consists of finite
coloured graphs rather than colours.

Nevertheless, the essence of the two construction is very similar, because in the cylindric
algebra constructed from $A$ and $B$, the relation algebra atom structure based also on $A$ and $B$
is coded in the cylindric atom structure, but the latter has additional shades of yellow
that are used to label $n-1$ hyperedges coding the cylindric information.

The strategy for \pe\ in a rainbow game is try black, if it doesn't work try white, and finally if it doesn't work
try red. In the latter case she is kind of cornered, so it is the hardest part in the game.
She never uses green.

In the cylindric algebra case, the most difficult part
for \pe\ is to label the edge between apexes of two cones (a cone is a special coloured graph) having
a common base, when she is also forced a red.

So in both cases  the choice of a red, when she is forced one,
 is the most difficult part for \pa\ and if she succeeds in every round then she wins.

Indeed,  it is always the case that \pa\ wins on a red clique,
using his greens to force \pa\  play an inconsistent triple of reds.
In case of cylindric algebra \pa\ bombards \pe\ with cones having green tints, based on the same base.

For Monk-like algebras, this number (the number of pebbles) is a Ramsey large uncontrollable number;
though it can be sometimes controlled, in the availability  of amalgamation moves, like in the case
of the proof used to solve the famous neat embedding problem, initiated by Monk and formulated
as  problem 2.12 in \cite{HH}. Therefore, there is this 'semi rational feeling in the air' among algebraic logicians that what can be done
with a Monk-like algebra can also be done by a rainbow algebra, but there is no meta theorem to this effect;
only a lot of evidence supporting this partially rational feeling and no evidence for
its contrary.

\subsection{Splitting up and bluring a rainbow algebra}

Now we split up and blur a finite {\it rainbow} relation algebra.
emphasizing on the affinity with
the hitherto discussed blow up and blur construction of Andr\'eka et all \cite{ANT}.
We follow the notation in \cite{HHbook}, cf lemmas 17,32, 34, 35, 36.

\begin{theorem} Let $\R=\A_{K_m, K_n}$, $m>n>2$.
Let $T$ be the algebra obtained by splitting the reds. Then $T$ has exactly two blurs
$\delta$ and $\rho$. $\rho$ is a flexible non principal ultrafilter consisting of reds with distinct indices and $\delta$ is the reds with common indices.
Furthermore, $T$ is representable, but $\Cm\At T\notin  S\Ra\CA_{m+2}$, in particular, it is not representable.
\end{theorem}

\begin{proof} \pe\ has a \ws\ over $\At\R$ using in $m+2$ rounds,
hence $\R\notin \RA_{m+2}$, hence is not
in $S\Ra\CA_{m+2}$.  $\Cm\At T$ is also not in the latter class
$\R$ embeds into it, by mapping ever red
to the join of its copies. Let $D=\{r_{ll}^n: n<\omega, l\in n\}$, and $R=\{r_{lm}^n, l,m\in n, l\neq m\}$.
If $X\subseteq R$, then $X\in T$ if and only if $X$ is finite or cofinite in $R$ and same for subsets of $D$.
Let $\delta=\{X\in T: X\cap D \text { is cofinite in $D$}\}$,
and $\rho=\{X\in T: X\cap R\text { is cofinite in $R$}\}$.
Then these are {\it the} non principal ultrafilters, they are the blurs and they are enough
to represent $T$ as follows. Let $\Delta$ be the graph $n\times \omega\cup m\times {\omega}$.
Let $\B$ be the full rainbow algebras over $\At\A_{K_m, \Delta}$ by
deleting all red atoms $r_{ij}$ where $i,j$ are
in different connected components of $\Delta$.

Obviously \pe\ has a \ws\ in $\EF_{\omega}^{\omega}(K_m, K_m)$, and so it has a \ws\ in
$G_{\omega}^{\omega}(\A_{K_m, K_m})$.
But $\At\A_{K_m, K_m}\subseteq \At\B\subseteq \At_{K_m, \Delta}$, and so $\B$ is representable.
Then one  defines a bounded morphism from $\At\B$ to the ultrafilters of $T$ namely $T^+$. The blurs are images of elements from
$K_m\times \{\omega\}$, by mapping the red with equal double index,
to $\delta$, for distinct indices to $\rho$.
The first copy is  reserved to define the rest of the red atoms the obvious way.
The other atoms are the same in both structures.

Let $|I|=n$ be the set of atoms and $J$ be the set of blurs. Define $\Delta$ as above $\omega$ copies of atoms, and one copy of the blurs;
they are disjoint.
This graph codes the principal ultrafilters in the first component, and the non principal ones in the second.
Let $\B$ be as defined above, just by replacing$m$ by $J$. Each atom of $\B$ (principal ultrafilter)  are mapped to the copies of the red which form
a principal ultrafilter. The elements of $J\times \{\omega\}$ are mapped to the corresponding non principal
ultrafilters, each $W$ to $\Uf^W$. This defines a bounded morphism, for $\B$ to the canonical extension of $\Tm$.
\end{proof}

This construction, as clearly indicated in the above proof,
 has a lot of affinity to the Andr\'eka N\'emeti blow up and blur construction.
The major difference, in fact the {\it only} significant difference, is that in the latter a Maddux finite algebra is used, while in the former case
a finite rainbow algebra is used.

In the latter case, we can only infer that the complex algebra is non representable,
in the former case we can know and indeed we can prove more. The reason basically is  that non-representability
of the Maddux algebra on finite sets, depends on an uncontrollable big Ramsey number (that is a function in the dimension),
while for  rainbow algebras we can control {\it when the algebra stops to be representable} by \pa\ s moves.
\pa\ forces a win by using greens, it is precisely this number, that  determines the extra dimensions in which the complex
algebras stop to be neatly embeddable
into, it is precisely the point at which it outfits the reds.

What can be done here is substitute a Maddux  finite relation algebra used by Andr\'eka and N\'emeti, by the rainbow algebra mentioned used
by Hirsch and Hodkinson and using the arguments of Andr\'eka and N\'emeti, to get a stronger result.
Or maybe we can lift Hirsch and  Hodkinsons
construction from relation algebras whose atoms are colours to cylindric algebras
whose atoms are coloured graphs, $n$ dimensional basis will not do here,
because the $n$ complex blurs are {\it not enough} in the sense of \cite{ANT}.

The relation algebra obtained by Hirsch and Hodkinson  does not have an $n$ dimensional cylindric basis, except for $n=3$,
and using this it was proved that only for the lowest value of $n$ namely $n=3$, the class $S\Nr_3\CA_k$, $k\geq 6$
is not closed under completions.
A substitute of $n$ dimensional cylindric basis that serves in this context is surjections from $n$
to coloured graphs used in rainbow construction of cylindric algebras.

We have two relation algebras, a Maddux one, and a rainbow one, that do almost the same thing  at least
they provide weakly representable
atom structures that are not strongly representable. We want an $n$ dimensional cylindric algebra.

\subsection{The desired cylindric algebra, blown up and blurred}

This is done in \cite{can} by blurring up and blowing a finite rainbow cylindric algebra namely $R(\Gamma)$
where $\Gamma=n+1$, and the greens
are  ${\sf G}=\{\g_i:i<n-1\}
\cup \{\g_0^{i}: i\in n+2\}.$

Let $\At$ be the rainbow atom structure in \cite{Hodkinson} except that we have $n+2$ greens and
$n+1$ reds.
The rainbow signature now consists of $\g_i: i<n-1$, $\g_0^i\i\in n+2$, $\r_{kl}^t: k,l\in n+1$, $t\in \omega$,
binary relations and $\y_S$ $S\subseteq \Z$,
$S$ finite and a shade of red $\rho$; the latter is outside the rainbow signature,
but it labels coloured graphs during the game, and in fact \pe\ can win the $\omega$ rounded game
and build the $n$ homogeneous model $M$ by using $\rho$ when
she is forced a red.

Then $\Tm\At$ is representable; this can be proved exactly as in \cite{Hodkinson}.
The atoms of $\Tm\At$ are coloured graphs whose edges are not labelled by
the one shade of red  $\rho$; it can also be viewed as a set algebra based on $M$
by relativizing semantics discarding assignments whose edges are labelled
by $\rho$. A coloured graph (an atom) in $\CA_{n+2, n+1}$
is one such that at least one of its edges is labelled red.
Now $\CA_{n+2, n+1}$ embeds into $\Cm\At\A$,
by taking every red graph to the join of its copies, which exists because $\Cm\At$ is complete
(these joins do not exist in the (not complete) term algebra; only joins of finite or cofinitely many reds do, hence it serves non representability.)
A copy of a red graph is one that is isomorphic to this graph, modulo removing superscripts of reds.
Another way to do this is to take every coloured graph to the interpretation of an infinite disjunct of the $MCA$ formulas
(as defined in \cite{Hodkinson}), and to be dealt with below; such formulas define coloured graphs whose edges are not labelled by the shade of red,
hence the atoms, corresponding
to its copies, in the relativized semantics; this defines an embedding,  because $\Cm\At$ is isomorphic to
the set algebra based on the same relativized semantics
using $L_{\infty,\omega}^n$ formulas in the rainbow signature. Here again $M$ is the $n$ homogeneous model constructed
in the new rainbow signature, though the construction is the
same.
But \pa\ can win a certain finite rounded game on $\CA_{n+1, n+2},$ hence it is
outside $S\Nr_n\CA_{n+4}$ and so is $\Cm\At$,  because the former is embeddable in the latter
and $S\Nr_n\CA_{n+4}$ is a variety; in particular, it is closed
under forming subalgebras.

\subsection{Multi modal logic when modalities do not commute}

Finding modal fragments of first order logic is an old problem in modal correspondence theory.
The aim is to find fragments of first order
logic that contain the standard translations of various modal logics and share their nice properties
for the same reasons.

Clique guarded semantics reflect that the 'undesirable properties' (like commutativity of cylindrifiers)
are only witnessed locally on  finite
cliques of the model or representation. However, below we show that the omitting types theorem fails
for finite variable fragments of first order logic with $n\geq 3$,
even if we ask only for `$n$ guarded models' omitting single non principle ones in countable languages. There may not be one.

We know how to build atom structures,  or indeed frames, from atomic algebras,
and conversely subalgebras of complex algebras from frames.
We are happy when we are able to preserve crucial logical properties.

But something seems to be missing. Modal logicians rarely study frames in isolation,
rather they are interested in constructing new frames from old ones using bounded morphisms, generated subframes,
disjoint unions, zig-zag products (the latter is a less familiar notion).

An algebraic logician adopts an analogous perspective but on
different (algebraic) level, via such constructions as
homomorphisms, subalgebras and direct products. So it appears that modal logicians work in a universe
that is distant from that of the algebraic logicians. In this connection,
it is absolutely natural to ask whether these universes are perhaps systematically related. And the answer is:
they are, very much so, and {\it duality theory} is devoted to studying such links.
For example a representation theorem for algebras is the dual of representing abstract state frames by what
Venema calls assignment frames.

The most prominent citizens of first order logic are the quantifiers, whose meaning is defined as follows:
$$\M\models \exists v_i\phi[s]\Longleftrightarrow \exists d\in M: \M\models \phi[s_i^d].$$
Here $s_i^d$ is the sequence which agrees with $s$ except at $i$ where its value is $d$.
There is an obvious modal view on this definition. Let $A={}^{\omega}M$
be the set of all assignments, and define a set of equivalence relations $\equiv_i$ on $A\times A$ by $s\equiv_i t$ iff $s(j)=t(j)$ for all $j\neq i$.
Moving the assignment to the front, we get a familiar ``modal" pattern:
$$\M,s\models \exists v_i \phi\Longleftrightarrow (\exists t\in A)(t\equiv_i s \land \M, t\models \phi.)$$
According to the Tarskian tradition, $A$ is the set of all evaluations of the variables into $\M$.
Here we treat elements of $t\in A$ as states or possible worlds in a certain Kripke frame $(A, \equiv_i, \ldots)$. The various logical connectives, like
$\phi\mapsto \exists x\phi$ will re-appear here as modalities whose accessibility relations (the relations $\equiv_i$) on $A$ will determine their meanings.
So we can view first order logic as one particular instance of the {\it modal logic of assignments.}
Relativized semantics can be obtained by looking at subsets of $A$.
It is harmless and sometimes even desirable to look at only some of the states and discard the rest.

Let $n$ be finite. Consider a frame $\F=(W, T_i, P_{ij}, E_{ij})_{i,j<n}$ of polyadic type.
(That is both $T_i$ and $P_{ij}$ are binary relations and $E_{ij}$ is unary).

Then we can form the complex algebra of $\F$ which we denote abusing notation slightly by $\wp(\F)$.
$\wp(\F)$ is of the same similarity type as $\PEA_n$.
If $\K$ is a class of frames,
then $\K^+=\{\A: \A\cong \wp(\F) \text { for some } \F\in \K\}.$ Let $\K$ be a class of frames,
let $\F\in \K$ and let ${\sf Term}(X)$ be the set of terms generated from
a countable set of variables $X$ (in the language of $\PEA_n$).
Let $v$ be a function from $X$ to $\wp(\F)$. We call $\M=(\F, v)$ a Kripke model over $\F$.

A truth relation $\models$ can be defined by recursion as follows:
For $s\in F$ we define
$$\M, s\models x\Longleftrightarrow   s\in v(x).$$
The booleans are as expected and the extra non boolean operations, the cylindrifiers, say, are defined as follows:
$$\M,s\models {\sf c}_i\tau\Longleftrightarrow  (\exists t\in F)sT_it\text { and } \M, t\models \tau.$$
The definition of the polyadic operations the ${\sf p}_{ij}$s is entirely analogous.
A {\it concrete} frame is of the form $(V, \equiv_i, P_{ij}, D_{ij})$ where $V\subseteq {}^{\alpha}U$ for some set $U$
and for $s,t\in V$ we have
$$s\equiv_it\Longleftrightarrow  s(j)=t(j)\ \  \forall j\neq i$$
$$s{P}_{ij}t \Longleftrightarrow s\circ [i,j]=t$$
$${D}_{ij}=\{s\in V: s_i=s_j\}.$$
$\K_{\sf cube}$ is the class of frames whose domain is of the form
$^{\alpha}U$  and $\K$ is the class of arbitrary concrete frames.

But then there is a whole landscape of
classes of frames between $\K$ and $\K_{\sf cube}$
obtained by imposing extra conditions on $V$ the domain of frames in question.

Let $L$ be such a class where $V$ is not necessarily a cartesian square.
The logic corresponding to $SPL^+$ has syntax like first order logic, and it can be viewed as
multi-modal propositional logic enriched with
constants. In particular, we can look at quantifiers $\exists v_i$
as if they are modal operators $\Diamond_i$ whose meaning is given by the relation $\equiv_i$.

The semantics in a frame with domain
$V$ are {\it relativized} to {\it the states in $V$.}
For example
$$\M,s\models \Diamond_i \phi\text { iff } (\exists t\in V)(t\equiv_i s \land \M,t\models \phi.)$$
Such logics, from which the guarded fragment arose, behave nicer than $n$ variable fragments of
ordinary first order logic
in many respects (particularly concerning decidability and completeness), witness Van Benthem's article in \cite{1}.

Below we shall show that a very interesting representation theorem proved by Ferenczi \cite{Fer}
for such Kripke frames, gives very nice
multi dimensional modal logics which has cylindrifiers viewed as diamonds
and a very natural relativized semantics.

Such logics are strong in expressive power, but exhibit very nice modal behaviour,
like the loosely guarded and packed fragments of $L_n$.
They have the finite base property and also the universal, hence equational
theory, of the class of subdirect products of complex algebras, the representable algebras for short,
are decidable \ref{Khaled}.

However, we will also show that the {\it relativization to $m$ clique guarded semantics},
when $m\geq n+3$, $n$ is the finite dimension $>2$,
negative properties persist to hold. The commutativity of cylindrifiers even when only locally witnessed
in $<n+3$ cliques, is already  harmful. Full fledged commutativity of cylindrifiers is not much worse.

\section*{On the notation}

We follow basically \cite{1} which is in conformity with the notation adopted
in the monographs \cite{HMT1}, \cite{HMT2}.
We deviate from the notation in {\it op.cit} when we feel it is compelling to do
so. For example, we write $f\upharpoonright X$, for the restriction of a function $f$ to a set $X$, rather
than the other way round.

For games we basically follow Hirsch Hodkinson's article \cite{HHbook2}, but not always.
We write $G_k(\A)$ for $G^k(\A)$, or simply $G_k$, is the $k$ rounded atomic game played on atomic networks of a cylindric algebra $\A$.
For rainbow cylindric algebras, we also deviate from the notation therein,  we find it more convenient
to write $\CA_{\G, \Gamma}$ - treating the greens $\sf G$, as a parameter that can vary -
for the rainbow algebra $R(\Gamma)$.
The latter is defined to be $\Cm(\rho(\K))$ where $\K$ is a class of models in the rainbow signature satisfying the
$L_{\omega_1,\omega}$ rainbow theory, and $\rho(\K)$ is the rainbow atom structure.

Our view is the conventional (more restrictive) one adopted in \cite{HH} and \cite{Hodkinson}, we view these models as coloured graphs,
that is complete graphs labelled by the rainbow colours.
As usual,  we interpret `an edge colored by a green say', to mean that the pair defining
the edge satisfies the corresponding green binary relation.
Our atom structures  will consist of surjections from $n$ (the dimension) to finite coloured graphs,
roughly finite coloured graphs.
When $k=\omega$, we simply write $G$.
$\PEA$ denotes polyadic equality algebras, $\PA$ denotes polyadic algebras, $\Sc$ denotes Pinter's algebras, and of course
$\CA$ denotes cylindric algebras  \cite{1}.
$\PEA_n$ denotes $\PEA$s of dimension $n$ and same for all other algebras.
The following information is folklore: $\CA$s and $\PA$s are proper reducts of $\PEA$s and $\Sc$s
are proper  reducts of all. Finally, following the usual notation again,
$\Df$ denotes diagonal free $\CA$s, and these are
proper reducts of all of the above.

For infinite dimension, $\sf QEA$ will denote the class quasipolaydic equality algebras,
these are substantially different than $\PA$s, they ahve only finite cylindrifiers and
substitutions. For an infinite ordinal $\alpha$, $\sf QEA_{\alpha}$ denotes the $\sf QEA$s
of dimension $\alpha$.

For an atomic algebra $\A$, we write $\At\A$ for its atom strcture.
For an atom structure $\At$, we write $\Cm\At$ for its complex algebra. In particular, for an atomic algebra $\A$
that is completely additive $\Cm\At\A$ is
its completion.
For neat reducts and neat embeddings we follow \cite{Sayedneat},
for complete representations and omitting types we follow \cite{Sayed}.

\section*{Layout}

After this magical long and winding tour in the introduction, there is still much more to say.
The paper is divided into two parts. In a nutshell part 1 deals with various forms of the omitting types theorems, 
and the second part, deals with neat embeddings and games played 
on rainbow atom structure, not in isolation, but ultimately linking up with omitting types.

In section 2 we study the omitting types theorem for $L_n$, proving several statements mentioned in \cite{Sayed} without proof
In the course of our investigation we propose a solution to a problem
of  Hirsch's in \cite{r} on relation algebras.
We show that the omitting types
theorem fails for any finite  first order
definable expansion of $L_n$ and we complete an independence proof (partially given in \cite{Sayed}).
We approach the  the omitting types theorem for various multi dimensional modal logics
$L_n$, including equality free ones, and cylindrfiier free
ones.
We study Vaught's theorem (on existence of atomic models for atomic theories)
in fragments of $L_n$ proving positive results after weakening
commutativity of cylindrifier. We also address the infinite dimensional case studying well behaved fragments of Keisler's logic.

In section 3, we use many rainbow constructions, to show that various classes
that consist of algebras having a neat embedding property (meaning that
they neatly embed into algebras in larger dimension in special way) are not elementary, and if such classes are varieties
then they may  not be closed under \d\ completions.
We use both our rainbow constructions, and a refined blow up and blur construction,
we show that the omitting types theorem fails even in the clique guarded semantics
(that is if we insist on commutativity of cylindrifiers only locally).

Several examples are sprinkled throughout to show that our results are best possible.
As a sample, we give an example to show that a maximality
condition in an omitting types theorem of Shelah's restricted to
$L_n$ is necessary. The example is constructed from a
relation algebra having an $\omega$ dimensional
hyperbasis.

\section*{Part One}

\section{Fixing loose ends, omitting types in $L_n$}

Here we prove several statements declared in \cite{Sayed} concerning omitting types, without proof, namely theorems 3.2.8, 3.2.9,
and the statement in the last
paragraph. We quote {\it we have an uncountable simple cylindric algebras in $\Nr_n\CA_{\omega}$
that is not completely representable.} Here we show that we indeed
have one.

But first we briefly discuss some algebraic repercussions of the above construction proving the main result in \cite{ANT}.
Based on the he above blow up and blur construction, we start by showing how such a construction is flexible;
indeed we re prove Monk's seminal result.  There are several parameters used to define the relation algebra above.
Let $l\in \omega$, $l\geq 2$, and let $\mu$ be a non-zero cardinal. Let $I$ be a finite set,
$|I|\geq 3l.$ Let
$$J=\{(X,n): X\subseteq I, |X|=l,n<\mu\}.$$
$I$ is the atoms of $\bold M$. $J$ is the set of blurs.

Pending on $l$ and $\mu$, let us call these atom structures ${\cal F}(l,\mu).$
In the example referred to above the atoms of $M$ are $I$, $J\subseteq \wp(I)$
consisting of $2$ element subsets of $I$
so it is just  ${\cal F}(2,1)$,

If $\mu\geq \omega$, then $J$ would be infinite,
and $\Uf$, the set of non principal ultrafilters corresponding to the blurs, will be a proper subset of the ultrafilters.
It is not difficult to show that if $l\geq \omega$
(and we relax the condition that $I$ be finite), then
$\Cm{\cal F}(l,\mu)$ is completely representable,
and if $l<\omega$ then $\Cm{\cal F}(l,\mu)$ is not representable.

Blurs can be witnesses as colours.
So in the former case we have infinitely many colours, so that the chromatic number
of the graph is infinite, while in the second case the chromatic number is finite.
Informally, if the blurs get arbitrarily large, then in the limit, the resulting algebra will be completely representable, and so its complex algebra
will be representable. If we take, a sequence of blurs, each finite,
we get a sequence of Monk (non-representable) algebras whose
limit is completely representable; this gets us back to Monk's seminal
results on non finite axiomatizability of the classes of representable relation algebras
and representable cylindric algebras of dimension $>2$.

Formally:
\begin{corollary}
\begin{enumarab}
\item  The classes ${\sf RRA}$ is not finitely axiomatizable.
\item  The elementary
closure of the class ${\sf CRA}$ is not finitely axiomatizable.
\end{enumarab}
\end{corollary}
\begin{demo}{Proof}
Let ${\cal D}$ be a non-
trivial ultraproduct of the atom structures ${\cal F}(i,1)$, $i\in \omega$. Then $\Cm{\cal D}$
is completely representable.
Thus $\Tm{\cal F}(i,1)$ are ${\sf RRA}$'s
without a complete representation while their ultraproduct has a complete representation.
Also the sequence of complex algebras $\Cm{\cal F}(i,1)$, $i\in \omega$
consists of algebras that are non-representable with a completely representable ultraproduct.
\end{demo}
Note that because our algebras posses $n$ dimensional cylindric basis, the result lifts easily to cylindric algebras.
Actually Monk's original proof did not go this way. Monk proved his result for relation algebras first in 1963,
and the analogous result for cylindric algebras took him 6 other years, establishing the first tie between
Ramsey's theorem and non representability result, introducing graph theory in
algebraic logic \cite{Monk}. The interplay between algebraic logic and graph theory and more generally combinatorics is now huge.
It was Maddux \cite{Maddux} who, much later though,
constructed such non representable cylindric algebras
from relation algebras possessing $n$ dimensional cylindric basis;
they were also generated by a single element.

For finite variable fragments $\L_n$ for $n\geq 3$, the situation for omitting types turns out to be drastically different than  first order logic using
all $\omega$ variables.

But first a definition.

\begin{definition}
Assume that $T\subseteq \L_n$. We say that $T$ is $n$ complete iff for all sentences $\phi\in \L_n$
we have either $T\models \phi$ or $T\models \neg \phi$. We say that $T$
is $n$ atomic iff  there is $\psi\in \L_n$ such that  for all $\eta\in \L_n$ either $T\models \psi\to \eta$
or $T\models \psi\to \neg \eta.$
\end{definition}

The next theorem \ref{finite} is proved using algebraic result in \cite{ANT} mentioned above.
\begin{theorem}\label{finite} Assume that $\L$ is a countable first order language containing a binary relation symbol. For $n>2$ and
$k\geq 0$,  there are a consistent $n$ complete
and $n$ atomic theory $T$ using only $n$ variables, and a set $\Gamma(x_1)$ using only $3$ variables
(and only one free variable) such that $\Gamma$ is realized in all models of $T$ but each $T$-witness for $T$ uses
more that $n+k$ variables
\end{theorem}
\begin{demo}{Proof} \cite{ANT}, \cite{Sayed}.
\end{demo}

Next, we recall a tremendously deep result of Shelah, that we use to prove theorem 3.2.9 in \cite{Sayed}.

\begin{lemma} Suppose that $T$ is a theory,
$|T|=\lambda$, $\lambda$ regular, then there exist models $\M_i: i<{}^{\lambda}2$, each of cardinality $\lambda$,
such that if $i(1)\neq i(2)< \chi$, $\bar{a}_{i(l)}\in M_{i(l)}$, $l=1,2,$, $\tp(\bar{a}_{l(1)})=\tp(\bar{a}_{l(2)})$,
then there are $p_i\subseteq \tp(\bar{a}_{l(i)}),$ $|p_i|<\lambda$ and $p_i\vdash \tp(\bar{a}_ {l(i)})$ ($\tp(\bar{a})$ denotes the complete type realized by
the tuple $\bar{a}$).
\end{lemma}
\begin{proof} \cite{Shelah} Theorem 5.16, Chapter IV.
\end{proof}
\begin{corollary} For any countable theory, there is a family of $< {}^{\omega}2$ countable models that overlap only on principal types
\end{corollary}
We give a sketch of Shelah's ideas:

\begin{theorem}(Shelah) If $T$ is a countable theory and we have $<{}^{\omega}2$
many non principal maximal types, then they can be omitted in a countable model
\end{theorem}
\begin{demo}{Proof}

The idea is that one can build several models such that they overlap only on isolated types.
One can build {\it two} models so that every maximal type which is realized in both is isolated. Using the jargon of Robinson's finite forcing
implemented via games, the idea
is that one distributes this job of building the two models among experts, each has a role to play, and that all have
winning strategies. There is no difficulty in stretching the above idea to make the experts build three, four or any finite number of models
which overlap only at principal types.
With a pinch of diagonalisation we can extend the number to $\omega$.

To push it still further to $^{\omega}2$ needs uses ideas of Shelah, here it is a typical 'binary tree' construction, where at each root
reached at a finite stage
is split into two other models,  we end up with continuum many, these are not necessarily pairwise non isomorphic,
we mean here by a type overlapping in two distinct models with respect to an indexing set that has the power of the continuum,
but the map from this set to the class of models obtained is not necessarily injective.
(Otherwise, we would obtain that any countable theory has continuum many models,
which of course is entirely absurd.)

Now assume not.  Let ${\bold F}$ be the given set of non principal ultrafilters. Then for all $i<{}^{\omega}2$,
there exists $F$ such that $F$ is realized in $\B_i$. Let $\psi:{}^{\omega}2\to \wp(\bold F)$, be defined by
$\psi(i)=\{F: F \text { is realized in  }\B_i\}$.  Then for all $i<{}^{\omega}2$, $\psi(i)\neq \emptyset$.
Furthermore, for $i\neq j$, $\psi(i)\cap \psi(j)=\emptyset,$ for if $F\in \psi(i)\cap \psi(j)$ then it will be realized in
$\B_i$ and $\B_j$, and so it will be principal.  This implies that $|\bold F|={}^{\omega}2$ which is impossible.

\end{demo}

The following provides a proof  of a result actually stronger than that stated in \cite{Sayed} without a proof, namely, theorem
3.2.9. since it addresses a strictly larger class than $\Nr_n\A_{\omega}$ addressed in the above mentioned result,
namely, the class $S_c\Nr_n\CA_{\omega}.$
We will give an example shortly by showing that indeed the inclusion $\Nr_n\CA_{\omega}\subseteq S_c\Nr_n\CA_{\omega}$ is proper,
(This is not trivial, for example its relation algebra analogue is an open problem, see \cite{r}, to be also answered in a while. )

\begin{theorem}\label{uncountable} Let $\A=S_c\Nr_n\CA_{\omega}$. Assume that $|A|=\lambda$, where $\lambda$ is an uncountable
cardinal. Let  $\kappa< {}^{\lambda}2$,
and $(F_i: i<\kappa)$ be a system of non principal ultrafilters.
Then there exists
a set algebra $\C$ with base $U$ such that $|U|\leq \lambda$, $f:\A\to \C$ such that $f(a)\neq 0$ and for all $i\in \kappa$, $\bigcap_{x\in X_i} f(x)=0.$

\end{theorem}

\begin{proof}\label{omit} Let $\A\subseteq_c \Nr_n\B$, where $\B$ is $\omega$ dimensional, locally finite and has the same cardinality as $\A$.
This is possible by taking $\B$ to be the subalgebra of which $\A$ is a strong neat reduct generated by  $A$, and noting that we gave countably many
operations.
The $F_i$'s correspond to maximal $n$ types in the theory $T$ corresponding to $\B$, that is, the first order theory $T$ such that $\Fm_T\cong \B$.
Assume that ${\bold F}$ be the given set of non principal ultrafilters, with no model omitting them.
Then for all $i<{}^{\lambda}2$, for every representation $\B_i$ of $\B$,
there exists $F$ such that $F$ is realized in $\B_i$. Let $\psi:{}^{\lambda}2\to \wp(\bold F)$, be defined by
$\psi(i)=\{F: F \text { is realized in  }\B_i\}$.  Then for all $i<{}^{\lambda}2$, $\psi(i)\neq \emptyset$.
Furthermore, for $i\neq j$, $\psi(i)\cap \psi(j)=\emptyset,$ for if $F\in \psi(i)\cap \psi(j)$ then it will be realized in
$\B_i$ and $\B_j$, and so it will be principal. But this means that $\|\bold F|=2^{\aleph_0}$ which is not the case.
So there exists a model omitting the given set of maximal non principal types; algebraically
there exists $f:\B\to \wp(^{\omega}M)$
such that $\bigcap_{x\in F_i} f(x)=\emptyset.$
The restriction of $f$ to $\A$ defined the obvious way is as required.
\end{proof}

Now we give two metalogical readings of the last two theorems. The first is given in \cite{Sayed}, theorem 3.2.10,
but we include it, because it prepares for the other one, which is an entirely new omitting types theorems for cylindric
algebras of sentences. Cylindrifiers in such algebras can be defined because we include individual constants; the number of
these determines the dimension of the algebra in
question, this interpretation was given in  \cite{Amer}; in the context of representing
algebras of sentences as (full) neat reducts.

\begin{theorem}\label{Shelah} Let $T$ be an $\L_n$ consistent theory that admits elimination of quantifiers.
Assume that $|T|=\lambda$ is a regular cardinal.
Let $\kappa<2^{\lambda}$. Let $(\Gamma_i:i\in \kappa)$ be a set of non-principal maximal types in $T$. Then there is a model $\M$ of $T$ that omits
all the $\Gamma_i$'s
\end{theorem}
\begin{demo}{Proof} If $\A=\Fm_T$ denotes the cylindric algebra corresponding to $T$, then since $T$ admits elimination of quantifiers, then
$\A\in \Nr_n\CA_{\omega}$. This follows from the following reasoning. Let $\B=\Fm_{T_{\omega}}$ be the locally finite cylindric algebra
based on $T$ but now allowing $\omega$ many variables. Consider the map $\phi/T\mapsto \phi/T_{\omega}$.
Then this map is from $\A$ into $\Nr_n\B$. But since $T$ admits elimination of quantifiers the map is onto.
The Theorem now follows.
\end{demo}

We now give another natural omitting types theorem for certain uncountable languages.
Let $L$ be an ordinary first order language with
a list $\langle c_k\rangle$ of individual constants
of order type $\alpha$. $L$ has no operation symbols, but as usual, the
list of variables is of order type $\omega$. Denote by $Sn^{L_{\alpha}}$ the set
of all $L$ sentences, the subscript $\alpha$ indicating that we have $\alpha$ many constants Let $\alpha=n\in \omega$.
Let $T\subseteq Sn^{L_0}$ be consistent. Let $\M$  be an $\L_0$  model of $T$.
Then any $s:n\to M$ defines an expansion of $\M$ to $L_n$ which we denote by $\M{[s]}$.
For $\phi\in L_n$ let $\phi^{\M}=\{s\in M^n: \M[s]\models \phi\}$. Let
$\Gamma\subseteq Sn^{L_{n}}$.
The question we address is: Is there a model $\M$ of $T$ such that for no expansion $s:n\to M$ we have
$s\in  \bigcap_{\phi\in \Gamma}\phi^{\M}$.
Such an $\M$ omits $\Gamma$. Call $\Gamma$ principal over $T$ if there exists $\psi\in L_n$ consistent with $T$ such that
$T\models \psi\to \Gamma.$
Other wise $\Gamma$ is non principal over T.

\begin{theorem}  Let $T\subseteq Sn^{L_0}$ be consistent and assume that $\lambda$ is a regular cardinal, and $|T|=\lambda$.
Let $\kappa<2^{\lambda}$. Let $(\Gamma_i:i\in \kappa)$ be a set of non-principal maximal types in $T$.
Then there is a model $\M$ of $T$ that omits
all the $\Gamma_i$'s
That is, there exists a model $\M\models T$ such that there is no $s:n\to \M$
such that   $s\in \bigcap_{\phi\in \Gamma_i}\phi^{\M}$.
\end{theorem}
\begin{demo}{Proof}
Let $T\subseteq Sn^{L_0}$ be consistent. Let $\M$ be an $\L_0$  model of $T$.
For $\phi\in Sn^L$ and $k<\alpha$
let $\exists_k\phi:=\exists x\phi(c_k|x)$ where $x$ is the first variable
not occurring in $\phi$. Here $\phi(c_k|x)$ is the formula obtained from $\phi$ by
replacing all occurrences of $c_k$, if any, in $\phi$ by $x$.
Let $T$ be as indicated above, i.e, $T$ is a set of sentences in which no constants occur. Define the
equivalence relation $\equiv_{T}$ on $Sn^L$ as follows
$$\phi\equiv_{T}\psi \text { iff } T\models \phi\equiv \psi.$$
Then, as easily checked $\equiv_{T}$ is a
congruence relation on the algebra
$$\Sn=\langle Sn,\land,\lor,\neg,T,F,\exists_k, c_k=c_l\rangle_{k,l<n}$$
We let $\Sn^L/T$ denote the quotient algebra.
In this case, it is easy to see that $\Sn^L/T$ is a $\CA_n$, in fact is an $\RCA_n$.
Let $L$ be as described above. But now we denote it $L_n$,
the subscript $n$ indicating that we have $n$-many
individual constants. Now enrich $L_{n}$
with countably many constants (and nothing else) obtaining
$L_{\omega}$.
Recall that both languages, now, have a list of $\omega$ variables.
For $\kappa\in \{n, \omega\}$
let $\A_{\kappa}=\Sn^{L_{k}}/{T}$.
For $\phi\in Sn^{L_n}$, let $f(\phi/T)=\phi/{T}$.
Then, as easily checked, $f$ is an embedding  of $\A_{n}$
into $\A_{\omega}$. Moreover $f$ has the additional property that
it maps $\A_{n}$, into (and onto) the neat $n$ reduct of $\A_{\beta}$,
(i.e. the set of $\alpha$ dimensional elements of $A_{\beta}$).
In short, $\A_{n}\cong \Nr_{n}\A_{\omega}$. Now again putting $X_i=\{\phi/T: \phi\in \Gamma_i\}$ and using that the
$\Gamma_i$'s are maximal non isolated, it follows that the
$X_i's$ are non-principal ultrafilters
Since $\Nr_n\CA_{\omega}\subseteq S_c\Nr_n\CA_{\omega}$, then our result follows.
\end{demo}

\subsection{The inclusion $\Nr_n\CA_{\omega}\subset S_c\Nr_n\CA_{\omega}$ is proper}

Here we split up and blur a finite atom structure $\At$ twice, to get two algebras, denoted below by $\A$ and $\B$.
For $n$ finite $n\geq 3$, $\A\in \Nr_n\PEA_{\omega}$, while $\Rd_{Sc}\B\notin \Nr_n{\sf Sc}_n.$

A variant of the following lemma, is available in \cite{Sayedneat}, theorem 5.1.3 with a sketch of proof; it is fully
proved in \cite{MLQ}. If we require that a (representable) algebra be a neat reduct,
then quantifier elimination of the base model guarantees this, as indeed illustrated below.

However, in \cite{Sayed} different relations symbols only had distinct interpretations, meaning that they could have non-empty intersections;
here we strengthen
this condition to that they have {\it disjoint} interpretations. We need this stronger
condition to show that our constructed algebras
are atomic. A more basic construction of a similar model
was constructed in \cite{IGPL}.

Also we endow the uncountable index set of the ternary relation symbols in our signature
an abelian  group structure. We prove our result for $n=3$. The proof is the same for larger finite $n$.
For infinite dimension, which is much more harder, the reader is referred to \cite{IGPL} giving a unified proof to a more general case encompassing
all dimensions $>1$.
The idea is also splitting. A finite  atom structure of a representable algebra $\At$ is fixed in advance.
Then its atoms are
split twice.  Once, each atom is split into uncountably many, and once each into uncountably many except for one atom which
is only split is into countably many atoms. These atoms are called big atoms, which mean that they are cylindrically equivalent to their
original. (This is a general theme in splitting arguments).

The first splitting gives an algebra that is a full neat reduct of an algebra in arbitrary extra dimensions;
the second gives an algebra that is not a full neat reduct
of an algebra in just one extra dimensions, hence in any higher
dimensions. Both algebras are representable, in fact strongly representable, that is the complex algebra
of their atom structure, namely, their completions, are representable.
However, these two algebras are elementary equivalent
(indeed their atom structures are $L_{\infty, \omega}$ equivalent)
because first order logic does not witness this infinite cardinality twist.

The model on which the algebra will be based. For cylindric and polyadic algebras
$\At$ is just $\At(\wp(^33))={}^33.$

\begin{lemma} Let $V=(\At, \equiv_i, {\sf d}_{ij})_{i,j<3}$ be a finite cylindric atom structure,
such that $|\At|\geq |{}^33.|$
Let $L$ be a signature consisting of the unary relation
symbols $P_0,P_1,P_2$ and
uncountably many ternary predicate symbols.
For $u\in V$, let $\chi_u$
be the formula $\bigwedge_{u\in V}  P_{u_i}(x_i)$.
Then there exists an $L$-structure $\M$ with the following properties:
\begin{enumarab}

\item $\M$ has quantifier elimination, i.e. every $L$-formula is equivalent
in $\M$ to a boolean combination of atomic formulas.

\item The sets $P_i^{\M}$ for $i<n$ partition $M$, for any permutation $\tau$ on $3,$
$\forall x_0x_1x_2[R(x_0,x_1,x_2)\longleftrightarrow R(x_{\tau(0)},x_{\tau(1)}, x_{\tau(2)}],$

\item $\M \models \forall x_0x_1(R(x_0, x_1, x_2)\longrightarrow
\bigvee_{u\in V}\chi_u)$,
for all $R\in L$,

\item $\M\models  \forall x_0x_1x_2 (\chi_u\land R(x_0,x_1,x_2)\to \neg S(x_0,x_1,x_2))$
for all distinct ternary $R,S\in L$,
and $u\in V.$

\item For $u\in V$, $i<3,$
$\M\models \forall x_0x_1x_2
(\exists x_i\chi_u\longleftrightarrow \bigvee_{v\in V, v\equiv_iu}\chi_v),$

\item For $u\in V$ and any $L$-formula $\phi(x_0,x_1,x_2)$, if
$\M\models \exists x_0x_1x_2(\chi_u\land \phi)$ then
$\M\models
\forall x_0x_1x_2(\exists x_i\chi_u\longleftrightarrow
\exists x_i(\chi_u\land \phi))$ for all $i<3$
\end{enumarab}
\end{lemma}
\begin{proof}Cf. \cite{MLQ}, with minor modifications.
\end{proof}

In the next lemma $\tau_4$ is called a $4$ witness and $\tau$ is called an approximate witness.
An approximate witness
agrees with the witness on the global level, that is in the finite atom structure. On the level of the big atoms, obtained after the splitting,
it only dominates it, possibly strictly. The $4$ witness, sees the cardinality twist in one extra dimensions, which terms in the algebra do not and cannot see.
\begin{lemma}\label{term}
\begin{enumarab}

\item For $\A\in \CA_3$ or $\A\in \SC_3$, there exist
a unary term $\tau_4(x)$ in the language of $\SC_4$ and a unary term $\tau(x)$ in the language of $\CA_3$
such that $\CA_4\models \tau_4(x)\leq \tau(x),$
and for $\A$ as above, $u\in \At={}^33$ and any $\C$ such that $\A=\Nr_3\C$, we have
$$\tau_4^{\C}(\chi_{u})=\tau^{\A}(\chi_{u})=\chi_{\tau^{\wp(^33)}(u).}$$

\item   For $\A\in \PEA_3$ or $\A\in \PA_3$, there exist a binary
term $\tau_4(x,y)$ in the language of $\SC_4$ and another  binary term $\tau(x,y)$ in the language of $\SC_3$
such that $PEA_4\models \tau_4(x,y)\leq \tau(x,y),$
and for $\A$ as above, $u,v\in \At={}^33$ and any $\C$ such that $\A=\Nr_3\C$, we have
$$\tau_4^{\C}(\chi_{u}, \chi_{v})= \tau^{\A}(\chi_{u}, \chi_{v})=\chi_{\tau^{\wp(^33)}(u,v)}.$$
(Careless about the dimension of $\C$ but doesn't matter).


\end{enumarab}
\end{lemma}

\begin{proof}

For all reducts of polyadic algebras, these terms are given in \cite{Fm}, and \cite{MLQ}.
For cylindric algebras $\tau_4(x)={}_3 {\sf s}(0,1)x$ and $\tau(x)={\sf s}_1^0{\sf c}_1x.{\sf s}_0^1{\sf c}_0x$.
For polyadic algebras, it is a little bit more complicated because the former term above is definable.
In this case we have $\tau(x,y)=\c_1(\c_0x.\s_1^0\c_1y).\c_1x.\c_0y$, and $\tau_4(x,y)={\sf c}_3({\sf s}_3^1{\sf c}_3x.s_3^0{\sf c}_3y)$

\end{proof}

The atom structures of the above algebras are extremely close to each other.
The following result unifies (model theoretically) all the main results in \cite{Note}, \cite{MLQ}, \cite{IGPL}, \cite{Fm}.

\begin{theorem}\label{ef}
There exists an atomic $\A\in \Nr_3\sf QEA_{\omega}$
with an elementary equivalent cylindric  uncountable algebra
$\B$ which is strongly representable, and its $\sf Sc$ reduct is not in $\Nr_3\Sc_4$.
Furthermore, the latter is a complete subalgebra of the former
and  $\At\A\equiv_{\infty, \omega}\At\B.$
\end{theorem}

\begin{proof}  Let $\L$ and $\M$ as above. Let
$\A_{\omega}=\{\phi^M: \phi\in \L\}.$
Clearly $\A_{\omega}$ is a locally finite $\omega$-dimensional quasipolyadic  set algebra with equality.
The proof for $\CA$s; and its relatives are very similar. We add a condition to our constructed model.
We assume that the relation symbols are indexed by an uncountable set $I$.
and that  there is a group structure on $I$, such that for distinct $i\neq j\in I$,
we have $R_i\circ R_j=R_{i+j}$.
We take $\At=({}^33, \equiv_i, \equiv_{ij}, d_{ij})$, where
for $u,v\in \At$, $u\equiv_i v$ iff $u$ and $v$ agree off $i$ and $v\equiv_{ij}u$ iff $u\circ [i,j]=v$.
We denote $^33$ by $V$.

By the symmetry condition we have $\A$ is a ${\sf PEA_3}$, and
$\A\cong \Nr_3\A_{\omega}$, the isomorphism is given by
$\phi^{\M}\mapsto \phi^{\M}.$

In fact, $\A$ is not just a polyadic equality algebras, it is also closed under all first order definable
operations using extra dimensions for quantifier elimination in $\M$ guarantees that this map is onto, so that $\A$ is the full  neat reduct.
For $u\in {}V$, let $\A_u$ denote the relativisation of $\A$ to $\chi_u^{\M}$

i.e $\A_u=\{x\in A: x\leq \chi_u^{\M}\}.$ Then $\A_u$ is a Boolean algebra.
Furthermore, $\A_u$ is uncountable and atomic for every $u\in V$
because by property (iv) of the above lemma,
the sets $(\chi_u\land R(x_0,x_1,x_2)^{\M})$, for $R\in L$
are disjoint of $\A_u$. It is easy to see that $A_u$ is actually isomorphic to the finite co-finite Boolean algebra on  a set of cardinality $I$.

Define a map $f: \Bl\A\to \prod_{u\in {}V}\A_u$, by
$f(a)=\langle a\cdot \chi_u\rangle_{u\in{}V}.$
We expand the language of the Boolean algebra $\prod_{u\in V}\A_u$ by constants in
and unary operations, in such a way that
$\A$ becomes interpretable in the expanded structure.

Let $\P$ denote the
following structure for the signature of boolean algebras expanded
by constant symbols $1_u$, $u\in {}V$ and ${\sf d}_{ij}$, and unary relation symbols
${\sf s}_{[i,j]}$ for $i,j\in 3$:

\begin{enumarab}
\item The Boolean part of $\P$ is the boolean algebra $\prod_{u\in {}V}\A_u$,

\item $1_u^{\P}=f(\chi_u^{\M})=\langle 0,\cdots0,1,0,\cdots\rangle$
(with the $1$ in the $u^{th}$ place)
for each $u\in {}V$,

\item ${\sf d}_{ij}^{\P}=f({\sf d}_{ij}^{\A})$ for $i,j<\alpha$.

\item ${\sf s}_{[i,j]}^{\P}(x)= {\sf s}_{[i,j]}^{\P}\langle x_u: u\in V\rangle= \langle x_{u\circ [i,j]} : u\in V\rangle.$

\end{enumarab}

Define a map $f: \Bl\A\to \prod_{u\in {}V}\A_u$, by
$$f(a)=\langle a\cdot \chi_u\rangle_{u\in{}V}.$$

Then there are quantifier free formulas
$\eta_i(x,y)$ and $\eta_{ij}(x,y)$ such that
$\P\models \eta_i(f(a), b)$ iff $b=f(c_i^{\A}a)$ and
$\P\models \eta_{ij}(f(a), b)$ iff $b=f({\sf s}_{[i,j]}a).$
The one corresponding to cylindrifiers is exactly like the $\CA$ case, the one corresponding to substitutions in $y={\sf s}_{[i,j]}x.$
Now, like the $\CA$ case, $\A$ is interpretable in $\P$, and indeed the interpretation is one dimensional and quantifier free.

For this it is enough to show that
$f$ is one to one and that $\rng(f)$
(Range of $f$) and the $f$-images of the graphs of the cylindric algebra functions in $A$
are definable in $P$. Since the $\chi_u^M$ partition
the unit of $A$,  each $a\in A$ has a unique expression in the form
$\sum_{u\in {}^33}(a.\chi_u^M),$ and it follows that
$f$ is boolean isomorphism: $bool(A)\to \prod_{u\in {}^33}bool(A_u).$
So the $f$-images of the graphs of the boolean functions on
$A$ are trivially definable.
$f$ is bijective so $\rng(f)$ is
definable, by $x=x$. For the diagonals, $f(d_{ij}^A)$ is definable by $x=d_{ij}$.
Finally we consider cylindrifications. For $S\subseteq {}^33$, $i<3$,
let $t_S$ be the closed term
$$\sum\{1_v: v\in {}^33, v\equiv_i u\text { for some } u\in S\}.$$
Let
$$\eta_i(x,y)=\bigwedge_{S\subseteq {}^33}(\bigwedge_{u\in S} x.1_u\neq 0\land
\bigwedge_{u\in {}^33\smallsetminus S}x.1_u=0\longrightarrow y=t_S).$$
We claim that for all $a\in A$, $b\in P$, we have
$$P\models \eta_i(f(a),b)\text { iff } b=f(c_i^Aa).$$
To see this, let $f(a)=\langle a_u\rangle_{u\in {}^33}$, say.
So in $A$ we have $a=\sum_ua_u.$
Let $u$ be given; $a_u$ has the form $(\chi_i\land \phi)^M$ for some $\phi\in L^3$, so
${\sf c_i}^A(a_u)=(\exists x_i(\chi_u\land \phi))^M.$
By the above lemma if  $a_u\neq 0$, this is
$(\exists x_i\chi_u)^M$;  by property $5$,
this is $(\bigvee_{v\in {}^33, v\equiv_iu}\chi_v)^M.$
Let $S=\{u\in {}^33: a_u\neq 0\}.$
By normality and additivity of cylindrifications we have,
$${\sf c}_i^A(a)=\sum_{u\in {}^33} {\sf c}_i^Aa_u=
\sum_{u\in S}{\sf c}_i^Aa_u=\sum_{u\in S}(\sum_{v\in {}^33, v\equiv_i u}\chi_v^M)$$
$$=\sum\{\chi_v^M: v\in {}^33, v\equiv_i u\text { for some } u\in S\}.$$
So $P\models f({\sf c}_i^{A}a)=t_S$. Hence $P\models \eta_i(f(a),f(c_i^Aa)).$
Conversely, if $P\models \eta_i(f(a),b)$, we require $b=f(c_ia)$.
Now $S$ is the unique subset of $^33$ such that
$$P\models \bigwedge_{u\in S}f(a).1_u\neq 0\land \bigwedge_{u\in {}^33\smallsetminus S}
f(a).1_u=0.$$  So we obtain
$$b=t_S=f({\sf c}_i^Aa).$$
Substitutions is more direct.
We have proved that $P$ is interpretable in $A$.


For $v\in V$, let $\B_v$ be a complete countable elementary subalgebra of $\A_v$.
Then proceed like the $\CA$ case, except that we take a different product, since we have a different atom structure, with unary relations
for substitutions:
Let $u_1, u_2\in V$ and let $v=\tau(u_1,u_2)$, as given in the above lemma.
Let $J=\{u_1,u_2, {\sf s}_{[i,j]}v: i\neq  j<3\}$.
Let  $\B=((\A_{u_1}\times \A_{u_2}\times \B_{v}\times \prod_{i,j<3, i\neq j} \B_{s_{[i,j]}v}\times \prod_{u\in V\sim J} \A_u), 1_u, d_{i,j}, {\sf s}_{i,j}x)$
inheriting the same interpretation. Then by the Feferman Vaught theorem,
which says that replacing a component in a possibly infinite product by  elementary equivalent
algebra, then the resulting new product is elementary equivalent to the original one, so that $\B\equiv \P$,
hence $\B\equiv \A$. (If a structure is interpretable in another structure then any structure
elementary equivalent to the former structure is elementary equivalent to the last).
Notice to  that $\B$ is atomic, because $\P$ is, and atomicity is a first order property.

Now $\B$ is strongly representable. \pe\ has a winning strategy in all finite rounded atomic games.
But $\B$ is easily seen to be completely representable, and indeed \pe\ can indeed win the $\omega$ rounded
usual atomic game, guided by a complete representation.

In our new product we made all the permuted versions of $\B_v$ countable, so that $|\B_v|$ {\it remains} countable,
because substitutions corresponding to transpositions
are present in our signature, so if one of the permuted components is uncountable, then $\B_{v}$ would be uncountable, and we do not want that.

Note that $\B\in {\sf RPEA}_3$. The contradiction follows from the fact that had  $\Rd_{sc}\B$ been a neat reduct, say $\B=\Nr_3\D$
then the term $\tau_3$ as in the above lemma, using $4$ variables, evaluated in $\D$
will force the component $\B_v$ to be uncountable, which is not the case by construction, indeed $\tau_3^{\D}(f(R_i), f(R_j))=f(R_{i+j})$.

For $u\in S_n$. Then $\{1_u: u\in S_n\}$ forms a partition of the unit
$^nW$ of $\A$.
Then $\A_u$ and $\B_u$ are atomic boolean algebras.

We devise a pebble game similar to the rainbow pebble game, except that it is a back and forth \ef\
game, not just a forth game.
Furthermore, each player has the option to choose
an element from {\it both} structures, and not just stick to one.  Pairs of pebbles are outside the board.
\pa\ as usual starts the game by placing a pebble on an element of one of the structures. \pe\
responds by placing the other pebble on the an element on the other structure.

Between them they choose an atom $a_i$ of $\At\A$
and an atom  $b_i$ of $\At\B$, under the restriction that player \pe\
must choose from the other structure from player \pa\ at each step.
A win for \pe\ if the binary relation resulting from the choices of the two players $R=\{(a,b): a\in \At(\A), b\in \At(\B)\}$ is a partial isomorphism.

At each step, if the play so far $(\bar{a}, \bar{b})$ and \pa\ chooses an atom $a$
in one of the structures, we have one of two case.
Either $a.1_u=a$ for some $u\neq Id$
in which case
\pe chooses the same atom in the other structure.
Else $a\leq 1_{Id}$
Then \pe\ chooses a new atom below $1_{Id}$
(distinct from $a$ and all atoms played so far.)
This is possible since there finitely many atoms in
play and there are infinitely many atoms below
$1_{u}$.
This strategy makes \pe\ win, since atoms below $1_u$ are cylindrically equivalent..
Let $J$ be a back and forth system which exists.
Order $J$ by reverse inclusion, that is $f\leq g$
if $f$ extends $g$. $\leq$ is a partial order on $J$.
For $g\in J$, let $[g]=\{f\in J: f\leq g\}$. Then $\{[g]: g\in J\}$ is the base of a
topology on
$J.$ Let $\C$ be the complete
Boolean algebra of regular open subsets of $J$ with respect to the topology
defined on $J.$
Form the boolean extension $\M^{\C}.$
We want to define an isomorphism in $\M^{\C}$ of $\breve{\A}$ to
$\breve{\B}.$
Define $G$ by
$$||G(\breve{a},\breve{b})||=\{f\in {J}: f(a)=b\}$$
for $c\in \A$ and $d\in \B$.
If the right-hand side,  is not empty, that is it contains a function $f$, then let
$f_0$ be the restriction of $f$ to the substructure of $\A$ generated by $\{a\}$.
Then $f_0\in J.$ Also $$\{f\in J:  f(c)=d\}=[f_0]\in \C.$$
$G$ is therefore a $\C$-valued relation. Now let $u,v\in \M$.
Then
$$||\breve{u}=\breve{v}||=1\text { iff }u=v,$$
and
$$||\breve{u}=\breve{v}||=0\text { iff } u\neq v$$
Therefore
$$||G(\breve{a},\breve{b})\land G(\breve{a},\breve{c})||\subseteq ||\breve{b}=\breve{c}||.$$
for $a\in \A$ and $b,c\in \B.$
Therefore ``$G$ is a function." is valid.
It is one to one because its converse is also a function.
(This can be proved the same way).
Also $G$ is surjective.

One can alternatively show that $\A\equiv_{\infty\omega}\B$ using "soft model theory" as follows:
Form a boolean extension $\M^*$ of $\M$
in which the cardinalities of $\A$ and $\B$ collapse to
$\omega$.  Then $\A$ and $\B$ are still back and forth equivalent in $\M^*.$
Then $\A\equiv_{\infty\omega}\B$ in $M^*$, and hence also in $\M$
by absoluteness of $\models$.

\end{proof}

\subsection{Does it work for relation algebras?}

here we try to apply the above technique to relation algebra reducts.
Let $\A$ be as above and let $\At$ be  a finite atom structure of a representable relation algebra.
Let $k\geq 4$. A $k$ witness $\tau_k$ is $m$-ary term of $\CA_k$ with rank $m\geq 2$ such
that $\tau_k$ is not definable in the language of relation algebras (so that $k$ has to be $\geq 4$)
and for which there exists an approximate  witness $\tau$ expressible in the language of relation algebras, such that
$\CA_k\models \tau_k(x_1,\ldots x_m)\leq \tau(x_1,\ldots x_m)$
(this is an implication between two first order formulas using $k$-variables)
and for any $u_1,\ldots u_m\in \At$, we have
$$\tau_k^{\C}(\chi_{u_1},\ldots \chi_{u_m})=\chi_{\tau^{\Cm\At}(u_1,\ldots u_m)}$$
for any $\C\in \CA_{k}$ such that $\A=\Nr_3\C$.

Add the following condition (inductively) to the model $\M$, we are still assuming that the uncountable  index set
of the ternary relation symbols carry a group structure. So we further assume that in addition to
the $6$ conditions formulated above, we have:
$$\M\models \exists z_0, z_1\ldots z_4(z_0=x\land z_1=y\land z_2=t\land \bigwedge R_{i\in J}(\bar{z})\longleftrightarrow R_{\sum J})$$
Let $\tau_5$ be the corresponding term.

\begin{theorem}\label{witness} If there exists an approximate witness to $\tau_5$,
then there exists a completely representable relation algebra $\B\in S_c\Ra\CA_{\omega}$
but not in ${\sf UpUr}\Ra\CA_k$ for $k\geq 5.$
\end{theorem}
\begin{proof}
The $\Ra$ reduct of $\A$ is a generalized reduct of $\A$, hence $\Ra\A$ is first order interpretable in $\P$, as well.
It follows that there are closed terms and  a unary relation symbol, and formula $\eta$, and $\mu$ built out of these closed terms and unary
relation symbol such that
$\P\models \eta(f(a), b, c)\text { iff }b= f(a\circ^{\Ra\A} c),$ and $\P\models \mu(f(a),b)\text { iff }b=\breve{a}$
where the composition is taken in $\Ra\A$. The former formula is built, like cylindrifiers from only closed terms, $1_u$, $u\in \At$
while converse is defined by
the unary relation symbol.

Here $\At$ defined depends on $\tau_5$, the term corresponding to the above Jonssons $Q$ operation,
we require that $\At$ is the atom structure of a finite relation algebra.

As before, for each $u\in \At$, choose any countable Boolean elementary
complete subalgebra of $\A_{u}$, $\B_{u}$ say. Assume that $\tau$ be an approximate witness.

Let $u_i: i<m$ be elements in $\At$, and let $v=\tau(u_1,\ldots u_m)=Id$.

Let $$\B=((\prod_{u_i: i<m}\A_{u_i}\times \B_{v}\times \times \B_{\breve{v}}\times \prod_{u\in {}V\smallsetminus \{u_1,\ldots u_m, v, \breve{v}\}}
\A_u), 1_u, R, Id) \equiv$$
$$(\prod_{u\in V} \A_u, 1_u, R, Id)=\P.$$

Let $\B$ be the result of applying the interpretation given above to $Q$.
Then $\B\equiv \Ra\A$ as relation  algebras, furthermore $\Bl\B$ is a complete subalgebra of $\Bl\A$.
Again $\B\in {\sf RRA}$, but it is not a full $\Ra$ reduct.

We use essentially the same argument. We force the $\tau(u_1,\ldots u_m)$
component together with its permuted versions (because we have converse) countable;
the resulting algebra will be a complete elementary subalgebra of the original one, but $\tau_k$
will force our twisted countable component to be uncountable, arriving at a contradiction.

In more detail, assume for contradiction that $\B=\Ra\D$ with $\D\in \CA_k$.
Then $\tau_k^{\D}(f(\chi_{u_1}),\ldots f(\chi_{u_n}))$, is uncountable in $\D$.
Because $\B$ is a full ${\sf RA}$ reduct,
this set is contained in $\B.$

For simplicity assume that $\tau^{\Cm\At}(u_1\ldots u_m)=Id.$
With $x_i\leq \chi_{u_i}$, let $\bar{x_i}=(0\ldots x_i,\ldots)$
with $x_i$ in the $uth$ place.
Then we have
$$\tau_k^{\D}(\bar{x_1},\ldots \bar{x_m})\leq \tau(\bar{x_1}\ldots \bar{x_m})\in \tau(f(\chi_{u_1}),\ldots f({\chi_{u_m}}))
=f(\chi_{\tau(u_1\ldots u_m)})=f(\chi_{Id}).$$

But this is a contradiction,
since  $\B_{Id}=\{x\in B: x\leq \chi_{Id}\}$ is  countable and $f$ is a Boolean isomorphism.
\end{proof}
We point out that this, modulo the existence of an approximate witness,
answers two question of Hirsch in \cite{r}.

\subsection{ Omitting types for finite first order definable operations}

Here we prove a result mentioned in \cite{ANT} without proof,
namely, that the omitting types theorem fails for any finite first order definable extension of $L_n$, first order logic restricted to the first $n$ variables,
when $n>2$. We add that our result even extends to stronger logics like ones endowed with operations of transitive closure as well \cite{Maddux}.
We recall what we mean by first order definable algebraic operations.
Such operations, as the name suggests, are built using first order formulas, one for each formula.
Later, we will show that in the corresponding logic they correspond to newly added
connectives defined by such formulas, and the arity of connectives is determined
by the number of {\it relation symbols} viewed as formula schemes in the defining formula.

\begin{definition}
Let $\Lambda$ be a first order language with countably many
relation symbols, $R_0, \ldots R_i,\ldots : i\in \omega$
each of arity $n$.
Let $\Cs_{n,t}$ denote the following class of similarity type $t$:
\begin{enumroman}
\item $t$ is an expansion of $t_{\CA_n}.$
\item  $S\Rd_{ca}\Cs_{n,t}=\Cs_n.$ In particular, every algebra in $\Cs_{n,t}$ is a boolean
field of sets with unit $^nU$ say,
that is closed under cylindrifications and contains diagonal elements.
\item For any $m$-ary operation $f$ of $t$, there exists a first order formula
$\phi$ with free variables among $\{x_0,\ldots, x_n\}$
and having exactly $m,$ $n$-ary relation symbols
$R_0, \ldots R_{m-1}$ such that,
for any set algebra ${\A}\in \Cs_{n,t}$
with base $U$, and $X_0, \ldots X_{m-1}\in {\cal A}$,
we have:
$$\langle a_0,\ldots a_{n-1}\rangle\in f(X_0,\ldots X_{m-1})$$
if and only if
$${\cal M}=\langle U, X_0,\ldots X_{n-1}\rangle\models \phi[a_0,\ldots a_{n-1}].$$
Here $\cal M$ is the first order structure in which for each $i\leq m$,
$R_i$ is interpreted as $X_i,$ and $\models$ is the usual satisfiability relation.
Note that cylindrifications and diagonal elements are so definable.
Indeed for $i,j<n$,  $\exists x_iR_0(x_0\ldots x_{n-1})$
defines $C_i$ and $x_i=x_j$ defines $D_{ij}.$
\item With $f$ and $\phi$ as above,
$f$ is said to be a first order definable operation with $\phi$ defining $f$,
or simply a first order definable
operation, and $\Cs_{n,t}$ is said to be a first order definable
expansion of $\Cs_n.$
\item $\RCA_{n,t}$ denotes the class $SP\Cs_{n,t}$, i.e. the class of all subdirect products
of algebras
in $\Cs_{n,t}.$ We also refer to
$\RCA_{n,t}$ as a first order definable expansion of $\RCA_n.$
\end{enumroman}
\end{definition}

From now on, fix a {\it finite} $t$ (that is an expansion of $t_{\CA_n})$
and fix a first order language $\Lambda$
with countably many relation symbols each of
arity $n$.
For $\omega\geq m>0$ , let $\Fm_r^{\Lambda_{n+m}}$,
or $\Fm_r^{n+m}$, denote the set of {\it restricted} formulas
built up of $n+m$  variables.
Here restricted means that the
variables occurring in atomic subformulas, appear in their natural order. For example
$R_0(x_0, x_1\ldots x_{n-1})$ is a restricted formula while
$R_0( x_1, x_0, \ldots x_{n-1})$ is not.
Since the variables in restricted formulas
appear in their natural order, we might as well dispense with them altogether,
and write $R_i$ instead of the more cumbersome
$R_i(x_0, x_1, \ldots x_{n-1})$, which we do.
The arity of $R_i$ will be clear from context.
For example, we write $\exists x_1R_5$
instead of $\exists x_1R_5(x_0, x_1,\ldots x_{n-1})$.
In light of the above the notion
of free variables needs further clarification,
since we omitted reference to free variables altogether.
A variable $x$ is free in a
formula $\phi$ if $\exists x\phi$ is not equivalent to
$\phi.$ In passing, we note that for $m=\omega$, allowing only restricted formulas is not
a restriction at all. In the presence of infinitely many variables,
any formula is equivalent to a restricted one.
In the next definition, we establish a one to one correspondence
between restricted formulas and $\CA$ terms \cite[\S 4.3]{HMT1}, more precisely:

\begin{definition}
\begin{enumroman}
\item Let $k>0$. Then by the discourse language of $\CA_k$
we understand the language consisting of
a countable list of (meta) variables,  which we denote by
$\bold x_0, \bold x_1, \ldots \bold x_i \ldots: i\in \omega$.
We use bold letters to distinguish between meta variables and variables.
We also have a binary connective $\cdot$, a unary
connective  $-$,
a set of unary function symbols
${\sf c}_i:  i<k$ and constants ${\sf d}_{ij}$, $i, j<k$.
\item Let $m>0.$ Now for any $\phi\in \Fm_r^{n+m}$,
we associate a term
$\tau(\phi)$ in the discourse language of $\CA_{n+m}$
by an easy recursion as follows:
$$\tau(R_i)={\bold x_i}, i\in \omega,$$
$$\tau(\phi\land \psi)=\tau(\phi)\cdot\tau(\psi),$$
$$\tau(\neg \phi)=-[\tau(\phi)],$$
$$\tau[(\exists x_i)\phi]={\sf c}_i(\tau(\phi),$$
$$\tau(x_i=x_j)={\sf d}_{ij}.$$
\end{enumroman}
\end{definition}
For example, if $\phi$ is $(\exists x_n)(R_0\land x_n=x_0)$ then $\tau(\phi)$ is
$c_n(\bold x_0\cdot d_{n0}).$

If $m=0$, then the above translation makes one move back and forth between valid
equivalences and equations, in the following sense.
$$\models \phi\longleftrightarrow \psi \text { if and only if }
\RCA_n\models \tau(\phi)=\tau(\psi).$$
The fact that $\RCA_n$ is not finitely axiomatizable, reflects
the metalogical fact that
validities
cannot be captured
by any finite Hilbert style axiomatization using only finitely many variables.
This non-finte axiomatizability result
cannot be circumvented when we add finitely many first order definable
operations. We now turn to defining certain abstract algebras based on neat reducts

\begin{definition}
\begin{enumroman}
\item Let $\omega\geq m>0.$ For $\A\in \CA_{n+m},$
$\Nr_n{\A}$ denotes the neat $n$-reduct of $\A$.
Let $t$ be a finite expansion of $t_{\CA_n}$
such that $\Cs_{n,t}$ is a first order definable expansion
of $\Cs_n.$
Assume further, that for any $f\in t$ there is a formula $\phi(f)$
defining $f$ and having {\it free} variables among the first $n$. Fix such $\phi(f).$
We let ${\Ra}\A$
denote the following algebra of type $t$ which is an expansion of ${\Nr}_n{\A}$:
$${\Ra}{\A}=
\langle {\Nr}_n{\A}, \tau^{\A}(\phi(f))\rangle_{f\in t\smallsetminus t_{\CA_n}}.$$

\item For $m\geq 0$,  $\K_{n,m}$ denotes the following class of type $t$:
$$\K_{n,m}=S\{{\Ra}{\A}: {\A}\in \CA_{n+m}\}.$$
\end{enumroman}
\end{definition}

In the definition of ${\Ra}\A$, the extra operations in $t$
are defined by forming the term corresponding to the beforehand
fixed formula $\phi(f)$ defining (the first order definable operation) $f$.
Note that if
$\A$ is representable then the choice of the formula
defining $f$ is immaterial, any two such equivalent formulas give
the same thing when interpreted in the algebra.

In principle, this might not be the case when $\A$
is not representable and finite dimensional. That is assume that $\A\in \Nr_n\CA_{n+k}$.
Assume that $f\in t$
is definable in set algebras via $\phi$, say, and $\psi$ is equivalent to $\phi$,  and both formulas useless that $n+k$
variables, then
$\tau^{\A}(\psi)$ may not be equal to $\tau^{\A}(\phi.)$
(A counterexample, though is not easy, but the idea is that, though the two formulas use $< n+k$ variables,
the proof of their {\it equivalence}may  needs more $> n+k$ variables.)

Note too, that because the free variables
occurring in every $\phi(f)$, are
among the first $n$,
we do have
$\tau^{\cal A}(\phi(f))$ is in $\Nr_n\A, $ so that in any event the definition of ${\Ra}{\A}$
is indeed sound.

\begin{definition}
Let $t$ be an expansion of $t_{\CA_n}$ such that $\Cs_{n,t}$ is a first order definable
expansion of $\Cs_n$. An algebra $\A$ of type $t$
is said to have the neat embedding property, if
$\A$ is a subalgebra of ${\Ra}\B$
for some ${\B}\in \CA_{\omega}.$
\end{definition}

Note that in principle $\B$ may not be unique. One can
construct $\B$ and ${\C}\in  \CA_{\omega}$ such that
$\Nr_3{\C}\cong \Nr_3{\B}$ but $\B$ and $\C$ are not isomorphic.

Actually this is quite easy to do, for any $n\geq 3$. Take $\A=\Nr_n\B$, where $\B$ is not locally finite.
Take $\C=\Sg^{\B}\A$, then $\C$ is locally finite, $\A=\Nr_n\C$ and $\B$ and $\C$ are not isomorphic.
This is the argument used to show
that $\Nr_n\CA_{\omega}=\Nr_n{\sf Lf}_{\omega}$.

In any case, what really counts at the end is the existence of at least
one such algebra
$\B$.
Note too that for such an algebra $\A$, its cylindric reduct is a cylindric algebra.

The following is proved in \cite{basim} reproving
Biro's result using only cylindric algebras, without reference to relation algebras.

\begin{theorem}
\begin{enumarab}
\item Let $t$ be an expansion of $t_{\CA_n}$ such that $\Cs_{n,t}$ is a first order definable
expansion of $\Cs_n$.
$\A$ of type $t$  is representable if and only if $\A$
has the neat embedding property.
\item Let $3\leq n\leq m\leq \omega.$
If $\K_{n,m}\supseteq \RCA_{n,t},$ then $m=\omega.$
In other words, $\K_{n,m}$ is properly contained in $\RCA_{n,t}$
when $m$ is finite.
\end{enumarab}
\end{theorem}

The condition on finiteness of models often provides us (in a certain sense) with more complex theories,
but the condition of finiteness for variable sets usually make theories simpler
(for example, algorithmically) nevertheless some desirable properties can be lost. Our next result is of such kind.
Letting $L_n$ denote first order logic with $n$ variables, we have:

\begin{theorem} Foe $n>2$, there is no finite first order definable expansion of $L_n$ that is both sound and complete
\end{theorem}
\begin{proof}\cite{basim} This was proved using Monk's algebras, so languages contained infinitely many countable relation symbols.
\end{proof}

Now using Maddux's algebras one can show that the above theorem holds
if we only have one binary relation symbol; the algebra $\A$ constructed in \cite{ANT} gives the same result but for the omitting types theorem

\begin{theorem} The same result holds when we have one ternary symbol; furthermore, the omitting types theorem fails.
\end{theorem}
\begin{proof}

Here we follow closely the formalism for relation algebras presented in \cite{Biro} which is that introduced in the Tarski Givant monograph; however,
we address $L_n$. Let $n$ be finite $n\geq 3$. We have $\L_n$ is defined in the following signature.
$E$ a binary symbol and $=$ are formulas.
If $\phi, \psi$ are formulas then so are $\phi\land \psi$, $\exists \phi$, $(\phi\land \psi)$ and
$\neg \phi$.

A model is a pair $\M=(M ,E)$
such that $E\subseteq {}M\times M$. Satisfiability is defined by recursion, the clause that
deserves recalling is the cylindrifier clause defined like first order logic.

We define an auxiliary language $\L^*$ which countably many
relation symbols $R_0, R_1\ldots R_m\ldots $ in addition to $E$, each of arity $n$.

The language $\L^{f}$ is a first order expansion of $\L_n$
if the following hold:

\begin{enumarab}
\item It is obtained from $\L_n$ by adding finitely many new symbols
each with a non-negative rank, with formation rules as above, namely, if $c$ is a
new $m$ ary symbol and $\phi_0,\ldots, \phi_{m-1}$
are formulas in $\L^{f}$, then $c(\phi_0,\ldots \phi_{m-1})$ is a formula.

\item The class of models of the expanded language is the same.

\item With every symbol $c$ of $\L^{f}$ we associate
a formula $\phi_c$ of $L^{*}$ such that
$\phi_c$ involves only $E, R_0,\ldots R_{m-1}$ where
$m$ is the rank of $c$ and the free variables of $\phi_c$
are among $v_0,\ldots  v_{m-1}$.
For $c\in \L^{f}$, the formula $\phi_c$ is defined as expected, example,
$\phi_E=: E(v_0v_1)$, $\phi_{=}=: v_0=v_1$ and
$\phi_{c_k}=\exists v_k(R_0(v_0v_2v_3)).$

\item We define
$$M\models c(\phi_0\ldots \phi_{m-1})[\bar{s}]\Longleftrightarrow
(U, \phi^M\ldots \phi_{m-1}^M)\models \phi_c[\bar{s}],$$
for every symbols different from
$E$ and $=$ and every $s\in {}^nM$.

\end{enumarab}

To deal with the proof theory of such a formalism one adjoins to $\L^f$ new meta variables symbols $
X_1, X_2, \dots $
called formula variables, thereby forming a new language $\L_{sh}$ (here these are sometimes called formula variables).
Indeed assume that $\Sigma$ is sound and complete in $\L^f$.
Now formulas are formed the same way treating the new formula
variables like $=$ and $E$.  Formulas are called {\it schemas}. An instance
of a schema $\sigma(X_1,\ldots X_n)$ is obtained from $\sigma$ by
replacing  all occurrences of the formula variables
$X_1, \ldots X_n$ in $\sigma$
by formulas in $\L^f$.

Let $C$ be the finite set of symbols different from $E$ and $=$.
Let $t'$ be a similarity type that has an operation symbol $H_c$ for any $c\in C$, and whose arity
is the rank $c$. Since $C$ is finite we may assume that
$t'$ is a finite expansion of $\CA_n$.
Let the set of variable symbols of
that algebraic language of $t'$ be $\{w_i: i\in \omega\}$.
For any schema $\theta$
define $\tau(\theta)$ by recursion, example,
$\tau(=)=1, \tau(E)=w_0$ $\tau(X_i)=w_{i+1}$
$\tau(c(\phi_0, \ldots \phi_{m-1})=H_c(\tau(\phi_0), \ldots \tau\phi_{m-1})$

Now the operations on set algebra $\A\in {\sf Cs}_3$
are interpreted by;
for any $c\in C$ $\R_0,\ldots \R_{m-1}\in \A$
$$H_c(\bar{\R})^{\A}=\{s\in ^nU:  (U, \R_0, \ldots \R_{m-1})\models \phi_c[s]\}.$$

Now define a term $\tau\theta$ of the language $t'$
for each formula schema of the expanded language, exactly as before
Let $\Delta=\{\tau(\chi)=1: \chi\in \Sigma\},$
Then $\Delta$ hence $\Sigma$ cannot be finite.

Now we show that the omitting types theorem fails for $\L$.

Let $\A\in \sf RCA_n\cap \Nr_n\CA_{n+k}$ be an algebra that is atomic but not completely representable. Exists.
We translate this algebra to an $\L$ theory, with a non principal type that cannot be
omitted. Now $A=\Fm_T$, where $T$ is an $\L_n$ theory.

Expand the language of $T$ by including syntactically $\phi_c$ as a new connective, to obtain $\L^{f}$.
Let $\Gamma$ be the non principal type of co-atoms of $\A$; that is is $\Gamma=\{\neg\phi: \phi_T\in \At\A\}$.
Assume, for contradiction, that $\M=(M, E)$ is a model omitting $\Gamma$, in the expanded language, where $\phi_c$
is interpreted semantically, as defined above, by
$$\M\models {\phi}_c(\psi_1,\ldots \psi_{m-1})[s]\Longleftrightarrow (\M, \psi_1^{\M}, \ldots \psi_{m-1}^{\M})\models \phi_c[s].$$
Let $\B$ be the corresponding set algebra, with the semantics for $\phi_c$ defined as above, so that
we have
$$[\phi_c(R_1, \ldots R_{m-1})]^{\B}=\{s\in {}^nM: (M, R_0,\ldots R_{n-1})\models \phi_c[s]\}.$$
Then $\B$ is closed under the operations (this can be proved by an easy induction exactly as above).

But the the reduct of $\B$ to the language of $\CA_n$ gives a complete representation of $\A$
which is impossible.

\end{proof}

\subsection{Omitting types in other contexts, positive and negative results}

An unpublished result of Andr\'eka and N\'emeti shows that the omitting types theorem fails for $L_2$ though
Vaught's theorem on existence of atomic models for atomic theories hold, \cite{Sayed}, \cite{HHbook}, \cite{Khaled}, \cite{Vaught}
In the next example we show that even Vaught's theorem, hence $OTT$, fails when we consider logics without equality
reflected algebraically by many reducts of polyadic algebras.

We first start with algebras that are cylindrifier free reducts of polyadic algebras. In
this case set algebras are defined exactly like polyadic set algebras by discarding
cylindrifiers. Such algebras are expansions of Pinter's algebras studied by S\'agi,
and explicitly mentioned by Hodkinson \cite{AU}
in the context of searching for algebras, where atomicity coincides with complete representability.

Assem showed that for any ordinal $\alpha>1$, and any infinite cardinal $\kappa$,
there is an atomic set algebra (having as extra Boolean operations only finite substitutions)
with $|A|=\kappa$,  that is not  completely representable.
In particular, $\A$ can be countable, and so the omitting types theorem, and for that matter Vaught's theorem fail.
This works for all dimensions, except that in
the infinite dimensional case, semantics is relativized to weak set algebras.
Do we have an analogous result, concerning failure of the omitting types theorem for fragments of $L_n$ without equality, but with quantifiers.
The answer is yes.

This theorem holds for Pinter's algebras and polyadic algebras, let $\K$ denote either class.
It suffices to show that there exists $\B\in {\sf RK}_n\cap \Nr_n\K_{n+k}$ that is not completely completely representable.
But this is not hard to show. Let $\A$ be the cylindric algebra of dimension $n\geq 3$, $n$ finite,
provided by theorem 1.1 in \cite{ANT}.
Then first we can expand $\A$ to a polyadic equality
algebra because it is a subalgebra of the complex algebra based on the atom structure of basic matrices.
This new algebra will also be in ${\sf RPEA}_n\cap \Nr_n{\sf PEA}_{n+k}$. Its reduct, whether the polyadic or the Pinter, will be as desired.

Indeed consider the $\sf PA$ case, with the other case completely analogous,
this follows from the fact that $\Nr_n\K_n\subseteq \Rd\Nr_n{\sf PEA_{n+k}}=\Nr_n\Rd{\sf PEA}_{n+k}\subseteq \Nr_n{\sf PA}_{n+k}$,
and that $\A$ is completely representable if and only if its diagonal free reduct is.
(This is proved by Hodkinson in \cite{AU}, the proof depends essentially on the fact that algebras considered
are binary generated).

Now what if we only have cylindrifiers, that is dealing with ${\sf Df}_n$, $n\geq 3$.
Let $\A$ be the cylindric algebra as in the previous paragraph. Assume that
there a type $\Gamma$,
that is realized in every representation of $\A$ but has no witness using extra $k$ variables. Let $\B=\Rd_{df}\A$.

Let $f:\B\to \C$ be a diagonal free representation of $\B$. The point is that though $\Gamma$ is realized in every {\it cylindric} representation of $\A$,
there might be a representation of its diagonal free reduct that omits $\Gamma$,
these are more, because we do not require preservation of the diagonal
elements. This case definitely needs further research, and we are tempted to think that it is not easy.

But what we are sure
of  that the ordinary omitting types theorem fails
for $L_n$ without equality (that is for ${\sf Df}_n$) for $n\geq 3$. One way, among many other, is to construct a
representable countable atomic algebra $\A\in {\sf RDf}_n$, that is not completely representable.
The diagonal free reduct of the cylindric algebra constructed in \cite{ANT} is such.
Now what about ${\sf Df_2}$? We do not know.
But if we have only {\it one} replacement then it fails.
For higher dimensions, the result follows from the following example from
\cite{AGMNS}.
We give the a sketch of the proof, the interested reader can work out the details himself
or either directly consult \cite{AGMNS}.

\begin{example}

Let $\B$ be an atomless Boolean set algebra with unit $U$, that has the following property:
For any distinct $u,v\in U$, there is $X\in B$ such that $u\in X$ and $v\in {}\sim X$.
For example $\B$ can be taken to be the Stone representation of some atomless Boolean algebra.
The cardinality of our constructed algebra will be the same as $|B|$.
Let $R=\{X\times Y: X,Y\in \B\}$
and
$A=\{\bigcup S: S\subseteq R: |S|<\omega\}.$
Then $|R|=|A|=|B|$ and
$\A$ is a subalgebra of $\wp(^2U)$.
Also the only subset of $D_{01}$ in $\A$ is the empty set.
Let $S=\{X\times \sim X: X\in B\}.$
$\bigcup S={}\sim D_{01}.$
and $\sum{}^{\A}S=U\times U.$
But
$S_0^1(X\times \sim X)=(X\cap \sim X)\times U=\emptyset.$
for every $X\in B$.
Thus $S_0^1(\sum S)=U\times U$
and
$\sum \{S_{0}^1(Z): Z\in S\}=\emptyset.$

For $n>2$, one takes $R=\{X_1\times\ldots\times X_n: X_i\in \B\}$ and the definition of $\A$ is the same. Then,
in this case, one takes $S$ to be
$X\times \sim X\times U\times\ldots\times U$
such that $X\in B$. The proof survives verbatim.
By taking $\B$ to be countable, then $\A$ can be countable, and so it violates the omitting types theorem.
\end{example}

\begin{example} We consider a very simple case, when we have only transpositions. In this case omitting types theorems
holds for countable languages and
atomic theories have atomic models. Here all substitutions corresponding to bijective maps are definable.
This class is defined by translating a finite presentation of $S_n$, the symmetric group on $n$ to equations,
and postulating in addition that
the substitution operators are Boolean endomorphisms. In this case, given an abstract algebra $\A$ satisfying these equations and
$a\in A$, non zero, and $F$ {\it any} Boolean ultrafilter containing $a$,
then the map $f:\A\to \wp(S_n)$ defined by $\{\tau\in S_n: s_{\tau}a\in F\}$
defines a Boolean endomorphism such that $f(a)\neq 0$.

\begin{enumarab}

\item Now we show that the omitting types theorem holds. We use a fairly standard Baire category argument.
Each $\eta\in S_n$  is a composition of transpositions, so that $s_{\eta}$, a composition of complete endomorphisms,
is itself complete. Therefore $\prod s_{\eta}X=0$ for all $\eta\in S_n$.
Then for all $\eta\in S_n$, $B_{\eta}=\bigcap_{x\in X} N_{s_{\eta}}x$ is nowhere dense in the Stone topology
and $B=\bigcup_{\eta\in S_n} B_{\eta}$ is of the first category (in fact, $B$ is also nowhere dense, because it is only a finite union of nowhere dense sets).

Let $F$ be an ultrafilter that contains $a$ and is outside $B$. This ultrafilter exists by the celebrated Baire category theorem,
 since the complement of $B$ is dense. (Stone spaces are compact and Hausdorff). Then for all $\eta\in S_n$,
there exists $x\in X$ such that ${s}_{\tau}x\notin F$. Let $h:\A\to \wp(S_n)$ be the usual
representation function; $h(x)=\{\eta\in S_n: { s}_{\eta}x\in F\}$.
Then clearly $\bigcap_{x\in X} h(x)=\emptyset.$

\item An further, with no restriction on cardinalities, every atomic algebra is completely
representable. Indeed, let $\B$ be an atomic transposition algebra, let $X$ be the set of atoms, and
let $c\in \B$ be non-zero. Let $S$ be the Stone space of $\B$, whose underlying set consists of all Boolean ultrafilters of
$\B$. Let $X^*$ be the set of principal ultrafilters of $\B$ (those generated by the atoms).
These are isolated points in the Stone topology, and they form a dense set in the Stone topology since $\B$ is atomic.
So we have $X^*\cap T=\emptyset$ for every nowhere dense set $T$ (since principal ultrafilters,
which are isolated points in the Stone topology,
lie outside nowhere dense sets).
Recall that for $a\in \B$, $N_a$ denotes the set of all Boolean ultrafilters containing $a$.

Now  for all $\tau\in S_n$, we have
$G_{X, \tau}=S\sim \bigcup_{x\in X}N_{s_{\tau}x}$
is nowhere dense. Let $F$ be a principal ultrafilter of $S$ containing $c$.
This is possible since $\B$ is atomic, so there is an atom $x$ below $c$; just take the
ultrafilter generated by $x$. Also $F$ lies outside the $G_{X,\tau}$'s, for all $\tau\in S_n$
Define, as we did before,  $f_c$ by $f_c(b)=\{\tau\in S_n: s_{\tau}b\in F\}$.
Then clearly for every $\tau\in S_n$ there exists an atom $x$ such that $\tau\in f_c(x)$, so that $S_n=\bigcup_{x\in \At\A} f_c(x)$
Now for each $a\in A$, let
$V_a=S_n$ and let $V$ be the disjoint union of the $V_a$'s.
Then $\prod_{a\in A} \wp(V_a)\cong \wp(V)$. Define $f:\A\to \wp(V)$ by $f(x)=g[(f_ax: a\in A)]$.
Then $f: \A\to \wp(V)$ is an embedding such that
$\bigcup_{x\in \At\A}f(x)=V$. Hence $f$ is a complete representation.
\end{enumarab}
\end{example}

\subsection{The proof of an independence result partially proved in \cite{Sayed} completed}

It is stated in \cite{Sayed}, cf. theorem 3.2.8,  without proof,
that it is possible that $covK$ many non isolated types cannot be omitted in $L_n$ countable theories
whose Lindenbaum Tarski algebras belong to
$\Nr_n\CA_{\omega}$, witness theorem 3.2.8.  (This is known for usual first order logic as our example clearly manifests.)
We show, here that this also holds for $L_n$ $n>2$; this is best estimate. By accomplishing this,
we  complete the independence proof in \cite{Sayed}.

We should mention that this cardinal has many re-incarnations in the literature, example it is the least cardinal
such that the Baire Category theorem fails for compact Hausdorff second countable spaces, and it is
also the largest cardinal for which Martin's axiom restricted to countable partially ordered set holds.

\begin{example}\label{cov} Let $n>2$. To show that $OTT(cov K)$ could be false, we adapt
an example in \cite{CF} p.242. Fix $n\geq 2$. There the example is constructed for $L_{\omega,\omega}$ to adapt
to $L_n$ some care is required.
Let $T$ be a theory such that for this given $n$, in $S^n(T)$, the Stone space of $n$ types,  the isolated points are not dense.
(In \cite{CF}, a theory $T$ is chosen which does not have a prime model. This implies that there is an $n$ such that the isolated types in
$S_n(T)$ are not dense;  {\it here we need a fixed $n$, given in advance, so not {\it any} theory without a
prime model will do, for the number witnessing its primeness could be greater than $n$}.).
It is easy to construct such theories, for any fixed $n$.
(For example the theory of random graphs the isolated types are not dense for any $n$).

Let $X$ be the space $S^0(T)$ of all complete $0$ types which are consistent with $T$. For an ordinal $\alpha$, let $X^{(\alpha)}$
be the $\alpha$-iterated Cantor-Bendixson derivative of $X$. Recall that for ordinal numbers $\alpha$ the $\alpha$
Cantor-Bendixson derivative of a topological space is defined by
transfinite induction
\begin{itemize}
\item $X^0=X$
\item $X^{\alpha+1}=[X^{\alpha}]'$
\item $X^{\beta}=\bigcap_{\alpha\in \beta} X_{\alpha}$.
\end{itemize}
The transfinite sequence of Cantor-Bendixson derivatives of $X$ must eventually stop.
The smallest ordinal $\alpha$ such that $X^{\alpha+1}=X^{\alpha}$
is called the Cantor-Bendixson  rank of $X$.The language is countable, there is some $\alpha<\omega_1$
such that $X^{(\alpha)}=X^{(\alpha+1)}$
and $X\setminus X^{\alpha}$ is countable. $X^{\alpha}$ is a perfect set and
therefore it is homeomorphic to the Cantor space $^{\omega}2$ or it is empty.
We associate a set $P_{\infty}$ of $\leq covK$ many types with $X^{\alpha}$. Assume that $X^{\alpha}$ is non-empty, since it is a
closed set in $X$,
there is some extension $T_{\infty}$ of $T$
such that in $X$
$$X^{\alpha}=\bigcap_{\sigma\in T^{\infty}}[\sigma].$$
Hence the space $S^0(T_{\infty})$ is homeomorphic to $X^{\alpha}$ and to $^{\omega}2$. Then there are $Y_{\beta} (\beta<covK)$
closed nowhere dense sets in $S^0(T_{\infty})$
such that
$$S^0(T_{\infty})=\bigcup_{\beta<covK}Y_{\beta}.$$type $p_{\beta}$ such that in $S^{0}(T_{\infty})$
$$Y_{\beta}=\bigcap_{\sigma\in p_{\beta}}[\sigma].$$
As $Y_{\beta}$ is nowhere dense $p_{\beta}$ is non principal in $T_{\infty}$. Assuming, without loss, that $T_{\infty}\subseteq p_{\beta}$ we get that
$p_{\beta}$ is non
principal in $T$.
Set
$$P_{\infty}=\{p_{\beta}:\beta<covK\}.$$
Let us consider the $0$ types in $X\setminus X^{\alpha}$. These are complete consistent extensions of $T$.
For every $T'\in X\setminus X^{\alpha}$ we shall define a set $P_{T'}$ of $\leq covK$ many $n$ types that are not omitted in $T'$.
If $T'$ is not a finite extension of $T$, set $P_{T'}=\{T'\}$. Otherwise, in $S^n(T')$ the isolated types are not dense.
Hence there is some non-empty $Y\subseteq
S^n(T')$ clopen and perfect. Now we can cover $Y$ with a family of $covK$ many closed nowhere dense sets of $n$ types.
Since $Y$ is clopen in $S^{n}(T')$, these sets are closed nowhere dense sets in $S^{n}(T')$,
so we obtain a family of $covK$ many non principal $n$ types that cannot be omitted. We may assume that
$T'\subseteq p$ for every $p\in P_T'$ and therefore every type in $P_T'$ is non principal in $T$.
Define
$$P=P_{\infty}\cup\bigcup\{P_{T'}: T'\in X\setminus X^{\alpha}\}.$$
Now $P$ is a family of non-principal types $|P|=covK$ that cannot be omitted.

Let $\A=\Fm/T$ and for $p\in P$ let $X_p=\{\phi/T:\phi\in p\}$. Then $X_p\subseteq \Nr_n\A$, and $\prod X_p=0$.
However for any $0\neq a$, there is no set algebra $\C$ with countable base
$M$ and $g:\A\to \C$ such that $g(a)\neq 0$ and $\bigcap_{x\in X_i}f(x)=\emptyset$.

But in principle, if we take the  neat $n$ reduct, representations preserving meets can exist.
We exclude this possibility by showing that
such representations necessarily lift to
all $\omega$ dimensions.

Let $\B=\Nr_n\A$. Let $a\neq 0$. Assume, seeking a contradiction, that there exists
$f:\B\to \D'$ such that $f(a)\neq 0$ and $\bigcap_{x\in X_i} f(x)=\emptyset$. We can assume that
$B$ generates $\A$ and that $\D'=\Nr_n\D$ where $\D\in \Lf_{\omega}$. Let $g=\Sg^{\A\times \D}f$. We will show
that $g$ is a one to one function with domain $\A$ that preserves the $X_i$'s which is impossible (Note that by definition $g$ is a homomorphism).
We have
$$\dom g=\dom\Sg^{\A\times \D}f=\Sg^{\A}\dom=\Sg^{\A}\Nr_{n}\A=\A.$$
By symmetry it is enough to show that $g$ is a function.  We first prove the following (*)
 $$ \text { If } (a,b)\in g\text { and }  {\sf c}_k(a,b)=(a,b)\text { for all } k\in \omega\sim n, \text { then } f(a)=b.$$
Indeed,
$$(a,b)\in \Nr_{n}\Sg^{\A\times \D}f=\Sg^{\Nr_{n}(\A\times \D)}f=\Sg^{\Nr_{n}\A\times \Nr_{n}\D}f=f.$$
Here we are using that $\A\times \D\in \Lf_{\omega}$, so that  $\Nr_{n}\Sg^{\A\times \D}f=\Sg^{\Nr_{n}(\A\times \D)}f.$
Now suppose that $(x,y), (x,z)\in g$.
Let $k\in \omega\sim n.$ Let $\Delta$ denote symmetric difference. Then
$$(0, {\sf c}_k(y\Delta z))=({\sf c}_k0, {\sf c}_k(y\Delta z))={\sf c}_k(0,y\Delta z)={\sf c}_k((x,y)\Delta(x,z))\in g.$$

Also,
$${\sf c}_k(0, {\sf c}_k(y\Delta z))=(0,{\sf c}_k(y\Delta z)).$$
Thus by (*) we have  $$f(0)={\sf c}_k(y\Delta z) \text { for any } k\in \omega\sim n.$$
Hence ${\sf c}_k(y\Delta z)=0$ and so $y=z$.
We conclude that there exists a countable $\B\in \Nr_n\CA_{\omega}$ and $(X_i:i<covK)$ such that $\prod X_i=0$ but there is no representation that
preserves the $X_i$'s. In more detail. Give any $a\in \B$, if $a$ is non zero, $\C$ is a set algebra with countable base
and $f:\B\to \C$ is a homomorphism such
that $f(a)\neq 0$, then there exists $i<covK$, such that $\bigcap_{x\in X_i} f(x)\neq \emptyset.$
Therefore $OTT$ is false in a model of $ZFC+\neg CH$.
This example finishes the independence proof mentioned in \cite{Sayed}.

\end{example}

\subsection{Vaught's theorem, in other contexts}

A classical theorem of Vaught, that is an immediate consequence of the Orey- Henkin omitting types theorem $(OTT)$,
says that every atomic complete theory has an atomic model. Algebraically, if
if $\A\in \Lf_{\omega}$, such that $\Nr_n\A$ is atomic for every $n\in \omega$ and $a\in A$ is non-zero,
then there is a set algebra  $\B$ with unit $V$, a homomorphism
from $\A\to \wp(V)$ such that $f(a)\neq 0$ and
$\bigcup\{f(x): x\in \At\Nr_n\A\}=V$ for every $n\in \omega$.
Roughly, $f$ {\it restricted} to all $n$ reducts is an atomic representation.
One applies, then,  the omitting types theorem by finding a countable model
that omits the co-atoms in every $n$ neat reduct.

When  we have an $OTT$ we have a Vaught's theorem but the converse may not be true, and it does happen.
There are contexts where Vaught theorem holds but
$OTT$ fails; indeed this is the case for first order logic restricted to the first two variables.
The point is that Vaught's conjecture is {\it a special case} of $OTT$ when the types required
to be omitted are the co-atoms.

Our first example addresses $\L_2$, first order logic restricted to two variable.
In this case, we have a Vaught's theorem regardless of cardinalities. This was proved in \cite{Khaled}; here
we give a different proof inspired by duality theory in modal logic applied to $\RCA_2$ viewed as subdirect product of complex
algebras of square Kripke frames.

\begin{theorem}$\L_2$ enjoys a Vaught's theorem, but not the omitting types theorem.
\end{theorem}

\begin{proof} For the first part, it suffices to show that every atomic $\RCA_2$  is completely representable.
Here we give a nice proof of Hirsch and Hodkinson's \cite{HHbook2}. Let $\A\in \RCA_2$;
the latter variety is conjugated and defined by Sahlqvist
equations, hence is closed under \d\ completions. The first order correspondents of such equations
are valid over the atom structure $\At\A$, which is a bounded
morphic image of a disjoint union of square frames $\F_i$.
By duality this induces  a bounded morphism from $\A$ into $\prod \F_i$.
which is an atomic, hence a complete representation.

The second part is an unpublished result of Andr\'eka and N\'emeti.
\end{proof}

The idea of guarded fragments proposed by Andr\'eka, van Benthem and N\'emeti.
is to look at quantification patterns. Only relativized quantification along the accessibility relation of the
Kripke frame is allowed in modal formulas; so
in the guarded fragment all quantification must be relativized to
an atomic formula. This atomic formula is called a {\it guard}.  This fragment was inspired by investigating relativized set
algebras, in which cylindrifiers are only restricted to the unit,
which may not be a Tarskian square, so cylindrifiers may not commute. These fragments were proved decidable (using mosaics)
and  have the Los Tarski theorem, and interpolation.

Next we address certain guarded fragments of $L_n$, reflected algebraically by
transposition algebras; and the guard's role  is relativize to the units of algebras; which gives a certain natural
class of algebras referred to as $\sf G_n$; the units are locally cube they are closed under (finite)
substitutions.
The representation of such algebras is proved by Ferenczi \cite{Fer} and a different proof is given by the
author and Mohamed Khaled; using games, which we outline here. The
{\it $n$ variable
polyadic logic} corresponding to such polyadic equality like
algebras show very nice modal behaviour.
Besides completeness proved by Ferenczi, it has the finite base property.
This means that if a finite theory has a model, then it has a finite one; the satisfiability problem is
thus decidable.

The technique uses combinatorial methods of Herwig extending partial isomorphism on a finite structure
to automorphisms on a larger finite structure whose size is controlled by the small structure.
Decidability is proved similar
to filtration in modal logic.

\begin{definition}[Class $\TEA_{\alpha}$]

A transposition equality algebra of dimension $\alpha$ is
an algebra$$\A=\langle A, +, \cdot, -, 0, 1, \c_i, \s^{i}_{j}, \s_{ij}, \sf d_{ij}\rangle_{i,j\in\alpha},$$
where $\c_i, \s^{i}_{j}, \s_{ij}$ are unary operations,
$\d_{ij}$ are constants, the axioms ($F_0$)-($F_9$) below are valid for every $i,j,k<\alpha$:
\begin{description}
\item[$(Fe_0)$] $\langle A, +, \cdot, -, 0, 1\rangle$ is a boolean algebra, $\s^i_i=\s_{ii}=\sf d_{ii}=Id\upharpoonright A$ and $\s_{ij}=\s_{ji}$,
\item[$(Fe_1)$] $x\leq\c_i x$,
\item[$(Fe_2)$] $\c_i(x+y)=\c_i x+\c_i y$,
\item[$(Fe_3)$] $\s^{i}_{j}\c_i x=\c_i x$,
\item[$(Fe_4)$] $\c_i\s^{i}_{j}x=\s^{i}_{j}x$, $i\not=j$,
\item[$(Fe_5)^*$] $\s^{i}_{j}\s^{k}_{m}x=\s^{k}_{m}\s^{i}_{j}x$ if $i,j\notin\{k,m\}$,
\item[$(Fe_6)$] $\s^{i}_{j}$ and $\s_{ij}$ are boolean endomorphisms,
\item[$(Fe_7)$] $\s_{ij}\s_{ij}x=x$,
\item[$(Fe_8)$] $\s_{ij}\s_{ik}x=\s_{jk}\s_{ij}x$, $i,j,k$ are distinct,
\item[$(Fe_9)$] $\s_{ij}\s^{i}_{j}x=\s^{j}_{i}x$,
\item[$(Fe_{10})$] $\s^{i}_{j}\sf d_{ij}=1$,
\item[$(Fe_{11})$] $x\cdot\sf d_{ij}\leq\s^{i}_{j}x$.
\end{description}
\end{definition}

\begin{definition}
\begin{itemize}
\item Let $\A\in \TEA_{n}$. A relativized $\A$ pre-network is a pair $N=(N_1, N_2)$
where $N_1$ is a finite set of nodes $N_2:N_1^n\to \A$ is a partial map, such that if $f\in \dom N_2$,
and $i,j<n$ then $f_{f(j)}^i\in \dom N_2$. $N$ is atomic if $\rng N\subseteq \At\A$.
We write $N$ for any of $N, N_1, N_2$ relying on context, we write $\nodes(N)$ for $N_1$ and $\edges(N)$ for $\dom(N_2)$.
$N$ is said to be a network if
\begin{enumerate}
\item for all $\bar{x}\in \edges(N)$, we have $N(\bar{x})\leq \sf d_{ij}$ iff $x_i=x_j$.
\item if $\bar{x}\equiv_i \bar{y}$, then $N(\bar{x})\cdot \c_i N (\bar{y})\neq 0.$
\end{enumerate}
\item Let $\A\in \TEA_{n}$. A relativized $\A$ pre-network is a pair $N=(N_1, N_2)$
where $N_1$ is a finite set of nodes $N_2:N_1^n\to \A$ is a partial map, such that if $f\in domN_2$,
and $\tau$ is a finite transformation then $\tau| f\in \dom N_2$. Again $N$ is atomic if $\rng N\subseteq \At\A$.
Also we write $N$ for any of $N, N_1, N_2$ relying on context, we write $nodes(N)$ for $N_1$ and $edges(N)$ for $dom(N_2)$.
$N$ is said to be a network if
\begin{enumerate}
\item for all $\bar{x}\in \edges(N)$, we have $N(\bar{x})\leq \sf d_{ij}$ iff $x_i=x_j$,
\item if $\bar{x},\bar{y}\in \edges(N)$ and $\bar{x}\equiv_i \bar{y}$, then $N(\bar{x})\cdot \c_i N (\bar{y})\neq 0$,
\item $N([i,j]|\bar{x})=\s_{[i,j]}N(\bar{x})$, for all $\bar{x}\in \edges(N)$ and all $i,j<n$.
\end{enumerate}
\end{itemize}
\end{definition}

\begin{definition}Let $\A\in \TEA_n$. We define a game denoted by
$G_{\omega}(\A)$ with $\omega$ rounds, in which the players $\forall$ (male) and $\exists$ (female)
build an infinite chain of relativized $\A$ pre-networks
$$\emptyset=N_0\subseteq N_1\subseteq \ldots.$$
In round $t$, $t<\omega$, assume that $N_t$ is the current prenetwork, the players move as follows:

\begin{enumerate}
\item $\forall$ chooses a non-zero element $a\in \A$, $\exists$
must respond with a relativized prenetwork $N_{t+1}\supseteq N_t$ containing an edge $e$ with
$N_{t+1}(e)\leq a$,
\item $\forall$ chooses an edge $\bar{x}$ of $N_t$ and an element $a\in \A$. $\exists$ must respond with a pre-network
$N_{t+1}\supseteq N_t$ such that either $N_{t+1}(\bar{x})\leq a$ or $N_{t+1}(\bar{x})\leq -a$,
\item or $\forall$ may choose an edge $\bar{x}$ of $N_t$ an index $i<n$ and $b\in \A$ with $N_t(\bar{x})\leq {\sf c}_ib$.
$\exists$ must respond with a prenetwork $N_{t+1}\supseteq N_t$ such that for some $z\in N_{t+1},$
$N_{t+1}(\bar{x}^i_z)=b$.
\end{enumerate}
$\exists$ wins if each relativized pre-network $N_0,N_1,\ldots$
played during the game is actually a relativized network.

Otherwise, $\forall$ wins. There are no draws.
\end{definition}

${\sf G_n}$ denotes the class of set algebra of dimension $n$ whose units are locally cube, that is
closed under
substitutions.

\begin{theorem} Let $\A\in \TEA_n$. Then \pe\ can win any play of $G_{\omega}(\A)$.
Hence, if $\A\in \TEA_{n}$, then $\A\in {\sf IG}_{n}$.
\end{theorem}
\begin{proof}
 Let $\A\in \TEA_n$. We want to build an isomorphism from $\A$ to some $\B\in D_{n}$. Consider a play
$N_0\subseteq N_1\subseteq \ldots $ of $G_{\omega}(\A)$ in which $\exists$ plays as in the previous lemma
and $\forall$ plays every possible legal move.
The outcome of the play is essentially a relativized representation of $\A$ defined as follows.
Let $N=\bigcup_{t<\omega}\nodes(N_t)$, and $\edges(N)=\bigcup_{t<\omega}\edges(N_t)\subseteq {}^nN$.
By the definition of the networks, $\wp(\edges(N))\in D_{n}$. We make $N$ into a representation by
defining $h:\A\rightarrow\wp(\edges(N))$ as follows
$$h(a)=\{\bar{x}\in \edges(N): \exists t<\omega(\bar{x}\in N_t \& N_t(\bar{x})\leq a)\}.$$
$\forall$-moves of the second kind guarantee that for any $n$-tuple $\bar{x}$
and any $a\in\A$, for sufficiently large $t$ we have either $N_{t}(\bar{x})\leq a$ or $N_{t}(\bar{x})\leq -a$.
This ensures that $h$ preserves the boolean operations.
$\forall$-moves of the third kind ensure that the cylindrifications are respected by $h$.
Preserving diagonals follows from the definition of networks.
The first kind of $\forall$-moves tell us that $h$ is one-one.
But the construction of the game under consideration ensures that $h$ is onto, too.
In fact $\bar{h}$ is a representation from $\A$ onto $\B\in D_n$. This follows from the definition of networks.
\end{proof}

\begin{corollary} Every atomic $\TEA_n$ is completely representable
\end{corollary}

\begin{proof} This can be easily discerned from the proof of the representability result using games.
However, a more direct proof, and indeed far easier proof,  can be obtained by playing {\it atomic} games,
like those that test complete representability in cylindric algebra
\cite{HHbook2}, where networks are labelled by atoms (as opposed to arbitrary elements), and \pa\ has only
one move, namely, the cylindrifier move, the third move above. Unlike ordinary game \pe\ does not have
the option to accept or reject. in response to \pa\ move, she has to choose a new network
whose edges are labelled appropriately by atoms according to the rules
of the games. \pe\ essentially uses his above \ws\ by discarding responses to
other moves.
\end{proof}

\begin{theorem} $\TEA_n$ has the finite algebra on finite base property, in particular if $T$ is a finite atomic theory then it has a finite
atomic model.
\end{theorem}

\begin{proof} Let $\A\in \TEA_n$ be finite, then $\A=\Fm_T$, where $T$ is a finite first order theory in the finite
language $L(\A)$. This is obtained by adding an $n$ ary predicate for every element of $\A$.
A model  $\M$ of $T$ is a representation of $\A$, and we know that one exists.
But $\bigwedge T$ can be coded in the loosely guarded fragment of $L_n$
(this can be proved by noting that the first order correspondents of the positive equations
axiomatizing $TA_n$ are are equivalent to
formulas in the loosely guarded fragment), hence
if $\bigwedge T$ has a model,
then it has a finite model by a result of Herwig.

A more basic proof \cite{AHN}.
Let $\F$ be any finite set of terms in the language of $\TA_n$.
Let $\A\in \TA_n$ and let $u$ be an assignment of the variables
in $\F$ to elements of $\A$.
It suffices to  construct an algebra
$\B$ with finite base and an assignment $v$ of the same variables to elements of $\B$, such that
$$\A\models \tau=\sigma[u]\Longleftrightarrow \B\models \tau=\sigma[v].$$
Let $U$ be the base of $\A$ and let $Q$ be an arbitrary finite subset
of $U$. Let $\F^*$ be the relational language containing
an $n$ predicate  for every $\tau\in \F$
Now let $\F^*$ be the first order structure with
$U$ at its base.
For any $s\in {}^nU$, define $\U(s)$ a finite substructure with domain $Q\cup \rng s$, expanded with $n$ constants.
The structure $\U(s)$ have at most $|Q|+n$ elements and their language is finite.
Hence it is possible to find  a finite $\K\subseteq U$ such that for all $s\in {}^nU$,
there is a $t\in K^n$ with $\U(t)\cong \U(s)$.

Let $\R$ be the substructure of $U$ with domain $K$.
It has a finite relational signature, so we obtain by Herwig's theorem a finite $F^*$ structure $\R^+$
such that any partial isomorphism of $\R$ is
induced by an automorphism of $\R^+$.

Let $G$ be the group of automorphisms of $\R^+$ that fix $Q$ pointwise.
And define the unit of our algebra
$H=\{sg: s\in 1: g\in G\}.$
We now obtain an algebra with universe
$\wp(H)$. Now $H$ is an arbitrary set of sequences.
We want to make it locally square
Set $\eta(x)$ to be $x=x\land \bigwedge _{\sigma:n\to n}{\sf s}_{\sigma}1$, $\tau(x.z):=x.z$, $\tau_-(x):=-x $and
$\tau_{\c_i}x:=-\c_ix$, etc.
Now  we have the formula $\eta$, and terms $\tau_1,\ldots \tau_{m-1}.$
For a tuple $\bar{a}$,
let $\eta^{\A}(\bar{a})=\{b\in A: \A\models \eta(\bar{a},b)\}$, this is closed under
$\tau_i^{\A}(\bar{a})$.

We obtain an algebra $\A_{\bar{a}}$ in the signature $f_1,\ldots f_k$ of $\TEA_n$.
Its domain is $A_{\bar{a}}$ is  $\eta^{\A}(\bar{a})$ and the function
symbols in the signature of $\TEA_n$ are interpreted as restrictions
to $A_{\bar{a}}$ of the functions defined by $\tau_i^{\A}(\bar{a}$, for every term as defined
above.   In the terminology of \cite{ANT}, $\TEA_n$ is a $\chi=(\eta, \tau_i)$  subreduct of ${\sf Crs}_n$.

Let $\psi(x)$ be a quantifier
free formula in the language of $\TEA_n$ and assume that $\A'$ falsifies
$\forall \bar{x}\psi(\bar{x})$ for some some $\A'\in \TEA_n$. Let
Let $\A\in {\sf Crs_n}$ and $\bar{a}$ an $m$ tuple of
elements of $\A$,
$\A'\subseteq \A_{\bar{a}}$. Assume that $\eta^{\A}(a)$ is not empty, otherwise there is nothing more to prove.

$$\A\models \exists \bar{y} \exists \bar{x} \land \bigwedge_{1\leq i\le {\sf length}(\bar{x}} \eta(\bar{y}, x_i)\land \neg \psi').$$
Hence there is a $\B\in {\sf Crs}_n$ with finite base
in which the same formula holds
Hence, $\forall{\bar{x}}$ fails in the  subreduct
of $\B$.
\end{proof}

We obtain \cite{AHN}:
\begin{corollary}\label{Khaled}
Let $\A\in  \TEA_n$, let $\psi(\bar{x})$ be a quantifier free formula of the signature
of $\TEA_n$ and let $\bar{a}\in \A$, with $|\bar{a}|=|x|$. Then there is a loosely guarded sentence
$\tau(\psi(a))$ such that if $M$ is a representation of $\A$,
then $\A\models \psi(\bar{a})$ iff $M\models \tau_{\A}(\psi(\bar{a)})$.
Every universal formula valid in algebras having a finite representation,
is valid in $\TEA_n$. Hence, the universal first order theory of
$\TEA_n$ is decidable.
\end{corollary}
We now state an easy theorem that says that an atomic representation omits any non principal type.
This is the analogue of Vaught's theorem that atomic countable theories, have countable atomic models,
countable atomic models are prime, and that a prime models omit any non family of principal type (regardless of their number).
Also it is exactly the Hirsch and  Hodkinsdon's
result that says that an atomic representation is necessarily a complete one (preserves arbitrary suprema).

\begin{theorem} Let $f:\A\to \wp(V)$ be an atomic representation of $\A\in \CA_n$.
Then for any given family $(Y_i:i\in I)$ of subsets of $\A$,  if $\prod Y_i=0$ for all $i\in I$, then we have
$\bigcap_{y\in Y_i} f(y)=\emptyset$ for all $i\in I$.
\end{theorem}
\begin{proof} Let $i\in I$. Let $Z_i=\{-y: y\in Y_i\}$. Then $\sum Z_i=1$. Let $x$ be an atom. Then $x.\sum Z_i=x\neq 0$.
Hence there exists $z\in Z_i$, such that
$x.z\neq 0$. But $x$ is an atom, hence $x.z=x$ and so $x\leq z$. We have shown that for every atom $x$, there exists $z\in Z_i$ such that $x\leq z$.
It follows immediately that
$V=\bigcup_{x\in \At\A}f(x)\leq \bigcup_{z\in Z_i} f(z)$, and so $\bigcap_{y\in Y_i} f(y)=\emptyset,$
and we are done.
\end{proof}

\subsection {The infinite dimensional case}

It was proved by the author that Vaught's theorem holds for various reducts of polyadic algebras,
including polyadic algebras, themselves, the reducts studied by Sain and cylindric polyadic algebras obtained from $\PA$ by allowing only
finite cylindrifiers, under the condition of complete additivity, which is also necessary, see both \cite{Sayed} and \cite{Sayedneat}
for an overview.

When we have diagonal the situation is problematic.
Nevertheless,  when we relativize semantics (weaken commutativity of cylindrifiers),
then we get that for the resulting cylindric-polyadic {\it equality} algebras,
atomic algebras are completely representable.
In this context, due to the presence of diagonal elements, the operations
are completely additive anyway, so this condition is superfluous.

We show that the class of atomic cylindric-polyadic algebras, introduced also by Ferenczi, like Boolean algebras,
coincides with the class of completely representable
ones. We consider the class ${\CPEA}_{\alpha}$, introduced and studied by Ferenczi,
whose similarity type is the restriction of polyadic equality type to finite cylindrifiers.

However, the axioms are obtained by weakening commutativity of cylindrifiers.
Ferenczi \cite{Fer} proved a strong representation theorem for such a class. Our next theorem proves more.
Indeed, if $\A$ is such an algebra then $\CPEA_{\alpha}$ defined by Sahlqvist equations is canonical,
hence $\A^+$ is $\sf CPEA_{\alpha}$; furthermore it is atomic; by our proof
it has a {\it complete (atomic)  representation},
hence $\A$ is representable. In more detail, we have:

\begin{theorem} For any infinite $\alpha$,
Every atomic ${\sf CPEA}_{\alpha}$ is completely representable
\end{theorem}
\begin{proof}

Compare with example above; the proof also uses the topological fact that principal ultrafilters lie outside
nowhere dense sets. Let $c\in A$ be non-zero. We will find a $\B\in \sf Gp_{\alpha}$ and a homomorphism
from $f:\A \to \B$ that preserves arbitrary suprema
whenever they exist and also satisfies that  $f(c)\neq 0$.
Now there exists $\B\in \sf CPEA_{\mathfrak{n}}$, $\mathfrak{n}$ a regular cardinal.
such that $\A\subseteq \Nr_{\alpha}\B$ and $A$ generates $\B$.
Note that  $|\mathfrak{n}\sim \alpha|=|\mathfrak{n}|$ and

This dilation also has Boolean reduct isomorphic to $F({}^\mathfrak{n}\alpha, \A)$, in particular, it is atomic because $\A$ is atomic.
Also cylindrifiers are defined on this minimal functional dilation exactly like above by restricting to singletons.
Let $adm$ be  the set of admissible substitutions. $\tau\in \B$ is admissible if
$\dom\tau\subseteq \alpha$ and $\rng\tau\cap \alpha=\emptyset$.
Then we have
for all $i< \mathfrak{n}$ and $\sigma\in adm$,
\begin{equation}\label{tarek1}
\begin{split}
s_{\sigma}{\sf c}_{i}p=\sum {\sf s}_{\sigma}{\sf s}_i^jp
\end{split}
\end{equation}
This uses that ${\sf c}_k=\sum {\sf s}_k^i x$, which is proved like the cylindric case; the proof depends on diagonal elements.
Let $X$ be the set of atoms of $\A$. Since $\A$ is atomic, then  $\sum^{\A} X=1$. By $\A=\Nr_{\alpha}\B$, we also have $\sum^{\B}X=1$.
Because substitutions are completely additive we have
for all $\tau\in {}^{\alpha}\mathfrak{n}$
\begin{equation}\label{tarek2}
\begin{split}
\sum {\sf s}_{\bar{\tau}}^{\B}X=1.
\end{split}
\end{equation}
Let $S$ be the Stone space of $\B$, whose underlying set consists of all boolean ultrafilters of
$\B$, and let $F$ be a principal ultrafilter chosen as before.
Let $\B'$ be the minimal completion of $\B$. Exists by complete additivity. Take the filter $G$ in $\B'$ generated by the generator of $F$
and let $F=G\cap \B$. Then $F$ is a perfect ultrafilter. Because our algebras have diagonal algebras,
we have to factor our base by a congruence relation that reflects
equality.
Define an equivalence relation on $\Gamma=\{i\in \beta:\exists j\in \alpha: {\sf c}_i{\sf d}_{ij}\in F\}$,
via $m\sim n$ iff ${\sf d}_{mn}\in F.$ Then $\Gamma\subset \alpha$ and
the desired representation is defined on a $Gp_{\alpha}$ with base
$\Gamma/\sim$. We omit the details.
\end{proof}

For cylindric algebras and quasi polyadic equality algebras ${\sf QEA_{\alpha}}$, for countable infinite dimension,
we have the following omitting types theorem,
where we also have to relativize
semantics but in different way, we do not get out of $\RCA_n$.

The unit of the algebra omitting the non principle types will be a weak space, that is, one of the form $^{\omega}U^{(p)}$
which is the set of $\omega$-ary sequences that agree cofinitely
with the fixed beforehand sequence $p\in {}^{\omega}U$. Such algebras are cylindric algebras indeed
they are representable. The theorem holds for any countable ordinal (with
the same proof.)

\begin{theorem} Let $\A\in S_c\Nr_{\omega}\CA_{\omega+\omega}$. Let $(X_i:i<\lambda)$, where $\lambda<covK$,
be a family of non principal types.
Then there is a weak set algebra that omits theses types.
In particular, if $\A$ is atomic, then $\A$ is completely representable on weak units.
\end{theorem}
\begin{proof} \cite{Sayed}
\end{proof}

Now we give an example showing that the condition of countability condition cannot be omitted.
We use a lifting argument implemented via ultraproducts.

\begin{example}$\alpha$ be an infinite ordinal.
Then there exists $\B\in \Nr_{\alpha}\CA_{\alpha+\omega}$ that is atomic, uncountable and not completely
representable.
Let $\C(k)\in \Nr_k\CA_{\omega}$ be an atomic countable and not completely representable.
That such algebras exist for every finite $k\geq 3$ will be proved below, witness item 2 of theorem
\ref{Robinsexample}.

Let $I=\{\Gamma: \Gamma\subseteq \alpha,  |\Gamma|<\omega\}$.
For each $\Gamma\in I$, let $M_{\Gamma}=\{\Delta\in I: \Gamma\subseteq \Delta\}$,
and let $F$ be an ultrafilter on $I$ such that $\forall\Gamma\in I,\; M_{\Gamma}\in F$.
For each $\Gamma\in I$, let $\rho_{\Gamma}$
be a one to one function from $|\Gamma|$ onto $\Gamma.$
Let ${\C}_{\Gamma}$ be an algebra similar to $\CA_{\alpha}$ such that
$\Rd^{\rho_\Gamma}{\C}_{\Gamma}={\C}(|\Gamma|)$. In particular, ${\C}_{\Gamma}$ has an atomic Boolean reduct.
Let
$\B=\prod_{\Gamma/F\in I}\C_{\Gamma}$
We will prove that
$\B\in \Nr_\alpha\CA_{\alpha+\omega},$  $\B$ is atomic and $\B$ is not completely representable. The last two requirements are
easy. $\B$ is atomic, because it is an ultraproduct of atomic algebras.
$\B$ is not completely representable, even on weak units, because, $\Nr_n\B={\C}(n)$, and so  such a representation induces
a complete (square) representation of its $n$ neat reducts, $n\geq 3.$
For the first part, for each $\Gamma\in I$, we know that $\C(|\Gamma|+k) \in\CA_{|\Gamma|+k}$ and
$\Nr_{|\Gamma|}\C(|\Gamma|+k)\cong\C(|\Gamma|)$.
Let $\sigma_{\Gamma}$ be a one to one function
 $(|\Gamma|+\omega)\rightarrow(\alpha+\omega)$ such that $\rho_{\Gamma}\subseteq \sigma_{\Gamma}$
and $\sigma_{\Gamma}(|\Gamma|+i)=\alpha+i$ for every $i<\omega$. Let
$\A_{\Gamma}$ be an algebra similar to a
$\CA_{\alpha+\omega}$ such that
$\Rd^{\sigma_\Gamma}\A_{\Gamma}=\C(|\Gamma|+k)$.  Then $\Pi_{\Gamma/F}\A_{\Gamma}\in \CA_{\alpha+\omega}$.

We prove that $\B= \Nr_\alpha\Pi_{\Gamma/F}\A_\Gamma$.
Recall that $\B^r=\Pi_{\Gamma/F}\C^r_\Gamma$ and note
that $\C^r_{\Gamma}\subseteq A_{\Gamma}$
(the base of $\C^r_\Gamma$ is $C(|\Gamma|)$, the base of $\A_\Gamma$ is $C(|\Gamma|+k)$).
Now using the fact that neat reducts commute with forming ultraproducts, so, for each $\Gamma\in I$,
\begin{align*}
\Rd^{\rho_{\Gamma}}\C_{\Gamma}^r&=\C((|\Gamma|)\\
&\cong\Nr_{|\Gamma|}\C(|\Gamma|+k)\\
&=\Nr_{|\Gamma|}\Rd^{\sigma_{\Gamma}}\A_{\Gamma}\\
&=\Rd^{\sigma_\Gamma}\Nr_\Gamma\A_\Gamma\\
&=\Rd^{\rho_\Gamma}\Nr_\Gamma\A_\Gamma
\end{align*}
We deduce that
$\Pi_{\Gamma/F}\C^r_\Gamma\cong\Pi_{\Gamma/F}\Nr_\Gamma\A_\Gamma=\Nr_\alpha\Pi_{\Gamma/F}\A_\Gamma$.
\end{example}

Squaring the unit is also problematic as the following example taken from \cite{HH}
illustrates.

\begin{example}

There are countable atomic representable $\CA$s that are not completely representable on square units.
For $\alpha\geq \omega$, the class of square completely representable algebras is not elementary.
In particular, there is an algebra that is completely representable, but not square completely representable.
Let $\C\in SA_{\alpha}$ such that $\C\models d_{01}<1$.  Such algebras exist, for example one can take $\C$ to be
$\wp(^{\alpha}2).$ Assume that $f: \C\to \wp(^{\alpha}X)$ is a square  complete representation.
Since $\C\models d_{01}<1$, there is $s\in h(-d_{01})$ so that if $x=s_0$ and $y=s_1$, we have
$x\neq y$. For any $S\subseteq \alpha$ such that $0\in S$, set $a_S$ to be the sequence with
$ith$ coordinate is $x$, if $i\in S$ and $y$ if $i\in \alpha\sim S$.
By complete representability every $a_S$ is in $h(1)$ and so in
$h(\mu)$ for some unique atom $\mu$.

Let $S, S'\subseteq \alpha$ be distinct and assume each contains $0$.   Then there exists
$i<\alpha$ such that $i\in S$, and $i\notin S'$. So $a_S\in h(d_{01})$ and
$a_S'\in h (-d_{01}).$ Therefore atoms corresponding to different $a_S$'s are distinct.
Hence the number of atoms is equal to the number of subsets of $\alpha$ that contain $0$, so it is at least $^{|\alpha|}2$.
Now using the downward Lowenheim Skolem Tarski theorem, take an elementary substructure $\B$ of $\C$ with $|\B|\leq |\alpha|.$
Then in $\B$ we have $\B\models d_{01}<1$. But $\B$ has at most $|\alpha|$ atoms, and so $\B$ cannot be {\it square}
completely representable
though it is completely representable on a weak unit).

\end{example}

\section*{Part Two}

\section{Games and neat embeddings}

We will ultimately show that the omitting types theorem fails
for finite variable fragments of first order logic with $n\geq 3$,
even if we ask only for `$n$ guarded models' omitting single non principle ones in countable languages.
There may not be one. In $n$ guarded semantics the witness for cylindrifiers can  only found on $<n$ cliques.
We use rainbow constructions for cylindric algebras.

Our metalogical result will follow from the algebraic result to be proved below
that there is a countable $\PEA_n$ that is elementary equivalent to a completely representable one but
its $\Sc$ reduct is
not even in $S_c\Nr_n\Sc_{n+3}$.

This is the cylindric-analogue of a relation algebra result proved by Hirsch \cite{r}.
In this section we do some of the required algebra.
Later we do the logic.

We need some preparing to do:
\begin{definition}\label{subs}
Let $n$ be an ordinal. An $s$ word is a finite string of substitutions $({\sf s}_i^j)$,
a $c$ word is a finite string of cylindrifications $({\sf c}_k)$.
An $sc$ word is a finite string of substitutions and cylindrifications
Any $sc$ word $w$ induces a partial map $\hat{w}:n\to n$
by
\begin{itemize}

\item $\hat{\epsilon}=Id$

\item $\widehat{w_j^i}=\hat{w}\circ [i|j]$

\item $\widehat{w{\sf c}_i}= \hat{w}\upharpoonright(n\sim \{i\}$

\end{itemize}
\end{definition}

If $\bar a\in {}^{<n-1}n$, we write ${\sf s}_{\bar a}$, or more frequently
${\sf s}_{a_0\ldots a_{k-1}}$, where $k=|\bar a|$,
for an an arbitrary chosen $sc$ word $w$
such that $\hat{w}=\bar a.$
$w$  exists and does not
depend on $w$ by \cite[definition~5.23 ~lemma 13.29]{HHbook}.
We can, and will assume \cite[Lemma 13.29]{HHbook}
that $w=s{\sf c}_{n-1}{\sf c}_n.$
[In the notation of \cite[definition~5.23,~lemma~13.29]{HHbook},
$\widehat{s_{ijk}}$ for example is the function $n\to n$ taking $0$ to $i,$
$1$ to $j$ and $2$ to $k$, and fixing all $l\in n\setminus\set{i, j,k}$.]
The following is the $\CA$ analogue of \cite[lemma~19]{r}.
In the next definition we extend the definition of atomic cylindric networks to polyadic ones.
For diagonal free reducts, the definition is modified the obvious way. For example for $\Df$s, we only have the
first condition, for $\sf Sc$s we do not have the second,
and the fourth ${\sf s}_{[i,j]}$ and $[i,j]$ are  replaced, respectively, by
${\sf s}_i^j$ and $[i|j].$

Let $\delta$ be a map. Then $\delta[i\to d]$ is defined as follows. $\delta[i\to d](x)=\delta(x)$
if $x\neq i$ and $\delta[i\to d](i)=d$. We write $\delta_i^j$ for $\delta[i\to \delta_j]$.

\begin{definition}
From now on let $2< n<\omega.$ Let $\C$ be an atomic ${\sf PEA}_{n}$.
An \emph{atomic  network} over $\C$ is a map
$$N: {}^{n}\Delta\to \At\C$$
such that the following hold for each $i,j<n$, $\delta\in {}^{n}\Delta$
and $d\in \Delta$:
\begin{itemize}
\item $N(\delta^i_j)\leq {\sf d}_{ij}$
\item $N(\delta[i\to d])\leq {\sf c}_iN(\delta)$
\item $N(\bar{x}\circ [i,j])= {\sf s}_{[i,j]}N(\bar{x})$ for all $i,j<n$.

\end{itemize}
\end{definition}
Note than $N$ can be viewed as a hypergraph with set of nodes $\Delta$ and
each hyperedge in ${}^{\mu}\Delta$ is labelled with an atom from $\C$.
We call such hyperedges atomic hyperedges.
We write $\nodes(N)$ for $\Delta.$ We let $N$ stand for the set of nodes
as well as for the function and the network itself. Context will help.
We assume that $\nodes(N)\subseteq \N$.
Formulated for $\sf Df$s only, the next definition applies to all algebras considered.

\begin{definition}\label{def:games}
Let $2\leq n<\omega$. For any ${\sf PEA_n}$
atom structure $\alpha$ and $n\leq m\leq
\omega$, we define two-player games $G(\alpha),$ \; and
$F^m(\alpha)$,
each with $\omega$ rounds.
\begin{enumarab}
\item   $G(\alpha)$ is the usual atomic game on networks \cite{HHbook2}.

\item $F^m$ is very similar, except that \pa\ can choose from only $m>n$ pebbles,
but he can re use them.

Let $m\leq \omega$.
In a play of $F^m(\alpha)$ the two players construct a sequence of
networks $N_0, N_1,\ldots$ where $\nodes(N_i)$ is a finite subset of
$m=\set{j:j<m}$, for each $i$.  In the initial round of this game \pa\
picks any atom $a\in\alpha$ and \pe\ must play a finite network $N_0$ with
$\nodes(N_0)\subseteq  m$,
such that $N_0(\bar{d}) = a$
for some $\bar{d}\in{}^{n}\nodes(N_0)$.
In a subsequent round of a play of $F^m(\alpha)$ \pa\ can pick a
previously played network $N$ an index $\l<n$, a ``face"
$F=\langle f_0,\ldots f_{n-2} \rangle \in{}^{n-2}\nodes(N),\; k\in
m\setminus\set{f_0,\ldots f_{n-2}}$, and an atom $b\in\alpha$ such that
$b\leq {\sf c}_lN(f_0,\ldots f_i, x,\ldots f_{n-2}).$
(the choice of $x$ here is arbitrary,
as the second part of the definition of an atomic network together with the fact
that $\cyl i(\cyl i x)=\cyl ix$ ensures that the right hand side does not depend on $x$).
This move is called a \emph{cylindrifier move} and is denoted
$(N, \langle f_0, \ldots f_{n-2}\rangle, k, b, l)$ or simply $(N, F,k, b, l)$.
In order to make a legal response, \pe\ must play a
network $M\supseteq N$ such that
$M(f_0,\ldots f_{i-1}, k, f_{i+1},\ldots f_{n-2}))=b$
and $\nodes(M)=\nodes(N)\cup\set k$.

\pe\ wins $F^m(\alpha)$ if she responds with a legal move in each of the
$\omega$ rounds.  If she fails to make a legal response in any
round then \pa\ wins.
\end{enumarab}
\end{definition}

We need some more technical lemmas which are  generalizations of lemmas formulated for relation algebras
in \cite{r}.

\begin{definition}\label{def:hat}
For $m\geq 5$ and $\C\in\CA_m$, if $\A\subseteq\Nr_n(\C)$ is an
atomic cylindric algebra and $N$ is an $\A$-network with $\nodes(N)\subseteq m$, then we define
$\widehat N\in\C$ by
\[\widehat N =
 \prod_{i_0,\ldots i_{n-1}\in\nodes(N)}{\sf s}_{i_0, \ldots i_{n-1}}N(i_0\ldots i_{n-1})\]
$\widehat N\in\C$ depends
implicitly on $\C$.
\end{definition}

The next definition is for ${\sf Sc}$s, and of course to applies to all its expansion studied here.
We write $\A\subseteq_c \B$ if $\A$ is a complete subalgebra of $\B$.

\begin{lemma}\label{lem:atoms2}
Let $n<m$ and let $\A$ be an atomic $\CA_n$,
$\A\subseteq_c\Nr_n\C$
for some $\C\in\CA_m$.  For all $x\in\C\setminus\set0$ and all $i_0, \ldots i_{n-1} < m$ there is $a\in\At(\A)$ such that
${\sf s}_{i_0\ldots i_{n-1}}a\;.\; x\neq 0$.
\end{lemma}
\begin{proof}
We can assume, see definition  \ref{subs},
that ${\sf s}_{i_0,\ldots i_{n-1}}$ consists only of substitutions, since ${\sf c}_{m}\ldots {\sf c}_{m-1}\ldots
{\sf c}_nx=x$
for every $x\in \A$.We have ${\sf s}^i_j$ is a
completely additive operator (any $i, j$), hence ${\sf s}_{i_0,\ldots i_{\mu-1}}$
is too  (see definition~\ref{subs}).
So $\sum\set{{\sf s}_{i_0\ldots i_{n-1}}a:a\in\At(\A)}={\sf s}_{i_0\ldots i_{n-1}}
\sum\At(\A)={\sf s}_{i_0\ldots i_{n-1}}1=1$,
for any $i_0,\ldots i_{n-1}<n$.  Let $x\in\C\setminus\set0$.  It is impossible
that ${\sf s}_{i_0\ldots i_{n-1}}\;.\;x=0$ for all $a\in\At(\A)$ because this would
imply that $1-x$ was an upper bound for $\set{{\sf s}_{i_0\ldots i_{n-1}}a:
a\in\At(\A)}$, contradicting $\sum\set{{\sf s}_{i_0\ldots i_{n-1}}a :a\in\At(\A)}=1$.
\end{proof}

For networks $M, N$ and any set $S$, we write $M\equiv^SN$
if $N\restr S=M\restr S$, and we write $M\equiv_SN$
if the symmetric difference $\Delta(\nodes(M), \nodes(N))\subseteq S$ and
$M\equiv^{(\nodes(M)\cup\nodes(N))\setminus S}N$. We write $M\equiv_kN$ for
$M\equiv_{\set k}N$.

Recall that $F^m$ is the usual atomic game on networks, except that the nodes are $m$ and \pa\ can re use nodes. Then:

\begin{theorem}\label{thm:n}
Let $n<m$, and let $\A$ be an atomic $\CA_m.$
If $\A\in{\bf S_c}\Nr_{n}\CA_m, $
then \pe\ has a \ws\ in $F^m(\At\A)$. In particular, if $\A$ is countable and completely representable, then \pe\ has a \ws\ in $F^{\omega}(\At\A).$
In the latter case since $F^{\omega}(\At\A)$ is equivalent to the usual atomic rounded game on networks,
the converse is also true.
\end{theorem}
\begin{proof}

The proof of the first part is based on repeated use of
lemma ~\ref{lem:atoms2}. We first show:
\begin{enumerate}
\item For any $x\in\C\setminus\set0$ and any
finite set $I\subseteq m$ there is a network $N$ such that
$\nodes(N)=I$ and $x\;.\;\widehat N\neq 0$.
\item
For any networks $M, N$ if
$\widehat M\;.\;\widehat N\neq 0$ then $M\equiv^{\nodes(M)\cap\nodes(N)}N$.
\end{enumerate}
We define the edge labelling of $N$ one edge
at a time. Initially no hyperedges are labelled.  Suppose
$E\subseteq\nodes(N)\times\nodes(N)\ldots  \times\nodes(N)$ is the set of labelled hyper
edges of $N$ (initially $E=\emptyset$) and
$x\;.\;\prod_{\bar c \in E}{\sf s}_{\bar c}N(\bar c)\neq 0$.  Pick $\bar d$ such that $\bar d\not\in E$.
Then there is $a\in\At(\c A)$ such that
$x\;.\;\prod_{\bar c\in E}{\sf s}_{\bar c}N(\bar c)\;.\;{\sf s}_{\bar d}a\neq 0$.
Include the edge $\bar d$ in $E$.  Eventually, all edges will be
labelled, so we obtain a completely labelled graph $N$ with $\widehat
N\neq 0$.
it is easily checked that $N$ is a network.
For the second part, if it is not true that
$M\equiv^{\nodes(M)\cap\nodes(N)}N$ then there are is
$\bar c \in^{n-1}\nodes(M)\cap\nodes(N)$ such that $M(\bar c )\neq N(\bar c)$.
Since edges are labelled by atoms we have $M(\bar c)\cdot N(\bar c)=0,$
so $0={\sf s}_{\bar c}0={\sf s}_{\bar c}M(\bar c)\;.\; {\sf s}_{\bar c}N(\bar c)\geq \widehat M\;.\;\widehat N$.

Next, we show that:

\begin{enumerate}
\item\label{it:-i}
If $i\not\in\nodes(N)$ then ${\sf c}_i\widehat N=\widehat N$.

\item \label{it:-j} $\widehat{N Id_{-j}}\geq \widehat N$.

\item\label{it:ij} If $i\not\in\nodes(N)$ and $j\in\nodes(N)$ then
$\widehat N\neq 0 \rightarrow \widehat{N[i/j]}\neq 0$.
where $N[i/j]=N\circ [i|j]$

\item\label{it:theta} If $\theta$ is any partial, finite map $n\to n$
and if $\nodes(N)$ is a proper subset of $n$,
then $\widehat N\neq 0\rightarrow \widehat{N\theta}\neq 0$.
\end{enumerate}

The first part is easy.
The second
part is by definition of $\;\widehat{\;}$. For the third part suppose
$\widehat N\neq 0$.  Since $i\not\in\nodes(N)$, by part~\ref{it:-i},
we have ${\sf c}_i\widehat N=\widehat N$.  By cylindric algebra axioms it
follows that $\widehat N\;.\;d_{ij}\neq 0$.  From the above
there is a network $M$ where $\nodes(M)=\nodes(N)\cup\set i$ such that
$\widehat M\;.\widehat N\;.\;d_{ij}\neq 0$.  From the first part, we
have $M\supseteq N$ and $M(i, j)\leq 1'$.  It follows that $M=N[i/j]$.
Hence $\widehat{N[i/j]}\neq 0$.
For the final part
(cf. \cite[lemma~13.29]{HHbook}), since there is
$k\in n\setminus\nodes(N)$, \/ $\theta$ can be
expressed as a product $\sigma_0\sigma_1\ldots\sigma_t$ of maps such
that, for $s\leq t$, we have either $\sigma_s=Id_{-i}$ for some $i<n$
or $\sigma_s=[i/j]$ for some $i, j<n$ and where
$i\not\in\nodes(N\sigma_0\ldots\sigma_{s-1})$.
Now apply parts~\ref{it:-j} and \ref{it:ij}.

Now we prove the required. $\A\subseteq\Nr_n\C$ for some $\C\in\CA_m$ then \pe\ always
plays networks $N$ with $\nodes(N)\subseteq n$ such that
$\widehat N\neq 0$. In more detail, in the initial round, let \pa\ play $a\in \At\A$.
\pe\ plays a network $N$ with $N(0, \ldots n-1)=a$. Then $\widehat N=a\neq 0$.
At a later stage suppose \pa\ plays the cylindrifier move
$(N, \langle f_0, \ldots f_{n-2}\rangle, k, b, l)$
by picking a
previously played network $N$ and $f_i\in \nodes(N), \;l<n,  k\notin \{f_i: i<n-2\}$,
and $b\leq {\sf c}_lN(f_0,\ldots  f_{i-1}, x, f_{i+1}, \ldots f_{n-2})$.

Let $\bar a=\langle f_0\ldots f_{i-1}, k, f_{i+1}, \ldots f_{n-2}\rangle.$
Then ${\sf c}_k\widehat N\cdot {\sf s}_{\bar a}b\neq 0$.
Then by the above there is a network  $M$ such that
$\widehat{M}.\widehat{{\sf c}_kN}\cdot {\sf s}_{\bar a}b\neq 0$.

Hence
$M(f_0,\dots, f_{i-1}, k, f_{i-2}, \ldots f_{n-2})=b$, and $M$ is the required response.
\end{proof}

\begin{definition}

We construct our desired example using the rainbow construction for cylindric algebras. We follow \cite{HHbook2}, with two deviations.
Given a graph $\Gamma$, we denote $\R(\Gamma)$ by $\PEA_{\sf G, \Gamma}$, where $\sf G$ is the complete irreflexive graph consisting
of the greens. $\Gamma$ is the set of red colours. We also consider polyadic equality algebras; the polyadic
accessibility relation can be defined easily by
$[f]\equiv _{ij} [g]$ iff $f\circ [i,j]=g$.

We define a new class $\K$ consisting of coloured graphs
(which are the models of the rainbow signature satisfying the $L_{\omega_1, \omega}$ theory as formulated
in \cite{HHbook2}). Let $\N^{-{1}}$ be $\N$ with reverse order, and let $g:\N\to \N^{-1}$ be the identity map, we denote $g(n)$ by $-n$.
We assume that $0\in \N$. We take ${\sf G}=\N^{-1}$ and $\Gamma=\N$.
\begin{definition}
Let $i\in \N^{-1}$, and let $\Gamma$ be a coloured graph  consisting of $n$ nodes
$x_0,\ldots  x_{n-2}, z$. We call $\Gamma$ an $i$ - cone if $\Gamma(x_0, z)=\g^0_i$
and for every $1\leq j\leq n-2$ $\Gamma(x_j, z)=\g_j$,
and no other edge of $\Gamma$
is coloured green.
$(x_0,\ldots x_{n-2})$
is called the center of the cone, $z$ the apex of the cone
and $i$ the tint of the cone.
\end{definition}
Translating this theory from models of the rainbow signature to coloured graphs,
we get that $\K$ consists of:

\begin{enumarab}

\item $M$ is a complete graph.

\item $M$ contains no triangles (called forbidden triples)
of the following types:

\vspace{-.2in}
\begin{eqnarray}
&&\nonumber\\
(\g, \g^{'}, \g^{*}), (\g_i, \g_{i}, \w),
&&\mbox{any }i\in n-1\;  \\
(\g^j_0, \g^k_0, \w_0)&&\mbox{ any } j, k\in \N\\
\label{forb:pim}(\g^i_0, \g^j_0, \r_{kl})&&\mbox{unless } \set{(i, k), (j, l)}\mbox{ is an order-}\\
&&\mbox{ preserving partial function }\N^{-1}\to\N\nonumber\\
\label{forb:match}(\r_{ij}, \r_{j'k'}, \r_{i^*k^*})&&\mbox{unless }i=i^*,\; j=j'\mbox{ and }k'=k^*
\end{eqnarray}
and no other triple of atoms is forbidden.

\item If $a_0,\ldots   a_{n-2}\in M$ are distinct, and no edge $(a_i, a_j)$ $i<j<n$
is coloured green, then the sequence $(a_0, \ldots a_{n-2})$
is coloured a unique shade of yellow.
No other $(n-1)$ tuples are coloured shades of yellow.

\item If $D=\set{d_0,\ldots  d_{n-2}, \delta}\subseteq M$ and
$\Gamma\upharpoonright D$ is an $i$ cone with apex $\delta$, inducing the order
$d_0,\ldots  d_{n-2}$ on its base, and the tuple
$(d_0,\ldots d_{n-2})$ is coloured by a unique shade
$\y_S$ then $i\in S.$

\end{enumarab}

\end{definition}
We denote the complex algebra on this atom structure by $\PEA_{\N^{-1}, \N}$.

\begin{theorem} \pa\ has a winning strategy in
$F^{n+3}(\At\Rd_{sc}\PEA_{\N^-1,\N})$
\end{theorem}
\begin{proof}
In the initial round \pa\ plays a graph $\Gamma$ with nodes $0,1,\ldots n-1$ such that $\Gamma(i,j)=\w$ for $i<j<n-1$
and $\Gamma(i, n-1)=\g_i$
$(i=1, \ldots n-2)$, $\Gamma(0,n-1)=\g_0^0$ and $\Gamma(0,1\ldots n-2)=\y_{B}$.
In the following move \pa\ chooses the face $(0,\ldots n-2)$ and demands a node $n$
with $\Gamma_2(i,n)=\g_i$ $(i=1,\ldots n-2)$, and $\Gamma_2(0,n)=\g_0^{-1}.$
\pe\ must choose a label for the edge $(n+1,n)$ of $\Gamma_2$. It must be a red atom $r_{mn}$. Since $-1<0$ we have $m<n$.
In the next move \pa\ plays the face $(0, \ldots n-2)$ and demands a node $n+1$, with $\Gamma_3(i,n)=\g_i$ $(i=1,\ldots n-2)$,
such that  $\Gamma_3(0,n+2)=\g_0^{-2}$.
Then $\Gamma_3(n+1,n)$ $\Gamma_3(n+1,n-1)$ both being red, the indices must match.
$\Gamma_3(n+1,n)=r_{ln}$ and $\Gamma_3(n+1, n-1)=r_{lm}$ with $l<m$.
In the next round \pa\ plays $(0,1\ldots n-2)$ and reuses the node $2$ such that $\Gamma_4(0,2)=\g_0^{-3}$.
This time we have $\Gamma_4(n,n-1)=\r_{jl}$ for some $j<l\in \N$.
Continuing in this manner leads to a decreasing sequence in $\N$.
\end{proof}

Let $k>n$ be given. We show that \pe\ has a \ws\ in the usual graph game in $k$ rounds (there is no restriction then on the size of the graphs).
We denote this game by $H_k$; that will further be strengthened below (in the sense that it becomes more difficult for \pe\.)
We define \pe\ s strategy for choosing labels for edges and $n-1$ tuples in response to \pa\ s moves.
Assume that we are at round $r+1$.

Let $M_0, M_1,\ldots M_r$ be the coloured graphs at the start of a play of $H_k(\alpha)$ just before round $r+1$.
Assume inductively that \pe\ computes a partial function $\rho_s:\N^{-1}\to \N$, for $s\leq r$, that will help her choose
the suffices of the chosen red in the critical case.  Inductively
for $s\leq r$:

\begin{enumarab}
\item  If $M_s(x,y)$ is green then $(x,y)$ belongs  \pa\ in $M_s$ (meaning he coloured it).

\item $\rho_0\subseteq \ldots \rho_r\subseteq\ldots$
\item $\dom(\rho_s)=\{i\in \Z: \exists t\leq s, x, x_0, x_1,\ldots x_{n-2}
\in \nodes(M_t)\\
\text { where the $x_i$'s form the base of a cone, $x$ is its appex and $i$ its tint }\}.$

The domain consists of the tints of cones created at an earlier stage.

\item $\rho_s$ is order preserving: if $i<j$ then $\rho_s(i)<\rho_s(j)$. The range
of $\rho_s$ is widely spaced: if $i<j\in \dom\rho_s$ then $\rho_s(i)-\rho_s(j)\geq  3^{m-r}$, where $m-r$
is the number of rounds remaining in the game.

\item For $u,v,x_0\in \nodes(M_s)$, if $M_s(u,v)=\r_{\mu,\delta}$, $M_s(x_0,u)=\g_0^i$, $M_s(x_0,v)=\g_0^j$,
where $i,j$ are tints of two cones, with base $F$ such that $x_0$ is the first element in $F$ under the induced linear order,
then $\rho_s(i)=\mu$ and $\rho_s(j)=\delta.$

\item $M_s$ is a a  coloured graph.

\item If the base of a cone $\Delta\subseteq M_s$  with tint $i$ is coloured $y_S$, then $i\in S$.

\end{enumarab}

To start with if \pa\ plays $a$ in the initial round then $\nodes(M_0)=\{0,1,\ldots n-1\}$, the
hyperedge labelling is defined by $M_0(0,1,\ldots n)=a$.

In response to a cylindrifier move for some $s\leq r$, involving a $p$ cone, $p\in \Z$,
\pe\ must extend $\rho_r$ to $\rho_{r+1}$ so that $p\in \dom(\rho_{r+1})$
and the gap between elements of its range is at least $3^{m-r-1}$. Properties (3) and (4) are easily
maintained in round $r+1$. Inductively, $\rho_r$ is order preserving and the gap between its elements is
at least $3^{m-r}$, so this can be maintained in a further round.
If \pa\ chooses a green colour, or green colour whose suffix
already belong to $\rho_r$, there would be fewer
elements to add to the domain of $\rho_{r+1}$, which makes it easy for \pe\ to define $\rho_{r+1}$.

Now assume that at round $r+1$, the current coloured graph is $M_r$ and that   \pa\ chose the graph $\Phi$, $|\Phi|=n$
with distinct nodes $F\cup \{\delta\}$, $\delta\notin M_r$, and  $F\subseteq M_r$ has size
$n-1$.  We can  view \pe\ s move as building a coloured graph $M^*$ extending $M_r$
whose nodes are those of $M_r$ together with the new node $\delta$ and whose edges are edges of $M_r$ together with edges
from $\delta$ to every node of $F$.

Now \pe\ must extend $M^*$ to a complete graph $M^+$ on the same nodes and
complete the colouring giving  a graph $M_{r+1}=M^+$ in $\K$ (the latter is the class of coloured graphs).
In particular, she has to define $M^+(\beta, \delta)$ for all nodes
$\beta\in M_r\sim F$, such that all of the above properties are maintained.

\begin{enumarab}

\item  If $\beta$ and $\delta$ are both apexes of two cones on $F$.

Assume that the tint of the cone determined by $\beta$ is $a\in \Z$, and the two cones
induce the same linear ordering on $F$. Recall that we have $\beta\notin F$, but it is in $M_r$, while $\delta$ is not in $M_r$,
and that $|F|=n-1$.
By the rules of the game  \pe\ has no choice but to pick a red colour. \pe\ uses her auxiliary
function $\rho_{r+1}$ to determine the suffices, she lets $\mu=\rho_{r+1}(p)$, $b=\rho_{r+1}(q)$
where $p$ and $q$ are the tints of the two cones based on $F$,
whose apexes are $\beta$ and $\delta$. Notice that $\mu, b\in \N$; then she sets $N_s(\beta, \delta)=\r_{\mu,b}$
maintaining property (5), and so $\delta\in \dom(\rho_{r+1})$
maintaining property (4). We check consistency to maintain property (6).

Consider a triangle of nodes $(\beta, y, \delta)$ in the graph $M_{r+1}=M^+$.
The only possible potential problem is that the edges $M^+(y,\beta)$ and $M^+(y,\delta)$ are coloured green with
distinct superscripts $p, q$ but this does not contradict
forbidden triangles of the form involving $(\g_0^p, \g_0^q, \r_{kl})$, because $\rho_{r+1}$ is constructed to be
order preserving.  Now assume that
$M_r(\beta, y)$ and $M_{r+1}(y, \delta)$ are both red (some $y\in \nodes(M_r)$).
Then \pe\ chose the red label $N_{r+1}(y,\delta)$, for $\delta$ is a new node.
We can assume that  $y$ is the apex of a $t$ cone with base $F$ in $M_r$. If not then $N_{r+1}(y, \delta)$ would be coloured
$\w$ by \pe\   and there will be no problem. All properties will be maintained.
Now $y, \beta\in M$, so by by property (5) we have $M_{r+1}(\beta,y)=\r_{\rho+1(p), \rho+1(t)}.$
But $\delta\notin M$, so by her strategy,
we have  $M_{r+1}(y,\delta)=\r_{\rho+1(t), \rho+1(q)}.$ But $M_{r+1}(\beta, \delta)=\r_{\rho+1(p), \rho+1(q)}$,
and we are done.  This is consistent triple, and so have shown that
forbidden triples of reds are avoided.

\item If this is not the case, and  for some $0<i<n-1$ there is no $f\in F$ such
that $M^*(\beta, f), M^* (f,\delta)$ are both coloured $\g_i$ or if $i=0$, they are coloured
$\g_0^l$ and $\g_0^{l'}$ for some $l$ and $l'$.
She chooses $\w_i$, for $M^+{(\beta,\delta)}$, for definiteness let it be the least such $i$.
It is clear that this also avoids all forbidden triangles (involving greens and whites).

\item Otherwise, if there is no $f\in F$,
such that $M^*(\beta, f), M*(\delta ,f)$ are coloured $\g_0^t$ and $\g_0^u$
for some $t,u$, then \pe\ defines $M^+(\beta, \delta)$ to  be $\w_0$.
Again, this avoids all forbidden triangles.
\end{enumarab}

She has not chosen green maintaining property one.  Now we turn to colouring of $n-1$ tuples,
to make sure that $M^+$ is a coloured graph maintaining property (6).

Let $\Phi$ be the graph chosen by \pa\, it has set of node $F\cup \{\delta\}$.
For each tuple $\bar{a}=a_0,\ldots a_{n-2}\in {M^+}^{n-1}$, $\bar{a}\notin M^{n-1}\cup \Phi^{n-1}$,  with no edge
$(a_i, a_j)$ coloured green (we already have all edges coloured), then  \pe\ colours $\bar{a}$ by $\y_S$, where
$$S=\{i\in A: \text { there is an $i$ cone in $M^*$ with base $\bar{a}$}\}.$$
We need to check that such labeling works, namely that last property is maintained.

Recall that $M$ is the current coloured graph, $M^*=M\cup \{\delta\}$ is built by \pa\ s move
and $M^+$ is the complete labelled graph by \pe\, whose nodes are labelled by \pe\ in response to \pa\ s moves.
We need to show that $M^+$ is labelled according to
the rules of the game, namely, that it is in $\K$.

Let us check that $(n-1)$ tuples are labelled correctly, by yellow colours.
(The  following argument is used in \cite{Hodkinson} p. 16, \cite{HH} p.844,
and \cite{HHbook2}.)
Let $D$ be a given  set of $n$ nodes of $M^+$, and suppose that $M^+\upharpoonright D$
is an $i$ cone with apex
$\delta$ and base $\{d_0,\ldots d_{m-2}\}$, and that the tuple $\bar{d}=(d_0,\ldots d_{m-2})$ is labelled $\y_S$ in $M^+$.
We need to show that $i\in S$.

If $D\subseteq M$, then inductively the graph
$M$ constructed so far is in $\K$, and therefore
$i\in S$. If $D\subseteq \Phi$ (recall that $\Phi$ the graph with $n$ nodes played by \pa\, namely $\Phi=F\cup \{\delta\}$,
then as \pa\ chose $\Phi$ in $\K$, we get also $i\in S$. If neither holds, then $D$ contains $\alpha$
and also some
$\beta\in M\sim \Phi$. \pe\ chose the colour $M^+(\alpha,\beta)$ and her strategy ensures her that it is not green.
Hence neither $\alpha$ or $\beta$ can be the apex of the cone $M^+\upharpoonright D$,
so they must both lie in the base $\bar{d}$.

This implies that
$\bar{d}$ is not yet labelled in $M^*$  ($M^*$'s underlying set is $M$ with the new node),
so \pe\ has applied her strategy to choose the colour $\y_S$ to label $\bar{d}$ in $M^+$.
But then \pe\ will have chosen $S$ containing $i$, by her \ws\ since $M^*\upharpoonright D$ is already a cone
in $N^*$.  Also \pe\ never chooses a green edge, so all green edges of $M^+$ lie in $M^*$.

We now check edge colours of triangles. The new triangles in $M^+$ are two kinds. There of the form $(\beta, \delta, f)$
for some $f\in F$ and $\beta\in M\sim F$, and those of the
form $(\beta, \beta', \delta)$ for distinct $\beta, \beta'\in M\sim F$.
The first kind is easy to check.
That leaves one (hard) case, where there are two nodes $\beta, \beta',
\in M$, \pe\ colours both $(\beta, \alpha)$ and $(\beta',
\alpha)$ red, and the old edge $(\beta, \beta')$ has already been
coloured red (earlier in the game).
If $(\beta, \beta')$ was coloured by \pe\ , that is \pe\ is their owner, then there is no problem, by \pe\ first move described above.
We show that this is what actually happened.
(This is precisely the argument used in \cite{HHbook2} starting the second half of p. 85, though the latter deals with the much more general case.)

So suppose, for a contradiction, that $(\beta, \beta')$ was coloured by
\pa\ . Since \pe\ chose red colours for $(\alpha, \beta)$
and $(\alpha, \beta')$, it must be the case that there are cones in
$N^*$ with apexes $\alpha, \beta, \beta'$ and the same base,
$F$, each inducing the same linear ordering $\bar{f} = (f_0,\ldots,
f_{m-2})$, say, on $F$. Now $\bar{f}$  must be labelled by some
yellow colour, $\y_S$, say. Since $\Phi\in \K$, so
the tint $i$ (say) of the cone from $\alpha$ to $\bar{f}$ lies in
$S$.

Suppose that $\lambda$ was the last node of $ F \cup \{ \beta,
\beta' \}$ to be created. As $ |F \cup \{
\beta, \beta' \}| = n + 1$, we see that \pe\ must have chosen
the colour of at least one edge in this : say, $( \lambda, \mu )$.
Now all edges from $\beta$ into $F$ are green,  so \pa\ is the owner of them
as well as of  $(\beta, \beta')$.

Hence $\lambda, \mu \in F$.
We can now see that it was \pe\ who chose the colour $\y_S$ of
$\bar{f}$. For $\y_S$ was chosen in the round when $F$'s last node,
i.e., $\lambda$ was created. It could only have been chosen by
\pa\ if he also picked the colour of every edge in $F$
involving $\lambda$. This is not so, as the edge $(\lambda, \mu)$
was coloured by \pe\, and lies in $F$.

As $i \in S$, it follows from \pe\ s \ws\ that at the time when $\lambda$ was added, there was already an
$i$-cone with base $\bar{f}$, and apex $x$ say.
We claim that $ F \cup \{ \alpha \}$ and $ F \cup \{ x \}$ are
isomorphic over $F$. For this, note that the only $(n - 1)$-tuples of
either $ F \cup \{ \alpha \}$ or $ F \cup \{ x \}$ with a
yellow colour are in $F$ ( since all others involve a green edge
). But this means that \pe\ could have taken $\alpha = x$ in
the current round, and not extended the graph. This is contrary to
our original assumption, and completes the proof.

This defines the colour of edges. Now for hyperedges,
for  each tuple of distinct elements
$\bar{a}=(a_0,\ldots a_{n-2})\in {}^{n-1}(\Gamma^+)$
such that $\bar{a}\notin {}^{n-1}\Gamma\cup {}^{n-1}\Delta$ and with no edge $(a_i, a)$
coloured greed in  $\Gamma^+$, \pe\ colours $\bar{a}$ by $\y_{S}$
where
$S=\{i <\omega: \text { there is a $i$ cone with base  } \bar{a}\}$.
Notice that $|S|\leq F$. This strategy works.

Then the last property is maintained, and \pe\
survives another round, and we are done.

\begin{theorem} There exists an atomic countable representable $\A\in \PEA_n$, such that
$\Rd_{sc}\A\notin S_c\Nr_n\Sc_{n+3}$ but $\A\in {\sf UpUr}S_c\Nr_n{\sf QPEA}_{\omega}$.
In particular, for any class $\K$ of Pinter's algebras, cylindric algebras and polyadic equality algebras
and $m\geq 3$, the class $S_c\Nr_n\K_m$ is not elementary.
\end{theorem}
\begin{proof}
Let $\A=\sf PEA_{\N^{-1},\N}$ and let $\beta=\At\sf PEA_{\N^{-1}, \N}$.
Then we showed that \pa\ can win the game $F^{n+3}$, on $\Rd_{sc}{\sf PEA}_{\N^{-1},\N}={\sf Sc}_{\N^{-1}, \N}$
Hence $\Rd_{sc}\sf PEA_{\N^{-1},\N}\notin S_c\Nr_n\sf Sc_{n+3}$.

For $n<\omega,$ \pe\ has a \ws\ $\sigma_n$ in $G_n(\A)$, this is proved above.
We can assume that $\sigma_n$ is deterministic.
Let $\B$ be a non-principal ultrapower of $\A$.  Then
\pe\ has a \ws\ $\sigma$ in $G(\B)$   --- essentially she uses
$\sigma_n$ in the $n$'th component of the ultraproduct so that at each
round of $G(\B)$ \pe\ is still winning in co-finitely many
components, this suffices to show she has still not lost.
Now use an elementary chain argument to construct countable elementary
subalgebras $\A=\A_0\preceq\A_1\preceq\ldots\preceq \B$.  For this,
let $\A_{i+1}$ be a countable elementary subalgebra of $\B$
containing $\A_i$ and all elements of $\B$ that $\sigma$ selects
in a play of $G_\omega(\B)$ in which \pa\ only chooses elements from
$\A_i$. Now let $\A'=\bigcup_{i<\omega}\A_i$.  This is a
countable elementary subalgebra of $\B$ and \pe\ has a \ws\ in
$H(\A')$. Hence by the elementary chain argument
there is a countable $\A'$ such that \pe\ can win the $\omega$ rounded game on its atom structure, hence $\A'\equiv \A$ but
the former is in $\Nr_n{\sf PEA}_{\omega}$  and we also have $\Rd_{sc}{\sf PEA}_{\N^{-1},\N}$ is not in
$S_c\Nr_n{\sf Sc}_{n+3}$,  hence  we are done.
\end{proof}

\begin{corollary} The class of completely
representable algebras and atom structures are not elementary, for any class $\K$ between
$\sf Df_n$ and $\sf PEA_n$.
\end{corollary}
\begin{proof} \pe\ can win all finite rounded games on $\At_{\N^{-1}, \N}$ hence
$\sf PEA_{\N^{-1}, \N}$, is elementary equivalent to a countable completely representable algebra;
using an ultrapower followed by an elementary chain
argument.  But $\Rd_{sc}\B\notin S_c\Nr_n\CA_{n+3}$, since \pa\ can win the game $F^{n+3}$, {\it a fortiori}, it
is not in $S_c\Nr_n\Sc_{\omega}$ (as in the proof above).  Since $\A$ is countable and atomic,
an atom structure of an algebra is interpretable in the algebra,
and an atom structure holds a completely representable algebra iff all algebras having this atom structure
are completely representable,  the required concerning atom structures follows. The $\sf Df$ reduct of
$\sf \PEA_{\N^{-1}, \N}$ is also not completely representable, for a complete representation of it, will induce a complete
representation of $\sf \PEA_{\N^{-1},\N}$, since the latter is generated by elements whose dimension set $<n$.
\end{proof}

In the next theorem we provide the algebra that is missing from the last paragraph of \cite{Sayed}. We take the opportunity
to discuss neat embedding
properties to
the notions of complete and strong representability in terms of {\it cardinality},
by a strongly representable algebra $\A$ in this context,  we mean, as expected
an  algebra such that $\Cm\At\A$ is representable. The Lyndon conditions are defined in \cite{HHbook}.
A Lyndon $n$th condition, or simple a Lyndon condition, is a first order sentence that codes that \pe\ has a \ws\
in the usual atomic game on (atomic) networks having
$n$ rounds \cite{HHbook2}.

\begin{theorem}\label{Robinsexample}
\begin{enumarab}
\item There exists countable atomic strongly representable algebras that are not completely representable.
In fact, such algebras can be chosen to fail infinitely many Lyndon conditions.

\item The classes ${\sf CRA_n}$ of completely representable algebras and ${\sf UpUr}\Nr_n\CA_{\omega}$
are not related both ways, that is they are mutually
distinct.

\item Every countable algebra in $\Nr_n\CA_{\omega}$ is completely representable, hence satisfies the Lyndon conditions.
There is an uncountable algebra in $\Nr_n\CA_{\omega}$ that is not completely representable
(The last algebra shows that the condition of maximality of Shelah's result restricted to $L_n$, see theorem \ref{Shelah} cannot be omitted,
and proves the alleged statement in the last paragraph of \cite{Sayed}).

\item Regardless of cardinalities any algebra in $\Nr_n\CA_{\omega}$ satisfies the Lyndon conditions, hence is strongly
representable.
\end{enumarab}
\end{theorem}
\begin{proof}
\begin{enumarab}

\item(a) The first part is easy. Let $\A=\CA_{\omega, \omega}$ be the rainbow algebra defined in \cite{HH}.
Then $\A$ is not completely representable, but its $\sf Df$ reduct satisfies the Lyndon
conditions hence it is strongly representable.

(b) For the second part, we use a Monk algebra, whose atom structure, is constructed from a class of models as defined in \cite{HHbook2}.
Let $\Gamma$ be any graph with infinite chromatic number,
and large enough finite girth $m$.  Let $m$ be also large enough so that any $3$ colouring of the edges of a complete graph
of size $m$ must contain a monochromatic triangle; this $m$ exists by Ramseys's theorem, and the graph exists
by Erdos' graphs having large chromatic number
and girth. $\M(\Gamma)$ is representable if and only if $\Gamma$ has infinite chromatic number.
This is taken from \cite{HHbook2}, so it will be sketchy.
Let $\Gamma^* $ be the ultrafilter extension of $\Gamma$.
We first define a strong bounded morphism $\Theta$
form $\M(\Gamma)_+$ to $\rho(I(\Gamma^*))$, as follows:
For any $x_0,\ldots x_{n-2}<n$ and $X\subseteq \Gamma^*\times n$, define the following element
of $\M(\Gamma^*)$:
$$X^{(x_0,\ldots x_{n-2})}=\{[f]\in \rho(I(\Gamma^*)): \exists p\in X[M_f\models p(f(x_0),\ldots f(x_{n-2})]\}.$$
Let $\mu$ be an ultrafilter in $\M(\Gamma).$
Define $\sim $ on $n$ by $i\sim j$ iff $d_{ij}\in \mu.$
Let $g$ be the projection map from $n$ to $n/\sim$.

Define a $\Gamma^*\times n$ coloured graph by with domain $n/\sim$ as follows. For each $v\in \Gamma^*\times n$
and $x_0,\ldots x_{n-2}<n$, we let
$$M_{\mu}\models v(g(x_0),\ldots g(x_{n-2})\Longleftrightarrow  X^{(x_0,\ldots x_{n-2})}\in \mu.$$
One can show that $\Cm(\M(\Gamma)_+)=\M(\Gamma)^{\sigma}$ is completely representable, by providing
\pe\ with  a \ws\ in the $\omega$ rounded atomic game on networks, by identifying networks with
structures.  Indeed, let $\Theta$ be as defined above. Let $N$ be a $\M(\Gamma)^{\sigma}$ network.

Then $\theta(N)$ is an $\M(\Gamma)^{\sigma}$ network. Identify $\Theta(N)$ with a structure $N^*$
with same domain and such that for
$x_0,\ldots x_{n-1}\in N$ with $\Theta(N(x_0,\ldots x{n-1})=[f]$, say, each $i<n$ and each $p\in \Gamma^*\times n$,
we have
$$N^*\models p(x_0,\ldots x_{i-1}, x_{i+1},\ldots x_{n-1})$$
if
$M_f\models p(f(0),\ldots f(i-1), f(i+1),\ldots f(n-1))$.
Assume that \pa\ chose $x\in {}^nN^*$, $i<n$ and an atom $[f]$ with $[f]\leq c_iN^*(x)$.

We can well assume that $f(i)\neq f(j)$ else she would have chosen the same network.
Let $y=x[i|z]\in {}^nN\cup \{z\}$, and $Y=\{y_0,\ldots y_{n-1}\}.$ Define
$q_j\in \Gamma^*\times n$ as follows. If for $j<n$, $\bar{y}\sim y_j$
are pairwise distinct, let $q_j\in \Gamma^*\times n$ be the unique element satisfying

$M\models q_j (y_0\ldots y_{j-1}, y_{j+1}, y_{n-1})$.

Else we choose $q_j$ arbitrarily,
then choose a new copy and let $d$ be the reflexive node in this copy.
Define $M\models d(t_0,\ldots t_{n-2})$ whenever $t_0,\ldots t_{n-2}\in M$
are distinct and $z\in \{t_0,\ldots t_{n-2}\}\nsubseteq Y$.
This structure is as required, it describes \pe\ s response to the play of \pa\ ,
and we are done.

The converse is also true, it follows from the fact that a representation of an algebra based on a graph with finite chromatic number,
necessarily contradicts Ramsey's theorem. The idea is that any representation will force a monochromatic forbidden triangle,
a typical Monk's argument.

Then $\M(\Gamma)$, the complex algebra constructed on $\Gamma$, as defined in \cite{HHbook2}
will be representable, hence $\rho(I(\Gamma))$ will be strongly representable, but it will fail $\rho_k$ for all $k\geq m$.
$\rho_k$ is the $k$th Lyndon condition, which is a first order sentence coding that \pe\ has a \ws\ in
a standard $k$ rounded atomic game played on networks, with $\omega$ many nodes. If $\A$ is  completely representable, then
$\A$ will satisfy all Lyndon conditions, so now we know it is not completely representable.

The idea is
that \pa\ can win in the $m$ rounded atomic game coded by $\rho_m$, by forcing a forbidden monochromatic triangle.
We can assume that $m>n$ where $n$ is the dimension.
Let $N$ be an atomic network with $m$ nodes. Choose a set $X$ of  $max\{n,6\}$ nodes of $\Gamma$,
such that the colour of $N(\bar{x})$ is constant say
$r$, for every hyperedge of $X$.
For $\bar{x}\in X$, of distinct elements, let $v(\bar{x})\in \Gamma$
be such that $N(\bar{x})=r$, and let $\Delta$ be the induced subgraph with nodes $\{v{\bar{x}}: \bar{x}\in X\}$ of $\Gamma$.
Since the girth is sufficiently large, $\Delta$ is $2$ colourable and its nodes can be partitioned into two distinct sets, each independent
and monochromatic. But any $2$ colouring of the edges of a complete graph of  size $\geq 6$, has  {\it an independent monochromatic triangle.}

The algebra will fail infinitely many Lyndon conditions $\rho_k$ for $k\geq m$,
and any countable elementary subalgebra  will be as required. It will also fail
infinitely many Lyndon conditions, hence will not be completely representable,
hence will not be in $S_c\Nr_n\CA_{\omega}$, hence will not
be a full neat reduct of a $\CA_{\omega}$.

\item The algebras in \cite{SL} is completely representable but is not in ${\sf UpUr}\Nr_n\CA_{\omega}$,
as we proceed to show. Assume $\alpha\leq \omega$. Let $\F$ is field of characteristic $0$.
Let $$V=\{s\in {}^{\alpha}\F: |\{i\in \alpha: s_i\neq 0\}|<\omega\},$$
Note that $V$ is a vector space over the field $\F$.
Let $${\C}=(\wp(V),
\cup,\cap,\sim, \emptyset , V, {\sf c}_{i},{\sf d}_{ij})_{i,j\in \alpha}.$$
Then clearly $\wp(V)\in \Nr_{\alpha}\CA_{\alpha+\omega}$.
Indeed let $W={}^{\alpha+\omega}\F^{(0)}$. Then
$\psi: \wp(V)\to \Nr_{\alpha}\wp(W)$ defined via
$$X\mapsto \{s\in W: s\upharpoonright \alpha\in X\}$$
is an isomorphism from $\wp(V)$ to $\Nr_{\alpha}\wp(W)$.
We shall construct an atomic algebra $\A$, $\A\notin \Nr_{\alpha}\CA_{\alpha+1}$
Let $y$ denote the following $\alpha$-ary relation:
$$y=\{s\in V: s_0+1=\sum_{i>0} s_i\}.$$
Let $y_s$ be the singleton containing $s$, i.e. $y_s=\{s\}.$
Define
${\A}\in \CA_{\alpha}$
as follows:
${\A}=\Sg^{\C}\{y,y_s:s\in y\}.$
Then we claim that
$\A\notin \Nr_{\alpha}\CA_{\alpha+1}$
The proof is similar to that in \cite{SL}. In what follows, we code the idea of the proof in a first order sentence that also says
that $\A$ is neither an elementary nor complete subalgebra of $\wp(V)$; we also show that it is completely representable, this last cannot
be represented by a first order sentence.
Let $\At(x)$ be the first order formula asserting that $x$ is an atom.
Let $$\tau(x,y) ={\sf c}_1({\sf c}_0x\cdot {\sf s}_1^0{\sf c}_1y)\cdot {\sf c}_1x\cdot {\sf c}_0y.$$
Let $$Rc(x):=c_0x\cap c_1x=x,$$
$$\phi:=\forall x(x\neq 0\to \exists y(\At(y)\land y\leq x))\land
\forall x(\At(x) \to Rc(x)),$$
$$A(x,y):=\At(x)\land x\leq y,$$
and  $\psi (y_0,y_1)$ be the following first order formula
$$\forall z(\forall x(A(x,y_0)\to x\leq z)\to y_0\leq z)\land
\forall x(\At(x)\to \At(c_0x\cap y_0)\land \At(c_1x\cap y_0))$$
$$\to [\forall x_1\forall x_2(A(x_1,y_0)\land A(x_2,y_0)\to \tau(x_1,x_2)\leq y_1)$$
$$\land \forall z(\forall x_1 \forall x_2(A(x_1,y_0)\land A(x_2,y_0)\to
\tau(x_1,x_2)\leq z)\to y_1\leq z)].$$
Then
$$\Nr_{\alpha}\CA_{\alpha}\models \phi\to \forall y_0 \exists y_1 \psi(y_0,y_1).$$
But this formula does not hold in $\A$.
We have $\A\models \phi\text {  and not }
\A\models \forall y_0\exists y_1\psi (y_0,y_1).$
In words: we have a set $X=\{y_s: s\in V\}$ of atoms such that $\sum^{\A}X=y,$ and $\A$
models $\phi$ in the sense that below any non zero element there is a
{\it rectangular} atom, namely a singleton.
Let $Y=\{\tau(y_r,y_s), r,s\in V\}$, then
$Y\subseteq \A$, but it has {\it no supremum} in $\A$, but {\it it does have one} in any full neat reduct $\B$ containing $\A$,
and this is $\tau_{\alpha}^{\B}(y,y)$, where
$\tau_{\alpha}(x,y) = {\sf c}_{\alpha}({\sf s}_{\alpha}^1{\sf c}_{\alpha}x\cdot {\sf s}_{\alpha}^0{\sf c}_{\alpha}y).$
In $\wp(V)$ this last is $w=\{s\in {}^{\alpha}\F^{(\bold 0)}: s_0+2=s_1+2\sum_{i>1}s_i\}.$ That $w\notin \A$, is proved exactly
like as in \cite{SL}. So for $y_0=y$, there is no $y_1\in \A$ satisfying $\psi(y_0,y_1)$.
The complete representability of $\A$ follows by observing that $\A$ contains all singletons, so that
$V=\bigcup_{s\in V} \{s\}=\bigcup \At\A$.
Finally, the atom structure consisting of the singletons, and for $s,t\in V$, we have ${s}\equiv _i {t}$ if $t(j)=s(j)$ for all $i\neq j$ is as required.

The other non inclusion will be proved in the next item.

\item We show that any  countable algebra in $\A\in \Nr_n\CA_{\omega}$ is completely representable.
Let $\B\in \CA_{\omega}$ and $\A\subseteq_c \Nr_n\B.$ We can assume that
that $\B\in {\sf Lf}_{\omega}$ and is countable, if not just replace by $\Sg^{\B}A$.
Let $X$ be the set of co-atoms of $\A$ (a co-atom is the complement of an atom), then
$$\prod^{\A'}X=\prod^{\Nr_nB}X=\prod ^{\B}X=0.$$
By usual omitting types theorem for first order logic, or rather the algebraic counterpart of, one can find for
every non zero $a\in \B$ a set algebra $\C$ with a square unit $V$, and a homomorphism $f:\B\to \C$ such that $f(a)\neq 0$ and
$-0=-\bigcap f(x)=\bigcup f(-x)=V$. The restriction of $f$ to $\A'$ in the obvious way gives a complete representation.

This example shows that the maximality condition cannot be omitted from the omitting
types theorem obtained from Shelah's omitting types theorem restricted to $L_n$,
see theorem \ref{Shelah}, and theorems 3.2.8, 3.2.9 in \cite{Sayed}.
It also provides the algebra
claimed to exist in the last paragraph of \cite{Sayed}. Without the maximality condition theorem \ref{Shelah}, even restricted
to the countable case, becomes an instance of Martin's axiom restricted to countable Boolean algebras, hence will
be independent.

We show that  ${\sf Up Ur}\Nr_n\CA_{\omega}\nsubseteq {\sf CRA_n}$.
We will exhibit an {\it uncountable} neat reduct, that is, an algebra in $\Nr_n\CA_{\omega}$,
that is not completely representable.

This is very similar to the Rainbow relation algebra when we have
$\omega_1$ many green atoms and $\omega$
red atoms. It will be used to
show that the condition of countability cannot be omitted, there are uncountable full neat reducts that fail to be completely
representable. However, it is strongly representable.
Here we allow the greens to be of cardinality $2^{\kappa}$ for any
infinite cardinal $\kappa$.

As usual we specify the atoms and forbidden triples. We will show that the relation algebra has an $\omega$ complete representation, that
is a representation that witnesses cylindrifiers on arbitrary large cliques, but it not have a complete one.
Furthermore it will be the full ${\sf Ra}$ reduct of a locally finite cylindric algebra.

The atoms are $1', \; a_0^i:i<2^{\kappa}$ and $a_j:1\leq j<
\kappa$, all symmetric.  The forbidden triples of atoms are all
permutations of $(1',x, y)$ for $x \neq y$, \/$(a_j, a_j, a_j)$ for
$1\leq j<\kappa$ and $(a_0^i, a_0^{i'}, a_0^{i^*})$ for $i, i',
i^*<2^{\kappa}.$  In other words, we forbid all the monochromatic
triangles.

Write $a_0$ for $\set{a_0^i:i<2^{\kappa}}$ and $a_+$ for
$\set{a_j:1\leq j<\kappa}$. Call this atom
structure $\alpha$.

Let $\A$ be the term algebra on this atom
structure; the subalgebra of $\Cm\alpha$ generated by the atoms.  $\A$ is a dense subalgebra of the complex algebra
$\Cm\alpha$. We claim that $\A$, as a relation algebra,  has no complete representation.

Indeed, suppose $\A$ has a complete representation $M$.  Let $x, y$ be points in the
representation with $M \models a_1(x, y)$.  For each $i<\omega_1$ there is a
point $z_i \in M$ such that $M \models a_0^i(x, z_i) \wedge a_1(z_i, y)$.

Let $Z = \set{z_i:i<2^{\kappa}}$.  Within $Z$ there can be no edges labeled by
$a_0$ so each edge is labelled by one of the $\kappa$ atoms in
$a_+$.  The Erdos-Rado theorem forces the existence of three points
$z^1, z^2, z^3 \in Z$ such that $M \models a_j(a^1, z^2) \wedge a_j(z^2, z^3)
\wedge a_j(z^3, z_1)$, for some single $j<\kappa$.  This contradicts the
definition of composition in $\A$.

Let $S$ be the set of all atomic $\A$-networks $N$ with nodes
 $\omega$ such that\\ $\set{a_i: 1\leq i<\omega,\; a_i \mbox{ is the label
of an edge in }
 N}$ is finite.
Then it is straightforward to show $S$ is an amalgamation class, that is for all $M, N
\in S$ if $M \equiv_{ij} N$ then there is $L \in S$ with
$M \equiv_i L \equiv_j N.$
Hence the complex cylindric algebra $\Ca(S)\in \CA_\omega$.

Now let $X$ be the set of finite $\A$-networks $N$ with nodes
$\subseteq\omega$ such that
\begin{enumerate}
\item each edge of $N$ is either (a) an atom of
$\c A$ or (b) a cofinite subset of $a_+=\set{a_j:1\leq j<\kappa}$ or (c)
a cofinite subset of $a_0=\set{a_0^i:i<2^{\kappa}}$ and
\item $N$ is `triangle-closed', i.e. for all $l, m, n \in nodes(N)$ we
have $N(l, n) \leq N(l,m);N(m,n)$.  That means if an edge $(l,m)$ is
labeled by $1'$ then $N(l,n)= N(mn)$ and if $N(l,m), N(m,n) \leq
a_0$ then $N(l,n).a_0 = 0$ and if $N(l,m)=N(m,n) =
a_j$ (some $1\leq j<\omega$) then $N(l,n).a_j = 0$.
\end{enumerate}
For $N\in X$ let $N'\in\Ca(S)$ be defined by
\[\set{L\in S: L(m,n)\leq
N(m,n) \mbox{ for } m,n\in nodes(N)}\]
For $i,\omega$, let $N\restr{-i}$ be the subgraph of $N$ obtained by deleting the node $i$.
Then if $N\in X, \; i<\omega$ then $\cyl i N' =
(N\restr{-i})'$.
The inclusion $\cyl i N' \subseteq (N\restr{-i})'$ is clear.

Conversely, let $L \in (N\restr{-i})'$.  We seek $M \equiv_i L$ with
$M\in N'$.  This will prove that $L \in \cyl i N'$, as required.
Since $L\in S$ the set $X = \set{a_i \notin L}$ is infinite.  Let $X$
be the disjoint union of two infinite sets $Y \cup Y'$, say.  To
define the $\omega$-network $M$ we must define the labels of all edges
involving the node $i$ (other labels are given by $M\equiv_i L$).  We
define these labels by enumerating the edges and labeling them one at
a time.  So let $j \neq i < \omega$.  Suppose $j\in \nodes(N)$.  We
must choose $M(i,j) \leq N(i,j)$.  If $N(i,j)$ is an atom then of
course $M(i,j)=N(i,j)$.  Since $N$ is finite, this defines only
finitely many labels of $M$.  If $N(i,j)$ is a cofinite subset of
$a_0$ then we let $M(i,j)$ be an arbitrary atom in $N(i,j)$.  And if
$N(i,j)$ is a cofinite subset of $a_+$ then let $M(i,j)$ be an element
of $N(i,j)\cap Y$ which has not been used as the label of any edge of
$M$ which has already been chosen (possible, since at each stage only
finitely many have been chosen so far).  If $j\notin \nodes(N)$ then we
can let $M(i,j)= a_k \in Y$ some $1\leq k < \omega$ such that no edge of $M$
has already been labeled by $a_k$.  It is not hard to check that each
triangle of $M$ is consistent (we have avoided all monochromatic
triangles) and clearly $M\in N'$ and $M\equiv_i L$.  The labeling avoided all
but finitely many elements of $Y'$, so $M\in S$. So
$(N\restr{-i})' \subseteq \cyl i N'$.

Now let $X' = \set{N':N\in X} \subseteq \Ca(S)$.
Then the subalgebra of $\Ca(S)$ generated by $X'$ is obtained from
$X'$ by closing under finite unions.
Clearly all these finite unions are generated by $X'$.  We must show
that the set of finite unions of $X'$ is closed under all cylindric
operations.  Closure under unions is given.  For $N'\in X$ we have
$-N' = \bigcup_{m,n\in \nodes(N)}N_{mn}'$ where $N_{mn}$ is a network
with nodes $\set{m,n}$ and labeling $N_{mn}(m,n) = -N(m,n)$. $N_{mn}$
may not belong to $X$ but it is equivalent to a union of at most finitely many
members of $X$.  The diagonal $\diag ij \in\Ca(S)$ is equal to $N'$
where $N$ is a network with nodes $\set{i,j}$ and labeling
$N(i,j)=1'$.  Closure under cylindrification is given.
Let $\C$ be the subalgebra of $\Ca(S)$ generated by $X'$.
Then $\A = \Ra(\C)$.
Each element of $\A$ is a union of a finite number of atoms and
possibly a co-finite subset of $a_0$ and possibly a co-finite subset
of $a_+$.  Clearly $\A\subseteq\Ra(\C)$.  Conversely, each element
$z \in \Ra(\C)$ is a finite union $\bigcup_{N\in F}N'$, for some
finite subset $F$ of $X$, satisfying $\cyl i z = z$, for $i > 1$. Let $i_0,
\ldots, i_k$ be an enumeration of all the nodes, other than $0$ and
$1$, that occur as nodes of networks in $F$.  Then, $\cyl
{i_0} \ldots
\cyl {i_k}z = \bigcup_{N\in F} \cyl {i_0} \ldots
\cyl {i_k}N' = \bigcup_{N\in F} (N\restr{\set{0,1}})' \in \A$.  So $\Ra(\C)
\subseteq \A$.
$\A$ is relation algebra reduct of $\C\in\CA_\omega$ but has no
complete representation.
Let $n>2$. Let $\B=\Nr_n \C$. Then
$\B\in \Nr_n\CA_{\omega}$, is atomic, but has no complete representation.
However, this algebra is strongly representable, that is $\Cm\At\A$ is representable

Now we show that it is s $\omega$ completely representable, namely a representation that respects all finite cliques.
For $\A$ we play the usual $\omega$ rounded atomic game on atomic networks. We show that \pe\ has a \ws\ in the $\omega$ rounded game.
Assume that \pe\ survives till the $r$ round, and that the current play is in round $r+1$. \pa\ chooses a previously played
network $N_s$, the edge $x,y\in N_t$ and atoms $a,b\in A$ such that $N_t(x,y)\leq a;b.$
Then \pa\ has to choose a witness for this composition and enlarge the network such that $N_{t+1}(x,z)=a$ and
$N_t+1)(x,y)=b$, and we can assume that
$N_{t+1}$ has just one extra node. If there is a witness in $N_t$ there
is nothing to prove so assume not.

Let \pa\ play the triangle move $(N_s, i, j, k,a, b)$
in round $r+1$. \pe\ has to choose labels for the edges $\{(x, k), (k,x)\}$,
$x\in \nodes(N_s)\sim \{i,j\}$.

She chooses the labels for  the edges $(x,k)$ one at a time
and then determines the labels of the reverse edge $(k, x)$
uniquely. We give the uncountably many atoms the green colour \g,
and the countably many the colour $r$.

We have several cases:
If it is not the case that $N_s(x, i)$ and $a$ are both green,
and it is not the case that $N_s(x, j)$ and $b$ are both green, \pe\ lets $N_{s+1}(x,k)$ a new green $\g$.

If $N_s(i,j)=\r$, $N_s(x, i)=\g,$ $N_s(x, j)=\g$ and $a=b$, then \pe\ lets $N_{s+1}(x,k)$ a new $\r$.
If neither, then $N(x, i)=\g$, $a=\g$ and $N_s(x, j)=b$ or $N_s(x, i)=a$ and $N_{s}(x,j)=\g$
and $b=\g$, she lets
$N_{s+1}(x, k)$ a new $\r$.
\item Assume $\A\in \Nr_n\CA_{\omega}$ is atomic. By the above,
we can assume that it is uncountable. Let $F^m$ be as above, like the usual $\omega$ rounded game
defined on atomic networks except that \pa\ s moves are restricted to $m>n$ pebbles with the option that he can re use them.
Now \pe\ has a \ws\ in $F^m$ for all $m>n$, hence \pe\ can win all finite rounded games. This implies that $\At\A$ satisfies the
Lyndon conditions.

\end{enumarab}
\end{proof}

\begin{definition}

\begin{enumarab}
\item  An atom structure $\At$ is weakly representable if $\Tm\At$ is representable; this class is denoted by ${\sf WRCA}_n$
\item An atom structure $\At$ is Lyndon, if $\Cm\At$ satisfies the Lyndon conditions, denoted by ${\sf LCA_n}$
\item An atom structure $\At$ is first order definable if the subalgebra of $\Cm\At$
consisting of all sets of atoms that are first order definable with parameters from
$\At$ is representable.
By first order definable by parameters we mean sets of the form
$$\{a\in \At: \At\models \phi(a \bar{b})\}$$
for some formula $\phi(x, \bar{y})$ of the signature
of $\At$.

\item We denote the rainbow algebra $\sf PEA_{m+3,m+2}$ by $\A_m$.
\end{enumarab}
\end{definition}

\begin{theorem}
\begin{enumarab}
\item These classes are all elementary, but are not finitely axiomatizable.
\item ${\sf CRA}_n \subseteq {\sf LCA}_n\subseteq {\sf FOCA}_n \subseteq {\sf WSA}_n$ and the inclusions are proper
\end{enumarab}

\end{theorem}
\begin{proof}
\begin{enumarab}
\item We show that both ${\sf WRCA}_n$ and ${\sf FCA_n}$ are elementary.
The proof is the same (though the classes are distinct as shown below.)
In any completely  additive variety, like ${\sf RCA_n}$,
we we can tighten the connection between $\RCA_n$ and $\At \RCA_n$.
Let $\At$ be an atom structure. If $\A$ is any atomic algebra in $V$ with atom structure
$\At$, and $\A$ is any subalgebra of with the same atom structure,
then the subalgebra generated by the atoms will be isomorphic to $\Tm \At.$
Then $\At\in \At\RCA_n$ iff $\Tm \At\in \RCA_n$ for every $\At$.
But $\Tm\At$ is completely additive, hence for each term $t\in \Tm\F$, the set of atoms
below $t$ is definable in $\At$ by a first order formula with parameters in $\At$.
By translating these formulas to the equational theory of $\RCA_n$,
we get a set of first order sentences $\Sigma_V$ axiomatizing the atom structures.

To prove non finite axiomatizability. It is not hard to show  that \pa\ has a \ws\ in $G_{\omega}(\A_m)$.
As $\A_m$ is finite it is is not representable and so $\At\A_m\notin {\sf WRCA_n}$.
This holds for each $m<\omega$. But, it is also not hard to show that  \pe\ has a \ws\ in $G_m(\A_m)$, hence
for any non principal ultrafilter over $\omega,$
the ultraproduct $\prod_D\A_m$ satisfies all Lyndon conditions.
But $\prod_F \At\A_m \cong \At(\prod_F \A_m)\in \sf LCA_n$
and ${\sf LCA_n}\subseteq  {\sf WCA_n}$.
A direct  application of Los theorem gives the required.

\item For the first part, we have $\sf LCA_n$ is elementary by definition; also
if $\A$ is completely representable then \pe\ has a \ws\ in the $\omega$ rounded atomic game $G^{\omega}$, cf. \cite{HHbook2},
hence in all finite rounded games, so the first inclusion holds.
In fact, it is not difficult to show that ${\sf LCA}_n={\sf UpUrCRA}_n$.

To prove strictness, let $\A=\PEA_{\omega, \omega}$.
Then \pe\ can win the finite rounded games but the algebra is not completely representable.
Hence $\At\A\in \sf LCA_n\sim \sf CRA_n$; this algebra also witness that ${\sf CRA_n}$ is not elementary.

The second inclusion and last  inclusions are obvious. Strictness of the inclusions will be proved below.
\end{enumarab}
\end{proof}

\begin{corollary} For any class $\K$ having signature between that of $\Df$ and $\PEA$ and any $n>2$,
the class of representable algebras is
not finitely axiomatizable
\end{corollary}

\subsection{When Monk-like and Rainbow algebras do the same thing; a model theoretic approach}

Here we use Model Theory, to prove that non atom canonicity of several classes consisting of subneat reducts.
We need  some fairly standard model theoretic preparations.
This proof unifies Rainbow and Monk like constructions used in \cite{Hodkinson} and
\cite{weak}, and introduces  a new algebra based on another graph.

\begin{theorem}
Let $\Theta$ be an $n$-back-and-forth system
of partial isomorphism on a structure $A$, let $\bar{a}, \bar{b} \in {}^{n}A$,
and suppose that $ \theta = ( \bar{a} \mapsto \bar{b})$ is a map in
$\Theta$. Then $ A \models \phi(\bar{a})$ iff $ A \models
\phi(\bar{b})$, for any formula $\phi$ of $L^n_{\infty \omega}$.
\end{theorem}
\begin{proof} By induction on the structure of $\phi$.
\end{proof}
Suppose that $W \subseteq {}^{n}A$ is a given non-empty set. We can
relativize quantifiers to $W$, giving a new semantics $\models_W$
for $L^n_{\infty \omega}$, which has been intensively studied in
recent times. If $\bar{a} \in W$:
\begin{itemize}
\item for atomic $\phi$, $A\models_W \phi(\bar{a})$
iff $A \models \phi(\bar{a})$

\item the boolean clauses are as expected

\item for $ i < n, A \models_W \exists x_i \phi(\bar{a})$ iff $A \models_W
\phi(\bar{a}')$ for some $ \bar{a}' \in W$ with $\bar{a}' \equiv_i
\bar{a}$.
\end{itemize}

\begin{theorem} If $W$ is $L^n_{\infty \omega}$ definable, $\Theta$ is an
 $n$-\textit{back-and-forth} system
of partial isomorphisms on $A$, $\bar{a}, \bar{b} \in W$, and $
\bar{a} \mapsto \bar{b} \in \Theta$, then $ A \models \phi(\bar{a})$
iff $ A \models \phi(\bar{b})$ for any formula $\phi$ of
$L^n_{\infty \omega}$.
\end{theorem}
\begin{proof} Assume that $W$ is definable by the $L^n_{\infty \omega}$
formula $\psi$, so that $W = \{ \bar{a} \in {}^{n}A:A\models \psi(a)\}$. We may
relativize the quantifiers of $L^n_{\infty \omega}$-formulas to
$\psi$. For each $L^n_{\infty
\omega}$-formula $\phi$ we obtain a relativized one, $\phi^\psi$, by
induction, the main clause in the definition being:
\begin{itemize}
\item $( \exists x_i \phi)^\psi = \exists x_i ( \psi \wedge
\phi^\psi)$.
\end{itemize}
 Then clearly, $ A \models_W \phi(\bar{a})$ iff $ A \models
 \phi^\psi(\bar{a})$, for all $ \bar{a} \in W$.
\end{proof}

The following theorem unifies and generalizes the main theorem in \cite{Hodkinson} and in \cite{weak}.
It shows that sometime Monk like algebras and rainbow algebras do the same thing.
We shall see that each of the constructions has its assets and liabilities.
In the rainbow case the construction can be refined by truncating the greens and reds to be finite, to give sharper results
as shown in theorem \label{smooth}.

While Monk's algebra in one go gives the required result for both relation and cylindric like algebras.
and it can be generalized to show that the class of strongly representable atom structures for
both relation and cylindric algebras is not elementary reproving a profound result
of Hirsch and Hodkinson, using Erdos probabilistic graphs, or {\it into Monk} ultraproducts.

\begin{theorem}\label{hodkinson}

\begin{enumarab}
\item There exists a polyadic equality atom structure $\At$ of dimension $n$,
such that $\Tm\At$ is representable as a $\sf PEA_n$, but not strongly representable.
In fact $\Rd_t\Cm\At$ is not representable for any signature $t$ between that of $\sf Df$ and $\sf Sc$.

\item Furthermore,
there exists an atomic  relation algebra $\R$ that that the set of all basic matrices
forms an $n$ dimensional polyadic basis
${\sf Mat}_n\At\R$ and
$\Tm{\sf Mat}_n\At\R$ is  representable as a polyadic equality algebra, while
$\Rd_{df}\Cm{\sf Mat_n}\At\R$ is not (as a diagonal free cylindric algebra of dimension $n$).
In particular, $\At\R$ is weakly but not strongly representable.
\end{enumarab}
\end{theorem}

\begin{proof}
$L^+$ is the rainbow signature consisting of the binary
relation symbols $\g_i :i<n-1 , \g_0^i: i< |\sf G|$, $\w, \w_i: i <n-2, \r_{jk}^i  i<\omega, j<k<|\R|)$
and the $(n-1)$ ary-relation symbols
$\y_S: S\subseteq {\sf G}$
together with a shade of red $\rho$ which is outside the rainbow signature,
but it is a binary relation, in the sense that it can label edges in coloured
graphs . Here we tale like Hodkinson $\sf G=\sf R=\omega$.
graphs. In the following theorem we shall see that by varying these parameters,
namely when $|\sf G|=n+2$ and $|\sf R|=n+1$ we get sharper results.

Let $\GG$ be the class of all coloured graphs in this rainbow signature.
Let $T_r$ denote the rainbow $L_{\omega_1,\omega}$ theory \cite{HHbook2}.
Let $\G$ by a countable disjoint union of cliques each of size $n(n-1)/2$ or $\N$ with edge relation defined by $(i,j)\in E$ iff $0<|i-j|<N$.
Let $L^+$ be the signature consisting of the binary
relation symbols $(a, i)$, for each $a \in \G \cup \{ \rho \}$ and
$ i < n$. Let $T_m$ denote the following Monk theory:

$M\models T_m$ iff
for all $a,b\in M$, there is a unique $p\in \G\cup \{\rho\}\times n$, such that
$(a,b)\in p$ and if  $M\models (a,i)(x,y)\land (b,j)(y,z)\land (c,l)(x,z)$, then $| \{ i, j, l \}> 1 $, or
$ a, b, c \in \G$ and $\{ a, b, c\} $ has at least one edge
of $\G$, or exactly one of $a, b, c$ -- say, $a$ -- is $\rho$, and $bc$ is
an edge of $\G$, or two or more of $a, b, c$ are $\rho$.

Then there is a countable coloured  $M\in \GG$ of both theories with the following
property:\\
$\bullet$ If $\triangle \subseteq \triangle' \in \GG$, $|\triangle'|
\leq n$, and $\theta : \triangle \rightarrow M$ is an embedding,
then $\theta$ extends to an embedding $\theta' : \triangle'
\rightarrow M$.

For both: Two players, $\forall$ and $\exists$, play a game to build a
labelled graph $M$. They play by choosing a chain $\Gamma_0
\subseteq \Gamma_1 \subseteq\ldots $ of finite graphs in $\GG$; the
union of
the chain will be the graph $M.$
There are $\omega$ rounds. In each round, $\forall$ and $\exists$ do
the following. Let $ \Gamma \in \GG$ be the graph constructed up to
this point in the game. $\forall$ chooses $\triangle \in \GG$ of
size $< n$, and an embedding $\theta : \triangle \rightarrow
\Gamma$. He then chooses an extension $ \triangle \subseteq
\triangle^+ \in \GG$, where $| \triangle^+ \backslash \triangle |
\leq 1$. These choices, $ (\triangle, \theta, \triangle^+)$,
constitute his move. $\exists$ must respond with an extension $
\Gamma \subseteq \Gamma^+ \in \GG$ such that $\theta $ extends to an
embedding $\theta^+ : \triangle^+ \rightarrow \Gamma^+$. Her
response ends the round.
The starting graph $\Gamma_0 \in \GG$ is arbitrary but we will take
it to be the empty graph in $\GG$.
We claim that $\exists$ never gets stuck -- she can always find a suitable
extension $\Gamma^+ \in \GG$.  Let $\Gamma \in \GG$ be the graph built at some stage, and let
$\forall$ choose the graphs $ \triangle \subseteq \triangle^+ \in
\GG$ and the embedding $\theta : \triangle \rightarrow \Gamma$.
Thus, his move is $ (\triangle, \theta, \triangle^+)$.
We now describe $\exists$'s response. If $\Gamma$ is empty, she may
simply plays $\triangle^+$, and if $\triangle = \triangle^+$, she
plays $\Gamma$. Otherwise, let $ F = rng(\theta) \subseteq \Gamma$.
(So $|F| < n$.) Since $\triangle$ and $\Gamma \upharpoonright F$ are
isomorphic labelled graphs (via $\theta$), and $\GG$ is closed under
isomorphism, we may assume with no loss of generality that $\forall$
actually played $ ( \Gamma \upharpoonright F, Id_F, \triangle^+)$,
where $\Gamma \upharpoonright F \subseteq \triangle^+ \in \GG$,
$\triangle^+ \backslash F = \{\delta\}$, and $\delta \notin \Gamma$.
We may view $\forall$'s move as building a labelled graph $ \Gamma^*
\supseteq \Gamma$, whose nodes are those of $\Gamma$ together with
$\delta$, and whose edges are the edges of $\Gamma$ together with
edges from $\delta$ to every node of $F$. The labelled graph
structure on $\Gamma^*$ is given by\\
$\bullet$ $\Gamma$ is an induced subgraph of $\Gamma^*$ (i.e., $
\Gamma \subseteq \Gamma^*$)\\
$\bullet$ $\Gamma^* \upharpoonright ( F \cup \{\delta\} ) =
\triangle^+$.
Now $ \exists$ must extend $ \Gamma^*$ to a complete
graph on the same node and complete the colouring yielding a graph
$ \Gamma^+ \in \GG$. Thus, she has to define the colour $
\Gamma^+(\beta, \delta)$ for all nodes $ \beta \in \Gamma \backslash
F$, in such a way as to meet the required conditions.  For rainbow case \pe\ plays as follows:
\begin{enumarab}

\item If there is no $f\in F$,
such that $\Gamma^*(\beta, f), \Gamma*(\delta ,f)$ are coloured $\g_0^t$ and $\g_0^u$
for some $t,u$, then \pe\ defined $\Gamma^+(\beta, \delta)$ to  be $\w_0$.

\item Otherwise, if for some $i$ with $0<i<n-1$, there is no $f\in F$
such that $\Gamma^*(\beta,f), \Gamma^*(\delta, f)$ are both coloured $\g_i$, then \pe\
defines the colour $\Gamma^+(\beta,\delta)$ to
to be $\w_i$ say the least such

\item Otherwise $\delta$ and $\beta$ are both the apexes on $F$ in $\Gamma^*$ that induce
the same linear ordering on (there are nor green edges in $F$ because
$\Delta^+\in \GG$, so it has no green triangles).
Now \pe\ has no choice but to pick a red. The colour she chooses is $\rho.$

\end{enumarab}
This defines the colour of edges. Now for hyperedges,
for  each tuple of distinct elements
$\bar{a}=(a_0,\ldots a_{n-2})\in {}^{n-1}(\Gamma^+)$
such that $\bar{a}\notin {}^{n-1}\Gamma\cup {}^{n-1}\Delta$ and with no edge $(a_i, a)$
coloured greed in  $\Gamma^+$, \pe\ colours $\bar{a}$ by $\y_{S}$
where
$S=\{i <\omega: \text { there is a $i$ cone with base  } \bar{a}\}$.
Notice that $|S|\leq F$. This strategy works.

For the Monk case \pe\ plays as follows:

Now $ \exists$ must extend $ \Gamma^*$ to a complete
graph on the same node and complete the colouring yielding a graph
$ \Gamma^+ \in \GG$. Thus, she has to define the colour $
\Gamma^+(\beta, \delta)$ for all nodes $ \beta \in \Gamma \backslash
F$, in such a way as to meet the conditions of definition 1. She
does this as follows. The set of colours of the labels in $ \{
\triangle^+(\delta, \phi) : \phi \in F \} $ has cardinality at most
$ n - 1 $. Let  $ i < n$ be a "colour" not in this set. $ \exists$
labels $(\delta, \beta) $ by $(\rho, i)$ for every $ \beta \in
\Gamma \backslash F$. This completes the definition of $ \Gamma^+$.
\\
It remains to check that this strategy works--that the conditions
from the definition of $\GG$ are met. It is
sufficient to note that

\begin{itemize}
 \item if $\phi \in F$ and $ \beta \in \Gamma \backslash F$, then
the labels in $ \Gamma^+$ on the edges of the triangle $(\beta,
\delta, \phi)$ are not all of the same colour ( by choice of $i$ )

\item if $ \beta, \gamma \in \Gamma \backslash F$, then two the
labels in $ \Gamma^+$ on the edges of the triangle $( \beta, \gamma,
\delta )$ are $( \rho, i)$.\\

\end{itemize}

Now there are only countably many
finite graphs in $ \GG$ up to isomorphism, and each of the graphs
built during the game is finite. Hence $\forall$ may arrange to play
every possible $(\triangle, \theta, \triangle^+)$ (up to
isomorphism) at some round in the game. Suppose he does this, and
let $M$ be the union of the graphs played in the game. We check that
$M$ is as required. Certainly, $M \in \GG$, since $\GG$ is clearly
closed under unions of chains. Also, let $\triangle \subseteq
\triangle' \in \GG$ with $|\triangle'| \leq n$, and $ \theta :
\triangle \rightarrow M$ be an embedding. We prove that $\theta$
extends to $\triangle'$, by induction on $d = | \triangle'
\backslash \triangle|.$ If this is $0$, there is nothing to prove.
Assume the result for smaller $d$. Choose  $ a \in \triangle'
\backslash \triangle $ and let $ \triangle^+ = \triangle'
\upharpoonright ( \triangle \cup \{ a \} ) \in \GG$. As, $
|\triangle| < n$, at some round in the game, at which the graph
built so far was $\Gamma$, say, $\forall$ would have played
$(\triangle, \theta, \triangle^+)$ (or some isomorphic triple).
Hence, if $\exists$ constructed $ \Gamma^+$ in that round, there is
an embedding $ \theta^+ : \triangle^+ \rightarrow \Gamma^+ $
extending $\theta$. As $ \Gamma^+ \subseteq M,  \theta^+$ is also an
embedding: $ \triangle^+ \rightarrow M$. Since $ |\triangle'
\backslash \triangle^+| < d, \theta^+$ extends inductively to an
embedding $\theta' : \triangle' \rightarrow M $, as required.

For the rainbow algebra, let $$W_r = \{ \bar{a} \in {}^n M : M \models ( \bigwedge_{i < j < n,
l < n} \neg \rho(x_i, x_j))(\bar{a}) \},$$
and for the Monk like algebra, $W_m$
is defined exactly the same way by discarding assignments whose edges are coloured by one of $n$ reds
$(\rho, i)$, in the former case ee are discarding assignments who have a $\rho$ labelled edge.
We denote both by $W$ relying on context.

The $n$-homogeneity built into
$M$, in all three cases by its construction implies that the set of all partial
isomorphisms of $M$ of cardinality at most $n$ forms an
$n$-back-and-forth system. But we can even go
further. We formulate our statement for the Monk algebra based on $\G$ whose underlying set
is $\N$ since this is the new case. (For the other case the reader is referred to \cite{Hodkinson} and \cite{weak}).

For any permutation $\chi$ of $\omega \cup \{\rho\}$, $\Theta^\chi$
is the set of partial one-to-one maps from $M$ to $M$ of size at
most $n$ that are $\chi$-isomorphisms on their domains. We write
$\Theta$ for $\Theta^{Id_{\omega \cup \{\rho\}}}$.
For any permutation $\chi$ of $\omega \cup \{\rho\}$, $\Theta^\chi$
is an $n$-back-and-forth system on $M$.
These notions are defined for the rainbow case completely analogously.

We may regard any non-empty labelled graph equally as an
$L^+$-structure, in the obvious way. The $n$-homogeneity built into
$M$ by its construction would suggest that the set of all partial
isomorphisms of $M$ of cardinality at most $n$ forms an
$n$-back-and-forth system. This is indeed true, but we can go
further.

Let $\chi$ be a permutation of the set $\omega \cup \{ \rho\}$. Let
$ \Gamma, \triangle \in \GG$ have the same size, and let $ \theta :
\Gamma \rightarrow \triangle$ be a bijection. We say that $\theta$
is a $\chi$-\textit{isomorphism} from $\Gamma$ to $\triangle$ if for
each distinct $ x, y \in \Gamma$,
\begin{itemize}
\item If $\Gamma ( x, y) = (a, j)$ with $a\in \N$, then there exist unique $l\in \N$ and $r$ with $0\leq r<N$ such that
$a=Nl+r$.
\begin{equation*}
\triangle( \theta(x),\theta(y)) =
\begin{cases}
( N\chi(i)+r, j), & \hbox{if $\chi(i) \neq \rho$} \\
(\rho, j),  & \hbox{otherwise.} \end{cases}
\end{equation*}
\end{itemize}

\begin{itemize}
\item If $\Gamma ( x, y) = ( \rho, j)$, then
\begin{equation*}
\triangle( \theta(x),\theta(y)) \in
\begin{cases}
\{( N\chi(\rho)+s, j): 0\leq s < N \}, & \hbox{if $\chi(\rho) \neq \rho$} \\
\{(\rho, j)\},  & \hbox{otherwise.} \end{cases}
\end{equation*}
\end{itemize}

We now have for any permutation $\chi$ of $\omega \cup \{\rho\}$, $\Theta^\chi$
is the set of partial one-to-one maps from $M$ to $M$ of size at
most $n$ that are $\chi$-isomorphisms on their domains. We write
$\Theta$ for $\Theta^{Id_{\omega \cup \{\rho\}}}$.

For any any permutation $\chi$ of $\omega \cup \{\rho\}$, $\Theta^\chi$
is an $n$-back-and-forth system on $M$.

Clearly, $\Theta^\chi$ is closed under restrictions. We check the
``forth" property. Let $\theta \in \Theta^\chi$ have size $t < n$.
Enumerate $\dom(\theta)$, $\rng(\theta).$ respectively as $ \{ a_0,
\ldots, a_{t-1} \}$, $ \{ b_0,\ldots b_{t-1} \}$, with $\theta(a_i)
= b_i$ for $i < t$. Let $a_t \in M$ be arbitrary, let $b_t \notin M$
be a new element, and define a complete labelled graph $\triangle
\supseteq M \upharpoonright \{ b_0,\ldots, b_{t-1} \}$ with nodes
$\{ b_0,\ldots, b_{t} \}$ as follows.\\

Choose distinct "nodes"$ e_s < N$ for each $s < t$, such that no
$(e_s, j)$ labels any edge in $M \upharpoonright \{ b_0,\dots,
b_{t-1} \}$. This is possible because $N \geq n(n-1)/2$, which
bounds the number of edges in $\triangle$. We can now define the
colour of edges $(b_s, b_t)$ of $\triangle$ for $s = 0,\ldots, t-1$.

\begin{itemize}
\item If $M ( a_s, a_t) = ( Ni+r, j)$, for some $i\in \N$ and $0\leq r<N$, then
\begin{equation*}
\triangle( b_s, b_t) =
\begin{cases}
( N\chi(i)+r, j), & \hbox{if $\chi(i) \neq \rho$} \\
\{(\rho, j)\},  & \hbox{otherwise.} \end{cases}
\end{equation*}
\end{itemize}

\begin{itemize}
\item If $M ( a_s, a_t) = ( \rho, j)$, then assuming that $e_s=Ni+r,$ $i\in \N$ and $0\leq r<N$,
\begin{equation*}
\triangle( b_s, b_t) =
\begin{cases}
( N\chi(\rho)+r, j), & \hbox{if $\chi(\rho) \neq \rho$} \\
\{(\rho, j)\},  & \hbox{otherwise.} \end{cases}
\end{equation*}
\end{itemize}

This completes the definition of $\triangle$. It is easy to check
that $\triangle \in \GG$. Hence, there is a graph embedding $ \phi : \triangle \rightarrow M$
extending the map $ Id_{\{ b_0,\ldots b_{t-1} \}}$. Note that
$\phi(b_t) \notin \rng(\theta)$. So the map $\theta^+ = \theta \cup
\{(a_t, \phi(b_t))\}$ is injective, and it is easily seen to be a
$\chi$-isomorphism in $\Theta^\chi$ and defined on $a_t$.
The converse,``back" property is similarly proved ( or by symmetry,
using the fact that the inverse of maps in $\Theta$ are
$\chi^{-1}$-isomorphisms).
We can also derive a connection between classical and
relativized semantics in $M$, over the set $W$:\\

Recall that $W$ is simply the set of tuples $\bar{a}$ in ${}^nM$ such that the
edges between the elements of $\bar{a}$ don't have a label involving
$\rho$. Their labels are all of the form $(Ni+r, j)$. We can replace $\rho$-labels by suitable $(a, j)$-labels
within an $n$-back-and-forth system. Thus, we may arrange that the
system maps a tuple $\bar{b} \in {}^n M \backslash W$ to a tuple
$\bar{c} \in W$ and this will preserve any formula
containing no relation symbols $(a, j)$ that are ``moved" by the
system. Indeed we show that the
classical and $W$-relativized semantics agree.
$M \models_W \varphi(\bar{a})$ iff $M \models \varphi(\bar{a})$, for
all $\bar{a} \in W$ and all $L^n$-formulas $\varphi$.

The proof is by induction on $\varphi$. If $\varphi$ is atomic, the
result is clear; and the boolean cases are simple.
Let $i < n$ and consider $\exists x_i \varphi$. If $M \models_W
\exists x_i \varphi(\bar{a})$, then there is $\bar{b} \in W$ with
$\bar{b} =_i \bar{a}$ and $M \models_W \varphi(\bar{b})$.
Inductively, $M \models \varphi(\bar{b})$, so clearly, $M \models_W
\exists x_i \varphi(\bar{a})$.
For the (more interesting) converse, suppose that $M \models_W
\exists x_i \varphi(\bar{a})$. Then there is $ \bar{b} \in {}^n M$
with $\bar{b} =_i \bar{a}$ and $M \models \varphi(\bar{b})$. Take
$L_{\varphi, \bar{b}}$ to be any finite subsignature of $L$
containing all the symbols from $L$ that occur in $\varphi$ or as a
label in $M \upharpoonright \rng(\bar{b})$. (Here we use the fact
that $\varphi$ is first-order. The result may fail for infinitary
formulas with infinite signature.) Choose a permutation $\chi$ of
$\omega \cup \{\rho\}$ fixing any $i'$ such that some $(i'N+r, j)$
occurs in $L_{\varphi, \bar{b}}$ for some $r<N$, and moving $\rho$.
Let $\theta = Id_{\{a_m : m \neq i\}}$. Take any distinct $l, m \in
n \setminus \{i\}$. If $M(a_l, a_m) = (i'N+r, j)$, then $M( b_l,
b_m) = (i'N+r, j)$ because $ \bar{a} = _i \bar{b}$, so $(i'N+r, j)
\in L_{\varphi, \bar{b}}$ by definition of $L_{\varphi, \bar{b}}$.
So, $\chi(i') = i'$ by definition of $\chi$. Also, $M(a_l, a_m) \neq
( \rho, j)$(any $j$) because $\bar{a} \in W$. It now follows that
$\theta$ is a $\chi$-isomorphism on its domain, so that $ \theta \in
\Theta^\chi$.
Extend $\theta $ to $\theta' \in \Theta^\chi$ defined on $b_i$,
using the ``forth" property of $ \Theta^\chi$. Let $
\bar{c} = \theta'(\bar{b})$. Now by choice of of $\chi$, no labels
on edges of the subgraph of $M$ with domain $\rng(\bar{c})$ involve
$\rho$. Hence, $\bar{c} \in W$.
Moreover, each map in $ \Theta^\chi$ is evidently a partial
isomorphism of the reduct of $M$ to the signature $L_{\varphi,
\bar{b}}$. Now $\varphi$ is an $L_{\varphi, \bar{b}}$-formula. Hence
by Theorem 6 applied to $L_{\varphi, \bar{b}}$, and lemma 11, we
have $M \models \varphi(\bar{a})$ iff $M \models \varphi(\bar{c})$.
So $M \models \varphi(\bar{c})$. Inductively, $M \models_W
\varphi(\bar{c})$. Since $ \bar{c} =_i \bar{a}$, we have $M
\models_W \exists x_i \varphi(\bar{a})$ by definition of the
relativized semantics. This completes the induction.

Let $L$ be $L^+$ without the red labels.
The logics $L^n, L^n_{\infty \omega}$ are taken in this
signature.

For an $L^n_{\infty \omega}$-formula $\varphi $,  define
$\varphi^W$ to be the set $\{ \bar{a} \in W : M \models_W \varphi
(\bar{a}) \}$.

Then the set algebra $\A$ (actually  in all three cases) is taken to be the relativized set algebra with domain
$$\{\varphi^W : \varphi \,\ \textrm {a first-order} \;\ L^n-
\textrm{formula} \}$$  and unit $W$, endowed with the algebraic
operations ${\sf d}_{ij}, {\sf c}_i, $ etc., in the standard way, and of course formulas are taken in the suitable
signature.
The completion of $\A$ is the algebra $\C$ with universe $\{\phi^W: \phi\in L_{\infty,\omega}^n\}$
with operations defined as for $\A$, namely, usual cylindrifiers and diagonal elements,
reflecting existential quantifiers, polyadic operations and equality.
Indeed, $\A$ is a representable (countable) atomic polyadic algebra of dimension $n$

Let $\cal S$ be the polyadic set algebra with domain  $\wp ({}^{n} M )$ and
unit $ {}^{n} M $. Then the map $h : \A
\longrightarrow S$ given by $h:\varphi ^W \longmapsto \{ \bar{a}\in
{}^{n} M: M \models \varphi (\bar{a})\}$ can be checked to be well -
defined and one-one. It clearly respects the polyadic operations. So it is a representation of $\A.$
A formula $ \alpha$  of  $L^n$ is said to be $MCA$
('maximal conjunction of atomic formulas') if (i) $M \models \exists
x_0\ldots x_{n-1} \alpha $ and (ii) $\alpha$ is of the form
$$\bigwedge_{i \neq j < n} \alpha_{ij}(x_i, x_j),$$
where for each $i,j,\alpha_{ij}$ is either $x_i=x_i$ or $R(x_i,x_j)$
for some binary relation symbol $R$ of $L$.

Let $\varphi$ be any $L^n_{\infty\omega}$-formula, and $\alpha$ any
$MCA$-formula. If $\varphi^W \cap \alpha^W \neq \emptyset $, then
$\alpha^W \subseteq \varphi^W $.
Indeed, take $\bar{a} \in  \varphi^W \cap \alpha^W$. Let $\bar{a} \in
\alpha^W$ be arbitrary. Clearly, the map $( \bar{a} \mapsto
\bar{b})$ is in $\Theta$. Also, $W$ is
$L^n_{\infty\omega}$-definable in $M$, since we have
$$ W = \{
\bar{a} \in {}^n M : M \models (\bigwedge_{i < j< n} (x_i = x_j \vee
\bigvee_{R \in L} R(x_i, x_j)))(\bar{a})\}.$$
We have $M \models_W \varphi (\bar{a})$
 iff $M \models_W \varphi (\bar{b})$. Since $M \models_W \varphi (\bar{a})$, we have
$M \models_W \varphi (\bar{b})$. Since $\bar{b} $ was arbitrary, we
see that $\alpha^W \subseteq \varphi^W$.
Let $$F = \{ \alpha^W : \alpha \,\ \textrm{an $MCA$},
L^n-\textrm{formula}\} \subseteq \A.$$
Evidently, $W = \bigcup F$. We claim that
$\A$ is an atomic algebra, with $F$ as its set of atoms.
First, we show that any non-empty element $\varphi^W$ of $\A$ contains an
element of $F$. Take $\bar{a} \in W$ with $M \models_W \varphi
(\bar{a})$. Since $\bar{a} \in W$, there is an $MCA$-formula $\alpha$
such that $M \models_W \alpha(\bar{a})$. Then $\alpha^W
\subseteq \varphi^W $. By definition, if $\alpha$ is an $MCA$ formula
then $ \alpha^W$ is non-empty. If $ \varphi$ is
an $L^n$-formula and $\emptyset \neq \varphi^W \subseteq \alpha^W $,
then $\varphi^W = \alpha^W$. It follows that each $\alpha^W$ (for
$MCA$ $\alpha$) is an atom of $\A$.

Define $\C$ to be the complex algebra over $\At\A$, the atom structure of $\A$.
Then $\cal C$ is the completion of $\A$. The domain of $\cal C$ is $\wp(\At\A)$. The diagonal ${\sf d}_{ij}$ is interpreted
as the set of all $S\in \At\A$ with $a_i=a_j$ for some $\bar{a}\in S$.
The cylindrification ${\sf c}_i$ is interpreted by ${\sf c}_iX=\{S\in \At\A: S\subseteq c_i^{\A}(S')\text { for some } S'\in X\}$, for $X\subseteq \At\A$.
Finally ${\sf p}_{ij}X=\{S\in \At\A: S\subseteq {\sf p}_{ij}^{\A}(S')\text { for some } S'\in X\}.$
Let $\cal D$ be the relativized set algebra with domain $\{\phi^W: \phi\text { an $L_{\infty\omega}^n$ formula }\}$,  unit $W$
and operations defined like those of $\cal A$.

${\cal C}\cong \cal D$, via the map $X\mapsto \bigcup X$.

The map is clearly injective. It is surjective, since
$$\phi^W=\bigcup\{\alpha^W: \alpha \text { an $MCA$-formula}, \alpha^W\subseteq \phi^W\}$$
for any $L_{\infty\omega}^n$ formula $\phi$.
Preservation of the Boolean operations and diagonals is clear.
We check cylindrifications. We require that for any $X\subseteq \At\A$,
we have
$\bigcup {\sf c}_i^{\cal C}X={\sf c}_i^{\cal D}(\bigcup X)$ that is
$$\bigcup \{S\in At\A: S\subseteq {\sf c}_i^{\A}S'\text { for some $S'\in X$ }\}=$$
$$\{\bar{a}\in W: \bar{a}\equiv_i \bar{a'} \text { for some } \bar{a'}\in \bigcup X\}.$$
Let $\bar{a}\in S\subseteq c_iS'$, where $S'\in X$. So there is $\bar{a'}\equiv_i \bar{a}$ with $\bar{a'}\in S'$, and so $\bar{a'}\in \bigcup X$.

Conversely, let $\bar{a}\in W$ with $\bar{a}\equiv_i \bar{a'}$ for some $\bar{a'}\in \bigcup X$.
Let $S\in At\A$, $S'\in X$ with $\bar{a}\in S$ and $\bar{a'}\in S'$.
Choose $MCA$ formulas $\alpha$ and $\alpha'$ with $S=\alpha^W$ and $S'=\alpha'^{W}$.
then $\bar{a}\in \alpha^{W}\cap (\exists x_i\alpha')^W$ so $\alpha^W\subseteq (\exists x_i\alpha')^W$, or $S\subseteq c_i^{\A}(S')$.
The required now follows. We leave the checking of substitutions to the reader.

We recall that for the rainbow signature, a formula $ \alpha$  of  $L^n$ is said to be $MCA$
('maximal conjunction of atomic formulas') if (i) $M \models \exists
x_0\ldots x_{n-1} \alpha $ and (ii) $\alpha$ is of the form
$$\bigwedge_{i \neq j < n} \alpha_{ij}(x_i, x_j) \land y_S(x_0,\ldots x_{n-1}),$$
where for each $i,j,\alpha_{ij}$ is either $x_i=x_i$ or $R(x_i,x_j)$
for some binary relation symbol $R$ of the rainbow signature.

A formula $\alpha$  is also an $MCA$ says that the set it defines in ${}^n M$
is nonempty, and that if $M \models \alpha (\bar{a})$ then the graph
$M \upharpoonright \rng (\bar{a})$ is determined up to isomorphism
and has no edge whose label is of the form $\rho$, in the rainbow case, and $(\rho, i)$, $i<n$ for the other case.
Now we have for any permutation $\chi$ of $\omega \cup \{\rho\}$, $\Theta^\chi$
is an $n$-back-and-forth system on $M$.
Hence, any two
tuples satisfying $\alpha$ are isomorphic and one is mapped to the
other by the $n$-back-and-forth system $\Theta$ of partial isomorphisms from $M$ to $M$; they are the {\it same} coloured graph.
No $L^n_{\infty \omega}$- formula can distinguish them. So $\alpha$
defines an atom of $\A$ --- it is literally indivisible. Since the
$MCA$ - formulas clearly 'cover' $W$, the atoms defined by them are
dense in $\A$. This is the informal idea proved rigourously.
So $\A$ again is atomic.

In the rainbow case there is a one to one correspondence between $MCA$ formulas and
$n$ coloured graphs whose edges are not labelled by the shade of red
$\rho$, so that in particular, those coloured graphs are the atoms of the algebra.
In the other cases, as previously mentioned,  $MCA$ formulas also define the atoms.

To prove item (2) we use Monk's algebras. Any of the two will do just as well:

We show that their atom structure consists of the $n$
basic matrices of a relation algebra.
In more detail, we define a relation algebra atom structure $\alpha(\G)$ of the form
$(\{1'\}\cup (\G\times n), R_{1'}, \breve{R}, R_;)$.
The only identity atom is $1'$. All atoms are self converse,
so $\breve{R}=\{(a, a): a \text { an atom }\}.$
The colour of an atom $(a,i)\in \G\times n$ is $i$. The identity $1'$ has no colour. A triple $(a,b,c)$
of atoms in $\alpha(\G)$ is consistent if
$R;(a,b,c)$ holds. Then the consistent triples are $(a,b,c)$ where

\begin{itemize}

\item one of $a,b,c$ is $1'$ and the other two are equal, or

\item none of $a,b,c$ is $1'$ and they do not all have the same colour, or

\item $a=(a', i), b=(b', i)$ and $c=(c', i)$ for some $i<n$ and
$a',b',c'\in \G$, and there exists at least one graph edge
of $G$ in $\{a', b', c'\}$.

\end{itemize}

$\alpha(\G)$ can be checked to be a relation atom structure. It is exactly the same as that used by Hirsch and Hodkinson in \cite{HHbook}, except
that we use $n$ colours, instead of just $3$, so that it a Monk algebra not a rainbow one. However, some monochromatic triangles
are allowed namely the 'dependent' ones.
This allows the relation algebra to have an $n$ dimensional cylindric basis
and, in fact, the atom structure of of $M(\G)$ is isomorphic (as a cylindric algebra
atom structure) to the atom structure $\Mat_n$ of all $n$-dimensional basic
matrices over the relation algebra atom structure $\alpha(\G)$.

We have in all cases a labelled graph  defined $M$ as a model of a first order theory in the Monk case
and of $L_{\omega_1, \omega}$ in the rainbow case. W
e have also relativized $^nM$ to  $W\subseteq {}^nM$ by deleting assignments
whose edges involve reds, and we defined the term algebra $\A$ as the atomic relativized set algebra
with unit $W$. In the case of  Monk's algebra we defined the relation algebra $\sf R$, such that
$\At\A\cong {\sf Mat}_n{\sf R}$. And in all cases the complex algebra $\C$
of the atom structure is isomorphic to the set algebra
$\{\phi^M: \phi\in L_{\infty, \omega}^n\}$.
Now the $\sf Df$ reduct of this algebra cannot be representable for else this induces a complete representation
of $\Rd_{df}\A$, hence a complete representation of $\A$, which in turn induces a representation of $\C$.

\end{proof}

We refer the reader to the notion of $m$ smooth representation \cite{HHbook}, which is a local form of representation where roughly cylindrifiers have witnesses only
on $<m$ cliques. Though formulated for relation algebras,
the definition lifts easily to cylindric algebras. (Several notions of
other relativized representations will be recalled below,
in the context of omitting types in clique guarded semantics, \ref{blurs}.)

\begin{theorem}\label{smooth} For every $n\geq 3$ there exists a
countable atomic $\sf PEA_n$, such that the $\CA$ reduct of
its completion  does not have an $n+4$ smooth representation, in particular, it is not  representable. Furthermore, its
${\sf Df}$ reduct is not representable.
\end{theorem}
\begin{proof}
Here we closely follow \cite{Hodkinson}; but our reds and greens are finite,
so we obtain a stronger result. We take $|{\sf G}|=n+2$, and ${\sf R}=n=1$.
Let $L^+$ be the rainbow signature consisting of the binary
relation symbols $\g_i :i<n-1 , \g_0^i: i< n+2, \w, \w_i: i <n-2, \r_{jk}^i  (i<n+1, j<k<n)$
and the $(n-1)$ ary-relation symbols
$\y_S: S\subseteq n+2)$, together with a shade of red $\rho$ that is outside the rainbow
signature but is a binary relation in the sense that it can label edges
of coloured graphs. Let $\GG$ be the class of corresponding rainbow coloured graphs.
By the same methods as above, there is a countable model  $M\in \GG$ with the following
property:\\
$\bullet$ If $\triangle \subseteq \triangle' \in \GG$, $|\triangle'|
\leq n$, and $\theta : \triangle \rightarrow M$ is an embedding,
then $\theta$ extends to an embedding $\theta' : \triangle'
\rightarrow M$. Now let $W = \{ \bar{a} \in {}^n M : M \models ( \bigwedge_{i < j < n,
l < n} \neg \rho(x_i, x_j))(\bar{a}) \}$. Then $\A$ with universe $\{\phi^W: \phi\in L_n\}$ and operations defined the usual way,
is representable, and its
completion, the complex algebra over the  above rainbow atom structure, $\C$ has universe $\{\phi^W: \phi\in L_{\infty, \omega}^n\}$.

We show that $\C$ is as desired. Assume, for contradiction, that $g:\C\to \wp(V)$ induces a relativized $n+4$ flat representation.
Then $V\subseteq {}^nN$ and we can assume that
$g$ is injective because $\C$ is simple. First there are $b_0,\ldots b_{n-1}\in N$ such $\bar{b}\in h(\y_{n+2}(x_0,\ldots x_{n-1}))^W$.
This tuple will be the base of finitely many cones, that will be used to force an inconsistent triple of reds.
This is because $\y_{n+2}(\bar{x})^W\neq \emptyset$.  For any $t<n+3$, there is a $c_t\in N$, such
that $\bar{b}_t=(b_0,\ldots b_{n-2},\ldots c_t)$ lies in $h(\g_0^t(x_0, x_{n-1})^W$ and in $h(\g_i(x_i, x_{n-1})^W$ for each $i$ with
$1\leq i\leq n-2$. The $c_t$'s are the apexes of the cones with base $\y_{n+2}$.
Take the formula
$$\phi_t=\y_{n+2}(x_0,\ldots ,x_{n-2})\to \exists x_{n-1}(\g_0^t(x_0, x_{n-1}))\land \bigwedge_{1\leq i\leq n-2}\g_i(x_i, x_{n-1})),$$
then $\phi_t^{W}=W$. Pick $c_t$ and $\bar{b_t}$ as above, and define for each $s<t<n+3,$ $\bar{c_{st}}$ to be
$(c_s,b_1,\ldots  b_{n-2}, c_t)\in {}^nN.$
Then $\bar{c}_{st}\notin h((x_0, x_{n-1})^W$. Let $\mu$ be the formula
$$x_0=x_{n-1}\lor \w_0(x_0, x_{n-1})\lor \bigvee \g(x_0, x_{n-1}),$$
the latter formula is a first order formula consisting of the disjunction of  the (finitely many ) greens.
For $j<k<n$, let $R_{jk}$ be the $L_{\infty\omega}^n$-formula $\bigvee_{i<\omega}\r_{jk}^i(x_0, x_{n-1})$.
Then
$\bar{c}_{st}\notin h(\mu^W)$, now for each $s<t< n+3$, there are $j<k<n$ with $c_{st}\in h(R_{jk})^W.$
By the pigeon- hole principle, there are $s<t< n+3$ and $j<k<n$
with $\bar{c}_{0s}, \bar{c}_{0t}\in h(R_{jk}^W)$. We have also $\bar{c}_{st}\in h(R_{j',k'}^W)$
for some $j', k'$ then the sequence $(c_0, c_s,\ldots,  b_2,\ldots b_{n-2},\ldots, c_t)\in h(\chi^W)$
where
$$\chi=(\exists_1R_{jk})(\land (\exists x_{n-1}(x_{n-1}=x_1\land \exists x_1 R_{jk})\land (\exists x_0(x_0=x_1)\land \exists x_1R_{j'k})),$$
so $\chi^W\neq \emptyset$. Let $\bar{a}\in \chi ^W$. Then $M\models _W  R_{jk}(a_0,a_{n-1})\land R_{jk}(a_0,a_1)\land R_{j'k'}(a_1, a_{n-1})$.
Hence there are
$i$, $i'$ and $i''<\omega$ such that
$$M\models _W\r_{jk}^{i}(a_0,a_{n-1})\land \r_{jk}^{i'}(a_0,a_1)\land \r_{j'k'}^{i''}(a_1, a_{n-1}).$$
But this triangle is inconsistent. Note that this inconsistent red was forced by an $n+4$ red clique
labelling edges between apexes of the same cone, with base labelled by $\y_{n+2}$.

For the last part, if its $\sf Df$ reduct is representable, then $\Rd_{df}\A$ will be completely representable,
hence $\A$ itself will be completely representable because it is generated by elements whose dimension set
$<n$, which is a contradiction.
\end{proof}
\begin{corollary}\label{can} We have $\C\notin S\Nr_n\CA_{n+4}$. In particular for any $k\geq 4$,  the variety
$S\Nr_n\CA_{n+k}$ is not atom canonical.
\end{corollary}
\begin{proof}
The first part. Assume, for contradiction,  that $\C\in S\Nr_n\CA_{n+4}$; let $\C\subseteq \Nr_n\D$.
Then $\C^+\in S_c\Nr_n\D^+$, and $\D^+$ is of course atomic. We show that $\C^+$ has an $n+4$ dimensional hyperbasis,
then that it has an $n+4$ smooth representation, which contradicts the previous
theorem. Here again hyperbasis are defined for cylindric algebras by a straightforward lifting from the relation algebra case.

First note that for every $n\leq l\leq m$, $\Nr_l\D^+$ is atomic.
Indeed, if $x$ is an atom in $\D^+$, and and $n\leq l<m$,
then ${\sf c}_{l}\ldots {\sf c}_{m-l+1}x$ is an atom in $\Nr_l\D^+$,
so if $c\neq 0$ in the latter, then there exists $0\neq a\in \At\D^+$,
such that $a\leq c$, and so ${\sf c}_{l}\ldots {\sf c}_{m-1+1}a\leq c_{l}\ldots {\sf c}_{m-1+1}c=c$.

Let $\Lambda=\bigcup_{k<n+3}\At\Nr_k\D^+$, and let $\lambda\in \Lambda$.
In this proof we follow closely section 13.4 in \cite{HHbook}. The details are omitted because they are identical
to the corresponding ones in op.cit.
For each atom $x$ of $\D$, define $N_x$, easily checked to be an $m$ dimensional   $\Lambda$ hypernetwork, as follows.
Let $\bar{a}\in {}^{n+4}n+4$ Then if $|a|=n$,  $N_x(a)$ is the unique atom $r\in \At\D$ such that $x\leq {\sf s}_{\bar{a}}r$.
Here substitutions are defined as above.
If $n\neq |\bar{a}| <n+3$, $N_x(\bar{a})$ the unique atom $r\in \Nr_{|a|}\D$ such that $x\leq s_{\bar{a}}r.$
$\Nr_{|a|}\D$ is easily checked to be atomic, so this is well defined.

Otherwise, put  $N_x(a)=\lambda$.
Then $N_x$ as an $n+4$ dimensional $\Lambda$ hyper-network, for each such chosen $x$ and $\{N_x: x\in \At\C\}$
is an $n+4$ dimensional $\Lambda$ hyperbasis.
Then viewing those as a saturated set of mosaics, one can can construct a complete
$n+4$ smooth representation of $M$ of $\C$.
Alternatively, one can use a standard step by step argument.Let $L(A)$ denotes the signature that contains
an $n$ ary predicate for every $a\in A$. For $\phi\in L(A)_{\omega,\infty}^n$,
let $\phi^{M}=\{\bar{a}\in C^{n+3}(M):\M\models \phi(\bar{a})\}$,
and let $\D$ be the algebra with universe $\{\phi^{M}: \phi\in L(C)_{\infty, \omega}^n\}$ with usual
Boolean operations, cylindrifiers and diagonal elements, cf. theorem 13.20 i
n \cite{HHbook}. The polyadic operations are defined
by swapping variables.

This is a $\sf PEA_n$; here semantics is defined as expected in the clique guarded fragment of first order logic.
Define $\D_0$ be the algebra consisting of those $\phi^{M}$ where $\phi$ comes from $L(C)_n$.
Assume that $M$ is $n$ square, then certainly $\D^0$ is a subalgebra of the $\sf Crs_n$
with domain $\wp(C^n(M))$ so $\D\in {\sf Crs_n}$. The unit $C^n(M)$ of $\D$ is symmetric,
closed under substitutions, so
$\D\in \sf G_n$, then $\D^0(M)\in \sf G_n$. If $M$ is $n$ flat we
have that cylindrifiers commute by definition,
hence we are done.

Now assume that
$M$ is infinitary $n$ flat. Then $\D\in \CA_n$ is the same. We show that $\D$ is atomic.
Let $\phi^C$ be a non zero element.
Choose $\bar{a}\in \phi^C$, and consider the infinitary $n$ type
$\{\psi\in L_{\infty}: M\models_C \psi(\bar{a})\}.$ Let $\tau$ be its conjunct.
Then $\tau\in L_{\infty}$, and $\tau$ is an atom, as required

Now the neat embedding is defined by $\theta(r)=r(\bar{x})$.
Preservation of operations is straightforward.  We show that $\theta$ is injective.
Let $r\in A$ be non-zero. But $M$ is a relativized representation, so there $\bar{a}\in M$
with $r(\bar{a})$ hence $\bar{a}$ is a clique in $M$,
and so $M\models r(\bar{x})(\bar{a})$ and $\bar{a}\in \theta(r)$. and we are done.

Finally, we check that it is a complete embedding under the assumption that
$M$ is a complete relativized representation.
$\A$ is atomic. Let $\phi\in L_{\infty}$ be such that $\phi^C\neq 0$.
Let $\bar{a}\in \phi^C$. Since
$M$ is complete and $a\in C^n(M$) there is $\alpha\in \At\A$, such
that $M\models \alpha(\bar{a})$, then $\theta(\alpha).\phi^C\neq 0.$
and we are done.

Then $\D_0$ is also a $\sf PEA_n$ and $\A$ embeds into the $n$ neat reduct of both.
If $M$ is complete, then the embedding is also complete.

For the second part $\A$ as above is in $\RCA_n=\bigcap _{k\in \omega}S\Nr_n\CA_{n+k}$,
but $\C$, its completion, is not in $S\Nr_n\CA_{n+k}$ for all $k\geq 4$,
since $S\Nr_n\CA_{n+k}\subseteq S\Nr_n\CA_{n+4}$ for each such
$k$.
\end{proof}

In theorem \ref{hodkinson} we used two Monk like algebras and one rainbow algebra.
We saw in the following theorem that the rainbow algebra is more flexible.

As mentioned in the introduction, the difference between a Monk like algebra and a rainbow algebra is that in the former
the number of pebbles used is determined by a Ramsey uncontrollable large number, while in rainbow construction the number of pebbles used
are basically the reds and the greens, and these can be controlled. As we have seen, in fact a \ws\ for \pa\
excludes neat embedability into extra dimensions
determined by the number of pebbles used  \pa\ if he has a \ws\ .
In Monk like algebras, it can happen that we have control the number of
pebbles when we are lucky enough to use amalgamation moves (see the next section). But when we don't, often, but not always,
rainbow constructions can offer
solace. (Witness Hirsch Hodkinson's result of the non finite axiomatizability of $\sf RA_{n+1}$ over $\RA_n$, when $n\geq 4$,
see \cite{HHbook} theorem 17.18.).

\subsection{Lifting Monk's algebras to the transfinite; the famous 2.12}

Now we give a solution to the famous problem 2.12
lifting it to infinite dimensions, using a lifting
argument due to Monk. A weaker solution of this problem is credited to Pigozzi in \cite{HMT1},
but we have not seen a published prove. A stronger result obtained jointly with Robin Hirsch proves this result
for many cylindric-like algebras including various diagonal free reducts of
cylindric and polyadic equality algebras.
This  construction  will
be outlined below, and slightly refined.
In the next theorem when we say 'finite schema axiomatizable' for a class of algebras of finite dimension,
we mean not finitely  axiomatizable.

\begin{theorem}\label{2.12}
For $\alpha>2$ (infinite included), and $k\in \omega$ and $r\in \omega$,
there is an $\A^r\in \Nr_{\alpha}{\sf CA}_{\alpha+k}\sim S\Nr_{\alpha}{\sf CA}_{\alpha+k+1},$
such that $\prod \A^r/F\in S\Nr_{\alpha}\CA_{\alpha+k+1}$.
In particular,  the variety
$S\Nr_{\alpha}\CA_{\alpha+k+1}$ is not axiomatizable by a finite schema over
$S\Nr_{\alpha}\CA_{\alpha+k}$. Furthermore, the variety $\RCA_{\alpha}$
is not finite schema axiomatizable over $S\Nr_{\alpha}\CA_{\alpha+k}$ for all finite$k$.
\end{theorem}
\begin{proof}
Assume that $3\leq m\leq n$, and let
$$\C(m,n,r)=\Ca(H_m^{n+1}(\A(n,r),  \omega)),$$
be as defined in \cite{HHbook}, definition 15.3.
Here $\A(n,r)$ is a finite Monk-like relation algebra \cite{HHbook2}, definition 15.2.
which has an $n$ dimensional hyperbasis $H_m^{n+1}(\A(n,r), \omega)$
consisting of all hypernetworks.
The hyperedges are labelled by $\omega$; it is an infinite cylindric algebra of dimension $m$, and it is a neat reduct.
In fact we have:
\begin{enumarab}
\item For any $r$ and $3\leq m\leq n<\omega$, we
have $\C(m,n,r)\in \Nr_m{\sf PEA}_n$.

\item $\C(m,n,r)\notin S\Nr_n\CA_{n+m+1}$ but $\prod_r \C(m,n,r)/F\in S\Nr_n\CA_{m+n+1}$ for any non principal ultrafilter
on $\omega$.
\item  For $m<n$ and $k\geq 1$, there exists $x_n\in \C(n,n+k,r)$ such that $\C(m,m+k,r)\cong \Rl_{x}C(n, n+k, r).$
An analogous result holds for quasi-polyadic equality algebras.
\end{enumarab}
Let us check this properties:
\begin{enumarab}
\item  $H_n^{n+1}(\A(n,r), \omega)$ is a wide $n$ dimensional $\omega$ symmetric hyperbases, so $\Ca H\in {\sf PEA}_n.$
But $H_m^{n+1}(\A(n,r),\omega)=H|_m^{n+1}$.
Thus
$$\C_r=\Ca(H_m^{n+1}(\A(n,r), \omega))=\Ca(H|_m^{n+1})\cong \Nr_m\Ca H$$

\item The next item is proved in \cite{HHbook} theorem 15.6.

\item Let $m<n$, let $$x_n=\{f\in F(n,n+k,r): m\leq j<n\to \exists i<m f(i,j)=Id\}.$$
Then $x_n\in \C(n,n+k,r)$ and ${\sf c}_ix_n\cdot {\sf c}_jx_n=x_n$ for distinct $i, j<m$.
Furthermore
\[{I_n:\C}(m,m+k,r)\cong \Rl_{x_n}\Rd_m {\C}(n,n+k, r).\]
via
\[ I_n(S)=\{f\in F(n, n+k, r): f\upharpoonright m\times m\in S, \forall j(m\leq j<n\to  \exists i<m\; f(i,j)=Id)\}.\]
\end{enumarab}
Now we use a lifting argument that is a generalization of Monk's argument in
\cite{HMT2}.

Let $k\in \omega$. Let $\alpha$ be an infinite ordinal.
Then $S\Nr_{\alpha}\K_{\alpha+k+1}\subset S\Nr_{\alpha}\K_{\alpha+k}.$
Let $r\in \omega$.
Let $I=\{\Gamma: \Gamma\subseteq \alpha,  |\Gamma|<\omega\}$.
For each $\Gamma\in I$, let $M_{\Gamma}=\{\Delta\in I: \Gamma\subseteq \Delta\}$,
and let $F$ be an ultrafilter on $I$ such that $\forall\Gamma\in I,\; M_{\Gamma}\in F$.
For each $\Gamma\in I$, let $\rho_{\Gamma}$
be a one to one function from $|\Gamma|$ onto $\Gamma.$

Let ${\C}_{\Gamma}^r$ be an algebra similar to $\CA_{\alpha}$ such that
\[\Rd^{\rho_\Gamma}{\C}_{\Gamma}^r={\C}(|\Gamma|, |\Gamma|+k,r).\]
Let
\[\B^r=\prod_{\Gamma/F\in I}\C_{\Gamma}^r.\]
We will prove that
\begin{enumerate}
\item\label{en:1} $\B^r\in S\Nr_\alpha\CA_{\alpha+k}$ and
\item\label{en:2} $\B^r\not\in S\Nr_\alpha\CA_{\alpha+k+1}$.  \end{enumerate}


For the first part, for each $\Gamma\in I$ we know that $\C(|\Gamma|+k, |\Gamma|+k, r) \in\K_{|\Gamma|+k}$ and
$\Nr_{|\Gamma|}\C(|\Gamma|+k, |\Gamma|+k, r)\cong\C(|\Gamma|, |\Gamma|+k, r)$.
Let $\sigma_{\Gamma}$ be a one to one function
 $(|\Gamma|+k)\rightarrow(\alpha+k)$ such that $\rho_{\Gamma}\subseteq \sigma_{\Gamma}$
and $\sigma_{\Gamma}(|\Gamma|+i)=\alpha+i$ for every $i<k$. Let $\A_{\Gamma}$ be an algebra similar to a
$\CA_{\alpha+k}$ such that
$\Rd^{\sigma_\Gamma}\A_{\Gamma}=\C(|\Gamma|+k, |\Gamma|+k, r)$.
Then, clearly
 $\Pi_{\Gamma/F}\A_{\Gamma}\in \CA_{\alpha+k}$.

We prove that $\B^r\subseteq \Nr_\alpha\Pi_{\Gamma/F}\A_\Gamma$.  Recall that $\B^r=\Pi_{\Gamma/F}\C^r_\Gamma$ and note
that $\C^r_{\Gamma}\subseteq A_{\Gamma}$
(the base of $\C^r_\Gamma$ is $C(|\Gamma|, |\Gamma|+k, r)$, the base of $A_\Gamma$ is $C(|\Gamma|+k, |\Gamma|+k, r)$).
So, for each $\Gamma\in I$,
\begin{align*}
\Rd^{\rho_{\Gamma}}\C_{\Gamma}^r&=\C((|\Gamma|, |\Gamma|+k, r)\\
&\cong\Nr_{|\Gamma|}\C(|\Gamma|+k, |\Gamma|+k, r)\\
&=\Nr_{|\Gamma|}\Rd^{\sigma_{\Gamma}}\A_{\Gamma}\\
&=\Rd^{\sigma_\Gamma}\Nr_\Gamma\A_\Gamma\\
&=\Rd^{\rho_\Gamma}\Nr_\Gamma\A_\Gamma
\end{align*}
$\Rd^{\rho_\Gamma}\A_\Gamma \in \K_{|\Gamma|}$, for each $\Gamma\in I$  then $\Pi_{\Gamma/F}\A_\Gamma\in \K_\alpha$.
Thus (using a standard Los argument) we have:
$\Pi_{\Gamma/F}\C^r_\Gamma\cong\Pi_{\Gamma/F}\Nr_\Gamma\A_\Gamma=\Nr_\alpha\Pi_{\Gamma/F}\A_\Gamma$,
proving \eqref{en:1}.

Now we prove \eqref{en:2}.
For this assume, seeking a contradiction, that $\B^r\in S\Nr_{\alpha}\CA_{\alpha+k+1}$,
$\B^r\subseteq \Nr_{\alpha}\C$, where  $\C\in \CA_{\alpha+k+1}$.
Let $3\leq m<\omega$ and  $\lambda:m+k+1\rightarrow \alpha +k+1$ be the function defined by $\lambda(i)=i$ for $i<m$
and $\lambda(m+i)=\alpha+i$ for $i<k+1$.
Then $\Rd^\lambda(\C)\in \CA_{m+k+1}$ and $\Rd_m\B^r\subseteq \Nr_m\Rd^\lambda(\C)$.
For each $\Gamma\in I$,\/  let $I_{|\Gamma|}$ be an isomorphism
\[{\C}(m,m+k,r)\cong \Rl_{x_{|\Gamma|}}\Rd_m {\C}(|\Gamma|, |\Gamma+k|,r).\]
Let $x=(x_{|\Gamma|}:\Gamma)/F$ and let $\iota( b)=(I_{|\Gamma|}b: \Gamma)/F$ for  $b\in \C(m,m+k,r)$.
Then $\iota$ is an isomorphism from $\C(m, m+k,r)$ into $\Rl_x\Rd_m\B^r$.
Then $\Rl_x\Rd_{m}\B^r\in S\Nr_m\CA_{m+k+1}$.
It follows that  $\C (m,m+k,r)\in S\Nr_{m}\CA_{m+k+1}$ which is a contradiction and we are done.
Now we prove the third part of the theorem, putting the superscript $r$ to use.
Recall that $\B^r=\Pi_{\Gamma/F}\C^r_\Gamma$, where $\C^r_\Gamma$ has the type of $\CA.$
and $\Rd^{\rho_\Gamma}\C^r_\Gamma=\C(|\Gamma|, |\Gamma|+k, r)$.
We know that $\Pi_{r/U}\Rd^{\rho_\Gamma}\C^r_\Gamma=\Pi_{r/U}\C(|\Gamma|, |\Gamma|+k, r) \subseteq \Nr_{|\Gamma|}\A_\Gamma$,
for some $\A_\Gamma\in\CA_{|\Gamma|+k+1}$.

Let $\lambda_\Gamma:|\Gamma|+k+1\rightarrow\alpha+k+1$
extend $\rho_\Gamma:|\Gamma|\rightarrow \Gamma \; (\subseteq\alpha)$ and satisfy
\[\lambda_\Gamma(|\Gamma|+i)=\alpha+i\]
for $i<k+1$.  Let $\F_\Gamma$ be a $\CA_{\alpha+k+1}$ type algebra such that $\Rd^{\lambda_\Gamma}\F_\Gamma=\c A_\Gamma$.
As before, $\Pi_{\Gamma/F}\F_\Gamma\in\CA_{\alpha+k+1}$.  And
\begin{align*}
\Pi_{r/U}\B^r&=\Pi_{r/U}\Pi_{\Gamma/F}\C^r_\Gamma\\
&\cong \Pi_{\Gamma/F}\Pi_{r/U}\C^r_\Gamma\\
&\subseteq \Pi_{\Gamma/F}\Nr_{|\Gamma|}\c A_\Gamma\\
&=\Pi_{\Gamma/F}\Nr_{|\Gamma|}\Rd^{\lambda_\Gamma}\F_\Gamma\\
&\subseteq\Nr_\alpha\Pi_{\Gamma/F}\F_\Gamma,
\end{align*}
proving the lemma.
\end{proof}

See \cite[definition~4.1.4]{HMT2} for the precise definition of finitely schema axiomatisability
and see \cite[theorem~4.1.7]{HMT2} to see how non finite schema axiomatisability follows from the above.

The next part describes a construction of Hirsch
that covers other algebras like Pinter's substitution algebras and quasipolyadic algebras
(it also covers the above two case.) We will only give a sketch to the somewhat intricate construction, the details can be found in \cite{t}.

We note that extending such non finite axiomatisability results
to diagonal free algebras is really hard.) These algebras constructed by Hirsch, are like the above algebras,
they are based on a relation algebra construction
see \cite[section~15.2]{HHbook},
modified here so that the elements become $n$-dimensional rather than two dimensional.
Still, although they are $n$-dimensional, all of their elements are generated by two dimensional elements.

\begin{definition}\label{def:cmnr}
Define a function $\kappa:\omega\times\omega\rightarrow\omega$
by $\kappa(x, 0)=0$ (all $x<\omega$) and $\kappa(x, y+1)=1+x\times\kappa(x, y))$ (all $x, y<\omega$).
For $n, r<\omega$ let
\[\psi(n, r)=
\kappa((n-1)r, (n-1)r)+1\]
and let $\psi(n, J)=\omega$, for any infinite linear order $J$.  [All of this is simply to ensure that $\psi(n, r)$ is sufficiently big compared to $n, r$ for the proof of non-embeddability to work.]
For any  $n<\omega$ and any linear order $J$, let
\[Bin(n, J)=\set{Id}\cup\set{a^k(i, j):i< n-1,\;j\in J,\;k<\psi(n, J)}\]
 For $i<n-1,\;j\in J,\; k<\psi(n, J)$
 let
\begin{align*} a(i,\_)&=\set{a^k(i, j):j\in J,\; k<\psi(n, J))},\\
a(\_, j)&=\set{a^k(i, j):i<n-1,\; k<\psi(n, J)},\\
a^k&=\set{a^k(i, j):i<n-1, \;j\in J},\\
a(i, j)&=\set{a^k(i, j):k<\psi(n, J)},\\
a(\_, >j)&=\bigcup_{j< j'\in J}a(\_, j')\\
a(\_, \leq j)&=\bigcup_{j'\leq j\in J}a(\_, j')\mbox{ and }\\
a&=\bigcup_{i<n-1}a(i, \_).
\end{align*}
Let $3\leq m\leq n<\omega$ and let $J$ be any linear order.
Let $F(m, n, J)$ be the set of all  functions $f:m\times m\to Bin(n, J)$ such that for all $x, y, z<m$ we have $f(x, x)=Id,\;f(x, y)=f(y, x)$, and $(f(x, y), f(y, z), f(x, z))\not\in Forb$, where $Forb$ is the following set of triples
 \[ \begin{array}{c}
 \set{(Id, b, c):b\neq c\in Bin(n, J)}\\
 \cup \\
 \set{(a^k(i, j), a^{k'}(i,j), a^{k^*}(i, j')): k, k', k^*< \psi(n, r), \;i<n-1, \; j'\leq j\in J}.
 \end{array}\]
For any $f, g\in F(m, n,  J)$ and $x<m$ we write $f\equiv_xg$ if for all $y, z\in m\setminus\set x$ we have $f(y, z)=g(y, z)$,  and for $\tau:m\to m$ we write $(f\tau)$ for the function defined by
\begin{equation}\label{eq:ftau}(f\tau)(x, y)=f(\tau(x), \tau(y)).\end{equation}
Clearly $(f\tau)\in F(m, n, J)$.  For the next couple of sections we will consider cases where $J$ is a finite linear order, i.e. some finite ordinal $r<\omega$.
\begin{definition}\label{def:cmnr}
The base of $\C(m, n, r)$ is the power set of $F(m, n, r)$ and the operators are
\begin{itemize}
\item  the boolean operators $+, -$ are union and set complement,
\item the diagonal $\diag xy=\set{f\in F(m, n, r):f(x, y)=Id}$,
\item  the cylindrifier $\cyl x(X)=\set{f\in F(m, n, r): \exists g\in X\; f\equiv_xg }$ and
\item the polyadic ${\sf s}_\tau(X)=\set{f\in F(m, n, r): f\tau \in X}$,
\end{itemize}
 for $x, y<m,\;  X\subseteq F(m, n, r)$ and  $\tau:m\to m$.
\end{definition}

Let $x, y<m$ and let $b\in Bin(n, r)$.  Define
\begin{equation}\label{eq:bxy}b_{x,y}=\set{f\in  F(m, n, r): f(x, y)= b}\;\;\in\C(m, n, r)\end{equation}
\end{definition}
Observe,
for any $x, y, z<m$ and  $\lambda, \mu, \rho\in Bin(n, r)$, that
\begin{equation}\label{eq:xyz}(\lambda, \mu, \rho)\in Forb \iff \lambda_{x, y}\cdot\mu_{y, z}\cdot\rho_{x, z}=0,\end{equation}
in particular we will use the case $(x, y, z)=(0, 1, 2)$, later.

\begin{lemma}\label{lem:qp}
For $3\leq m,\; 2\leq n$ and $r<\omega$ the algebra $\C(m, n, r)$ satisfies all of the axioms defining
 ${\sf PEA}_m$
except, perhaps, the commutativity of cylindrifiers $\cyl x\cyl y(X)=\cyl y\cyl x(X)$.
\end{lemma}\begin{proof} Routine.
\end{proof}
\begin{lemma}\label{lem:Nr}
If  $3\leq m\leq m'$ then $\C(m, n, r)\cong\Nr_m\C(m', n, r)$.
\end{lemma}
\begin{proof}
The isomorphism maps $X\subseteq F(m, n, r)$ to $\set{f\in F(m', n, r): f\restr{m\times m}\in X}$.
\end{proof}
\begin{lemma}\label{lem:pea}
For $3\leq m\leq n$ and  $r<\omega$, \/ $\C(m, n, r)\in{\sf PEA}_m$.
\end{lemma}
\begin{proof}  The proof available in \cite{t} is easy, so we include it for the reader's convenience.
If $r=0$ then $Bin(n, r)=\set{Id}$ so $\C(m,n,0)$ is the trivial algebra hence $\C(m, n, 0)\in {\sf PEA}_m$.  Now assume $r>0$.
In view of lemma~\ref{lem:qp} we only have to check the commutativity of cylindrifiers: $\cyl  x\cyl yX=\cyl y\cyl xX$, for $x, y<m$.
This equation is trivial if $x=y$ so assume not.  By additivity, it suffices to check the case where $X$ is an atom, $X=\set f$ for some $f\in F(m, n, r)$, that is we must show that $g\in\cyl x\cyl y \set f\iff g\in\cyl y\cyl x \set f$.  Thus given $g\equiv_{xy} f$, it suffices to find $h\in  F(m,n,r)$
such that $f\equiv_x h\equiv_y g$.  If there is $z<m,\; z\neq x, y$ and $f(y, z)=Id$ then the required $h$ is $g[y/z]$, or if
$g(z, x)=Id$ the required $h$ is $f[x/z]$.
Suppose there is no such $z$, so for each $z<m,\; z\neq x, y$ we have $f(y, z), g(x, z)\in a$.  Let $h:m\times m\to Bin(n, r)$ be identical to $f$ on pairs not involving $x$,  be identical to $g$ on pairs not involving $y$ (this is well-defined, since $f\equiv_{xy}g$) and  let $h(x, y)=h(y, x)=a^0(i, 0)$, where $i$ is the least number below $n-1$ such that it is not the case that there is $z\neq x, y<m$ and $f(y, z), g(x, z)\in a(i,\_)$.  Since $m\leq n$ and  there are only $m-2$ possible values of $z$ in $m\setminus\set{x, y}$ and $n-1$ possible values of $i$, such an $i$ must exist.  This defines $h$, it is now easy to check that $h\in F(m, n, r)$.

\end{proof}

\begin{theorem}
Let $3\leq m<n<\omega$ and $r<\omega$, then  $\Rd_\Sc\C(m, n, r)\not\in S\Nr_m\Sc_{n+1}$.
Furthermore, if $J$ be a countable linear order containing an infinite ascending sequence,
then $\C(m, n, J)\in \Nr_m{\sf PEA_{n+1}}$.
\end{theorem}
\begin{proof} \cite{t}
\end{proof}
Let ${\sf OPEA}_{\alpha}$ denote the class of quasi-polyadic equality algebras of dimension
$\alpha$, and ${\sf SC}_{\alpha}$ that of Pinter's substitution algebras of dimension $\alpha$.
Since $\Pi_{r/U}\C(m, n, r) \cong\C(m, n, \Pi_{r/U} r)$ and $\Pi_{r/U}r$ contains an infinite ascending sequence,
this proves that the new algebra satisfies the hypothesis of theorem \ref{2.12}.
Indeed for $m<n$, let in the previous case, set
 $$x_n=\{f\in F(n,n+k,r): m\leq j<n\to \exists i<m f(i,j)=Id\}.$$
Then $x_n\in \C(n,n+k,r)$ and ${\sf c}_ix_n\cdot {\sf c}_jx_n=x_n$ for distinct $i, j<m$.
Furthermore
\[{I_n:\C}(m,m+k,r)\cong \Rl_{x_n}\Rd_m {\C}(n,n+k, r).\]
via
\[ I_n(S)=\{f\in F(n, n+k, r): f\upharpoonright m\times m\in S, \forall j(m\leq j<n\to  \exists i<m\; f(i,j)=Id)\}.\]
But in this case we have, using the notation adopted
in the proof of theorem \ref{2.12}, that
$\B^r\subseteq \Nr_\alpha\Pi_{\Gamma/F}\A_\Gamma$  where $\Pi_{\Gamma/F}\A_\Gamma\in {\sf QPEA}_{\alpha+k}$,
so that $\B^r\in S\Nr_{\alpha}{\sf QPEA}_{\alpha+k}$, but  $\Rd_{sc}\B^r\notin S\Nr_{\alpha}{\sf SC}_{\alpha+k+1}$.
This shows that, for any infinite ordinal $\alpha$, and any $k\geq 1$, the variety  $S\Nr_{\alpha}\K_{\alpha+k+1}$ is not finitely
axiomatizable by a finite schema over $S\Nr_{\alpha}\K_{\alpha+k}$ for any $\K$ between
Pinter's substitution algebras and Halmos' quasi-polyadic equality
algebras.

\subsection{When splitting fails, do rainbows work?}

Here we address non finite variable axiomatizations of diagonal free reducts of cylindric and polyadic algebras.
Whether finite variable axiomatizations for such representable algebras exist for finite dimension $>2$ is a long standing open
problem in algebraic logic. We do not know its origin but it was formulated
officially in \cite{ST} published
in 1990, so it dates back at least to this history. In the diagonal free context the splitting technique of Andr\'eka's
does not work, in fact {\it it can be proved } that it does not work. But maybe rainbow do, as we proceed to show.

We start by a general well known result.

\begin{theorem}\label{d} Let $V$ be a discriminator variety.
Assume that there are simple algebras $\A$ and $\B$ such that $\A\in V$ and $\B\notin V$,
and for any equation involving $n$ variables,
$\A$ falsifies $e$ if and only if $\B$ falsifies $e$.
Then $V$ is not axiomatizable by any set of prenex universal sentences that uses only $n$ variables.
\end{theorem}
\begin{proof} If $\Sigma$ is any $n$ variable equational theory then $\A$ and $\B$ either both validate
$\Sigma$ or neither do. Since one algebra is in $V$ while the other is not, it follows that $\Sigma$ does not axiomatize $V$.
If $\A$ and $\B$ are simple, then they
are subdirectly irreducible. In a discriminator variety every universal prenex formula is equivalent in subdirectly irreducible members to an
equation using the same number of variables. Hence the desired.
\end{proof}

For every $n\geq 3$, two  finite rainbow polyadic algebras are constructed, one will be representable, the diagonal free reduct
of the other will not be representable. All colours are the same
except that one has more red graphs than the other (a red graph is a coloured graph such that at least one edge is coloured red).
In the usual atomic game on graphs using his excess of greens, \pa\ wins.
This prohibits the first algebra to be representable. In the second case the reds are more, and \pe\ can win the
$\omega$ rounded game (in finitely many rounds) on coloured graphs, hence the algebra will be representable.
Furthermore, {\it we conjecture} that $n$ variable (diagonal free) equations cannot distinguish the two algebras.

Let $\kappa$ be a finite number $>n$.

Let $\alpha=3.2^n$ and $\beta=(\alpha+1)(\alpha+2)/2.$

\begin{definition}
The colours we use:
\begin{itemize}

\item greens: $\g_i$ ($1\leq i<n-2)\cup \{\g^0_i: i\leq \alpha+2\}$,
\item whites : $\w_i, i <n$
\item yellow : $\y$
\item reds:  $\r_{i}$, $i\in \kappa$

\item shades of yellow : $\y_S: S\subseteq \alpha+2$

And coloured graphs are:
\begin{definition}
\begin{enumarab}

\item $M$ is a complete graph.

\item $M$ contains no triangles (called forbidden triples)
of the following types:

\vspace{-.2in}
\begin{eqnarray}
&&\nonumber\\
(\g, \g^{'}, \g^{*}), (\g_i, \g_{i}, \w),
&&\mbox{any }i\in n-1\;  \\
(\g^j_0, \g^k_0, \w_0)&&\mbox{ any } j, k\in \alpha+2\\
\label{forb:pim}(\g^i_0, \g^j_0, \r_{kl})&&\\
\label{forb:match}(\r_{i}, \r_{i}, \r_{j})&&
\end{eqnarray}
and no other triple of atoms is forbidden.

\item If $a_0,\ldots   a_{n-2}\in M$ are distinct, and no edge $(a_i, a_j)$ $i<j<n$
is coloured green, then the sequence $(a_0, \ldots a_{n-2})$
is coloured a unique shade of yellow.
No other $(n-1)$ tuples are coloured shades of yellow.

\item If $D=\set{d_0,\ldots  d_{n-2}, \delta}\subseteq M$ and
$\Gamma\upharpoonright D$ is an $i$ cone with apex $\delta$, inducing the order
$d_0,\ldots  d_{n-2}$ on its base, and the tuple
$(d_0,\ldots d_{n-2})$ is coloured by a unique shade
$y_S$ then $i\in S.$

\end{enumarab}
\end{definition}

\end{itemize}
\end{definition}
More concretely, our algebras are the polyadic equality algebras
$\A=\A_{\alpha+2,\beta}$ and $\B=\A_{\alpha+2,\alpha}$, here $\alpha+2$ is the number of greens

\begin{theorem} \pa\ has a \ws\  for $\B$  in $\alpha+2$ rounds; hence $\Rd_{df}\B\notin {\sf RDf_n}$
\end{theorem}
\begin{proof}
\pa\ plays a coloured graph $M \in \K$ with
nodes $0, 1,\ldots, n-1$ and such that $M(i, j) = \w (i < j <
n-1), M(i, n-1) = \g_i ( i = 1,\ldots, n), M(0, n-1) =
\g^0_0$, and $ M(0, 1,\ldots, n-2) = \y_{\alpha+2}$. This is a $0$-cone
with base $\{0,\ldots , n-2\}$. In the following moves, \pa\
repeatedly chooses the face $(0, 1,\ldots, n-2)$
and demands a node
$t<\alpha+2$ with $\Phi(i,\alpha) = \g_i (i = 1,\ldots,  n-2)$ and $\Phi(0, t) = \g^t_0$,
in the graph notation -- i.e., a $t$ -cone on the same base.
\pe\, among other things, has to colour all the edges
connecting nodes. By the rules of the game
the only permissible colours would be red. Using this, \pa\ can force a
win in $\alpha+2$ rounds eventually using her enough supply of greens,
which \pe\ cannot match using his $<$ number of reds. The conclusion now follows since $\B$ is generated by elements whose dimension sets
are $<n$.
\end{proof}

\begin{theorem} The algebra $\A\in {\sf RPEA_n}$
\end{theorem}

\begin{proof} If \pa\ plays like before, now \pe\ has more  reds, so \pa\ cannot force a win. In fact \pa\
can only force a red clique of size $\alpha+2$, not bigger. So \pe\ s
strategy within red cliques is to choose a label for each edge using
a red colour and to ensure that each edge within the clique has a label unique to this edge (within the clique).
Since there are $\beta$ many reds she can do that.

Let $M$ be a coloured graph built at some stage, and let \pa\ choose the graph $\Phi$, $|\Phi|=n$, then $\Phi=F\cup \{\delta\}$,
where  $F\subseteq M$ and $\delta\notin M$.
So we may view \pa\ s move as building a coloured graph $M^*$ extending $M$
whose nodes are those of $\Gamma$ together with $\delta$ and whose edges are edges of $\Gamma$ together with edges
from $\delta$ to every node of $F$.

Colours of edges and $n-1$ tuples in $M^*$ but not
in $M$ are determined by \pa\ moves.
No $n-1$ tuple containing both $\delta$ and elements of $M\sim F$
has a colour in $M^*$

Now \pe\ must extend $M^*$ to a complete the graph on the same nodes and
complete the colouring giving  a graph $M$ in $\K$.
In particular, she has to define $M^+(\beta, \delta)$
for all nodes  $\beta\in M\sim F$.
\begin{enumarab}
\item  if $\beta$ and $\delta$ are both apexes of cones on $F$, that induces the same linear ordering on $F$, the
\pe\ has no choice but to pick a  red atom, and as we described above, she can do that

\item Other wise, this is not the case, so for some $i<n-1$ there is no $f\in F$ such
that $M^*(\beta, f), M^* (f,\delta)$ are both coloured $\g_i$ or if $i=0$, they are coloured
$\g_0^l$ and $\g_0^{l'}$ for some $l$ and $l'$.
\end{enumarab}
In the second case \pe\ uses the normal strategy in rainbow constructions.
She chooses $\w_0$, for $M^+(\beta,\delta)$.

Now we turn to coluring of $n-1$ tuples. For each tuple $\bar{a}=a_0,\ldots a_{n-2}\in M^{n-1}$ with no edge
$(a_i, a_j)$ coloured green, then  \pe\ colours $\bar{a}$ by $\y_S$, where
$$S=\{i\in \alpha+2: \text { there is an $i$ cone in $M^*$ with base $\bar{a}$}\}.$$

Clearly the choice of whites avoids all forbidden triples, and the choice of reds works.
We need to check that the the labelling of $n-1$ types works.
\end{proof}

\begin{theorem} A coloured graph is red, if at least one of its edges are labelled red
\end{theorem}
We write $\r$ for $a:n\to \Gamma$, where $\Gamma$ is a red graph, and we call it a red atom.
(Here we identify an atom with its representative, but no harm will follow).

\begin{athm}{Conjecture}For any $n$ variable equation the two algebras $\A$ and $\B$,  falsify it together or satisfy it together.
\end{athm}
\begin{demo}{Idea}
Let $\R$ be the set of red graphs of $\A$, and $\R'$ be the set of red graphs in $\B$. Then $|\R|\geq |\R'|\geq 3.2^n$.
Assume that the equation $s=t$, using $n$ variables does not hold in $\A$.
Then there is an assignment $h:\{x_0,\ldots x_{n-1}\}\to \A$, such that $\A, h\models s\neq t$.
We construct an assignment $h'$ into $\B$ that also falsifies $s=t$.
Now $\A$ has more red atoms, but $\A$ and $\B$ have identical non-red atoms. So for any non red atom $a$ of $\B$, and for any
$i<n$, let $a\leq h'(x_i)$ iff $a\leq h(x_i)$.
The assignment $h$ induces a partition of $\R$ into $2^n$ parts $\R_S$, $S\subseteq \{0,1\ldots n-1\}$, by
$$\R_S=\{\r: \r\leq h(x_i) , i\in S, \r.h(x_i)=0, i\in n\sim S\}.$$
Partition  $\R'$ into $2^n$ parts: $S\subseteq n$ such that $|\R'_S|=|\R_S|$ if $|\R_S|<3$, and $|\R_S'|\geq 3$ iff $|\R_S|\geq 3$.
This possible because $|\R|\geq 3.2^n$.
Now for each $i<n$ and each red atom $r'$ in $\R'$, we complete the definition of
$h'(x_i)\in \B$  by
$r'\leq h'(x_i)$ iff $r'\in \R'_S$ for some $S$ such that $i\in S$.
We show inductively that for any term $\tau$ using only the first $n$ variables
and any $S\subseteq n$, we have

$$\R_S\subset h(\tau)\Longleftrightarrow  \R'_S\subset h'(\tau)$$
$$h(\tau)\sim \R = h'(\tau)\sim \R'$$
$$|h(\tau)\cap \R| =|h'(\tau) \cap \R'| \text {  iff   } |h(\tau)\cap \R|<3$$
$$|h(\tau) \cap \R| \geq 3\Longleftrightarrow |h'(\tau)\cap \R'|\geq 3.$$
Hence $\B$ does not model $s=t$.
The converse is entirely analogous.
\end{demo}

Modulo the above conjecture, and \ref{d} we readily obtain:

\begin{corollary} For $n\geq 3$, the class of representable algebras in any $\K$
between ${\sf Df}_n$ and ${\sf PEA_n}$
does not have a prenex universal axiomatization using $n$  variables.
\end{corollary}

\subsection{Neat atom structures}

Before dealing with neat atom structures; we devise {\it neat games} played on $\lambda$ neat hypergraphs.
We devise a game between \pe\ and \pa\ on its rainbow atom structure.
Recall that $F^m$  is like the ordinary $\omega$ rounded atomic game, except that the number of pebbles
used by \pa\ are limited to $m$; however he has the
option to reuse them. \pe\ will play this game on the $\Sc$ reduct of $\PEA_{\Z,\N}$, namely $\Sc_{\Z, \N}$.
This is a game  played on atomic networks, but they translate to equivalent games played
on coloured graphs, as indicated above, which we continue use.

We have another game $H$ with $\omega$ rounds that is stronger for it aims at {\it gripping} ${\omega}$ extra dimension.

$H$ is much more complicated, because it is played on hypernetworks,
which is a network endowed with a set of labelled hyperedges. Furthermore $H_k$ the game $H$ truncated to $k$ rounds will be played
on $\PEA_{\Z, \N}$, for each $k\geq n$. To describe such hypernetworks we need
some preparing to do:
For  $x,y\in \nodes(N)$, we set  $x\sim y$ if there exists $\bar{z}$ such that $N(x,y,\bar{z})\leq {\sf d}_{01}$.
The equivalence relation $\sim$ over the set of all finite sequences over $\nodes(N)$ is defined by
$\bar x\sim\bar y$ iff $|\bar x|=|\bar y|$ and $x_i\sim y_i$ for all
$i<|\bar x|$.(It can be checked that this indeed an equivalence relation.)

A \emph{ hypernetwork} $N=(N^a, N^h)$ over an atomic polyadic equality algebra $\C$
consists of a network $N^a$
together with a labelling function for hyperlabels $N^h:\;\;^{<
\omega}\!\nodes(N)\to\Lambda$ (some arbitrary set of hyperlabels $\Lambda$)
such that for $\bar x, \bar y\in\; ^{< \omega}\!\nodes(N)$
\begin{enumerate}
\renewcommand{\theenumi}{\Roman{enumi}}
\setcounter{enumi}3
\item\label{net:hyper} $\bar x\sim\bar y \Rightarrow N^h(\bar x)=N^h(\bar y)$.
\end{enumerate}
If $|\bar x|=k\in \N$ and $N^h(\bar x)=\lambda$ then we say that $\lambda$ is
a $k$-ary hyperlabel. $(\bar x)$ is referred to a a $k$-ary hyperedge, or simply a hyperedge.
(Note that we have atomic hyperedges and hyperedges)
When there is no risk of ambiguity we may drop the superscripts $a,
h$.
There are {\it short} hyperedges and {\it long} hyperedges (to be defined in a while). The short hyperedges are constantly labelled,
so that the atoms in the neat reduct are no smaller than the atoms
in the dilation. (When $\A=\Nr_n\B,$ it is common to call $\B$ a dilation of $\A$.)
We know that there is a one to one correspondence between networks and coloured graphs.
If $\Gamma$ is a coloured graph, then by $N_{\Gamma}$
we mean the corresponding network defined on $n-1$ tuples of the nodes of $\Gamma$ to
to coloured graphs of size $\leq n$.
\begin{definition}

\begin{enumarab}
\item A hyperedge $\bar{x}\in {}^{<\omega}\nodes (\Gamma)$ of length $m$ is {\it short}, if there are $y_0,\ldots y_{n-1}\in \nodes(N)$, such that
$N_{\Gamma}(x_i, y_0, \bar{z})\leq d_{01}$, or $N(_{\Gamma}(x_i, y_1, \bar{z})\ldots$ or $N(x_i, y_{n-1},\bar{z})\leq d_{01}$ for all $i<|x|$,
for some (equivalently for all)
$\bar{z}.$ Otherwise, it is called {\it long.}
\item A hypergraph $(\Gamma, l)$
is called {\it $\lambda$ neat} if $N_{\Gamma}(\bar{x})=\lambda$ for all short hyper edges.
\end{enumarab}
\end{definition}

In $H$  \pa\ has three moves.

The first is the normal cylindrifier move. There is no polyadic move.
The next two are amalgamation moves.
But the games are not played on hypernetworks, they are played on coloured hypergraphs, whose graph part
can be viewed as an $L_{\omega_1, \omega}$ model for the rainbow signature, and the part are the part dealing with hyperedges with a
labelling function.
The amalgamation moves reflect the fact, in case \pe\ wins, then for every $k\geq n$ there is a $k$ dimensional hyperbasis,
so that the small algebra embeds into cylindric algebras of arbitrary large dimensions.
The game is played on $\lambda$ neat hypernetworks,  translated to $\lambda$ neat hypergraphs,
where $\lambda$ is a label for the short hyperedges.

A piece of notation: Let $N$ be a network and let $\theta$ be any function.  The network
$N\theta$ is a complete labelled graph with nodes
$\theta^{-1}(\nodes(N))=\set{x\in\dom(\theta):\theta(x)\in\nodes(N)}$,
and labelling defined by
$(N\theta)(i_0,\ldots i_{\mu-1}) = N(\theta(i_0), \theta(i_1), \theta(i_{\mu-1}))$,
for $i_0, \ldots i_{\mu-1}\in\theta^{-1}(\nodes(N))$.  Similarly, for a hypernetwork
$N=(N^a, N^h)$, we define $N\theta$ to be the hypernetwork
$(N^a\theta, N^h\theta)$ with hyperlabelling defined by
$N^h\theta(x_0, x_1, \ldots) = N^h(\theta(x_0), \theta(x_1), \ldots)$
for $(x_0, x_1,\ldots) \in \;^{<\omega}\!\theta^{-1}(\nodes(N))$.

\pa\ can play a \emph{transformation move} by picking a
previously played hypernetwork $N$ and a partial, finite surjection
$\theta:\omega\to\nodes(N)$, this move is denoted $(N, \theta)$.  \pe\
must respond with $N\theta$.  Finally, \pa\ can play an
\emph{amalgamation move} by picking previously played hypernetworks
$M, N$ such that $M\equiv^{\nodes(M)\cap\nodes(N)}N$ and
$\nodes(M)\cap\nodes(N)\neq \emptyset$.
This move is denoted $(M,
N)$.  To make a legal response, \pe\ must play a $\lambda_0$-neat
hypernetwork $L$ extending $M$ and $N$, where
$\nodes(L)=\nodes(M)\cup\nodes(N)$.

\begin{lemma}Let $\A$ be a countable atomic algebra with atom structure $\alpha$.
If \pe\ can win the $\omega$ rounded game $H$ on $\alpha$,
then there exists a locally finite $\PEA_{\omega}$ such that
$\At\A\cong \At\Nr_n\C$. Furthermore, $\C$ can be chosen to be complete, and $\Cm\At\A=\Nr_n\C$.
The game $H$ can be strengthened so that a \ws\ for \pe\  guarantees that
$\A\cong \Nr_n\C$ (This is much stronger). Conversely, regardless of cardinalities if $\A\in \Nr_n\CA_{\omega}$ is atomic,
then $\A$ is strongly representable, and  \pe\ can win $H(\At\A)$.
\end{lemma}
\begin{proof}
For the first part. Fix some $a\in\alpha$. Using \pe\ s \ws\ in the game of neat hypernetworks, one defines a
nested sequence $N_0\subseteq N_1\ldots$ of neat hypernetworks
where $N_0$ is \pe's response to the initial \pa-move $a$, such that
\begin{enumerate}
\item If $N_r$ is in the sequence and
and $b\leq {\sf c}_lN_r(f_0, ldots, x, \ldots  f_{n-2})$,
then there is $s\geq r$ and $d\in\nodes(N_s)$ such
that $N_s(f_0, f_{l-1}, d, f_{l+1},\ldots f_{n-2})=b$.
\item If $N_r$ is in the sequence and $\theta$ is any partial
isomorphism of $N_r$ then there is $s\geq r$ and a
partial isomorphism $\theta^+$ of $N_s$ extending $\theta$ such that
$\rng(\theta^+)\supseteq\nodes(N_r)$.
\end{enumerate}
Now let $N_a$ be the limit of this sequence, that is $N_a=\bigcup N_i$, the labelling of $n-1$ tuples of nodes
by atoms, and the hyperedges by hyperlabels done in the obvious way.
This limit is well-defined since the hypernetworks are nested.
We shall show that $N_a$ is the base of a weak set algebra having unit  $V={}^{\omega}N_a^{(p)}$,
for some fixed sequence $p\in {}^{\omega}N_a$.

Let $\theta$ be any finite partial isomorphism of $N_a$ and let $X$ be
any finite subset of $\nodes(N_a)$.  Since $\theta, X$ are finite, there is
$i<\omega$ such that $\nodes(N_i)\supseteq X\cup\dom(\theta)$. There
is a bijection $\theta^+\supseteq\theta$ onto $\nodes(N_i)$ and $\geq
i$ such that $N_j\supseteq N_i, N_i\theta^+$.  Then $\theta^+$ is a
partial isomorphism of $N_j$ and $\rng(\theta^+)=\nodes(N_i)\supseteq
X$.  Hence, if $\theta$ is any finite partial isomorphism of $N_a$ and
$X$ is any finite subset of $\nodes(N_a)$ then
\begin{equation}\label{eq:theta}
\exists \mbox{ a partial isomorphism $\theta^+\supseteq \theta$ of $N_a$
 where $\rng(\theta^+)\supseteq X$}
\end{equation}
and by considering its inverse we can extend a partial isomorphism so
as to include an arbitrary finite subset of $\nodes(N_a)$ within its
domain.
Let $L$ be the signature with one $n$ -ary predicate symbol ($b$) for
each $b\in\alpha$, and one $k$-ary predicate symbol ($\lambda$) for
each $k$-ary hyperlabel $\lambda$. We work in $L_{\infty, \omega}$.

Here we have a sequence of variables of order type $\omega$, and two 'sorts' of formulas,
the $n$ predicate symbols uses only $n$ variables, and roughly
the $n$ variable formulas built up out of the first $n$ variables will determine the atoms of
neat reduct, the $k$ ary predicate symbols
will determine the atoms of algebras of higher dimensions as $k$ gets larger;
the atoms in the neat reduct will be no smaller than the atoms in the dilations.
This process will be interpreted in an infinite weak set algebra with base $N_a$, whose elements are
those  assignments satisfying such formulas.

For fixed $f_a\in\;^\omega\!\nodes(N_a)$, let
$U_a=\set{f\in\;^\omega\!\nodes(N_a):\set{i<\omega:g(i)\neq
f_a(i)}\mbox{ is finite}}$.
Notice that $U_a$ is weak unit (a set of sequences agreeing cofinitely with a fixed one.)


We can make $U_a$ into the universe an $L$ relativized structure ${\cal N}_a$;
here relativized means that we are only taking those assignments agreeing cofinitely with $f_a$,
we are not taking the standard square model.
However, satisfiability  for $L$ formulas at assignments $f\in U_a$ is defined the usual Tarskian way, except
that we use the modal notation, with restricted assignments on the left:
For $r\in \A
l_0, \ldots l_{n-1}, i_0 \ldots, i_{k-1}<\omega$, \/ $k$-ary hyperlabels $\lambda$,
and all $L$-formulas $\phi, \psi$, let
We can make $U_a$ into the base of an $L$-structure ${\cal N}_a$ and
evaluate $L$-formulas at $f\in U_a$ as follow.  For $b\in\alpha,\;
l_0, \ldots l_{\mu-1}, i_0 \ldots, i_{k-1}<\omega$, \/ $k$-ary hyperlabels $\lambda$,
and all $L$-formulas $\phi, \psi$, let
\begin{eqnarray*}
{\cal N}_a, f\models b(x_{l_0}\ldots  x_{n-1})&\iff&N_a(f(l_0),\ldots  f(l_{n-1}))=b\\
{\cal N}_a, f\models\lambda(x_{i_0}, \ldots,x_{i_{k-1}})&\iff&  N_a(f(i_0), \ldots,f(i_{k-1}))=\lambda\\
{\cal N}_a, f\models\neg\phi&\iff&{\cal N}_a, f\not\models\phi\\
{\cal N}_a, f\models (\phi\vee\psi)&\iff&{\cal N}_a,  f\models\phi\mbox{ or }{\cal N}_a, f\models\psi\\
{\cal N}_a, f\models\exists x_i\phi&\iff& {\cal N}_a, f[i/m]\models\phi, \mbox{ some }m\in\nodes(N_a)
\end{eqnarray*}
For any $L$-formula $\phi$, write $\phi^{{\cal N}_a}$ for the set of all $n$ ary assignments satisfying it; that is
$\set{f\in\;^\omega\!\nodes(N_a): {\cal N}_a, f\models\phi}$.  Let
$D_a = \set{\phi^{{\cal N}_a}:\phi\mbox{ is an $L$-formula}}.$
Then this is the universe of the following weak set algebra
\[\D_a=(D_a,  \cup, \sim, {\sf D}_{ij}, {\sf C}_i)_{ i, j<\omega}\]
then  $\D_a\in\RCA_\omega$. (Weak set algebras are representable).

Now we consider formulas in more than $n$ variables, corresponding to the $k$ hyperlabels.
Let $\phi(x_{i_0}, x_{i_1}, \ldots, x_{i_k})$ be an arbitrary
$L$-formula using only variables belonging to $\set{x_{i_0}, \ldots,
x_{i_k}}$.  Let $f, g\in U_a$ (some $a\in \alpha$) and suppose that $\{(f(i_j), g(i_j): j\leq k\}$
is a partial isomorphism of ${\cal N}_a$ (viewed as a weak mode), then one can easily prove by induction over the
quantifier depth of $\phi$ and using (\ref{eq:theta}), that
\begin{equation}
{\cal N}_a, f\models\phi\iff {\cal N}_a,
g\models\phi\label{eq:bf}\end{equation}

For any $L$-formula $\phi$, write $\phi^{{\cal N}_a}$ for
$\set{f\in\;^\omega\!\nodes(N_a): {\cal N}_a, f\models\phi}$.  Let
$Form^{{\cal N}_a} = \set{\phi^{{\cal N}_a}:\phi\mbox{ is an $L$-formula}}$
and define a cylindric algebra
\[\D_a=(Form^{{\cal N}_a},  \cup, \sim, {\sf D}_{ij}, {\sf C}_i, i, j<\omega)\]
where ${\sf D}_{ij}=(x_i= x_j)^{{\cal N}_a},\; {\sf C}_i(\phi^{{\cal N}_a})=(\exists
x_i\phi)^{{\cal N}_a}$.  Observe that $\top^{{\cal N}_a}=U_a,\; (\phi\vee\psi)^{{\cal N}_a}=\phi^{\c
N_a}\cup\psi^{{\cal N}_a}$, etc. Note also that $\D$ is a subalgebra of the
$\omega$-dimensional cylindric set algebra on the base $\nodes(N_a)$,
hence $\D_a\in {\sf Lf}_{\omega}\cap {\sf Ws}_\omega$, for each atom $a\in \alpha$, and is clearly complete.

Let $\C=\prod_{a\in \alpha} \D_a$. (This is not necessarily locally finite).
Then  $\C\in\RCA_\omega$, and $\C$ is also complete, will be shown to be is the desired generalized weak set algebra,
that is the desired dilation.
Note that unit of $\C$ is the disjoint union of the weak spaces.

An element $x$ of $\C$ has the form
$(x_a:a\in\alpha)$, where $x_a\in\D_a$.  For $b\in\alpha$ let
$\pi_b:\C\to \D_b$ be the projection map defined by
$\pi_b(x_a:a\in\alpha) = x_b$.  Conversely, let $\iota_a:\D_a\to \c
C$ be the embedding defined by $\iota_a(y)=(x_b:b\in\alpha)$, where
$x_a=y$ and $x_b=0$ for $b\neq a$.  Evidently $\pi_b(\iota_b(y))=y$
for $y\in\D_b$ and $\pi_b(\iota_a(y))=0$ if $a\neq b$.

Suppose $x\in\Nr_n\C\setminus\set0$.  Since $x\neq 0$,
it must have a non-zero component  $\pi_a(x)\in\D_a$, for some $a\in \alpha$.
Assume that $\emptyset\neq\phi(x_{i_0}, \ldots, x_{i_k})^{\D_a}= \pi_a(x)$ for some $L$-formula $\phi(x_{i_0},\ldots, x_{i_k})$.  We
have $\phi(x_{i_0},\ldots, x_{i_k})^{\D_a}\in\Nr_{n}\D_a$.  Let
$f\in \phi(x_{i_0},\ldots, x_{i_k})^{\D_a}$ and let
$$b=N_a(f(0), f(1), \ldots f(n-1))\in\alpha.$$
We first show that
$b(x_0, x_1, \ldots x_{n-1})^{\D_a}\subseteq
 \phi(x_{i_0},\ldots, x_{i_k})^{\D_a}$.  Take any $g\in
b(x_0, x_1\ldots x_{n-1})^{\D_a}$,
so $N_a(g(0), g(1)\ldots g(n-1))=b$.  The map $$\set{(f(0),
g(0)), (f(1), g(1))\ldots (f(n-1), g(n-1))}$$ is
a partial isomorphism of $N_a$.  By
 (\ref{eq:theta}) this extends to a finite partial isomorphism
 $\theta$ of $N_a$ whose domain includes $f(i_0), \ldots, f(i_k)$. Let
 $g'\in U_a$ be defined by
\[ g'(i) =\left\{\begin{array}{ll}\theta(i)&\mbox{if }i\in\dom(\theta)\\
g(i)&\mbox{otherwise}\end{array}\right.\] By (\ref{eq:bf}), ${\cal N}_a,
g'\models\phi(x_{i_0}, \ldots, x_{i_k})$. Observe that
$g'(0)=\theta(0)=g(0)$ and similarly $g'(n-1)=g(n-1)$, so $g$ is identical
to $g'$ over $n$ and it differs from $g'$ on only a finite
set of coordinates.  Since $\phi(x_{i_0}, \ldots, x_{i_k})^{\c
\ D_a}\in\Nr_{n}(\C)$ we deduce ${\cal N}_a, g \models \phi(x_{i_0}, \ldots,
x_{i_k})$, so $g\in\phi(x_{i_0}, \ldots, x_{i_k})^{\D_a}$.  This
proves that $b(x_0, x_1\ldots x_{n-1})^{\D_a}\subseteq\phi(x_{i_0},\ldots,
x_{i_k})^{\D_a}=\pi_a(x)$, and so
$$\iota_a(b(x_0, x_1,\ldots x_{n-1})^{\c \ D_a})\leq
\iota_a(\phi(x_{i_0},\ldots, x_{i_k})^{\D_a})\leq x\in\c
C\setminus\set0.$$
Hence every non-zero element $x$ of $\Nr_{n}\C$
is above a non-zero element
$$\iota_a(b(x_0, x_1\ldots n_1)^{\D_a})$$
(some $a, b\in \alpha$) and these latter elements are the atoms of $\Nr_{n}\C$.  So
$\Nr_{n}\C$ is atomic and $\alpha\cong\At\Nr_{n}\C$ --- the isomorphism
is $b \mapsto (b(x_0, x_1,\dots x_{n-1})^{\D_a}:a\in A)$; denote this map by $i$.

The second part follows from the fact, that \pe\ can win the game $F^m$ for every $m\geq n$.
The second part follows work in $L_{\infty,\omega}$ so that $\C$ is complete
by changing the defining clause (now allowing infinite disjunctions)  to
$$N_a, f\models (\bigvee_{i\in I} \phi_i) \text { iff } (\exists i\in I)(N_a,  f\models\phi_i)$$

By working in $L_{\infty, \omega},$ we assume that arbitrary joins hence meets exist,
so $\C_a$ is complete, hence so is $C$. But $\Cm\At\A\subseteq \Nr_n\C$ is dense and complete, so
$\Cm\At\A=\Nr_n\C$.

For the third required there are two ways. For the first part we use all elements of $\A$ as
$n$ predicate elements and assume that the algebra is simple.
This guarantees that the map is surjective. To define the first clause, one sets
$$N_a, f\models r(x_{l_0}\ldots  x_{n-1})\Longleftrightarrow N_a(f(l_0),\ldots  f(l_{n-1}))\leq r$$
The simplicity of $\A$ guarantees that this is well defined.

The second way is to use a stronger game, in which \pa\ has even more moves and this allows is to remove $\At$
from both sides of the above equation, obtaining a much stronger result
The main play of the stronger game $J(\A)$ is a play of the game $H(\A).$

The base of the main board at a certain point will be the atomic network $X$ and we write
$X(\bar{x})$ for the atom that labels the edge $\bar{x}$ on the main board.
But \pa\ can make other moves too, which makes it harder for \pe\ to win and so a \ws\ for \pe\ will give a stronger result.
An $n$  network is a finite complete graph with nodes including $n$
with all edges labelled by elements of $\A$. No consistency properties are assumed.

\pa\ can play an arbitrary $n$ network $N$, \pe\ must replace $N(n)$  by
some element $a\in A$. The idea, is that the constraints represented by $N$ correspond to an element of the $\RCA_\omega$ being constructed on $X$,
generated by $A$. This network is placed on the side of the main board. $N$ asserts that whenever it appears in $X$
you can never have an atom not below  holding between the embedded images of $n$.
But it also asserts that whenever an atom below $a$ holds in $X$,  there are also points in $X$ witnessing all the nodes of
$N$. The final move is that \pa\ can pick a previously played $n$ network $N$ and pick any  tuple $\bar{x}$
on the main board whose atomic label is below $N(\bar{n})$.

\pe\ must respond by extending the main board from $X$ to $X'$ such that there is an embedding $\theta$ of $N$ into $X'$
 such that $\theta(0)=x_0\ldots , \theta(n-1)=y_{n-1}$ and for all $\bar{i} \in N$ we have
$X(\theta(i_0)\ldots, \theta(i_{n-1}))\leq N(\bar{i})$. This ensures that in the limit, the constraints in
$N$ really define $a$.
If \pe\ has a \ws\ in $J(A)$ then the extra moves mean that every $n$ dimensional element generated by
$A$ in the $\RCA_\omega$ constructed in the play is an element of $A$.

\end{proof}

Next  we introduce several definitions an atom structures concerning neat embeddings:

\begin{definition}
\begin{enumarab}
\item Let $1\leq k\leq \omega$. Call an atom structure $\alpha$ {\it weakly $k$ neat representable},
if the term algebra is in $\RCA_n\cap \Nr_n\CA_{n+k}$, but the complex algebra is not representable.

\item Call an atom structure {\it $m$ weakly representable} if it is weakly representable, but $\Cm\alpha\notin S\Nr_n\CA_{n+m}$.

\item Call an atom structure {\it $m$ weakly $k$ neat representable}, if it is simultaneously weakly $k$ neat and $m$ weakly representable.

\item Call an atom structure $\alpha$  {\it neat gripping} if there exists $\A$ such that $\At\A=\alpha$,
and $\A\in \Nr_n\CA_{\omega}$  and any $\B$ based on this atom structure is in $\Nr_n\CA_{n+k}$ for $k\geq 1$.

\item Call an atom structure $\alpha$ {\it  weakly neat},
 if there is $\A\notin S_c\Nr_n\CA_{n+k}$, $\At\A=\alpha$,
and there exists $\B\equiv \A$, and $\At(\B)\cong \At\Nr_n\C$, for some locally finite $\CA_{\omega}.$

\item Call it strongly neat if, satisfies the above, with the additional condition that we can omit
$\At$ from both sides
so that that is $\B\cong \Nr_n\C$, for some locally finite $\CA_{\omega}$.
\end{enumarab}
\end{definition}

We have:

\begin{enumarab}

\item  In the blow up and blur construction in \cite{ANT} a $k$ weakly neat atom structure is given.
$k$ cannot be infinite, for else the term algebra will be countable atomic and in $\Nr_n\CA_{\omega}$; this
implies, see \cite{Sayedneat}, that it is  completely
representable and so strongly representable. More basically, a complete representation of an atomic algebra induces
a representation of its completion.

\item  There exists a $n$, $n+4$  weak atom structure, this was proved in theorem \ref{smooth}.
\item $m$ weakly $k$ neat atom structures will be dealt with below, see theorem \ref{blurs}

\item There exists an atom structure that is not neat gripping, in fact there exists $\alpha$, and atomic algebras
$\A$, $\B$  having the atom structure $\alpha$, such that $\B$ is countable $\A$ is uncountable $\A\in \Nr_n\CA_{\omega}$,
but $\B\notin {\sf UpUr}\Nr_n\CA_{n+1}.$

\item In the rainbow rainbow algebra $\PEA_{\Z,\N}$,
let  \pa\ has two more  amalgamation moves (together with the usual cylindrifier move), but now the game
is played on {\it hypergraphs}, generalizing the game played by Hirsch on {\it hypernetworks}.

Recall that a hypergraph has two parts, a coloured graph,
and hyperedges of arbitrarily long lengths and these are labelled. The game aims to capture a two sorted structure,
namely, a neat reduct.
Some of the hyperedges are called short, the others are called long.
The short hyperedges are called $\lambda$ neat, and they are constantly labelled to ensure that the
atoms in the small algebra are no bigger than those in the dilation (the big algebra in which  the small algebra neatly embeds).
Games are played on $\lambda$ neat hypergraphs, that is those hypergraphs in which the short
edges are labelled by the constant label $\lambda$.

In this new game \pe\ has two more amalgamation moves to respond to.
The first cylindrifier move is as before, except that now \pe\ has to provide labels for short hyperedges and long ones.
Roughly, in the former case she has no choice but to choose the label $\lambda$ and in the second case her strategy is straightforward, naturally
dictated by \pa\ s previous move.
In the first amalgamation move \pa\ is forced a response. In the second she has to amalgamate two given hypergraphs, by labelling edges and
$n-1$ hyperedges of the graphs and giving labels to the hyperedges labelling the short ones with $\lambda$.
The last strategy is easy. The response to the remaining amalgamation moves are very similar to
cylindrifier moves; this can be done by contracting the nodes, via the notion of {\it envelope} \cite{r}, to labelling edges done exactly
like the cylindrifier move.

Let us give more details. We show how \pe\ can win the $k$ rounded games of $H$, for every finite $k$ as follows.
We have already dealt with the graph part.
We now have to extend his strategy dealing with $\Lambda$ neat hypernetworks.

In a play, \pe\ is required to play $\lambda_0$ neat hypernetworks, so she has no choice about the
hyperedges for short edges, these are labelled by $\lambda_0$. In response to a cylindrifier move by \pa\
extending the current hypergraph providing a new node $k$,
and a previously played coloured graph $M$
all long hyperedges not incident with $k$ necessarily keep the hyperlabel they had in $M$.
All long hyperedges incident with $k$ in $M$
are given unique hyperlabels not occurring as the hyperlabel of any other hyperedge in $M$.
(We can assume, without loss of generality, that we have infinite supply of hyperlabels of all finite arities so this is possible.)
In response to an amalgamation move, which involves two hypergraphs required to be amalgamated, say $(M,N)$
all long hyperedges whose range is contained in $\nodes(M)$
have hyperlabel determined by $M$, and those whose range is contained in nodes $N$ have hyperlabel determined
by $N$. If $\bar{x}$ is a long hyperedge of \pe\ s response $L$ where
$\rng(\bar{x})\nsubseteq \nodes(M)$, $\nodes(N)$ then $\bar{x}$
is given
a new hyperlabel, not used in any previously played hypernetwork and not used within $L$ as the label of any hyperedge other than $\bar{x}$.
This completes her strategy for labelling hyperedges.

Now we turn to amalgamation moves. We need some notation and terminology taken from \cite{r}; they are very useful to economize on proofs.
Every edge of any hypernetwork has an owner \pa\ or \pe\ , namely, the one who coloured this edge.
We call such edges \pa\ edges or \pe\ edges. Each long hyperedge $\bar{x}$ in a hypernetwork $N$
occurring in the play has an envelope $v_N(\bar{x})$ to be defined shortly.
In the initial round of \pa\ plays $a\in \alpha$ and \pe\ plays $N_0$
then all irreflexive edges of $N_0$ belongs to \pa\ There are no long hyperedges in $N_0$. If in a later move,
\pa\ plays the transformation move $(N,\theta)$
and \pe\ responds with $N\theta$ then owners and envelopes are inherited in the obvious way.
If \pa\ plays a cylindrifier move requiring a new node $k$ and \pe\ responds with $M$ then the owner
in $M$ of an edge not incident with $k$ is the same as it was in $N$
and the envelope in $M$ of a long hyperedge not incident with $k$ is the same as that it was in $N$.
The edges $(f,k)$ belong to \pa\ in $M$, all edges $(l,k)$
for $l\in \nodes(N)\sim \phi\sim \{k\}$ belong to \pe\ in $M$.
if $\bar{x}$ is any long hyperedge of $M$ with $k\in \rng(\bar{x}$, then $v_M(\bar{x})=\nodes(M)$.
If \pa\ plays the amalgamation move $(M,N)$ and \pe\ responds with $L$
then for $m\neq n\in \nodes(L)$ the owner in $L$ of a edge $(m,n)$ is \pa\ if it belongs to
\pa\ in either $M$ or $N$, in all other cases it belongs to \pe\ in $L$.

If $\bar{x}$ is a long hyperedge of $L$
then $v_L(\bar{x})=v_M(x)$ if $\rng(x)\subseteq \nodes(M)$, $v_L(x)=v_N(x)$ and  $v_L(x)=\nodes(M)$ otherwise.
This completes the definition of owners and envelopes.

The next claim, basically, reduces amalgamation moves to cylindrifiers moves.
By induction on the number of rounds one can show:

\begin{athm}{Claim}\label{r} Let $M, N$ occur in a play of $H_k(\alpha)$ in which \pe\ uses the above labelling
for hyperedges. Let $\bar{x}$ be a long hyperedge of $M$ and let $\bar{y}$ be a long hyperedge of $N$.
\begin{enumarab}
\item For any hyperedge $\bar{x}'$ with $\rng(\bar{x}')\subseteq v_M(\bar{x})$, if $M(\bar{x}')=M(\bar{x})$
then $\bar{x}'=\bar{x}$.
\item if $\bar{x}$ is a long hyperedge of $M$ and $\bar{y}$ is a long hyperedge of $N$, and $M(\bar{x})=N(\bar{y})$
then there is a local isomorphism $\theta: v_M(\bar{x})\to v_N(\bar{y})$ such that
$\theta(x_i)=y_i$ for all $i<|x|$.
\item For any $x\in \nodes(M)\sim v_M(\bar{x})$ and $S\subseteq v_M(\bar{x})$, if $(x,s)$ belong to \pa\ in $M$
for all $s\in S$, then $|S|\leq 2$.
\end{enumarab}
\end{athm}

Again we proceed inductive, with the inductive hypothesis exactly as before except that now each $N_r$ is a
$\lambda$ neat hypergraph. All other inductive conditions are the same (modulo this replacement). Now,
we have already dealt with hyperlabels for long and short
hyperedges, we dealt with the graph part of the first hypergraph move.
All what remains is the amalgamation move. With the above claim at hand,
this turns out an easy task to implement guided by \pe\ s
\ws\ in the graph part.

We consider an amalgamation move $(N_s,N_t)$ chosen by \pa\ in round $r+1$.
We finish off with edge labelling first. \pe\ has to choose a colour for each edge $(i,j)$
where $i\in \nodes(N_s)\sim \nodes(N_t)$ and $j\in \nodes(N_t)\sim \nodes(N_s)$.
Let $\bar{x}$ enumerate $\nodes(N_s)\cap \nodes(N_t)$
If $\bar{x}$ is short, then there are at most two nodes in the intersection
and this case is similar to the cylindrifier move, she uses $\rho_s$ for the suffixes of the red.

if not, that is if $\bar{x}$ is long in $N_s$, then by the claim
there is a partial isomorphism $\theta: v_{N_s}(\bar{x})\to v_{N_t}(\bar{x})$ fixing
$\bar{x}$. We can assume that
$$v_{N_s}(\bar{x})=\nodes(N_s)\cap \nodes (N_t)=\rng(\bar{x})=v_{N_t}(\bar{x}).$$
It remains to label the edges $(i,j)\in N_{r+1}$ where $i\in \nodes(N_s)\sim \nodes (N_t)$ and $j\in \nodes(N_t)\sim \nodes(N_s)$.
Her strategy is similar to the cylindrifier move. If $i$ and $j$ are tints of the same cone she choose a red using $\rho_s$,
If not she  chooses   a white.
She never chooses a green.
Then she lets $\rho_{r+1}=\rho_r$ and the properties $II$, $III$, $IV$ remain true
in round $r+1$.

Concerning the last property to be maintained, and that is
colouring $n-1$ types. Let $M^+=N_s\cup M_s$, which is the graph whose edges are labelled according to the rules of the game,
we need to label $n-1$ hyperedges by shades of yellow.
For each tuple $\bar{a}=a_0,\ldots a_{n-2}\in {M^+}^{n-1}$, $\bar{a}\notin N_s^{n-1}\cup M_s^{n-1}$,  with no edge
$(a_i, a_j)$ coloured green (we have already labelled edges), then  \pe\ colours $\bar{a}$ by $\y_S$, where
$$S=\{i\in \Z: \text { there is an $i$ cone in $M^*$ with base $\bar{a}$}\}.$$
As before this can be checked to be $OK$.

A \ws\ for all finite rounded game of \pa\, implies that $\At\A'\cong \At\Nr_n\C$ for some locally finite algebra $\C$
This $\C$ can be also chosen to
be complete and $\Cm\At\A\cong \Nr_n\C$; this is done by working instead in $L_{\infty, \omega}$
(rather than in first order logic as done in \cite{r}, theorem 39).
and changing the defining clause by allowing infinite disjunction.
This, in turn,  will imply  that $\C$ is complete and $\Cm\At=\Nr_n\C$
since the former is complete and dense in $\Nr_n\C$ (which is also complete).

\item Now we give a sufficient condition for a countable atom structure to be neat.
We define a game that is strictly stronger than the game $H$, played on hypergraphs,
which are coloured graphs endowed with arbitrary long labelled hyperedges)
in which \pa\ has even more moves and this allows is to remove $\At$
from both sides of the above equation, obtaining a much stronger result
The main play of the stronger game $J(\A)$ is a play of the game $H(\A).$
The base of the main board at a certain point will be the atomic network $X$ and we write
$X(\bar{x})$ for the atom that labels the edge $\bar{x}$ on the main board.
But \pa\ can make other moves too, which makes it harder for \pe\ to win and so a \ws\ for \pe\ will give a stronger result.
An $n$  network is a finite complete graph with nodes including $n$
with all edges labelled by elements of $\A$. No consistency properties are assumed.

\pa\ can play an arbitrary $n$ network $N$, \pe\ must replace $N(n)$  by
some element $a\in A$. The idea, is that the constraints represented by $N$ correspond to an element of the $\RCA_\omega$ being constructed on $X$,
generated by $A$. This network is placed on the side of the main board. $N$ asserts that whenever it appears in $X$
you can never have an atom not below  holding between the embedded images of $n$.
But it also asserts that whenever an atom below $a$ holds in $X$,  there are also points in $X$ witnessing all the nodes of
$N$. The final move is that \pa\ can pick a previously played $n$ network $N$ and pick any  tuple $\bar{x}$
on the main board whose atomic label is below $N(\bar{n})$.

\pe\ must respond by extending the main board from $X$ to $X'$ such that there is an embedding $\theta$ of $N$ into $X'$
 such that $\theta(0)=x_0\ldots , \theta(n-1)=y_{n-1}$ and for all $\bar{i} \in N$ we have
$X(\theta(i_0)\ldots, \theta(i_{n-1}))\leq N(\bar{i})$. This ensures that in the limit, the constraints in
$N$ really define $a$.
If \pe\ has a \ws\ in $J(A)$ then the extra moves mean that every $n$ dimensional element generated by
$A$ in the $\RCA_\omega$ constructed in the play is an element of $A$.
\end{enumarab}

\begin{definition} Let $\K\subseteq \CA_n$, and $\L$ be an extension of first order logic.
$\K$ is {\it detectable} in $\L$, if for any $\A\in \K$, $\A$ atomic, and for any atom structure
$\beta$ such that $\At\A\equiv_{\L}\beta$,
if $\B$ is an atomic algebra such that $\At\B=\beta$, then $\B\in \K.$
\end{definition}

The class $\Nr_n\CA_m$ is not detectable in $L_{\infty, \omega}$. Namely,
there exists $\A$ and $\B$ such that $\At\A\equiv_{\infty, \omega}\At\B$, $\A\in \Nr_n\CA_{\omega}$,
$\B\in {\sf UpUr}\Nr_n\CA_{\omega}\sim \Nr_n\CA_{n+1}.$

We have also proved that the class of completely representable atom structures is not
detectable in first order logic. The class of strongly representable atom structure is also not detectable in first order logic
\cite{HHbook2}. Actually this reference includes the last two cases; the former
is proved, like we did, using a rainbow construction, while the latter uses {\it an anti ultraproduct of Monk's algebras} based
on Erdos probabilistic graphs. (We sketched the idea in the introduction under the subtitle 'the good and the bad').
By an anti ultraproduct we mean an ultraproduct of a
sequence of weakly representable algebras based on graphs
with finite chromatic number to one based on a graph with infinite chromatic number,
which is an inverse to Monk's algebras, an ultraproduct of bad Monk like (non representable ones)
converging to a good
(representable) one.

We have already proved (or at least provided references where proofs can be found) of
the following, except for item $(7)$ which we prove below
in theorem \ref{blurs}.

\begin{corollary}
\begin{enumarab}
\item There exists $\A$ and $\B$ such that $\At\A\equiv_{\infty, \omega}\At\B$, $\A\in \Nr_n\CA_{\omega}$,
$\B\in {\sf UpUr}\Nr_n\CA_{\omega}\sim \Nr_n\CA_{n+1}$
\item There exists $\A$ and $\B$ such that $\At\A=\At\B$,
$\A\in \Nr_n\CA_{\omega}$ and $\B\notin {\sf Up Ur}\Nr_n\CA_{n+1}$.
\item There exists $\A\notin S_c\Nr_n\CA_{n+3}$, hence not completely representable,  and $\B\equiv \A$, such that $\B$ is
completely representable, $\At\B\equiv \At\Nr_n\C$,
for some $\C\in {\sf Lf}_{\omega}$, furthermore $\C$ can be chosen to be complete, in which case $\Cm\At\B\cong \Nr_n\C$.
\item If there exists a countable atom structure $\alpha$ for which \pe\ can win the $k$ rounded $J$ game for every finite $k$, then in the above we
would have the stronger $\B\equiv \Nr_n\C$, where $\A$ is any algebra between $\Tm\alpha$ and $\Cm\alpha$.
\item Let $n$ be finite $n\geq 3$. Then there exists a countable weakly $k$ neat atom structure of dimension $n$ if and only if $k<\omega.$
\item There exists a $4$ weakly representable atom structure.
\item There exists a $n+k+1$ weakly neat if there exists a relation algebra with an adequate set of blurs, having an $n$ dimensional cylindric basis.
\item There exists an atom structure of a representable atom algebra that is not neat, this works for all dimensions $>1$.
\item The class of completely representable algebras, and strongly representable
ones of dimension $>2$, is not detectable in $L_{\omega,\omega}$, while the class
$\Nr_n\CA_m$ for any ordinals $1<n<m<\omega$, is not detectable
even in $L_{\infty,\omega}.$
\item For every $\alpha>2$ (infinite included), $r\in \omega$, $k\in \omega$, $k\geq 1$,  there exists
$\B_r\in S\Nr_{\alpha}\QEA_{\alpha+k}$, such that $\prod \B^r/F\in \sf RQEA_{\alpha}$
but $\Rd_{sc}\B_r\notin  S\Nr_{\alpha}\Sc_{\alpha+k}$.
In particular $S\Nr_{\alpha}\K_{\alpha+k+1}$ is not axiomatizable by a finite schema over
$S\Nr_{\alpha}\K_{\alpha+k}$ and $\sf RK_{\alpha}$ is also
not axiomatizable by a finite schema over $S\Nr_{\alpha}\K_{\alpha+k}$ for any $\K$ between
$\Sc$ and $\QEA$, and any finite $k$.
\end{enumarab}
\end{corollary}

Lifting from atom structures, we define several classes of atomic representable algebras.
${\sf CRA_n}$ denotes the class of completely representable algebras of dimension $n$.
Let ${\sf SRCA_n}$ be the class of strongly representable atomic algebras of dimension $n$,
$\A\in {\sf SRCA_n}$ iff $\Cm\At\A\in \RCA_n$.  ${\sf WRCA_n}$ denotes the class of weakly representable algebras of dimension $n$,
and this is just $\RCA_n\cap \At$.  We  have the following strict inclusions lifting them up from atom structures \cite{HHbook2}:
$${\sf CRA}_n\subset {\sf LCA_n}\subset {\sf SRCA_n}\subset {\sf WCRA}_n$$

The second  and fourth classes are elementary but not finitely axiomatizable, bad Monk  algebras converging to a good one,
can witness this, while ${\sf SRCA_n}$ is not closed under both ultraroots and ultraproducts,
good Monks algebras converging to a bad one witnesses
this. Rainbow algebras witness that ${\sf CRA_n}$ is not elementary.

For a cylindric algebra atom structure $\F$ the first order algebra over $\F$ is the subalgebra
of $\Cm\F$ consisting of all sets of atoms that are first order
definable with parameters from $S$. ${\sf FOCA_n}$ denotes the class of atomic such algebras of dimension $n$.

Let $\K$ be the class of atomic representable algebras having $NS$, and ${\sf L}$ be the class of atomic
representable algebras having  $NS$ the unique neat embedding property; these are defined in \cite{Sayed}.
Obviously, the latter is contained in the former, and both are contained in $\Nr_n\CA_{\omega}.$
The next theorem shows that there are a plethora of very interesting classes between the
atomic algebras in the amalgamation base of ${\sf RCA_n}$ and atomic algebras in
${\sf RCA_n}.$ Some are elementary, some are not.
Some can be explicitly defined in terms of the (strength of) neat embeddings, some are not, at least its not obvious
how they can.

For a class $\K$ with a Boolean reduct, $\K\cap \At$ denotes the class of atomic algebras in $\K$; the former is elementary iff the
latter is. Call an atomic representable algebra strongly Lyndon if it atomic and \pe\ can win
the $\lambda$ neat hypernetwork game $H_k$ for every $k$, and let ${\sf SLCA_n}$ denote the elementary closure of this  class.
\begin{theorem}
We have the following inclusions (note that $\At$ commutes with ${\sf UpUr})$:
$${\sf L}\cap \At\subseteq \K\cap \At\subseteq \Nr_n\CA_{\omega}\cap \At\subset {\sf UpUr}\Nr_n\CA_{\omega}\cap \At$$
$$\subseteq {\sf Up Ur}S_c\Nr_n\CA_{\omega}\cap \At={\sf UpUr}{\sf CRA_n}$$
$$\subseteq {\sf SLCA_n} \subseteq {\sf LCA}_n\subset {\sf SRCA_n}\subseteq {\sf UpUr}{\sf SRCA}_n\subseteq {\sf FOCA_n}$$
$$\subseteq S\Nr_n\CA_{\omega}\cap \At={\sf WRCA_n}= \RCA_n\cap \At.$$
\end{theorem}
\begin{proof}
Items (5) and (7) uses ideas of Ian Hodkinson and Robin Hirsch.
\begin{enumarab}
\item  The first inclusion is witnessed by an atomic algebra that lies in the amalgamation base of $\RCA_n$,
but not in the super amalgamation base of  $\RCA_n$. The second is witnessed by an atomic
algebra in $\Nr_n\CA_{\omega}$, that is not in the strong amalgamation base of $\RCA_n$.

\item The third inclusion is witnessed by the algebra $\B$ constructed in \cite{IGPL}. $\B$ is also completely representable, hence strongly
representable,
so it witnesses the strictness of fourth inclusion {\it without} the elementary closure operator (with it we do not know whether the inclusion is strict).

\item The fifth  inclusion is witnessed by $\PEA_{\Z, \N}$.

\item The sixth; we do not know whether  the game $H_k$ coded by $\rho_k$ is strictly harder than that coded by
$\sigma_k$, as far as \pa\ is concerned, but it probably is due to the hyperedges part.
A \ws\ for \pe\ in $H_k$ for every $k\geq n$, forces the algebra to be in
${\sf UpUr}\Nr_n\CA_{\omega}$ while a \ws\ for $\sigma_k$ for every $k\geq n$, forces the algebra to
be in ${\sf Up Ur }S_c\Nr_n\CA_{\omega}$. We know that they are unrelated both ways.
This is not telling much.

\item We now provide a concrete example of algebra that there is a strongly representable algebra that fails infinitely many Lyndon conditions.
(We know that one exists because ${\sf LCA_n}$ is elementary, by definition,  and it is contained in ${\sf SRCA_n}$ which is not elementary.)
Any atomic algebra satisfying the Lyndon conditions, will have their complex algebra of its atom structure also
satisfying the Lyndon conditions, hence will be strongly representable, \ref{Robinsexample}.
\item Let $\A$ be the $\bold G$ Monk algebra constructed in theorem \ref{hodkinson} or the algebra based
or the rainbow term construction obtained by blowing up and blurring a finite rainbow algebra, proving that
$S\Nr_{n}\CA_{n+4}$ is not atom canonical.
Recall that such algebras were defined using first order formulas, the first in a Monk's signature, the second in the rainbow signature
(the latter is first order since we had only
finitely many greens).  Though the usual semantics was altered, first order logic did not see the relativization, only infinitary formulas saw it,
and thats why the complex algebras could not be represented.
Another  distinction worthwhile highlighting here, is that the first algebra is based on an infinite graph
with finite chromatic number (disjoint cliques), and that is why the complex algebra is not
representable. The second rainbow algebra is based on a complete infinite irreflexive graph,
the graph of reds. Then $\A\in {\sf FOCA_n}$ but not in ${\sf SRCA_n}.$

\item We now show that ${\sf FOCA_n}\subset  {\sf WCA_n}$. This is not at all obvious because they are both elementary.
Take an $\omega$ copy of the an $3$ element graph with nodes $\{1,2,3\}$ and edges
$1\to 2\to\ 3$. Then of course $\chi(\Gamma)<\infty$. Now an $\Gamma$  has a three first order definable colouring.
Since $\M(\Gamma)$ is not representable, then the algebra of first order definable sets is not representable because $\Gamma$ is interpretable in
$\rho(\Gamma)$, the atom structure constructed from $\Gamma$ as defined in \cite{HHbook}.
However, the term algebra is representable. (This is not so easy to prove).

\end{enumarab}
\end{proof}
\begin{theorem}
Let $n\geq 3$ be finite. Then ${\sf LCA}_n={\sf Up Ur CRA_n}$. In particular, ${\sf LCA_n}={\sf Up Ur}S_c\Nr_n\CA_{\omega}\cap \At$.
\end{theorem}
\begin{proof}
Any algebra satisfying the Lyndon conditions has an ultrapower in which
\pe\ has a winning strategy in the $\omega$ rounded game, which in turn has a countable elementary subalgebra $\B$ in which \pe\
wins the $\omega$ rounded
game. Hence $\B$ is completely representable, and $\A\in {\sf Up Ur}\{\B\}$.
\end{proof}
For an atomic algebra $\A$, by an atomic subalgebra we mean a subalgebra of $\A$ containing all its atoms, we write
$S_{at}$ to denote this operation applied to an algebra or to a
class of algebras.

\begin{theorem} Assume that for every atomic representable algebra that is not strongly representable, there exists a graph $\Gamma$ with finite
chromatic number such that $\A\subseteq \M(\Gamma)$ and $\At\A=\rho(\M(\Gamma))$.
Assume also that for every graph $\Gamma$ with $\chi(\Gamma)<\infty$, there exists
$\Gamma_i$ with  $i\in \omega$, such that $\prod_{i\in F}\Gamma_i=\Gamma$, for some non principal ultrafilter $F$.
Then $S_{at}{\sf Up}{\sf SRCA}_n={\sf WRCA_n}.$
\end{theorem}
\begin{proof}
Assume that $\A$ is atomic, representable but not strongly representable. Let $\Gamma$ be a graph with $\chi(\Gamma)<\infty$ such that
$\A\subseteq \M(\Gamma)$ and $\At\A=\rho(\M(\Gamma))$.
Let $\Gamma_i$ be a sequence of graphs each with infinite chromatic number converging to $\Gamma$, that is, their ultraproduct
is $\Gamma$
Let $\A_i=\M(\Gamma_i)$. Then $\A_i\in {\sf SRSA_n}$, and we have:
$$\prod_{i\in \omega}\M(\Gamma_i)=\M(\prod_{i\in \omega} \Gamma_i)=\M(\Gamma).$$
And so $\A\subseteq_{at} \prod_{i\in \omega}\A_i$, and we are done.
\end{proof}

\begin{theorem}\label{r} Let $3\leq n<\omega$. Then the following hold:
\begin{enumroman}
\item Any $\K$ such that $S_c\Nr_n\CA_{\omega}\subseteq \K\subseteq S_c\Nr_n\CA_{n+3}$ is not elementary. 
\item Any $\K$ such that $S_c\Ra\CA_{\omega}\subseteq \K\subseteq S_c\Ra\CA_5$ is not elementary
\item If $n\geq 2$ (infinite included) and $\beta>n$, $\Nr_n\CA_{\beta}$ is not elementary
\item For $n\geq 2$, the inclusions $\Nr_n\CA_{k}\subseteq S_c\Nr_n\CA_{k}$ is proper
\item If there is a $5$ witness, then the inclusion $\Ra\CA_k\subseteq S_c\Ra\CA_k$ is proper for all $k\geq 5$.
\end{enumroman}
\end{theorem}
\begin{proof}
\begin{enumarab}
\item From our main theorem
\item The main result in \cite{r}
\item The main result in \cite{IGPL}
\item This inclusion is witnessed by $\B$ in \ref{ef}, since $\B$ is a complete subalgebra of $\A$,
hence is in $S_c\Nr_n\CA_{\omega},$ hence in $S_c\Nr_n\CA_k$ for any finite $k>n$, but it is not
in $\Nr_n\CA_{n+1}$.
\item Witnessed by $\B$ as in \ref{witness} modulo a witness.

\end{enumarab}
\end{proof}

\section{Metalogical application: Omitting types, yet again,  in clique guarded semantics}

One more algebraic result, then we formulate and prove our metalogical results, concerning
two new  omitting types theorems for $L_n$, first order logic restricted to the first $n$ variables, when $n>2$ is
finite. The reader is referred to \cite{ANT} for the definition of $n$ complex blurs for a relation algebra.

\begin{definition}  Let $\R$ be a finite relation algebra with non identity atoms $I$.
Let $U, V, W$ be subsets of $I$. We say that $(U,V,W)$ is safe, in symbols,
${\sf safe}(U,V,W)$ if $a\leq b;c$, for al $a\in U$, $b\in V$, and $c\in W$. Otherwise, the triple is not safe.
\end{definition}

\begin{theorem}\label{blurs}
Let $k\geq 1$ and $m\geq 3$. Let $n=m+k$.
Assume that there exists a finite relation algebra $\R$ with finitely many non identity atoms $I$ with $|I|\geq n^n$,
and let $l\leq 2n-2$, such that the set $J$ of $l$
blurs satisfy the following two properties
\begin{enumarab}
\item $(\forall P\in I)(\forall W\in J)(I\subseteq P;W)$
\item $(\forall P_2,\ldots P_n, q_2,\ldots Q_n\in I)(\forall W\in J)W\cap P_2;Q_n\cap \ldots P_n;Q_n\neq \emptyset$.
\end{enumarab}
Assume that $\R\notin S\Ra\CA_{n+1}$.
Then  there exists an atomic countable polyadic $\A\in \sf \RPEA_m\cap \Nr_m\PEA_{n}$
such that $\Rd_{ca}\Cm\At\A\notin S\Nr_m\CA_{n+1}$.
\end{theorem}

\begin{proof}
For the first part let $n=m+k$.
Then we have, by assumption, that $\R\notin S\Ra\CA_{m+k+1}$ \cite{HHbook}.
We proceed like above but more technically. We will blow up and blur $\R$.  Let $I$ be the set of  non identity atoms, then  $|I|\geq 6$.
Let $l\geq 2$ be finite and $J$ be the set of all $l$ element subsets of $I$.
That is $W\in J$ iff $|W|=l$ and $W\subseteq I$.
We assume that $l\geq 2n-1$,  and that  $|\At\A(n)|\geq {}^nn\geq (2n-1)l$.
Let
$$H=\{a_i^{P,W}: i\in \omega, P\in I, W\in J, P\in W\}.$$
For $P\in I$, and $W\in J$, set
$$H^P=\{a_i^{P,W}: i\in \omega, W\in J, P\in W\}.$$
$$E^W=\{a_i^{P,W}: i\in \omega, P\in W\};$$the two partitions.
Let
$$R=\{X: X\cap E^W \text { is finite or cofinite}, \forall W\in J\}.$$
For $i,j,k\in \omega$ $E(i,j,k)$ abbreviates that $i,j,k$ are {\it evenly distributed}, i.e.
$$E(i,j,k)\text { iff } (\exists p,q,r)\{p,q,r\}=\{i,j,k\}, r-q=q-p$$
We sometimes denote this ternary relation by $E$.
We define a relation atom structure on ${\cal F}=H\cup \{Id\}$.
All atoms are self-converse. We define the consistent triples as follows
(Involving identity are as usual $(a, b, Id): a\neq b).$
Let $i,j,k\in \omega$, $P,Q,R\in I$ and $S,Z,W\in J$ such that
$P\in S$, $Q\in Z$ and $R\in W$. Then the triple
$(a_i^{P,S},a_j^{Q,Z}, a_k^{R,W})$ is consistent iff
either
\begin{enumroman}
\item ${\sf safe}(S,Z,W)$
or
\item $E(i,j,k)\&P\leq Q;R.$
\end{enumroman}
Now we have
\begin{enumarab}
\item  $\Cm{\cal F}$ is not in $S\Ra\CA_{m+k+1}$.
\item ${\cal R}$ the term algebra over $\cal F$, has universe $R$, and is representable.
\end{enumarab}
It uses the first partition of $H$ showing
that is $\R$ embeds into $\Cm\cal F$ and we are done because $\R\notin S\Ra\CA_{m+k+1}$, and the latter is a variety, hence
closed under forming subalgebras.

The representability of the term algebra uses coloured graphs whose labels are the et $\Uf$ of all
ultrafilters of $\cal R$  and these actually provided a complete representation of $\Uf^+$ the canonical extension of  $\cal R$.

We have represented the term algebra  using all the non principal ultrafilters as labels,
and indeed we have completely represented its ultrafilter extension whose
set of atom is $\Uf$.

By $l\geq 2n-1$ and $|\At\R|\geq 2(n-1)l$, we have enough $n$ blurs that is,
$(J,E)$, or simple $J$, because we have only one $E$,  is a complex blur for $\R$ \cite{ANT}, in the following sense:
\begin{enumarab}
\item Each element of $J$ is non empty
\item $\bigcup J=I$
\item $(\forall P\in I)(\forall W\in J)(I\subseteq P;W)$
\item $(\forall V_1,\ldots V_n, w_2,\ldots W_n\in J)(\exists T\in J)(\forall 2\leq i\leq n)
{\sf safe}(V_i,W_i,T)$
\item $(\forall P_2,\ldots P_n, q_2,\ldots Q_n\in I)(\forall W\in J)W\cap P_2;Q_n\cap \ldots P_n;Q_n\neq \emptyset$.
\end{enumarab}

(1) and (2) are easy. (3) and (5) are given. We check (4).
We start by $(4)$. Let $V_1,\ldots V_n, W_2,\ldots W_n\in J$.
Then $U=\bigcup \{V_i\cup W_i: 2\leq i\leq n\}$ has cardinality at
most $(2n-2)l$, hence the cardinality of $I-U$ is $\geq k-(2n-2)l$, and ${\sf safe}(V_i,W_i,T)$
because $V_i\cap W_i\cap T=\emptyset$ for all $2\leq i\leq n$.

Now we check the final condition. Let $P_2,\ldots P_n, Q_2,\ldots Q_n\in I$
be arbitrary, then $U=\{P_2,\ldots Q_n\}$ has cardinality $\leq 2n-2$, and
so each $W\in J$ contains an $S\notin U$, by $l\geq 2n-1$, and $S\leq P_2;Q_2\ldots P_n;Q_n$ by the
definition of composition.

Now by these properties we can infer exactly like \cite{ANT}, that
$\cal R$ has a cylindric basis of dimension $n$; that is also symmetric.
Let $\B_n$ be the set of basic matrices of our blown up and blurred ${\cal R}$, and consider it as a polyadic equality algebra atom structure
defining accessibility relations corresponding to transpositions the obvious way.

In the first order language $L$ of $(\omega, <)$, which has quantifier elimination,
we define diagrams for each $K\subseteq n$ and $\phi\in L$ as in \cite{ANT}, via maps  $\bold e:K\times K\to {\cal R}$.
For an atom let $v(a)$ be its $i$ th co-ordinate.
The pair $\bold e$ and $\phi$ defines an element in the polyadic complex algebra $\Cm\B_n$,
called a diagram, that is a set of matrices, defined by
$$M(\bold e, \phi)=\{m\in B_n, i,j\in K, m_{ij}\leq \phi(\bold e_{ij},v(m_{ij}))\}.$$
A normal diagram is one whose entries are either atoms or finitely many blurs, that is elements of the form $E^W$,
in addition to the condition that $\phi$ implies $\phi_E$, where the latter is the first order formula defining
$e$. Any diagram can be approximated by normal ones; and actually it is a finite union of normal diagrams \cite{ANT}.
The  polyadic equality algebra, denoted by $\Bb_n(\R, J, E)$, consists of those diagrams,
is representable, and have for $t<n$
$$\Nr_t \Bb_n(\R, J, e)\cong \Bb_t(\R, J, e),$$
via $\Psi$ say.
In other words, $\cal R$ has an $n$ dimensional hyperbasis, and the above algebra contains the term algebra.
So we have
$\Tm{\sf Mat}_{n} {\cal R}\in \RCA_{n}$ and $\Tm{\sf Mat}_m{\cal R}\cong \Nr_m\Tm{\sf Mat}_{m+k}{\cal R}$, via the restriction of $\Psi$.
Finally, we have  $\Cm{\sf Mat_m{\cal R}}$ is not in
$S\Nr_m\PEA_{m+k+1}$, for if it were then ${\cal R}$ embeds into $\Ra\Cm{\sf Mat_m{\cal R}}$ which
will imply that  ${\cal R}\in S\Ra\CA_{m+k+1}$, and this is not possible.

By varying the relation $E$ like in \cite{ANT}, making it depending on the blurs,
this complex algebra, witnessing that $S\Nr_n\PEA_{m+k+1}$ is
not closed under completions,
can be chosen to be generated by a single element.

\end{proof}
\begin{definition}

We refer the reader to \cite{HHbook} for the notion of a relativized representation.
$M$ is a relativized representation of an abstract algebra $\A$, if there exists $V\subseteq {}^nM$, and a an injective
homomorphism $f:\A\to \wp(V)$. We write $M\models 1(\bar{s})$, if $\bar{s}\in V$.

\begin{enumarab}
\item  Let $M$ be a relativized representation of a $\CA_m$.
A clique in $M$ is a subset $C$ of $M$ such that $M\models 1(\bar{s})$ for all $\bar{s}\in {}^mC$.
For $n>m$, let $C^{n}(M)=\{\bar{a}\in {}^nM: \text { $\rng(\bar{a})$ is a clique in $M$}\}.$
\item Let $\A\in \CA_m$, and $M$ be a relativized representation of $\A$. $M$ is said to be $n$ square, $n>m$,
if whenever $\bar{s}\in C^n(M)$, $a\in A$, and $M\models {\sf c}_ia(\bar{s})$,
then there is a $t\in C^n(M)$ with $\bar{t}\equiv _i \bar{s}$,
and $M\models a(\bar{t})$.
\item  $M$ is infinitary $n$ flat if  for all $\phi\in L(A)_{\infty, \omega}^n$, for all $\bar{a}\in C^n(M)$, for all $i,j<n$, we have
$$M\models \exists x_i\exists x_j\phi\longleftarrow \exists x_j\exists x_i\phi(\bar{a}).$$
$M$ is just $n$ flat, if the above holds for first order formulas using $n$ variables,
namely, for $\phi\in L(A)^n$.
\item $M$ is said to be $n$ smooth if it is $n$ square, and there is are equivalence relations $E^t$, for $t=1,\ldots, n$
on $C^{t}(M)$ such that  $\Theta=\{(\bar{x}\to \bar{y}): (\bar{x}, \bar{y})\in \bigcup_{1\leq t\leq n} E^t\}$
(here   $\bar{x}\to \bar{y}$ is the map $x_i\mapsto y_i$)
is an $n$ back and for system of partial isomorphisms on $M$.
\end{enumarab}
\end{definition}

Let $3\leq m<n$. Let $\sf CRA_{m,n}$ be the class of algebras that have $n$ smooth complete representations.
It can be proved that for $m < k<l$, both $>m$, that  $\sf CRA_{m,l}\neq \sf CRA_{m,k}$,
and that they are all non elementary. The latter can be partially obtained from our first algebraic result in one go,
by observing that for $n\geq m+3$, $\sf CRA_{n,m}\subseteq S_c\Nr_n\CA_m$.
(This is proved exactly like the $\Ra$ case, \cite{HHbook}).

\begin{corollary} The classes $\sf CRA_{m,n}$ for $n\geq m+3$ are not elementary
\end{corollary}
\begin{proof} Since $S_c\Nr_m\CA_{\omega}\subseteq {\sf CRA_{m,n}}\subseteq S_c\Nr_m\CA_{m+3}$.
\end{proof}

We now give what we think is an interesting metalogical consequence of our algebraic result concerning omitting types.
The omitting types theorem has been proved
to fail for finite variable fragments, in a quite strong sense \cite{ANT}, when we consider classical semantics.

Here we show that it also fails for $L_n$, first order logic restricted to $n$ variables when $n>2$,
when we consider relativized semantics, so it fails in also a strong sense but in another way. For a countable
atomic \cite{ANT} theory, there may not be an $n+3$ relativized flat atomic model.
In other words, we prove the results in the abstract.

\begin{corollary} The omitting types theorem fails in even in the clique guarded semantics of first order logic.
In more detail, there is a countable theory $T$
and non principal $\Gamma$, such that no $n+3$ relativized flat model omits it.
\end{corollary}
\begin{proof} Let $\A=\sf PEA_{\Z,\N}$. Then $\A$ is atomic. Assume that $\A=\Fm_T$. Let $\Gamma$ be the set of co-atoms, then
$\prod \Gamma=0$. Assume, for contradiction,  that $T$ has an $n+3$ model $\M$ omitting $\Gamma$.
Let $L(A)$ denotes the signature that contains
an $n$ ary predicate for every $a\in A$. For $\phi\in L(A)_{\omega,\infty}^n$,
let $\phi^{\M}=\{\bar{a}\in C^{n+3}(\M): \M\models \phi(\bar{a})\}$, and let $\D$ be the algebra with universe $\{\phi^{\M}: \phi\in L\}$ with usual
Boolean operations, cylindrifiers and diagonal elements, cf. theorem 13.20 in \cite{HHbook}. The polyadic operations are defined
by swapping variables.
This is a $\sf PEA_n$; here semantics is defined as expected in the clique guarded fragment of first order logic.
Define $\D_0$ be the algebra consisting of those $\phi^{\M}$ where $\phi$ comes from $L$.
Then $\D_0$ is also a $\sf PEA_n$ and $\A$ embeds into the $n$ neat reduct of both \cite{HHbook}.
If $\M$ is complete, then the embedding is also complete, but this is impossible for it will imply that $\A\in S_c\Nr_n\PEA_{n+3}$.
\end{proof}

The classical Orey-Henkin omitting types theorem, or rather the contrapositive thereof, says that if
$\Gamma$ is a type that is realized in every model of a
countable theory $T$, then it is necessarily isolated by a formula.
We call this isolating formula an {\it $m$ witness},
if $m$ is the number of variables occurring in this formula.

\begin{theorem}
There is a countable theory $T$ a type realized in every $n+k+1$ smooth model, but there is no $n+k$ witness.
\end{theorem}
\begin{proof}  We use theorem \ref{blurs}. Let $\Gamma$ be the set of atoms. Then we claim
that $\Gamma$ is realized in all $n+k+1$
models. Towards proving this claim consider a model $\M$ of $T$. If  $\Gamma$ is not realized in $\M$,
then this gives an $n+k+1$ complete representation of $\A=\Fm_T$,
which is impossible,  because $\A$ is not in $S\Nr_n\CA_{n+k+1}$.

Assume that $\phi\in L_{k+n}$, we show that $\phi$ cannot be a $n+k$ witness.
Let $\A\in \RCA_n\cap \Nr_n\CA_{n+k}$ be the above algebra. Then $\A$ is simple, and so we can assume
without loss of generality, that it is set algebra with a countable
base. Let $\M=(M,R)$  be the corresponding model to this set algebra in the sense of \cite{HMT2} sec 4.3.
Then $\M\models T$ and $\phi^{\M}\in \A$.
But $T\models \exists x\phi$, hence $\phi^{\M}\neq 0,$
from which it follows that  $\phi^{\M}$ must intersect an atom $\alpha\in \A$ (recall that the latter is atomic).
Let $\psi$ be the formula, such that $\psi^{\M}=\alpha$. Then it cannot
be the case that
that $T\models \phi\to \neg \psi$,
hence $\phi$ is not a $k$ witness,
and we are done.
\end{proof}

The notion of relativized representations is important in algebraic logic, as indicated in the introduction..
$M$ is a relativized representation of an abstract algebra $\A$, if there exists $V\subseteq {}^nM$, and an injective
homomorphism $f:\A\to \wp(V)$. We write $M\models 1(\bar{s})$, if $\bar{s}\in V$.

${\sf CRA_{m,\kappa}}$, denotes the class of algebras having
$\kappa$ {\it complete} smooth relativized representations, where a $\kappa$ representation of an atomic
$\A$  is complete if it is atomic, namely, when $\bigcup_{x\in \At\A}f(x)=V,$ where $f$ is the representation.
Our next theorem says that the notion of $\omega$
complete representation is a notion strictly
weaker than complete representation, but this can be only witnessed on uncountable algebras.

\begin{theorem}\label{longer} Regardless of cardinalities, $\A\in {\sf CRA}_{m, \omega}$
iff \pe\ has a \ws\ in $G_{\omega}^{\omega}$.
\end{theorem}
\begin{proof} One side is obvious. Now assume that \pe\ has a \ws\ in the $\omega$ rounded game, using $\omega$ many pebbles.
We need to build an $\omega$ relativized complete representation.
The proof goes as follows. First the atomic networks are finite, so we need to convert them into $\omega$ dimensional
atomic networks. For a network $N$, and  a map $v:\omega\to N$,
let $Nv$ be the network  induced by $v$, that is $Nv(\bar{s})=N(v\circ \bar{s})$.
let $J$ be the set of all such $Nv$, where $N$ occurs in some play
of $G_{\omega}^{\omega}(\A)$ in which \pe\ uses his \ws\ and $v:\omega\to N$ (so via these maps we are climbing up $\omega$).

This  can be checked to be  an $\omega$ dimensional hyperbasis (extended to the cylindric case the obvious way).
So $\A\in S_c\Nr_n\CA_{\omega}$. We can use that the basis consists of $\omega$ dimensional atomic networks, such that for
each such network, there is a finite bound on the size of its strict networks.
Then a complete $\omega$ relativized representation  can be obtained in a step by step way, requiring inductively
in step $t$, there for any finite clique $C$ of $M_t$, $|C|<\omega$, there is
a network in the base, and an embedding $v:N\to M_t$ such that $\rng v\subseteq C$.
Here we consider finite sequences of arbitrarily large length, rather than fixed length $n$ tuples.
This  is because an $\omega$ relativized representation only requires cylindrifier witnesses over
finite sized cliques, not necessarily cliques that are uniformly bounded.
\end{proof}

\begin{theorem}\label{rel}${\sf CRA_m}\subset {\sf CRA}_{m,\omega}$, the strict inclusion can be only witnessed on
uncountable algebras. Furthermore, ${\sf CRA}_{m,\omega}$ is not elementary.
The classes $S_c\Nr_m\CA_{\omega}$, ${\sf CRA}_m$, and ${\sf CRA}_{m, \omega}$,
coincide on atomic countable
algebras. Uncountable algebras can be characterized by transfinite games \cite{HHbook2} theorem 3.3.3.
\end{theorem}
\begin{proof}That ${\sf CRAS}_{m,\omega}$ is not elementary is witnessed by the rainbow algebra
$\PEA_{K_{\omega}, K}$, where the latter
is a disjoint union of $\K_n$, $n\in \omega$. \pe\ has a \ws for all finite length games, but \pa\ can win the infinite rounded game.
Hence \pe\ can win the transfinite game on an uncountable non-trivial ultrapower of $\A$,
and using elementary chains one can find an elementary countable
subalgebra $\B$ of this ultrapower such that \pe\ has a \ws\ in the $\omega$ rounded game.
This $\B$ will have an $\omega$ square representation
hence will be in $\CRA_{m,\omega}$, and $\A$ is not in the latter class.

Now consider $\A=\PEA_{\omega_1, \omega}$. Then clearly \pe\ has a \ws\ in $G_{\omega}^{\omega}$, the usual atomic game
with $\omega$ rounds and $\omega$ nodes, cf. \cite{HHbook2},
and so $\A\in {\sf \CRA}_{m,\omega}$.  Assume for contradiction that  $M$ is a complete representation of $\Rd_{df}\A$.
Then in $\A$ there will be an uncountable red clique forced by the tints, that is the greens,
$\{\g_0^i: i<\omega_1\}$ of cones inducing the same order on a face,
and any edge within the clique is labelled by $\r_{ij}$ for some distinct
$i,j$ so the red indices must match which is impossible, because the reds are countable and the greens are uncountable.
Notice that this proof can show
that ${\sf CRA}_{m, \lambda}\subset {\sf CRA}_{m,\kappa}$ by $\CA_{\kappa, \lambda}$. (We will show in a while that
$\CA_{\omega_1, \omega}\in \Nr_n\CA_{\omega}$.)
The last statement follows from \cite{HHbook2}.
It is is known that the completely representable algebras coincide with atomic algebras in $S_c\Nr_n\CA_{\omega}$ on atomic algebras.
But is also clear that for a countable atomic algebra an $\omega$ complete representation
is equivalent to an ordinary complete representation because \pe\ can schedule all possible moves
by \pa\ , and the $\omega$ complete representations guides her to win the
game.
\end{proof}

\begin{theorem}  If $\A\in \CA_m$ is finite, and has an $n$ square relativized representation, then it has a finite $n$ square relativized
representation. If $n\geq m+3$,
this is not true for $n$ smooth relativized representations. Same for ${\sf PEA_m}$.
\end{theorem}
\begin{proof} The first follows from the fact that square representations can be coded in the clique guarded fragment of first order logic,
and indeed in the loosely guarded fragment of first order logic which
has the finite base property.
The second follows from the fact that  the problem of deciding whether a finite $\CA_m$ is in
$S\Nr_m\CA_n$, when $n\geq m+3$, is undecidable,
from which one can conclude that there are finite algebras in $S\Nr_m\CA_n$
that do not have a finite $n$ dimensional hyperbasis, and
these cannot possibly have finite representations.
\end{proof}

For any cardinal $\kappa$, $K_{\kappa}$ will denote the complete irreflexive graph with $\kappa$ nodes.
Let $p<\omega$, and $I$ a linearly irreflexive ordered set, viewed as model to a signature containg a binary relation $<$.
$M[p,I]$ is the disjoint union of $I$ and the complete graph $K_p$ with $p$ nodes.
$<$ is interpreted in this structure as follows $<^{I}\cup <{}^{K_p}\cup I\times K_p)\cup (K_p\times I)$
where the order on $K_p$ is the edge relation. Here dimensions are denoted by $m$ and we reserve $n$ for nodes.

In our next remark, rainbows and Monk like algebras can do the same thing again.
\begin{remark}

\item One can show that ${\sf CRA_{m,n}}$ is not elementary
by using $A=M[n-4, \Z]$ and $B=M[n-4, \N]$. \pe\ has a \ws\ for all finite rounded graph  (with $n$ nodes) on $\A=\PEA_{A,B}$,
hence she has a \ws\ in the $\omega$ rounded game
on the ultrapower, from which an elementary countable subalgebra can be extracted in which \pe\
still has a \ws\ and this algebra has an $n$ smooth relativized representation.

But it can also be shown that \pa\ can  the $\omega$ rounded game, also with $n$ nodes,
hence $\A$ does not have an $n$ complete relativized representation,
but it is elementary equivalent to one that does.

\item Alternatively one can prove non elementarity by using Monk-like algebras.
Let $\A(n)$ be an infinite atomic atomic relation algebra;
the atoms of $\A(n)$ are $Id$ and $a^k(i,j)$ for each $i<n-2$, $j<\omega$ $k<\omega_1$.
Then for any $r\geq 1$, $\A(n-1,r)$ embeds completely in $\A(n-1)$ the obvious way,
hence, the latter has no $n$ dimensional hyperbasis.
One then proves that $\A(n)$ has an $m$ dimensional hyperbasis for each $m<n-1$, by proving that \pe\ has
\ws\ in the hyperbasis game $G_r^{m,n}(\A(n), \omega)$, for any $r<\omega$, that is, for any finite rounded game.

Let $\C=\Ca(H_m^n(\A(n),\omega)$. Write $T$ for the game $G_r^{m,n}$
and consider $\M=\M(\A(n), m, n, \omega, \C)$ as a
$5$ sorted structure with sorts $\A(n), \omega, H_m^n(\A, \omega)$ and $\C$.
Then one can show that \pe\ has a \ws\ in $G(T\upharpoonright 1+2r, \M)$.

\pe\ has a \ws\ in $G(T|r, \M)$ for all finite $r>0$. So \pe\ has a \ws\ in $G(T, \prod_D \M)$.
Hence there is a countable elementary subalgebra $\E$ of $\prod_D\M$ such
that \pe\ has a \ws\ in $G(T,\E)$ Hence $\E$ has the form $\M(\B, m, n, \Lambda, \A)$  for some atomic $\B\in \RA$ countable set
$\Lambda$ and countable atomic $m$ dimensional $\A$ such
that $\At\A\cong \At\Ca(H_m^n(\B, \Lambda))$.
Furthermore, we have  $\B\prec \prod _D\A(n)$  and $\A\prec \prod_D\C$.
Thus \pe\ has a \ws\ in $G(T, \M, m, n, \Lambda \C)$
and she also has a \ws\ in $G_{\omega}^{m,n}(\B,\Lambda)$.

So $\B\in S\Ra\CA_n$ and $\A$ embeds into $\Ca(H_m^n(\B, \Lambda)\in \Nr_m\CA_n$. For
the latter let, let $H=H_{n-1}^n(\B, r)$ , then $H$ is a wide $n-1$
$\omega$ hyperbasis and $\Ca(H)\in \CA_n$, and
$\C=\Ca(H_m^n(\A, \omega)=\Ca(H|m^{n})\cong \Nr_m\Ca H$.
But $\A$ embeds into $\Ca(H_m^n(\B, \Lambda)$
and we are done.
In fact, one can show that both $\B$ and $\D$ are actually representable, by finding a representation of
$\prod \A(n)/F$ (or an elementary countable subalgebra of it it) embedding every
$m$ hypernetwork in such a way so that
the embedding respects $\equiv_i$ for every $i<m$, but we do not need that much.
\end{remark}

We end the paper with the following two corollaries. It is true for $k\geq 4$,
and conditionally true for $k=1,2,3$,  given that certain finite relation algebras (depending on $k$)
exist.
More precisely:
\begin{corollary} Let $k\geq 1$. Let $\K\in \{ S\Nr_n{\sf L}_{n+k}\}$ with $\sf L\in\{\Sc, \CA, \PA, \PEA\}$.
Assume the existence of a finite relation algebra $\R$ as in the hypothesis of theorem \ref{blurs}.
Then the following hold:
\begin{enumarab}
\item $\K$ is not atom canonical; even more there exists a countable atomic $\A\PEA_n$, such that $\Rd_{df}\Cm\At\A$ is not representable
and  $\Rd_{Sc}\A\notin  S\Nr_n\Sc_{n+k}.$
\item $\K$ is not closed under \d\ completions. There exists $\A\in \PEA_n$ such that the $\Sc$ reduct of its completion is not in $S\Nr_n\Sc_{n+k}$
\item There exists an algebra not in $\K$ with a dense representable algebra.
\item There exists an atomic countable representable a
algebra with no $n+k$ smooth flat representation.
\end{enumarab}
\end{corollary}
\begin{proof} We use the term cylindric algebra, proved to exist conditionally, namely,
$\A$ in \ref{blurs}.
\begin{enumarab}
\item We have $\Cm\At\A\notin S\Nr_n\CA_{n+k+1}$,
while $\A\in \bigcap_{k\in \omega}S\Nr_n\CA_{n+k}$ since it is representable.
\item $\Cm\At\A$ is the \d\ completion of $\A$ (even in the $\PA$ and $\Sc$ cases, because $\A$,
hence its $\Sc$ and $\PA$ reducts are completely additive)
\item $\A$ is dense in $\Cm\At\A$.
\item  Assume that $\A$ has an $n$ smooth complete representation $\M$. $L(A)$ denotes the signature that contains
an $n$ ary predicate for every $a\in A$.
For $\phi\in L(A)_{\omega,\infty}^n$,
let $\phi^{M}=\{\bar{a}\in C^n(M):M\models_C \phi(\bar{a})\}$,
and let $\D$ be the algebra with universe $\{\phi^{M}: \phi\in L_{\infty,\omega}^n\}$ with usual
Boolean operations, cylindrifiers and diagonal elements, cf. theorem 13.20 in \cite{HHbook}. The polyadic operations are defined
by swapping variables.
Define $\D_0$ be the algebra consisting of those $\phi^{M}$ where $\phi$ comes from $L^n$.
Assume that $M$ is $n$ square, then certainly $\D_0$ is a subalgebra of the $\sf Crs_n$
with domain $\wp(C^n(M))$ so $\D_0\in {\sf Crs_n}$. The unit $C^n(M)$ of $\D_0$ is symmetric,
closed under substitutions, so
$\D_0\in \sf G_n$.  Since $M$ is $n,$ flat we
have that cylindrifiers commute by definition,
hence $\D_0\in \CA_n$.

Now since $M$ is infinitary $n+k$ smooth then it is infinitary $n+k$ flat.
Then one prove that  $\D\in \CA_n$ in exactly the same way.
Clearly $\D$ is complete. We claim that $\D$ is atomic.
Let $\phi^M$ be a non zero element.
Choose $\bar{a}\in \phi^M$, and consider the infinitary conjunction
$\tau=\bigwedge \{\psi\in L_{\infty}: M\models_C \psi(\bar{a})\}.$
Then $\tau\in L_{\infty}$, and $\tau^{M}$ is an atom, as required

Now defined the neat embedding by $\theta(r)=r(\bar{x})^{M}$.
Preservation of operations is straightforward.  We show that $\theta$ is injective.
Let $r\in A$ be non-zero. But $M$ is a relativized representation, so there $\bar{a}\in M$
with $r(\bar{a})$ hence $\bar{a}$ is a clique in $M$,
and so $M\models r(\bar{x})(\bar{a})$, and $\bar{a}\in \theta(r)$. proving the required.

We check that it is a complete embedding under the assumption that
$M$ is a complete relativized representation.
Recall that $\A$ is atomic. Let $\phi\in L_{\infty}$ be such that $\phi^M\neq 0$.
Let $\bar{a}\in \phi^M$. Since
$M$ is complete and $\bar{a}\in C^n(M$) there is $\alpha\in \At\A$, such
that $M\models \alpha(\bar{a})$, then $\theta(\alpha).\phi^C\neq 0.$
and we are done.
Now $\A\in S_c\Nr_n\CA_{n+k}$; it embeds completely into $\Nr_n\D$, $\D$ is complete, then
so is $\Nr_n\D$, and consequently $\Cm\At\A\subseteq \Nr_n\D$, which is impossible, because we know that
$\Cm\At\A\notin S\Nr_n\CA_{n+k}.$
\end{enumarab}
\end{proof}

The following (relativized completeness theorem) is proved in
\cite{t}.
\begin{corollary} Let $\A\in \sf PEA_m$ be countable and atomic and assume that $m<n$. Then the following conditions are
equivalent.
\begin{enumarab}
\item $\A$ has an $n$ dimensional hyperbasis.
\item $\A$ has an infinitary $n$ flat complete representation.
\item $\A\in S_c\Nr_m(\PEA_n\cap \At)$
\item \pe has a \ws\ in $F_{\omega}^n$.
\item \pe\ has a \ws\ in $G_{\omega}^n$
\item $\A$ has $n$ smooth relativized complete representation.
\end{enumarab}
\end{corollary}

\begin{theorem}\label{first} For $n\geq m+3$,
the class $S_c\Nr_m(\CA_n\cap \At)$, hence all of the above are not elementary
\end{theorem}

\end{document}